\documentclass[11pt]{amsart}
\usepackage[T1]{fontenc}
\usepackage{fourier}
\usepackage[mathscr]{euscript}
\usepackage{amsmath,amsfonts,amssymb}
\usepackage{mathtools}
\usepackage[margin=1.8cm]{geometry}
\usepackage[english]{babel}
\usepackage{textcomp}
\usepackage{fancyhdr}
\usepackage{eso-pic}
\usepackage{bbold}
\usepackage{color}
\usepackage[colorinlistoftodos]{todonotes}
\usepackage{xcolor}
\usepackage{textgreek}
\usepackage{comment}
\usepackage{stmaryrd}
\usepackage{esint}
\usepackage{cite}
\usepackage[shortlabels]{enumitem}
\newtheorem{theorem}{Theorem}[section]
\newtheorem{lemma}[theorem]{Lemma}
\newtheorem{definition}[theorem]{Definition}

\theoremstyle{remark}
\newtheorem{remark}[theorem]{Remark}
\setlist[itemize]{label=$\diamond$}
\numberwithin{equation}{section}

\let\vph=\varphi
\let\eps=\varepsilon
\DeclarePairedDelimiter\absb\lvert\rvert
\DeclarePairedDelimiter\normb\lVert\rVert
\DeclarePairedDelimiter\scalarb\langle\rangle
\newcommand\scalar{\scalarb*}
\newcommand\norm{\normb*}
\newcommand\abs{\absb*}
\DeclareMathOperator\dvg{div}

\newcommand\R{\mathbb{R}}
\newcommand\N{\mathbb{N}}
\newcommand\D{\mathbb{D}}

\newcommand\dpt{\partial_t}

\newcommand\cC{\mathscr{C}}
\newcommand\cF{\mathscr{F}}
\newcommand\loc{\text{loc}}
\newcommand\pw{\text{pw}}
\definecolor{mycolor}{HTML}{D35400}

\definecolor{hide}{HTML}{A0AC81}
\definecolor{myred}{HTML}{8A1538}

\usepackage[colorlinks=true,urlcolor=cyan,citecolor=blue]{hyperref}
\usepackage{cleveref}
\crefformat{rema}{#2 Remark~#1#3}{}
\crefformat{lemma}{#2Lemma~#1#3}{}
\crefformat{theo}{#2Theorem~#1#3}{}
\crefformat{appendix}{#2Appendix~#1#3}{}
\crefformat{section}{#2Section~#1#3}{}

\begin{document}

\title[{D}iscontinuous {S}olutions for the {T}wo-{F}luid {M}odel]{{G}lobal-in-time {D}iscontinuous {S}olutions for the {T}wo-{P}hase {M}odel of {C}ompressible {F}luids with {D}ensity-{D}ependent {V}iscosity}

\author{Marcel Zodji}
\address{23 Boulevard de France, 91037 Évry-Courcouronnes, France}
\curraddr{}
\email{marcel.zodji@univ-evry.fr}
\thanks{}



\date{\today}

\dedicatory{}


\begin{abstract}
We are concerned with a model describing the motion of two compressible, immiscible fluids with density-dependent viscosity in the whole $\mathbb R^3$. The phases of the flow may have different pressure and
viscosity laws and are separated by a sharp interface, across which the (total) density is discontinuous. Our goal is to study the persistence of the regularity of this sharp interface over time.  More precisely,  
the dynamics of the flow are governed by three coupled equations: two hyperbolic equations (for the
volume fraction of one phase and for the density) and a parabolic equation for the velocity field.
We assume that, at the initial time, the density is $\alpha$-H\"older continuous on both sides of
a $\mathscr C^{1+\alpha}$-regular surface across which it may be discontinuous. We prove the existence
and uniqueness of a global-in-time weak solution in an intermediate regularity class that ensures the persistence
of the piecewise H\"older regularity of the density and the $\mathscr C^{1+\alpha}$ regularity of
the sharp interface.
\end{abstract}

\maketitle

\tableofcontents

\section{Introduction}
This paper deals with the following system of PDEs describing the motion of two compressible viscous fluids with density-dependent viscosity in the three-dimensional space $\R^3$:
\begin{equation}\label{intro:twofluid}
    \begin{cases}
        \dpt c+u\cdot \nabla c=0,\\
        \dpt \rho +\dvg (\rho u)=0,\\
        \dpt (\rho u)+\dvg (\rho u\otimes u)+\nabla P(\rho,c)=\dvg \big(2\mu(\rho,c)\D u\big)+\nabla \big(\lambda(\rho,c)\dvg u\big).
    \end{cases}
\end{equation}
Above, $c = c(t,x) \in [0,1]$ denotes the volume fraction, $\rho = \rho(t,x) \in [0,\infty)$ is the total density, and $ u = u(t,x) \in \R^3$ represents the single velocity field of the fluid mixture. These are the unknowns of the problem. Meanwhile, $P= P(\rho,c)$, $\mu=\mu(\rho,c)$, and $\lambda=\lambda(\rho,c)$ are respectively the pressure, shear viscosity, and bulk viscosity laws of the mixture. They are given $W_\loc^{2,\infty}$-regular functions of $\rho$ and $c$. The derivation of system \eqref{intro:twofluid} from the classical two-fluid model can be found in \cite{zodji2023well}. Specifically, this system corresponds to the single-velocity Baer-Nunziato model, where $\rho_+=\rho c$ and $\rho_-=\rho (1 - c)$ represent the densities of the fluid phases. 

For initially immiscible flows, each phase occupies its own domain, and they are separated by a sharp interface in such a way that we have ${\rho_+}_{|t=0}{\rho_-}_{|t=0} = 0$, or equivalently (for positive total density) $c_0(1 - c_0) = 0$. In order words, the volume fraction $c_0$ is the characteristic function of a  domain $D_0$. Based on DiPerna–Lions theory (see e.g \cite{DiPernaLions,vasseur2019global}), any weak solution of \eqref{intro:twofluid} in the classical energy framework, with $\dvg u\in L^1_\loc ([0,\infty), L^\infty(\R^3))$, preserves the patch structure of the volume fraction (and hence the immiscibility of the phases):
\begin{equation}\label{intropatchfrac}
c(1-c)=0,
\end{equation}
so that $c(t)$ remains the characteristic function of a  domain $D(t)$. In general, \(D(t)\) corresponds to the image 
of the initial domain \(D_0\) under the flow map associated with the fluid velocity.

In the case where both phases have the same constant viscosity coefficients, the existence of weak solutions in the classical energy framework has been the subject of several works since 2019 (see, for instance, \cite[Vasseur, Wen and Yu]{vasseur2019global}, \cite[Kra\u{c}mar, Kwon, Ne\u{c}esov\'a and Novotn\'y]{kracmar2022weak}, \cite[Novotn\'y]{novotny2020weak}, \cite[Novotn\'y and Pokorn\'y]{novotny2020weak1}, \cite[Wen]{wen2021global}, and others). The velocity field is relatively weak 
in the sense that it is only known to satisfy 
\begin{equation}\label{intro:enerregul}
\sqrt{\rho} u\in L^\infty((0,\infty), L^2(\R^3)), \nabla u \in L^2((0,\infty)\times \R^3).
\end{equation}
This regularity is not sufficient to define a flow map for the velocity field, and thus to describe the dynamics of the sharp interface that separates the fluid phases. Conversely, classical solutions (see, for instance, \cite{Burtea_pressurerelaxation,WangChanxinZhangglobalregultwophase}) are too strong, in the sense that they do not allow for discontinuities in the densities or in the volume fraction.

The main purpose of this paper is to construct global-in-time weak solutions of \eqref{intro:twofluid} in an intermediate regularity class which allow for the immiscibility of the fluid phases and  for the persistence of  the regularity of the sharp interface over time. 
In Section \ref{mainresult} below, we state our main results, which are then compared to existing results in Section \ref{review}.
\subsection{Main Results}\label{mainresult} 
This section is devoted to the statement of the main results.  To simplify the presentation, we first specify the assumptions on the initial data and on the pressure and viscosity laws.

\noindent
\paragraph{\textbf{Assumptions on the pressure and viscosity laws}}
Given that the density of the phase labeled $"+"$ (resp. $"-"$) is $\rho_+=\rho c$ (resp. $\rho_-=\rho(1-c)$), if $P_\pm$ denotes the pressure law of the phase $"\pm"$, then the pressure $P=P(\rho,c)$ in \eqref{intro:twofluid} is defined as follows:
 \begin{gather}\label{pressureG}
     P(\rho,c)= P_+(\rho c)+ P_-(\rho(1-c)).
 \end{gather}
A similar definition holds for the shear and bulk viscosity laws.

We always assume that the shear viscosity is strictly positive, the bulk viscosity is non-negative, and the pressure  is an increasing function of the density: for all   $\rho>0$, $c\in \{0,1\}$
\begin{gather}\label{intronovacuum}
\mu(\rho,c)>0,\quad \lambda(\rho,c)\geqslant 0,\quad \partial_\rho P(\rho,c)>0,
\end{gather}
Following our computations, we believe that the non-negativity assumption on the bulk viscosity can be waived.

We identify two particular cases where the algebraic structure of the model allows for proving that the pressure and shear viscosity jumps are stable: either
\begin{itemize}
    \item the shear viscosity depends only on the volume fraction, $\mu(\rho,c) = \mu(c)$ and there exists a $W^{2,\infty}_\loc$-regular  function $\Psi_P$ such that
    \begin{gather} \label{sec1:jumpcondition}
    P(\rho,c)= \Psi_P\big(f(\rho,c)\big),
\end{gather}
where the $(\rho,c)$-dependent function $f$ satisfies 
\begin{gather}\label{sec1:eq1}
\rho \partial_\rho f(\rho,c)=\nu(\rho,c), \;\text{ and where }\; \nu(\rho,c)= 2\mu(\rho,c)+\lambda(\rho,c)
\end{gather}
denotes the total viscosity. Differentiating \eqref{sec1:jumpcondition} with respect to $\rho$, we obtain 
\begin{gather}\label{psiderivative}
\Psi'_P\big(f(\rho,c)\big)=\frac{\rho \partial_\rho P(\rho,c)}{\nu(\rho,c)}.
\end{gather}
\item or the shear viscosity of one phase depends on the density  (namely $\partial_\rho\mu(\rho,c) \not\equiv0$) and there exist $W^{2,\infty}_\loc$-regular functions $\Psi_P$,  and  $\Psi_\mu$ such that 
\begin{gather}\label{sec1:jumpconditionbis}
P(\rho,c)= \Psi_P\big(f(\rho,c)\big),\text{ and } \mu(\rho,c)=\Psi_\mu \big(f(\rho,c)\big).
\end{gather}
\end{itemize}
In fact, in the latter situation, the jump of the pressure and shear viscosity at the sharp interface decay 
exponentially in time. If, in addition to \eqref{sec1:jumpconditionbis}, there exists a $W^{2,\infty}_\loc$-regular function $\Psi_\lambda$ such that 
\begin{gather}\label{sec1:exponentialdecay}
    \lambda(\rho,c)=\Psi_\lambda \big(f(\rho,c)\big),
\end{gather}
then the jump of the viscosity bulk viscosity and velocity gradient decay exponentially in time as well. 
Assumptions \eqref{sec1:jumpconditionbis}-\eqref{sec1:exponentialdecay} hold true, for example, when $\mu$, $\lambda$, and $P$ are independent of the volume fraction, which corresponds to the Navier–Stokes model.  Moreover, for power constitutive laws
\[
P_\pm (\rho)= b_\pm \rho^{\gamma_\pm},\quad \nu_\pm(\rho)=  a_\pm \rho^{\alpha_\pm}, \quad \text{with}\quad \gamma_\pm\in [1,\infty), \;\text{ and }\;  b_\pm, a_\pm, \alpha_\pm\in (0,\infty), 
\]
the assumption \eqref{sec1:jumpcondition} holds  if 
\[
\frac{\gamma_+}{\alpha_+}=\frac{\gamma_-}{\alpha_-}=\check{\gamma},\; \text{ and }\; b_+ \Bigg(\frac{\alpha_+}{a_+}\Bigg)^{\check{\gamma}}=b_- \Bigg(\frac{\alpha_-}{a_-}\Bigg)^{\check{\gamma}}=\check{b}, \;\text{ and we have }\; \Psi_P(\varrho)=\check{b} \varrho^{\check{\gamma}}.
\]
\noindent 
\paragraph{\textbf{Assumptions on the initial data}}
We consider the system \eqref{intro:twofluid}, supplemented with initial conditions 
\begin{gather}\label{intro:initialdata}
    (c,\rho,\rho u)_{|t=0}=(c_0,\rho_0,\rho_0u_0)
\end{gather}
satisfying the following assumptions: there exists 
\begin{itemize}
    \item a simply connected, bounded and open domain $D_0\subset \R^3$ of class  $\cC^{1+\alpha}$, such that
    \begin{equation}\label{initialvolum}
        c_0=\mathbb{1}_{D_0^c}.
    \end{equation}
    \item a constant pressure equilibrium $\widetilde P>0$ such that
    \begin{equation}\label{init:density-velocity}
        P(\rho_0,c_0)-\widetilde P\in L^2(\R^3),\, \text{ and }\, u_0\in H^1(\R^3),\;\text{ } \rho_0\in L^\infty(\R^3),\;\text{ } \inf_{x\in \R^3} \rho_0(x)>0.
    \end{equation}
    \item some $p_0\in (1, 6/5)$, such that  
    \begin{equation}\label{init:lowfrequency}
  P(\rho_0,c_0)-\widetilde P, \rho_0 u_0\in L^{p_0}(\R^3).
\end{equation}
This (low-frequency) assumption allows us to derive time-decay estimates for the solutions.
Finally, we consider a shear-viscosity equilibrium state, $\widetilde\mu$, such that:
\[
\mu(\rho_0,c_0)-\widetilde\mu\in L^2(\R^3).
\]
\end{itemize}

What we mean by the $\mathcal{C}^{1+\alpha}$ regularity of $D_0$ is the following.
\begin{definition}\label{def:interface} Let $D$ be an open, bounded and simply connected subset of $\R^3$, with boundary $\Sigma=\partial D$. 

We say that $D$ (or $\Sigma$) is of class $\cC^{1+\alpha}$ if there exist a finite number of open sets  $B_j\subset \R^2$, $j=1,\dots J$, and  maps $\tau_j \colon B_j \longrightarrow \Sigma\subset\R^3$ such that 
\begin{itemize}
    \item $\tau_j$ are $\mathcal{C}^{1+\alpha}$-regular, and are homeomorphisms onto their images; and $\Sigma = \bigcup_{j=1}^{J} \tau_j(B_j)$.

    \item for all $j, k$ such that $\tau_j(B_j) \cap \tau_k(B_k) \neq \emptyset$, the map $\tau_j^{-1} \circ \tau_k$ is a $\mathcal{C}^{1+\alpha}$-diffeomorphism from\\
$\tau_k^{-1}\big( \tau_j(B_j) \cap \tau_k(B_k) \big)$ onto $\tau_j^{-1}\big( \tau_j(B_j) \cap \tau_k(B_k) \big)$.
\end{itemize}
We introduce the following notations to quantify the regularity of $\Sigma$ (or $\tau_j$ if preferred):
\begin{equation}
    \norm{\Sigma}_{\text{Lip}}=\sup_j\norm{\nabla \tau_j}_{L^\infty}, \;\big|\Sigma\big|_{\dot \cC^{1+\alpha}}=\sup_j\abs{\nabla \tau_j}_{\dot \cC^\alpha},\;\text{ and }\;\; \big|\Sigma\big|_{\text{inf}}=\inf_j\inf_{\substack{s\neq s'\\
    s,s\in B_j}}\Bigg(\frac{\abs{\tau_j(s)-\tau_j(s')}}{\abs{s-s'}}\Bigg)\in (0,\infty).
\end{equation}
Furthermore, there exists\footnote{Following \cite{kiselev2016finite}, we construct $\varphi$  as solution to $-\Delta \varphi=g$ in $D$, and  $\varphi=0$ on $\Sigma$, where $0< g\in \cC^\infty(D)$. Standard estimates for elliptic equations (see, for instance, \cite{gilbarg1977elliptic}) imply that $0 < \varphi \in \mathcal{C}^{1+\alpha}(D)$. The fact that $\abs{\nabla \varphi}_{\text{inf}}>0$ is a consequence of Hopf’s lemma at boundary points, which holds for $\mathcal{C}^{1+\alpha}$-regular manifolds (see \cite{finn1957asymptotic}).} a $\mathcal{C}^{1+\alpha}$-regular function $\varphi \colon \R^3 \to \R$ such that
\[
\varphi=0 \text{ on } \Sigma,\;\;D=\big\{x\in \R^3\colon \vph(x)>0\big\}, \;\text{ and }\;  \abs{\nabla \vph }_{\text{inf}}=\inf_{x\in \Sigma}\abs{\nabla\varphi(x)}>0.
\]
Since $\varphi \circ \tau_j = 0$ on $B_j$, the outward unit normal vector field to $D$ is given by
\begin{gather}\label{sec2:eq16}
n_x(\tau_j(s)) = \frac{\nabla \varphi(\tau_j(s))}{\abs{\nabla \varphi(\tau_j(s))}}.
\end{gather}

Recall that the relevant situations for the immiscible two-phase model is when the total density $\rho=\rho_++\rho_-$ is discontinuous. Since the densities $\rho_\pm$ of the fluid phases are different and $\rho_+ \rho_-=0$, the class of piecewise regular functions (for the total density $\rho$) is suitable for the study of \eqref{intro:twofluid}.

    In this paper, we will work with the space $\cC^\alpha_{\pw,\Sigma}(\R^3)$ of discontinuous functions on $\R^3$, which are $\alpha$-H\"older continuous on either side of $\Sigma$ (recall that\;\;$\R^3= D \cup \Sigma \cup (\R^3\setminus \overline{D})$):
    \[
    \cC^\alpha_{\pw,\Sigma}(\R^3):= \Big\{g\in L^\infty(\R^3)\colon g\in \dot \cC^{\alpha}_{\pw,\Sigma}(\R^3)\Big\}, \;\text{ where }\;\dot \cC^{\alpha}_{\pw,\Sigma}(\R^3):= \dot \cC^{\alpha}(\overline{D})\cap \dot \cC^{\alpha}(\R^3\setminus D).
    \]
    The semi-norm (resp. norm) associated with $\dot \cC^\alpha_{\pw,\Sigma}(\R^3)$ (resp. $\cC^\alpha_{\pw,\Sigma}(\R^3)$) is defined as follows:
    \[
    \abs{g}_{\dot \cC^\alpha_{\pw,\Sigma}(\R^3)}= \abs{g}_{\dot \cC^\alpha(\overline{D})} + \abs{g}_{\dot \cC^\alpha(\R^3 \setminus D)} \quad\big(\text{resp.} \quad \norm{g}_{\cC^\alpha_{\pw,\Sigma}(\R^3)} = \norm{g}_{L^\infty(\R^3)} + \abs{g}_{\dot \cC^\alpha_{\pw,\Sigma}(\R^3)}\big).
    \]
Given a function $g \in \cC^{\alpha}_{\pw,\Sigma}(\R^3)$, we define its traces $g^+$ and $g^-$ on $\Sigma$, along with its average $\scalar{g}_{\text{avg}}$ and its jump $\llbracket g\rrbracket$ across $\Sigma$, as follows: for all $\sigma\in \Sigma$,
    \begin{gather}
        g^+(\sigma)=\lim_{r\to 0^+} g\big(\sigma+rn_x(\sigma)\big), \;\, g^-(\sigma)=\lim_{r\to 0^+} g\big(\sigma-rn_x(\sigma)\big),\;\,\scalar{g}_{\text{avg}}=\frac{g^++g^-}{2},\quad \llbracket g\rrbracket= g^+-g^-.
    \end{gather}
\end{definition}
In addition to the assumptions \eqref{initialvolum}-\eqref{init:density-velocity}-\eqref{init:lowfrequency}, we further suppose that the total density is H\"older continuous on both sides of $\Sigma_0 = \partial D_0$, across which it may be discontinuous:
\begin{equation}\label{density:piecewise}
    \rho_0\in \dot\cC^\alpha_{\pw,\Sigma_0}(\R^3), \;\text{ where }\; \alpha\in (0,1).
\end{equation}
We denote by $\Pi$ the Cauchy's tensor 
\[
\Pi =2\mu(\rho,c)\D u+\big(\lambda(\rho,c)\dvg u- P(\rho,c)+\widetilde P\big) I_3.
\]
When $\alpha\in [1/2,1)$, we additionally assume that there exists $\check{\alpha}\in (0,3/2-\alpha)$ such that ($\dot v=\dpt v+ (u\cdot\nabla) v$)
\begin{gather}\label{sec:intro:eq1}
(\rho \dot u)_{|t=0}=\dvg(\Pi)_{|t=0}\in \dot H^{-\check{\alpha}}(\R^3).
\end{gather}
If $\alpha\in (0,1/2)$, we will instead take $\check{\alpha}=1$.

Our result reads as follows:
\begin{theorem}\label{th1}
    Assume the above assumptions hold for the constitutive coefficients and the initial data. There exist constants 
    $\epsilon,\epsilon'>0$ such that if 
    \begin{gather}\label{intro:eq2}
    \normb{P(\rho_0,c_0)-\widetilde P}_{L^2\cap L^{p_0}(\R^3)}^2+\normb{\rho_0  u_0}_{L^2\cap L^{p_0}(\R^3)}^2+\normb{\nabla u_0}_{L^2(\R^3)}^2\leqslant \epsilon,
    \end{gather}
    and 
    \begin{itemize}
        \item either  
        \[
        \normb{\mu(\rho_0,c_0)-\widetilde\mu}_{L^\infty(\R^3)}\leqslant \epsilon',
        \]
        \item or  \eqref{sec1:jumpcondition}  holds and 
        \begin{gather}\label{intro:eq3}
        \normb{\llbracket \mu(\rho_0,c_0),\; P(\rho_0,c_0)\rrbracket}_{L^\infty(\partial D_0)} \leqslant \epsilon',
        \end{gather}
    \end{itemize}
    then there exists a unique global-in-time weak solution $(c,\rho, u)$ to the Cauchy problem \eqref{intro:twofluid}-\eqref{intro:initialdata} verifying: 
    \begin{gather*}
        P(\rho,c)-\widetilde P,\; u,\; \nabla u \in L^\infty((0,\infty), L^2(\R^3));\\
        \dot u,\;\sigma^{\frac{\check{\alpha}}{2}}\nabla \dot u  \in L^2((0,\infty)\times \R^3),\,\text{ and }\, \sigma^{\frac{\check\alpha}{2}} \dot u,\; \sigma^{\frac{1+\check{\alpha}}{2}}\nabla\dot u\in L^\infty((0,\infty), L^2(\R^3)); \\
        \sigma^{\frac{1+\check{\alpha}}{2}}\ddot u \in L^2((0,\infty)\times \R^3);\;\text{ where }\;\sigma(t)=\min \{1,t\}.
    \end{gather*}
    Additionally, we have 
    \[
    \nabla u\in L^1((0,\infty), L^\infty(\R^3)),
    \]
    and hence the velocity admits a well-defined flow map. This flow map transports the initial domain $D_0$ to a $\cC^{1+\alpha}$-regular domain $D=D(t)$ and 
    \[
    \int_0^\infty \normb{\nabla u(t')}_{\dot \cC^{\alpha}_{\pw,\partial D(t')}(\R^3)}dt' <\infty,\;\;\text{ and, \;for all }\;\; t\in (0,\infty),\;\,  \rho(t)\in \cC^\alpha_{\pw,\partial D(t)}(\R^3).
    \]
    Furthermore, if \eqref{sec1:jumpconditionbis}-\eqref{sec1:exponentialdecay} hold, then the jumps of the pressure, and the velocity gradient across  $\partial D(t)$ decay exponentially to zero in time. 
\end{theorem}

The proof of Theorem \ref{th1} is presented in Section \ref{sec:theorem},  based on the a priori
estimates in Section \ref{part2} whose proofs is given in Section \ref{lemmasProof}.

\begin{remark} The solution possesses additional regularity and decay rates described in Section \ref{part2}. 
\begin{itemize}
    \item In particular, the classical energy controls the volume of the moving domain $D(t)$ as follows:
    \begin{gather}
        \forall \, t\in (0,\infty),\;\; \abs{D(t)}\leqslant  \abs{D_0}+(\widetilde P)^{-1}\int_{\R^3} \rho_0(x)\left(\dfrac{\abs{u_0(x)}^2}{2}+\int_{\widetilde\rho_+\mathbb{1}_{D_0}(x)}^{\rho_0(x)} s^{-2}\left(P(s,c_0(x))-c_0(x) \widetilde P\right)ds\right)dx.
    \end{gather}
    This estimate is by no means predictable. Above, $\widetilde\rho_+$ denotes the constant equilibrium state of the density of the fluid in $\R^3\setminus D(t)$; it satisfies 
    \[
    P(\widetilde\rho_+,1)=\widetilde P. 
    \]
    Likewise, define $\widetilde\rho_-$ by
    \[
    P(\widetilde \rho_-,0)=\widetilde P.
    \]
    For brevity, we  will  write $\widetilde \rho(c)$ in place of $\widetilde \rho_\pm$ such that 
    \[
    P(\widetilde \rho(c),c)=\widetilde P. 
    \]
    The smallness conditions on the pressure fluctuation \eqref{intro:eq2}–\eqref{intro:eq3} mean that the 
    "+"-fluid density is close to $\widetilde\rho_+$, while the  "-"-fluid density is close to $\widetilde\rho_-$.  
    
    \item The low-frequency assumption (see \eqref{init:lowfrequency}) helps to derive sufficient time-decay estimates for the solution, which in turn imply the following uniform-in-time $\cC^{1+\alpha}$-regularity of $D$:
    \[
    \sup_{t>0}\Big(\norm{\partial D(t)}_{\text{Lip}}+\abs{\partial D(t)}_{\dot \cC^{1+\alpha}}+\abs{\partial D(t)}_{\text{inf}}^{-1}\Big)<\infty.
    \]
    \item Finally, note that assumption \eqref{sec:intro:eq1} implies 
    \[
    \nabla u_0\in L^{6/(1+2\check{\alpha})}(\R^3). 
    \]
    In general, this is the maximal integrability implied by \eqref{sec:intro:eq1}. 
\end{itemize}
    
\end{remark}
\begin{remark} A cornerstone of the proof of Theorem \ref{th1} is to derive a Lipschitz bound for the velocity. This is the minimal regularity needed to propagate the sharp interface’s regularity. To achieve this, we express the velocity 
gradient as (see \eqref{sec2:eq3}): 
\begin{align}
    \nabla u
            &=-\frac{1}{\mu(\rho,c)}\nabla(-\Delta)^{-1}\mathbb P (\rho \dot u)+\nabla (-\Delta)^{-1}\nabla\bigg(\frac{1}{\nu(\rho,c)}(-\Delta)^{-1}\dvg (\rho \dot u)\bigg)-\nabla (-\Delta)^{-1}\nabla \Bigg(\frac{P(\rho,c)-\widetilde P }{\nu(\rho,c)}\Bigg)\notag\\
            &+\frac{1}{\mu(\rho,c)}\big[\mathcal{K}, \mu(\rho,c)-\widetilde \mu\big](2\D u)-\nabla (-\Delta)^{-1}\nabla\bigg(\frac{1}{\nu(\rho,c)}\big[\check{\mathcal{K}},\mu(\rho,c)-\widetilde \mu\big](2\D u)\bigg),\label{intro:eq4}
\end{align}
where $\mathcal{K}$ and $\check{\mathcal{K}}$ are combinations of second-order Riesz operators.
    \begin{itemize}
        \item  As a first step, we perform energy estimates (\textit{\`a} la Hoff) for $\dot u$, which yield 
         $L^p(\R^3)$ estimate  for $\dot u$, with $p\in (3,\infty)$. These energy estimates involve $L^q(\R^3)$ norm of the velocity gradient and the pressure fluctuation (with $q\in \{3,4,6\}$). To derive these bounds, we use the fact that Riesz operators  are continuous on $L^q(\R^3)$. In particular, the norms of the first two terms of \eqref{intro:eq4} are obtained via an interpolation inequality, whereas the third is bounded by
        \[
        \norm{\nabla (-\Delta)^{-1}\nabla \Bigg(\frac{P(\rho,c)-\widetilde P}{\nu(\rho,c)}\Bigg)}_{L^q(\R^3)}\leqslant \frac{C}{\displaystyle \inf_{\R^3} \mu(\rho,c)} \norm{P(\rho,c)-\widetilde P}_{L^q(\R^3)}.
        \]
        Next, we observe that 
        \[
        \dfrac{d}{dt}\int_{\R^3} H_{q-1}(\rho,c)dx+\frac{1}{q}\int_{\R^3}\abs{P(\rho,c)-\widetilde P}^{q}dx\leqslant \frac{1}{q}\int_{\R^3}\abs{F}^{q}dx,
        \]
        where 
        \[
        \rho \partial_\rho H_{q-1}(\rho,c)- H_{q-1}(\rho,c)= \abs{P(\rho,c)-\widetilde P}^{q-2} (P(\rho,c)-\widetilde P), 
        \]
        and 
        \[
        F=-(-\Delta)^{-1}\dvg (\rho \dot u) +[\check{\mathcal{K}}, \mu(\rho,c)-\widetilde\mu](2\D u). 
        \]
        As for the first two terms of \eqref{intro:eq4}, we use an interpolation inequality to obtain an estimate for the $L^q(\R^3)$-norm of the first term of $F$. Next, we prove that the $L^q(\R^3)$-norm of the last term of $F$, as well as the last terms in \eqref{intro:eq4} are small compared to the $L^q$-norm of $\nabla u$, 
        provided the jump of  the shear viscosity $\mu(\rho,c)$ is small in $L^\infty(\partial D)$. In total, 
        one can close the energy estimate for $\dot u$, assuming the shear viscosity is piecewise 
         H\"older regular, its jump across $\partial D$ is small; and the density is upper bounded and far away from vacuum. Note that we use time weights, which makes the computations more technical than the heuristic above. 

         Let us notice that for constant shear viscosity coefficient, the $L^p(\R^3)$-bound for $\dot u$ 
         immediately yields the propagation of the $L^\infty$-bound of the density (see \cite{Hoff95a}). However, this seems not to hold in the case of density-dependent shear viscosity (see \cite{zodji2023discontinuous}), or in the case of general viscous tensor (see \cite{bresch2023extension,bocchi2022anisotropy}). Indeed, the pressure and the velocity gradient appear to be of the same order, and propagating the $L^\infty(\R^3)$-bound of the density requires a Lipschitz velocity.

        \item Given the integrability of $\dot u$ obtained above, the first term in \eqref{intro:eq4} 
        belongs to $L^\infty(\R^3)$. In fact, modulo the multiplicative factor $1/\mu(\rho,c)$, this term is H\"older continuous on the whole space.  The other terms are even-order Riesz transform of discontinuous functions, and  it is well-known that even-order Riesz operators are  not continuous on $L^\infty(\R^3)$. Recent results (see \cite{gancedo2022quantitative,gancedo2023global}) prove that these operators are actually continuous on $\cC^\alpha_{\pw,\partial D}(\R^3)\cap L^p(\R^3)$, with $p\in (1,\infty)$, with a norm depending on $\norm{\partial D(t)}_{\text{Lip}}$, $\abs{\partial D(t)}_{\text{inf}}$ and $\abs{\partial D(t)}_{\dot \cC^{1+\alpha}}$.  These quantities are known to grow exponentially with the time integral of the Lipschitz norm of the velocity: for instance
        \[
        \norm{\partial D(t)}_{\text{Lip}} \leqslant \norm{\partial D_0}_{\text{Lip}} e^{\int_0^t \norm{\nabla u}_{L^\infty(\R^3)}},\quad \abs{\partial D(t)}_{\text{inf}} \geqslant \abs{\partial D_0}_{\text{inf}} e^{-\int_0^t \norm{\nabla u}_{L^\infty(\R^3)}}. 
        \]
     Thus, a non-uniform-in-time estimate for 
     \begin{gather}\label{intro:eq5}
     \int_0^t \norm{\nabla u}_{L^\infty(\R^3)}
     \end{gather}
     leads to an exponential growth in time. For the one-phase model (Navier–Stokes system for compressible fluid), this exponential growth is compensated by the exponential-in-time decay of the pressure and velocity jumps. 
     For the two-phase model, we are able to prove the exponential-in-time decay for the pressure and velocity gradient jumps under the assumptions \eqref{sec1:jumpconditionbis}-\eqref{sec1:exponentialdecay}. This decay does not hold in general, as we will discuss in the following remark. Our strategy is to use the low-frequency assumption \eqref{init:lowfrequency} to derive uniform-in-time estimate for \eqref{intro:eq5}.  

     Leveraging the continuity of Riesz operators on $\cC^\alpha_{\pw,\partial D}(\R^3)\cap L^p(\R^3)$, we prove that 
     the $L^\infty(\R^3)$-bound  (resp. $\dot\cC^{\alpha}_{\pw,\partial D}(\R^3)$-bound) of the last two terms of \eqref{intro:eq4} is small compared to the $L^\infty(\R^3)$-bound (resp. $\dot\cC^{\alpha}_{\pw,\partial D}(\R^3)$-bound) of $\nabla u$; provided that the viscosity jump is small in $L^\infty(\partial D)$. We apply similar arguments to derive piecewise H\"older estimate for $F$, from which the propagation of  the density's 
     piecewise H\"older regularity   follows, since we have (from the mass equation): 
     \[
     \dpt f(\rho,c)+ u\cdot\nabla f(\rho,c)+ P(\rho,c)-\widetilde P=-F, \;\text{ where }\; \rho \partial_\rho f(\rho,c)= \nu(\rho,c).
     \]
     Above, the pressure fluctuation term is treated as a damping term. 
    \end{itemize}
\end{remark}
\begin{remark}\label{intro:rema1}
 We consider a simplified case of \eqref{intro:twofluid} in which the two phases have the same constant viscosity coefficient but different power-law pressures:
 \[
 \mu(\rho,c)=\mu, \;\, \lambda(\rho,c)=\lambda,\; \text{ and }\; P(\rho,c)= A(c)\rho ^{\beta(c)}.
 \]
 From the mass equation, one readily obtains 
 \[
(2\mu+\lambda) \big(\dpt \log P(\rho,c) +u\cdot\nabla \log P(\rho,c)\big)+ \beta(c)\big( P(\rho,c)-\widetilde P\big)=\beta(c) (-\Delta)^{-1}\dvg (\rho \dot u),
 \]
 and the jump across the moving interface is:
 \begin{align}
 (2\mu+\lambda)\dfrac{d}{dt}\llbracket \log \big(P(\rho,c)\big) \circ \tau_j\rrbracket&+ \scalar{\beta(c_0\circ \tau_{0,j})}_{\text{avg}}\llbracket \big(P(\rho,c)\big) \circ \tau_j\rrbracket\notag\\
 &=\llbracket\beta(c_0\circ \tau_{0,j})\rrbracket \Big((-\Delta)^{-1}\dvg (\rho \dot u)\circ \tau_j-\scalar{P(\rho,c)\circ \tau_j}_{\text{avg}}+\widetilde P\Big).\label{intro:eq6}
 \end{align}
 Similarly, from the mass equation we obtain 
 \[
 (2\mu+\lambda)\big[\dpt\log \rho+ u\cdot\nabla \log \rho\big]+ P(\rho,c)-\widetilde P= (-\Delta)^{-1}\dvg (\rho \dot u),
 \]
 and hence 
 \begin{gather}\label{intro:eq7}
 (2\mu+\lambda)\frac{d}{dt}\scalar{\log (\rho\circ \tau_j)}_{\text{avg}}=(-\Delta)^{-1}\dvg (\rho \dot u)\circ \tau_j-\scalar{P(\rho,c)\circ \tau_j}_{\text{avg}}+\widetilde P.
 \end{gather}
 From \eqref{intro:eq6}-\eqref{intro:eq7}, it follows that 
 \[
 \dfrac{d}{dt}\bigg\llbracket \log\big( A(c)^{\frac{2}{\beta(1)+\beta(0))}}\rho \big)\circ \tau_j \bigg\rrbracket+\frac{1}{2\mu+\lambda} \llbracket \big(P(\rho,c)\big) \circ \tau_j\rrbracket=0. 
 \]
 Under the assumption \eqref{sec1:jumpcondition}, the jumps appearing above are comparable in the sense that 
 \[
 \inf \Bigg(\frac{\llbracket \big(P(\rho,c)\big) \circ \tau_j\rrbracket}{\llbracket \log\big( A(c)^{\frac{2}{\beta(1)+\beta(0))}}\rho \big)\circ \tau_j \rrbracket}\Bigg)>0
 \]
 and the exponential-in-time decay of jumps follows. Otherwise, it is unclear whether the exponential-in-time holds in general. Specifically, under the assumption $A(c)=1$, we should require
 \[
 \frac{\varrho_1^{\beta(1)}-\varrho_2^{\beta(0)}}{\log(\varrho_1)- \log(\varrho_2)}>0,\quad \text{for}\; \varrho_1 \;\text{close to } \; \widetilde\rho(1), \;\text{ and } \;\varrho_2\; \text{ close to }\;\widetilde\rho(0);
 \]
 which does not hold in general.

\end{remark}
\subsection{Review of known results}\label{review}
Understanding the interaction between two or more fluids has become a central issue in modern science. Mixture flow occurs in various contexts:  biological systems (blood, respiratory tract fluid); geophysical phenomena (ice formation, river flooding, landslides, and snow-slides) and industrial applications (flows in chemical reactors or nuclear reactor channels), to name a few examples. For more interesting examples, we refer to  \cite{ishii2010thermo}. The constituents of the mixture
(or phases) may have different densities, pressure laws, temperatures and viscosities. 

An interesting problem is the study of the dynamic of the interface separating the components of a mixture. From a physical point of view, the interface is a narrow layer across which the fluid states (density, temperature) may experience large variations or discontinuities,  giving rise to two categories of interface models. On one hand, there are \emph{diffuse interface models} for which the components of the mixture are separated by a thin layer (non-zero thickness). Within this thin layer, flow states experience large gradients while continuously passing from one pure zone to another.   On the other hand, there are \emph{sharp interface models} for which the interface is a hypersurface, with zero thickness. Although these models are derived from physical principles naturally, it is very expensive to study them numerically. Moreover, changes in the topology of the interface are harder to track. 

The mathematical study of the sharp interface model was first addressed by Tani in 1984 \cite{tani1984two}, where the solutions were constructed in piecewise H\"older spaces. The initial density and velocity are $\cC^{1+\alpha}$ and $\cC^{2+\alpha}$, respectively, on both sides of a suitable $\cC^{2+\alpha}$ interface.  Taking into account surface tension, Denisova  constructed  solutions with Sobolev regularity for both initial data and interface in \cite{denisova2000problem}. 

In contrast to the aforementioned works based on a change in Lagrangian coordinates, other studies rely on a transformation to a flattened interface. Piecewise Sobolev-regular solutions of high order ($H^{4k}, \; k\geqslant 3$)  were constructed in \cite{jang2016compressible}, where the interface is represented by a graph, considering scenarios with and without surface tension. The Rayleigh-Taylor instability problem has been investigated in \cite{jang2016compressible-taylor,jang2016compressible-stability}, with the latter reference addressing the zero surface tension limit problem. 
Finally, akin to the previous works, there is a study by Kubo, Shibata and Soga \cite{kubo2016some} that relies on the maximal regularity property of the linearized system. 

Note that only Tani’s work addresses density-dependent viscosity coefficients, and the result is only local in time.
The other results concern cases where each phase of the flow has constant viscosity coefficients, possibly different.
Existence of global-in-time solutions is established in this latter case under smallness of the piecewise $H^{4k},\; k\geqslant 3$ norm, and with almost flat interface.  The present paper allows for density-dependent viscosity coefficients and low regular initial data and interface and establishes global-in-time result. 

In the aforementioned papers, the system is solved on either side of the sharp interface, with a transmission boundary condition. A different approach to tracking the interface is to construct weak solutions for the full model in a class that allows for the study of its dynamic. The construction of weak solutions in the classical energy framework to \eqref{intro:twofluid} was first addressed in 2019 by Vasseur, Wen, and Yu (see \cite{vasseur2019global}), in the case where \emph{the viscosities coefficients are constant and common for both phases}. The authors considered power-law for the pressure of the phases, with a closeness constraint on the adiabatic exponents. This result has been improved by Novotn\'y \cite{novotny2020weak1}, Wen \cite{wen2021global} and Kra$\check{c}$mar, Kwon, Ne$\check{c}$esov\'a,  Novotn\'y \cite{kracmar2022weak} by allowing a 
more general pressure law, and  also  by Novotn\'y and Pokorn\'y \cite{novotny2020weak} for multi-component fluid models. It is important to note that although the pressure laws 
of the components of the mixture are different, the viscosities of the components are supposed to be the same constant. 
The solutions that they constructed are too weak to track down the discontinuity in the density or to study the dynamic of the 
interface between the components of the mixture.  

In one-dimensional space, Bresch, Burtea and Lagoutière \cite{bresch2022mathematical,bresch2024mathematical} achieved global well-posedness for non-common constant viscosity, with numerical considerations in \cite{bresch2023mathematical}. However,  for higher dimensions, $d\in \{2,3\}$, it has been explained for instance in \cite{bocchi2022anisotropy,bresch2023extension,bresch2018global} that, even the algebraic structure of the anisotropic one-phase model completely eludes the theories developed in the isotropic case with constant viscosity coefficients \cite{Feiresisl2001,Hoff95a,LionsCompressible}. 
The issue stems from the complete degeneration of the quantity (namely, the effective flux) that links the velocity's divergence to the pressure, as it is now of the same order as the velocity gradient.

In a recent paper (see \cite{zodji2023well}), we establish the local-in-time well-posedness for the system 
\eqref{intro:twofluid} in a functional framework tailored to our analysis. This serves as a building block 
for proving the existence and uniqueness of discontinuous solutions to the two-dimensional (one-phase) Navier–Stokes system 
for a compressible fluid with density-dependent viscosity (see \cite{zodji2023discontinuous}). We account for 
viscosity coefficients arising from homogenization (see, e.g., \cite{gerard2022correction,AnalysisViscosityGerardMathieu}), 
which are not covered by the work of Bresch and Desjardins \cite{BreschDesjardins2007}. Notice that the proof relies strongly on
the fact that the jumps of the pressure and the velocity gradient across the discontinuity curve decay exponentially in time. 
As discussed in Remark \ref{intro:rema1} above, this exponential-in-time decay does not seem to hold in general for the two-fluid model. 
Thus, we impose a low-frequency condition on the initial data to derive sufficient decay for the solution, so that exponential decay of the jumps is not needed. Moreover, unlike  the aforementioned result (see \cite{zodji2023discontinuous}), which imposes
smallness of the shear viscosity fluctuation in $\cC^\alpha_{\pw,\partial D_0}$, Theorem \ref{th1} only requires 
the smallness  of the shear viscosity fluctuation in $L^\infty$ and, in some cases, only of its jump. 

We conclude this section by noting that, for the two dimensional incompressible model, a smallness condition on the viscosity fluctuation also appears in \cite{gancedo2023global,gancedo20252d}. Recent work (see \cite{liao2024globalInc}) shows that, when the interface has low regularity ($\partial D_0\in W^{2,2+\eps}$, for some $\eps$ in a sharp range $(0,\eps_0)$), the smallness on the viscosity fluctuation  can be replaced by a smallness of the initial data. In this work, the key argument depends heavily on the fact that the system is considered in the two dimensional space. The extension to three dimensional space, or to the propagation  of $W^{2,p}$-regularity (with $p\in (2+\eps_0,\infty)$) of interfaces, or to the compressible model  
appears particularly involved.

\paragraph{\textbf{Organization of the paper}}
The rest of the paper is organized as follows. In Section \ref{part2} we state all the a priori estimates 
needed for the proof of  Theorem \ref{th1}. Their proofs are presented in Section \ref{lemmasProof}. 
Finally, we prove Theorem \ref{th1} in  Section \ref{sec:theorem}.

\section{Statement of a priori estimates}\label{part2}
This section is devoted to deriving all the key a priori estimates needed for the proof of Theorem \ref{th1}. 
These estimates consist of piecewise H\"older and energy-type bounds for approximate solutions of system \eqref{ep2.1} below, 
from which the global-in-time persistence of the regularity of the discontinuity surface will follow. This final step, concluding the proof of Theorem \ref{th1}, is the focus of Section \ref{sec:theorem} below.

More precisely, we consider the following Cauchy problem (we temporarily omit the superscript  $n$):
\begin{gather}\label{ep2.1}
\begin{cases}
    \dpt c+u\cdot\nabla c=0,\\
    \dpt \rho +\dvg (\rho u)=0,\\
    \dpt (\rho u)+\dvg (\rho u\otimes u)+\nabla P(\rho,c)=\dvg \big(2\mu(\rho,c)\D u\big)+\nabla\big(\lambda(\rho,c)\dvg u\big),\\
    (c,\rho,u)_{|t=0}= (c_0,\rho_0, u_0^n),
\end{cases}
\end{gather}
where $ \big(u_0^n\big)_{n \in \N}$ is a sequence of initial velocities satisfying the same assumptions as $u_0$ in Theorem \ref{th1}, along with the compatibility condition
\begin{equation}\label{sec2:compa}
    (\rho \dot u)_{|t=0}=\dvg \big(2\mu(\rho_0,c_0)\D u_0^n+\big(\lambda(\rho_0,c_0)\dvg u_0^n-P(\rho_0,c_0)+\widetilde P\big) I_3\big)\in L^2(\R^3).
\end{equation}
The construction of the sequence $\big(u_0^n\big)_n$ and the proof of its convergence toward 
$u_0$ are given in Section \ref{sec:theorem} below.  Our main objective is to derive estimate for 
$\normb{\nabla u}_{L^1((0,t), \cC^\alpha_{\pw,\Sigma}(\R^3))}$, from which we will deduce the global-in-time 
existence of the solution, as well as the persistence of interface regularity. This will be the focus of the 
next two subsections.

\subsection{Piecewise H\"older bounds}\label{sec:2:1}
This subsection is devoted to deriving piecewise H\"older estimates for the velocity gradient and the pressure fluctuation. 
The first step in obtaining these bounds is to derive an expression for the velocity gradient. To this end, we begin by 
rewriting the momentum equations $\eqref{ep2.1}_3$ as follows \big(hereafter, $G(\rho,c)$ denotes the  pressure fluctuation, namely $G(\rho,c)=P(\rho,c)-\widetilde P$\big):
\begin{equation}\label{sec2:eq23}
\widetilde\mu \Delta u+\nabla \big( (\widetilde\mu+\lambda(\rho,c))\dvg u-G(\rho,c)\big)=\rho \dot u-\dvg \big( 2(\mu(\rho,c)-\widetilde \mu)\D u\big).
\end{equation}
Next, we apply the Leray projector, $\mathbb{P}$, to obtain:
\[
\widetilde\mu \Delta \mathbb{P} u=\mathbb{P}(\rho \dot u)-\dvg \mathbb{P}\big( 2(\mu(\rho,c)-\widetilde \mu)\D u\big),
\]
and whence
\begin{align*}
    \widetilde\mu \nabla \mathbb{P} u&=-\nabla(-\Delta)^{-1}\mathbb P (\rho \dot u)+\nabla (-\Delta)^{-1}\dvg \mathbb{P}\big( 2(\mu(\rho,c)-\widetilde \mu)\D u\big)\\
    &=-\nabla(-\Delta)^{-1}\mathbb P (\rho \dot u)-(\mu(\rho,c)-\widetilde\mu)\nabla \mathbb{P}u+\big[\nabla (-\Delta)^{-1}\dvg \mathbb{P}, \mu(\rho,c)-\widetilde \mu\big](2\D u).
\end{align*}
It follows that 
\begin{equation}\label{sec2:eq1}
    \mu(\rho,c)\nabla \mathbb{P} u=-\nabla(-\Delta)^{-1}\mathbb P (\rho \dot u)+\big[\nabla (-\Delta)^{-1}\dvg \mathbb{P}, \mu(\rho,c)-\widetilde \mu\big](2\D u).
\end{equation}
Similarly, applying the divergence operator to \eqref{sec2:eq23} yields
\[
\Delta \big( (2\widetilde\mu+\lambda(\rho,c))\dvg u- G(\rho,c)\big)=\dvg (\rho \dot u)-\dvg \dvg\big( 2(\mu(\rho,c)-\widetilde \mu)\D u\big),
\]
and hence 
\begin{align*}
    (2\widetilde\mu+\lambda(\rho,c))\dvg u- G(\rho,c)&= -(-\Delta)^{-1}\dvg (\rho \dot u)+(-\Delta)^{-1}\dvg \dvg\big( 2(\mu(\rho,c)-\widetilde \mu)\D u\big)\\
    &=-(-\Delta)^{-1}\dvg (\rho \dot u)-2(\mu(\rho,c)-\widetilde\mu)\dvg u\\
    &+\big[(-\Delta)^{-1}\dvg \dvg,\mu(\rho,c)-\widetilde \mu\big](2\D u).
\end{align*}
It follows that 
\begin{equation}\label{sec2:eq2}
\big(2\mu(\rho,c)+\lambda(\rho,c)\big)\dvg u- G(\rho,c)=-(-\Delta)^{-1}\dvg (\rho \dot u)+\big[(-\Delta)^{-1}\dvg \dvg,\mu(\rho,c)-\widetilde \mu\big](2\D u).
\end{equation}
Combining \eqref{sec2:eq1} and~\eqref{sec2:eq2}, we obtain the following expression for the velocity gradient:
\begin{align}
    \nabla u&=\nabla \mathbb{P} u+ \nabla (\mathbb I-\mathbb P) u\notag\\
            &=\nabla \mathbb{P} u-\nabla (-\Delta)^{-1}\nabla \dvg u\notag\\
            &=-\frac{1}{\mu(\rho,c)}\nabla(-\Delta)^{-1}\mathbb P (\rho \dot u)+\nabla (-\Delta)^{-1}\nabla\bigg(\frac{1}{\nu(\rho,c)}(-\Delta)^{-1}\dvg (\rho \dot u)\bigg)-\nabla (-\Delta)^{-1}\nabla \Bigg(\frac{G(\rho,c)}{\nu(\rho,c)}\Bigg)\notag\\
            &+\frac{1}{\mu(\rho,c)}\big[\mathcal{K}, \mu(\rho,c)-\widetilde \mu\big](2\D u)-\nabla (-\Delta)^{-1}\nabla\bigg(\frac{1}{\nu(\rho,c)}\big[\check{\mathcal{K}},\mu(\rho,c)-\widetilde \mu\big](2\D u)\bigg).\label{sec2:eq3}
\end{align}
Above, $\mathcal{K}$ and $\check{\mathcal{K}}$ denote the following even order singular operators:
\[
\mathcal{K}= \nabla (-\Delta)^{-1}\dvg \mathbb{P} \;\text{ and } \,\check{\mathcal{K}}=(-\Delta)^{-1}\dvg \dvg.
\]

We now assume that the density and the shear viscosity coefficients are bounded from above and below as follows:
\begin{equation}\label{sec2:eq29}
\begin{cases}
    \displaystyle 0<\rho_*:=\inf_{t,x}\rho(t,x)\leqslant \sup_{t,x}\rho(t,x):=\rho^*<\infty;\\
    \displaystyle 0<\mu_*:=\inf_{t,x}\mu(\rho,c))(t,x)\leqslant \sup_{t,x}\mu(\rho,c))(t,x)=:\mu^*<\infty.
\end{cases}
\end{equation}
Also, given that the pressure is strictly increasing with respect to the density, we have the following bounds
    \begin{gather}\label{sec2:eq25}
    0<a_*(c):=\inf_{\rho\in[\rho_*,\rho^*]} \Bigg(\frac{\rho \partial_\rho P(\rho,c)}{\nu(\rho,c)}\Bigg)\leqslant \sup_{\rho\in[\rho_*,\rho^*]} \Bigg(\frac{\rho \partial_\rho P(\rho,c)}{\nu(\rho,c)}\Bigg)=:a^*(c)<\infty,
    \end{gather}
 from which we define the ratio:
    \begin{gather}\label{sec2:eq26}
    \overline{a}=\sup_c \frac{a^*(c)}{a_*(c)}.
    \end{gather}
Throughout what follows, $C$ will denote a generic constant that is independent of the bounds of the solutions. 
It may depend on $\widetilde P$, $\widetilde \mu$, $\widetilde \lambda$, $\alpha$, or arise from embedding inequalities etc.
In contrast, $C_*$ will depend on $\rho^*$, $\mu_*$, $\mu^*$. Both $C$ and $C_*$ may vary from line to line.

Let us denote by $\varphi(t)=\varphi_0(\mathcal{X}^{-1}(t))$ the level-set function defining $D(t)$: 
\[
D(t)= \big\{x\in \R^3\colon \vph(t,x)>0\big\},\;\text{ and let } \Sigma(t)=\partial D(t).
\]
Under a smallness condition on the shear viscosity jump we establish, in the following lemma, bounds for the  functionals $\vartheta$, and $\check{\vartheta}$ defined as (we recall the definition of the $(\rho,c)$-dependent function $f$ in \eqref{sec1:eq1}, which we supplement with $f(\widetilde\rho_+,1)=f(\widetilde\rho_-,0)=0$):
\begin{align}
    \check{\vartheta}(t)&= \abs{\Sigma(t)}_{\dot\cC^{1+\alpha}}+\abs{\nabla\varphi(t)}_{\dot\cC^\alpha}+\abs{f(\rho(t),c(t))}_{\dot\cC^\alpha_{\pw,\Sigma(t)}(\R^3)}+\int_0^t\abs{\nabla u, G(\rho,c)}_{\dot\cC^\alpha_{\pw,\Sigma}(\R^3)};\notag\\
    \vartheta(t)&=\int_0^t\norm{G(\rho,c),\,\nabla u,\,\nabla \mathbb P u,\, \big[\mathcal{K},\mu(\rho,c)-\widetilde \mu\big](\D u),\, \big[\check{\mathcal{K}},\mu(\rho,c)-\widetilde \mu\big](\D u)}_{L^\infty(\R^3)}\notag.
\end{align}
Our result reads as follows.
\begin{lemma}\label{sec2:lem1}  Let us introduce the functionals $\phi_\Sigma$, $\check{\phi}_\Sigma$ and $\delta_{}$ defined by
\begin{gather}\label{sec2:eq27}
\begin{cases}
\phi_\Sigma(t)\displaystyle &=\displaystyle\sup_{[0,t]}\Big(\abs{\mu(\rho,c),\,\lambda(\rho,c)}_{\dot \cC^{\alpha/2}_{\pw,\Sigma}(\R^3)}+\ell_{\Sigma}^{-\alpha/2}\norm{\mu(\rho,c)-\widetilde\mu,\, \lambda(\rho,c)-\widetilde\lambda}_{L^\infty(\R^3)}\Big),\\
\check{\phi}_{\Sigma}(t)&=\displaystyle\sup_{[0,t]}\Big(\abs{\mu(\rho,c)}_{\dot \cC_{\pw,\Sigma}^{\alpha/2}(\R^3)}+ \big(\ell^{-\alpha/2}_{\Sigma}+\mathfrak{P}_{\Sigma}\abs{\Sigma}_{\dot \cC^{1+\alpha/2}}\big)\norm{\llbracket \mu(\rho,c)\rrbracket}_{L^\infty(\Sigma)}\Big),\\
\delta(t)&= \displaystyle \sup_{[0,t]} \norm{\llbracket \mu(\rho,c),\, \nu(\rho,c)\rrbracket}_{L^\infty\cap L^{p}(\Sigma)}
\end{cases}
\end{gather}
where $\mathfrak{P}_\Sigma$ is a polynomial in $\abs{\Sigma}$, $\norm{\Sigma}_{\text{Lip}}$ and $\abs{\Sigma}_{\text{inf}}^{-1}$.

We now fix some $p\in (6/\alpha, \infty)$. There exists a constant $C>0$ such that if 
    \begin{align}\label{sec2:eq7}
\frac{4C}{\mu_*}\Bigg[\overline{a}e^{\alpha\int_0^t\normb{\nabla u}_{L^\infty(\R^3)}}&\sup_{[0,t]}\norm{\llbracket \mu(\rho,c)\rrbracket}_{L^\infty(\Sigma)}\Big(1+\frac{1}{\mu_*}\norm{\llbracket\nu(\rho,c)\rrbracket}_{L^\infty(\Sigma)}\Big)\notag\\
&+\sup_{[0,t]} \Big(\overline{a}\abs{\llbracket \mu(\rho,c)\rrbracket}_{\dot \cC^{\alpha/2}(\Sigma)}+\frac{\mu^*}{\mu_*}\norm{\llbracket\nu(\rho,c)\rrbracket}_{L^\infty(\Sigma)}\Big)\Bigg]\leqslant 1,
    \end{align}
    then the following inequalities hold true:
    \begin{itemize}
        \item \textbf{Piecewise H\"older estimate}:
    \begin{align}
    \check{\vartheta}(t)&\leqslant C_*\Bigg( \check{\vartheta}(0)+\int_0^t\norm{\rho \dot u}_{L^{3/(1-\alpha)}(\R^3)}\Bigg) \exp\Bigg[C_*K_0e^{C\vartheta(t)}\Bigg(\vartheta(t)+\int_0^t \norm{\rho \dot u}_{L^6(\R^3)}^{1/2}\norm{\rho\dot u}_{L^2(\R^3)}^{1/2}\Bigg)\Bigg],\label{sec2:pwholder}
\end{align}
where $K_0$ depends on $\norm{\Sigma_0}_{\text{Lip}}$, $\abs{\nabla\vph_0}_{\text{inf}}^{-1}$, $\abs{\Sigma_0}_{\text{inf}}$, and $\abs{\Sigma}$.
\item \textbf{Lipschitz estimate:}
\begin{align}\label{sec2:Lipschtz}
  \vartheta(t)&\leqslant C_* \Bigg\{\norm{f(\rho_0,c_0)}_{L^\infty(\R^3)}+\int_0^t\norm{\dot u}_{L^{2}(\R^3)}+\int_0^t\phi_\Sigma\norm{G(\rho,c),\,\nabla u,\, (-\Delta)^{-1}\dvg (\rho \dot u)}_{L^{p}(\R^3)}\notag\\
&+\Big(1+\delta(t) e^{\alpha\vartheta(t)}\Big)\Bigg[\abs{f(\rho_0,c_0)}_{\dot \cC^{\alpha/2}_{\pw,\Sigma_0}(\R^3)}+\int_0^t\norm{\rho \dot u}_{L^{6/(2-\alpha)}(\R^3)}+\Big(\sup_{[0,t]}\abs{\mu(\rho,c)}_{\dot\cC^\alpha(\R^3)}+ \delta(t)\ell_{\Sigma}^{-\alpha}\Big)\int_0^t\norm{\nabla u}_{L^{6/\alpha}(\R^3)}\Bigg]\notag\\
&+\check{\vartheta}(t)\sup_{[0,t]} \norm{\llbracket \mu(\rho,c)\rrbracket}_{\cC^{\alpha/2}(\Sigma)}\Big(1+\delta(t)e^{\alpha\vartheta(t)}\Big)+ \int_0^t \Big(\ell_{\Sigma}^{-\frac{\alpha}{2}}\norm{\llbracket G(\rho,c)\rrbracket}_{L^\infty\cap L^{p}(\Sigma)}
        +\abs{\llbracket G(\rho,c)\rrbracket}_{\dot \cC^{\alpha/2}(\Sigma)}\Big)\notag\\
    &+\vartheta(t)\delta(t)\Bigg( \check{\phi}_\Sigma e^{\alpha \vartheta(t)}  +\big(\ell^{-\alpha/2}_{\Sigma}+\mathfrak{P}_{\Sigma} \abs{\Sigma}_{\dot \cC^{1+\alpha/2}}\big)\Bigg)\Bigg\}e^{C_*\int_0^t\normb{w}_{L^\infty(\R^3)}},
\end{align}
with
\begin{align*}
   \sup_{[0,t]}\norm{f(\rho,c)}_{L^\infty(\R^3)}+\int_0^t\norm{w}_{L^\infty(\R^3)}&\leqslant C\overline{a}\Bigg(\norm{f(\rho_0,c_0)}_{L^\infty(\R^3)}+\delta(t)\int_0^t\norm{\nabla u}_{\cC^{\alpha/2}_{\pw,\Sigma}(\R^3)}\Bigg)\notag\\
    &+C\overline{a}\int_0^t \big(\norm{\rho \dot u}_{L^2(\R^3)}+\norm{\rho \dot u}_{L^{6/(2-\alpha)}(\R^3)}+\phi_\Sigma\norm{\nabla u}_{L^{p}(\R^3)}\big). 
\end{align*}
\end{itemize}
\end{lemma}
The proof of the above Lemma is the purpose of Section \ref{sec3:proof:lem1} below.  The estimate of the functional $\vartheta$
involves some norms of the jump of the  pressure fluctuation $G(\rho,c)$.  These norms can be roughly be estimated as follows:
 \begin{equation}\label{sec2:eq22}
 \norm{\llbracket G(\rho,c)\rrbracket}_{L^q(\Sigma)}\leqslant \abs{\Sigma}^{1/q}  \norm{G(\rho,c)}_{L^\infty(\R^3)},\;\text{ with } q\in [1,\infty], \;\text{ and } \abs{\llbracket G(\rho,c)\rrbracket}_{\dot \cC^{\alpha/2}(\Sigma)}\leqslant 2 \abs{G(\rho,c)}_{\dot \cC^{\alpha/2}_{\pw,\Sigma}(\R^3)}, 
 \end{equation}
and we will instead  estimate the $\cC^{\alpha/2}_{\pw,\Sigma}(\R^3)$-bound  as in \eqref{sec2:eq50}-\eqref{sec2:eq30} and we will require the pressure fluctuations to be small. In particular situations such as \eqref{sec1:jumpcondition} or \eqref{sec1:jumpconditionbis}, we can derive sharper estimates for the norms of the pressure jump. 
\begin{lemma}\label{sec2:lem3}  Let $q\in [1,\infty]$, and $0<a_*=\min_c a_*(c)$, $a^*=\max_c a^*(c)$. The following 
estimates hold.
\begin{itemize}
    \item Assume that \eqref{sec1:jumpcondition} holds true. Then we have
    \begin{align}
        \sup_{[0,t]}\norm{\llbracket f(\rho,c)\rrbracket}_{L^q(\Sigma)}&+ \int_0^t\norm{\llbracket G(\rho,c)\rrbracket}_{L^q(\Sigma)}\notag\\
        &\leqslant \frac{a^*}{a_*} e^{\frac{1}{q}\int_0^t\normb{\nabla u}_{L^\infty(\R^3)}}\Bigg(\norm{\llbracket f(\rho_0,c_0)\rrbracket}_{L^q(\Sigma_0)}+C\sup_{[0,t]}\norm{\llbracket \mu(\rho,c)\rrbracket}_{L^q(\Sigma)}\int_0^t\normb{\nabla u}_{L^\infty(\R^3)}\Bigg),\label{sec2:eq37}
    \end{align}
    and 
    \begin{align}
\sup_{[0,t]}\abs{\llbracket f(\rho,c)\rrbracket}_{\dot\cC^\alpha(\Sigma)}&+\int_0^t\abs{\llbracket G(\rho,c)\rrbracket}_{\dot\cC^\alpha(\Sigma)}\notag\\
     &\leqslant  \frac{a^*}{a_*}e^{\alpha\int_0^t\normb{\nabla u}_{L^\infty(\R^3)}} \Bigg[\abs{\llbracket f(\varrho_0,c_0)\rrbracket}_{\dot\cC^\alpha(\Sigma_0)}+\sup_{[0,t]}\abs{\llbracket \mu(\rho,c)\rrbracket}_{\dot\cC^\alpha(\Sigma)}\int_0^t\norm{\nabla u}_{L^\infty(\R^3)}\notag\\
     &+\sup_{[0,t]}\norm{\llbracket \mu(\rho,c)\rrbracket}_{L^\infty(\Sigma)}\int_0^t\big(\ell_{\Sigma}^{-\alpha}\norm{\nabla u}_{L^\infty(\R^3)}+\abs{\nabla u}_{\dot \cC^\alpha_{\pw,\Sigma}(\R^3)}\big)\notag\\
     &+C_*\sup_{[0,t]}\abs{f(\rho,c)}_{\dot \cC^\alpha_{\pw,\Sigma}(\R^3)}\int_0^t\norm{\llbracket f(\rho,c)\rrbracket}_{L^\infty(\R^3)}\Bigg]\label{sec2:eq38}.
\end{align}
    \item Assume that \eqref{sec1:jumpconditionbis} holds true. Then we have, 
        \begin{align}\label{sec2:eq40}
      \norm{\llbracket f(\rho,c)(t)\rrbracket}_{L^q(\Sigma(t))}&\leqslant e^{- a_{*} t+C_*\int_0^t\norm{\nabla u}_{L^\infty(\R^3)}} \norm{\llbracket f(\rho_0,c_0)\rrbracket}_{L^q(\Sigma_0)},
    \end{align}
    and 
    \begin{align}\label{sec2:eq41}
     \abs{\llbracket f(\rho,c)(t)\rrbracket}_{\dot\cC^\alpha(\Sigma(t))}&\leqslant  C_*e^{-a_{*}t+ C_*\int_0^t\norm{\nabla u}_{L^\infty(\R^3)}e^{ \alpha\int_0^t\norm{\nabla u}_{L^\infty(\R^3)}}}\Bigg[\abs{\llbracket f(\varrho_0,c_0)\rrbracket}_{\dot\cC^\alpha(\Sigma_0)}\notag\\
     &+\norm{\llbracket f(\rho_0,c_0)\rrbracket}_{L^\infty(\Sigma_0)}\int_0^t\Big(\big(\ell_{\Sigma}^{-\alpha}+\abs{f(\rho,c)}_{\dot \cC^\alpha(\Sigma)}\big)\norm{\nabla u}_{L^\infty(\R^3)}+\abs{f(\rho,c),\,\nabla u}_{\dot\cC^\alpha_{\pw,\Sigma}(\R^3)}\Big)\Bigg].
\end{align}
\item  In addition to \eqref{sec1:jumpconditionbis}, assume that \eqref{sec1:exponentialdecay} holds true. Then  
\begin{align}\label{sec2:eq46}
    \norm{\llbracket \nabla u(t)\rrbracket}_{L^q(\Sigma(t))}\leqslant C_* e^{- a_{*} t+C_*\int_0^t\norm{\nabla u}_{L^\infty(\R^3)}} \norm{\llbracket f(\rho_0,c_0)\rrbracket}_{L^q(\Sigma_0)}\big(1+\norm{\nabla u(t)}_{L^\infty(\R^3)}\big).
\end{align}
\end{itemize}
\end{lemma}
The proof of the above lemma is the purpose of Section \ref{prooflem3} below. From these bounds, we immediately deduce that the jump of $f(\rho,c)$, viscosities, pressure, and velocity gradient decay exponentially in time, once the functionals 
$\vartheta$ and $\check{\vartheta}$ are  uniformly bounded in time. The estimate obtained for these functionals  in Lemma \ref{sec2:lem1} above, involves the $L^p(\R^3)$ norm of $\nabla u$, $G(\rho,c)$, with $p\in (6/\alpha,\infty)$; and  the $L^2(\R^3)\cap L^{3/(1-\alpha)}(\R^3)$ norm of the material acceleration $\dot u$. In the following section, we derive these bounds from time-weighted energy estimates.

\subsection{Time-weighted energy estimates}\label{sec:2:2}
This section is devoted to deriving $L^1((0,t),L^q(\R^3))$-bounds for the velocity gradient, pressure fluctuation, and for the material acceleration as these  appear in the estimates of the functionals $\vartheta$ and $\check{\vartheta}$ in Lemma \ref{sec2:lem1}. These $L^1((0,t),L^q(\R^3)$-bounds will follow from the fact that the following energy functionals are uniformly bounded in time:
\begin{align}\label{sec2:eq48}
\begin{cases}
    \mathcal{A}_1(t)&=\displaystyle\int_0^t\scalar{\cdot}^{1+r_0}\norm{\sqrt{\rho}\dot u}_{L^2(\R^3)}^2+\frac{1}{2}\sup_{[0,t]}\scalar{\cdot}^{1+r_0}\int_{\R^3} \big( 2\mu(\rho,c)\abs{\D u}^2+\lambda(\rho,c)\abs{\dvg u}^2\big),\\
    \mathcal{A}_2(t)&=\displaystyle\frac{1}{2}\sup_{[0,t]}\sigma\scalar{\cdot}^{\frac{3}{2}(1+r_0)}\norm{\sqrt{\rho}\dot u}_{L^2(\R^3)}^2
    +\int_0^t\sigma\scalar{\cdot}^{\frac{3}{2}(1+r_0)}
    \int_{\R^3}\big( 2\mu(\rho,c)\abs{\D \dot u}^2+\lambda(\rho,c)\abs{\dvg \dot u}^2\big),\\
    \mathcal{A}_3(t)&=\displaystyle\int_0^t\sigma^{1+\check{\alpha}}\scalar{\cdot}^{\frac{3}{2}(1+r_0)}\norm{\sqrt{\rho}\ddot u}_{L^2(\R^3)}^2+\frac{1}{2}\sup_{[0,t]}\sigma^{1+\check{\alpha}}\scalar{\cdot}^{\frac{3}{2}(1+r_0)}\int_{\R^3} \big( 2\mu(\rho,c)\abs{\D \dot u}^2+\lambda(\rho,c)\abs{\dvg \dot u}^2\big).
    \end{cases}
\end{align}
Above $0<r_0<\frac{3}{p_0}-\frac{5}{2}$ (we recall that $1< p_0<6/5$, see \eqref{init:lowfrequency}); $\check{\alpha}\in (0,1]$ is introduced in \eqref{sec:intro:eq1}; and the time weights $\scalar{\cdot}$ and  $\sigma$ are defined by:
\[
\sigma(t)=\min \{1,t\},\;\text{ and }\; \scalar{t}=\sqrt{1+t^2}.
\]
From now on, the constants $C$ and $C_*$may also depend on $r_0$, and $\check{\alpha}$.  $C_0$ will stand for a constant depending on $\norm{\rho_0 u_0,\, G(\rho_0,c_0)}_{L^{p_0}(\R^3)}^2$ and on the initial energies $E_0$ and $\check{E}_0$ (see their definitions in ) as positive powers. We denote by $\check{C}_0$ a constant that depends on the same quantities as $C_0$, and additionally on $\norm{\nabla u_0}_{L^2(\R^3)}^2$ and $\normb{(\rho \dot u)_{|t=0}}_{\dot H^{-\check\alpha}(\R^3)}^2$.

We obtain the following bounds.
\begin{lemma}\label{sec2:lem4}
   There exist constants $C=C(\alpha)>0$  and $c_*=c_*(\rho_*,\rho^*)$ such that if 
   \begin{gather}
       \frac{C}{\mu_*}\sup_{[0,t]} \norm{\llbracket \mu(\rho,c)\rrbracket}_{L^\infty(\Sigma)}\leqslant 1, \quad\text{and}\quad E_0\leqslant c_*, 
   \end{gather}
   then the following estimates holds true:
\begin{align}
\mathcal{A}_1(t)\leqslant C_*\Big(C_0\big(1+\check \phi_\Sigma^{\frac{6}{\alpha}} E_0\big)+ \norm{\nabla u_0}_{L^2(\R^3)}^2+E^{\frac{1}{4}}_0\big(\mathcal{A}_1(t)^{\frac{5}{4}}+C_0^{\frac{5}{4}}\big)\Big)\label{sec2:eq53},
\end{align}
\begin{align}
\mathcal{A}_2(t) \leqslant C_*C_0\Big(1+\mathcal{A}_1(t)^2\Big)+ C_*\mathcal{A}_1(t)\Big(1+\mathcal{A}_1(t)^2+\check\phi_\Sigma^{\frac{6}{\alpha}} C_0\Big),\label{sec2:eq54}
\end{align}
\begin{gather}\label{sec2:eq55}
    \mathcal{A}_3(t)\leqslant C_*\Big( \check C_0+ \mathcal{L}\bigg(C_0\phi_\Sigma^{\frac{2}{\alpha}},\,\mathcal{A}_1(t)\bigg)\Big)e^{C_*C_0\big(1+\check \phi_\Sigma^{\frac{6}{\alpha}}+\mathcal{A}_1(t)\big)},
\end{gather}
where $\mathcal{L}$ is a polynomial. 
\end{lemma}
The proof of the above lemma is the purpose of Section \ref{proof:sec2:lem4}. Prior to that, we derive in Section \ref{basic:estimate} some basic a priori estimates for the solutions to the system \eqref{ep2.1}. From these estimates, we immediately deduce: 
\begin{lemma}\label{sec2:lem5} Let $p\in [6,\infty)$. There exists a constant $\check k=\check k(p,\rho_*,\rho^*)>0$, such that if 
\[
 C\frac{\check{k}}{\mu_*}\sup_{[0,t]}\norm{\llbracket\mu(\rho,c)\rrbracket}_{L^\infty(\Sigma)}\leqslant \frac{1}{2},
\]
then we have
\begin{gather*}
    \int_0^t\Big(\norm{\nabla u, \, G(\rho,c)}_{L^p(\R^3)}+\norm{\dot u}_{L^{6/(2-\alpha)}(\R^3)}\Big)\leqslant C_*C_0\Bigg(1+\check{\phi}_\Sigma^{\frac{3}{\alpha}\big(1-\frac{2}{p}\big)}\Bigg)+ C_*\big(\mathcal{A}_1(t)^{\frac{1}{2}}+\mathcal{A}_2(t)^{\frac{1}{2}}\big);
\end{gather*}
and 
\[
\int_0^t \norm{\rho\dot u}_{L^{3/(1-\alpha)}(\R^3)}\leqslant C_*\check{\mathcal{L}}\big(\check{C}_0,C_0\check{\phi}_\Sigma^{\frac{2}{\alpha}}, \check{\phi}_\Sigma^{\frac{6}{\alpha}}\mathcal{A}_1(t),\, \mathcal{A}_1(t),\, \mathcal{A}_3(t)\big),
\]
where $\check{\mathcal{L}}$ is a polynomial. 
\end{lemma}
The proof of Lemma \ref{sec2:lem5} is the purpose of Section \ref{proof:sec2:lem5} below. This concludes the present section.

All the a priori estimates derived in this section will be closed with the help of a bootstrap argument in Section \ref{sec:theorem}. We also prove the compactness of the approximate solutions, thereby obtaining Theorem \ref{th1}.


\section{Proofs of the a priori estimates}\label{lemmasProof}
\subsection{Proof of Lemma \ref{sec2:lem1}}\label{sec3:proof:lem1}
We begin the proof of Lemma \ref{sec2:lem1} by rewriting the expression of the velocity gradient (see \eqref{sec2:eq3}):
\begin{align}
    \nabla u
            =&-\frac{1}{\mu(\rho,c)}\nabla(-\Delta)^{-1}\mathbb P (\rho \dot u)+\nabla (-\Delta)^{-1}\nabla\bigg(\frac{1}{\nu(\rho,c)}(-\Delta)^{-1}\dvg (\rho \dot u)\bigg)-\nabla (-\Delta)^{-1}\nabla \Bigg(\frac{G(\rho,c)}{\nu(\rho,c)}\Bigg)\notag\\
            &+\frac{1}{\mu(\rho,c)}\big[\mathcal{K}, \mu(\rho,c)-\widetilde \mu\big](2\D u)-\nabla (-\Delta)^{-1}\nabla\bigg(\frac{1}{\nu(\rho,c)}\big[\check{\mathcal{K}},\mu(\rho,c)-\widetilde \mu\big](2\D u)\bigg).\label{sec2:eq4}
\end{align}
\begin{proof}[Proof of \eqref{sec2:pwholder}]
We estimate each term appearing on the RHS of the expression above as follows:
\begin{itemize}
    \item \textbf{First term}: The first term in  \eqref{sec2:eq4} is estimated as follows:
    \begin{align*}
        &\abs{\frac{1}{\mu(\rho,c)}\nabla(-\Delta)^{-1}\mathbb P (\rho \dot u)}_{\dot \cC^\alpha_{\pw,\Sigma}(\R^3)}\notag\\
        &\leqslant \frac{1}{\mu_*^2}\abs{\mu(\rho,c)}_{\dot \cC^\alpha_{\pw,\Sigma}(\R^3)}\norm{\nabla(-\Delta)^{-1}\mathbb P (\rho \dot u)}_{L^\infty(\R^3)}
        +\frac{1}{\mu_*}\norm{\nabla(-\Delta)^{-1}\mathbb P (\rho \dot u)}_{\dot \cC^{\alpha}(\R^3)}\notag\\
        &\leqslant \frac{C}{\mu_*^2}\abs{\mu(\rho,c)}_{\dot \cC^\alpha_{\pw,\Sigma}(\R^3)}\norm{\nabla\nabla(-\Delta)^{-1}\mathbb P (\rho \dot u)}_{L^6(\R^3)}^{1/2}\norm{\nabla(-\Delta)^{-1}\mathbb P (\rho \dot u)}_{L^6(\R^3)}^{1/2}+\frac{C}{\mu_*}\norm{\rho \dot u}_{L^{3/(1-\alpha)}(\R^3)}.\\
        &\leqslant \frac{C}{\mu_*^2}\abs{\mu(\rho,c)}_{\dot \cC^\alpha_{\pw,\Sigma}(\R^3)}\norm{\rho\dot u}_{L^6(\R^3)}^{1/2}\norm{\rho \dot u}_{L^2(\R^3)}^{1/2}+\frac{C}{\mu_*}\norm{\rho \dot u}_{L^{3/(1-\alpha)}(\R^3)}.
    \end{align*}
    We used the Gagliardo-Nirenberg's inequality $\lVert g\rVert_{L^\infty(\R^3)}\leqslant \lVert\nabla g\rVert_{L^6(\R^3)}^{1/2}\lVert g\rVert_{L^6(\R^3)}^{1/2}$, the embedding  \\
    $\dot W^{1,3/(1-\alpha)}(\R^3)\hookrightarrow \dot\cC^{\alpha}(\R^3)$ and the continuity of Riesz transforms on $L^p(\R^3)$, with $p\in (1,\infty)$.
    \item \textbf{Second and third  terms}:  The second term in \eqref{sec2:eq4} is a second-order Riesz transform of a discontinuous function. Thus, we apply \eqref{ap:eq23} to obtain:
  \begin{align*}
        &\abs{\nabla (-\Delta)^{-1}\nabla\bigg(\frac{1}{\nu(\rho,c)}(-\Delta)^{-1}\dvg (\rho \dot u)\bigg)}_{\dot \cC^\alpha_{\pw,\Sigma}(\R^3)}\notag\\
        &\leqslant C \abs{\frac{1}{\nu(\rho,c)}(-\Delta)^{-1}\dvg (\rho \dot u)}_{\dot \cC^\alpha_{\pw,\Sigma}(\R^3)}+ C \mathfrak{P}_{\Sigma} \norm{\frac{1}{\nu(\rho,c)}(-\Delta)^{-1}\dvg (\rho \dot u)}_{L^\infty(\R^3)}\abs{\Sigma}_{\dot \cC^{1+\alpha}}\notag\\
        &\leqslant \frac{C}{\mu_*}\norm{\rho \dot u}_{L^{3/(1-\alpha)}(\R^3)} +\frac{C}{\mu_*}\norm{\rho \dot u}_{L^6(\R^3)}^{1/2}\norm{\rho \dot u}_{L^2(\R^3)}^{1/2}\Bigg(\frac{1}{\mu_*}\abs{\nu(\rho,c)}_{\dot \cC^\alpha_{\pw,\Sigma}(\R^3)}+\mathfrak{P}_{\Sigma}\abs{\Sigma}_{\dot \cC^{1+\alpha}}\Bigg).
    \end{align*}
    We derive a similar estimate for the third term
    \begin{align*}
        \abs{\nabla (-\Delta)^{-1}\nabla \Bigg(\frac{G(\rho,c)}{\nu(\rho,c)}\Bigg)}_{\dot \cC^\alpha_{\pw,\Sigma}(\R^3)}&\leqslant \frac{C}{\mu_*}\abs{G(\rho,c)}_{\dot\cC^\alpha_{\pw,\Sigma}(\R^3)}\\
        &+\frac{C}{\mu_*} \norm{G(\rho,c)}_{L^\infty(\R^3)}\Bigg(\frac{1}{\mu_*}\abs{\nu(\rho,c)}_{\dot \cC^\alpha_{\pw,\Sigma}(\R^3)}+\mathfrak{P}_{\Sigma}\abs{\Sigma}_{\dot \cC^{1+\alpha}}\Bigg).
    \end{align*}
    \item\textbf{Remaining terms}: Using \eqref{ap:eq19}, the next-to-last term in \eqref{sec2:eq4} is estimated as follows:
    {\small
    \begin{align}
        &\abs{\frac{1}{\mu(\rho,c)}\big[\mathcal{K}, \mu(\rho,c)-\widetilde \mu\big](2\D u)}_{\dot\cC^\alpha_{\pw,\Sigma}(\R^3)}\notag\\
        &\leqslant \frac{1}{\mu_*^2}\abs{\mu(\rho,c)}_{\dot \cC^\alpha_{\pw,\Sigma}(\R^3)}\norm{\big[\mathcal{K}, \mu(\rho,c)-\widetilde \mu\big]\D u}_{L^\infty(\R^3)}
        +\frac{1}{\mu_*}\abs{\big[\mathcal{K}, \mu(\rho,c)-\widetilde \mu\big]\D u}_{\dot \cC^\alpha_{\pw,\Sigma}(\R^3)}\notag\\
        &\leqslant \frac{1}{\mu_*^2}\abs{\mu(\rho,c)}_{\dot \cC^\alpha_{\pw,\Sigma}(\R^3)}\norm{\big[\mathcal{K}, \mu(\rho,c)-\widetilde \mu\big]\D u}_{L^\infty(\R^3)}\notag\\
        &+\frac{C}{\mu_*}\Big(\abs{\mu(\rho,c)}_{\dot \cC^{\alpha}_{\pw,\Sigma}(\R^3)}+ \big(\ell^{-\alpha}_{\Sigma}+\mathfrak{P}_{\Sigma}\abs{\Sigma}_{\dot \cC^{1+\alpha}}\big)\norm{\llbracket \mu(\rho,c)\rrbracket}_{L^\infty(\Sigma)}\Big)\norm{\nabla \mathbb P u, \D u}_{L^\infty(\R^3)}\notag\\
        &+\frac{C}{\mu_*}\Big( \abs{\mu(\rho,c)}_{\dot \cC^{\alpha/2}_{\pw,\Sigma}(\R^3)}+ \ell^{-\alpha/2}_{\Sigma}\norm{\llbracket \mu(\rho,c)\rrbracket}_{L^\infty(\Sigma)}\Big)\abs{\D u}_{\dot \cC^{\alpha/2}_{\pw,\Sigma}(\R^3)}+\frac{C}{\mu_*}\norm{\llbracket \mu(\rho,c)\rrbracket}_{L^\infty(\Sigma)}\abs{\nabla u}_{\dot \cC^\alpha_{\pw,\Sigma}(\R^3)}.\label{sec2:eq6}
    \end{align}
    }
    The same estimate also holds for $\absb{\frac{1}{\mu(\rho,c)}[\nabla(-\Delta)^{-1}\nabla, \mu(\rho,c)-\widetilde \mu](2\D u)}_{\dot\cC^\alpha_{\pw,\Sigma}(\R^3)}$.
    
      Similarly, we apply \eqref{ap:eq23}-\eqref{ap:eq19} to obtain:
\begin{align}
    &\abs{\nabla (-\Delta)^{-1}\nabla\bigg(\frac{1}{\nu(\rho,c)}\big[\check{\mathcal{K}},\mu(\rho,c)-\widetilde \mu\big](2\D u)\bigg)}_{\dot \cC^\alpha_{\pw,\Sigma}(\R^3)}\notag\\
    &\leqslant \abs{\frac{1}{\nu(\rho,c)}\big[\check{\mathcal{K}},\mu(\rho,c)-\widetilde \mu\big]\D u}_{\dot \cC^\alpha_{\pw,\Sigma}(\R^3)}+\frac{C}{\mu_*}\norm{\big[\check{\mathcal{K}},\mu(\rho,c)-\widetilde \mu\big]\D u}_{L^\infty(\R^3)}\mathfrak{P}_{\Sigma} \abs{\Sigma}_{\dot \cC^{1+\alpha}}\notag\\
    &\leqslant \frac{C}{\mu_*}\Big(\frac{1}{\mu_*}\abs{\nu(\rho,c)}_{\dot\cC^\alpha_{\pw,\Sigma}(\R^3)}+\mathfrak{P}_{\Sigma} \abs{\Sigma}_{\dot \cC^{1+\alpha}}\Big)\norm{\big[\check{\mathcal{K}},\mu(\rho,c)-\widetilde \mu\big]\D u}_{L^\infty(\R^3)}\notag\\
    &+\frac{C}{\mu_*}\Big(\abs{\mu(\rho,c)}_{\dot \cC^{\alpha}_{\pw,\Sigma}(\R^3)}+ \big(\ell^{-\alpha}_{\Sigma}+\mathfrak{P}_{\Sigma} \abs{\Sigma}_{\dot \cC^{1+\alpha}}\big)\norm{\llbracket \mu(\rho,c)\rrbracket}_{L^\infty(\Sigma)}\Big)\norm{ \nabla u}_{L^\infty(\R^3)}\notag\\
        &+\frac{C}{\mu_*}\Big( \abs{\mu(\rho,c)}_{\dot \cC^{\alpha/2}_{\pw,\Sigma}(\R^3)}+ \ell^{-\alpha/2}_{\Sigma}\norm{\llbracket \mu(\rho,c)\rrbracket}_{L^\infty(\Sigma)}\Big)\abs{\D u}_{\dot \cC^{\alpha/2}_{\pw,\Sigma}(\R^3)}+\frac{C}{\mu_*}\norm{\llbracket \mu(\rho,c)\rrbracket}_{L^\infty(\Sigma)}\abs{\nabla u}_{\dot \cC^\alpha_{\pw,\Sigma}(\R^3)}.\notag
\end{align}
\end{itemize}
Combining all these estimates, we arrive at:
{\small
\begin{align}
    \int_0^t\abs{\nabla u}_{\dot \cC^\alpha_{\pw,\Sigma}(\R^3)}
    &\leqslant \frac{C}{\mu_*}\int_0^t\norm{\rho \dot u}_{L^{3/(1-\alpha)}(\R^3)}+\frac{C}{\mu_*} \int_0^t\psi(s)\Big(\frac{1}{\mu_*}\abs{\mu(\rho,c)(s),\nu(\rho,c)(s)}_{\dot \cC^\alpha_{\pw,\Sigma(s)}(\R^3)}+\mathfrak{P}_{\Sigma(s)} \big|\Sigma(s)\big|_{\dot \cC^{1+\alpha}}\Big)\notag\\
    &+\frac{C}{\mu_*}\int_0^t\Big(\abs{\mu(\rho,c)}_{\dot \cC^{\alpha}_{\pw,\Sigma}(\R^3)}+ \big(\ell^{-\alpha}_{\Sigma}+\mathfrak{P}_{\Sigma} \abs{\Sigma}_{\dot \cC^{1+\alpha}}\big)\norm{\llbracket \mu(\rho,c)\rrbracket}_{L^\infty(\Sigma)}\Big)\norm{\nabla\mathbb P u,\nabla u}_{L^\infty(\R^3)}\notag\\
    &+\frac{C}{\mu_*}\Big( \abs{\mu(\rho,c)}_{\dot \cC^{\alpha/2}_{\pw,\Sigma}(\R^3)}+ \ell^{-\alpha/2}_{\Sigma}\norm{\llbracket \mu(\rho,c)\rrbracket}_{L^\infty(\Sigma)}\Big)\abs{\D u}_{\dot \cC^{\alpha/2}_{\pw,\Sigma}(\R^3)}\notag\\
        &+\frac{C}{\mu_*}\int_0^t\norm{\llbracket \mu(\rho,c)\rrbracket}_{L^\infty(\Sigma)}\abs{\nabla u}_{\dot \cC^\alpha_{\pw,\Sigma}(\R^3)}+\frac{C}{\mu_*}\int_0^t\abs{G(\rho,c)}_{\dot\cC^\alpha_{\pw,\Sigma}(\R^3)},\label{sec2:eq8}
\end{align}
}
where 
\[
\psi= \norm{\rho \dot u}_{L^6(\R^3)}^{1/2}\norm{\rho \dot u}_{L^2(\R^3)}^{1/2}+\norm{G(\rho,c), \big[\mathcal{K}, \mu(\rho,c)-\widetilde \mu\big]\D u, \big[\check{\mathcal{K}}, \mu(\rho,c)-\widetilde \mu\big]\D u}_{L^\infty(\R^3)}.
\]

To  estimate the last term above, we first observe that the $(\rho,c)$-dependent function  $f=f(\rho,c)$ satisfies 
\begin{equation*}
    \dpt f(\rho,c)+ u\cdot \nabla f(\rho,c)+ G(\rho,c)=-F.
\end{equation*}
We write the above equation along the velocity flow map $\mathcal{X}$ of the velocity and we obtain 
\begin{equation}\label{sec2:eq18}
\dpt f\big(\varrho(t,y),c_0(y)\big)+G\big(\varrho(t,y),c_0(y)\big)=-F\big(t,\mathcal{X}(t,y)\big),
\end{equation}
where $\varrho(t,y)=\rho(t,\mathcal{X}(t,y)$. Let $y_1$ and $y_2$ be two points located on the same side of $\Sigma$. Evaluating \eqref{sec2:eq18} at $y \in \{y_1, y_2\}$ and taking the difference, we obtain:
\begin{equation*}
    \dpt f\big(\varrho(t,y_j),c_0(y_1)\big)\Big|_{j=1}^{j=2} + g(t,y_1,y_2)f\big(\varrho(t,y_j),c_0(y_1)\big)\Big|_{j=1}^{j=2}= -F\big(t,\mathcal{X}(t,y_j)\big)\Big|_{j=1}^{j=2},
\end{equation*}
where (notice that $c_0(y_1)=c_0(y_2)$)
\[
g(t,y_1,y_2)=\frac{G\big(\varrho(t,y_j),c_0(y_1)\big)\Big|_{j=1}^{j=2}}{f\big(\varrho(t,y_j),c_0(y_1)\big)\Big|_{j=1}^{j=2}}.
\]
It follows that
\begin{align}
f\big(\varrho(t,y_j),c_0(y_1)\big)\Big|_{j=1}^{j=2}&=e^{-\int_0^t g(\tau,y_1,y_2)d\tau}f\big(\rho_0(y_j),c_0(y_1)\big)\Big|_{j=1}^{j=2}\notag\\
&-\int_0^t e^{\int_\tau^t g(\tau',y_1,y_2)d\tau' }F\big(\tau,\mathcal{X}(\tau,y_j)\big)\Big|_{j=1}^{j=2}d\tau,\label{sec2:eq12}
\end{align}
and 
\begin{align*}
    G\big(\varrho(t,y_j),c_0(y_1)\big)\Big|_{j=1}^{j=2}&=g(t,y_1,y_2) e^{-\int_0^t g(\tau,y_1,y_2)d\tau}f\big(\rho_0(y_j),c_0(y_1)\big)\Big|_{j=1}^{j=2}\\
    &-g(t,y_1,y_2)\int_0^t e^{-\int_\tau^t g(\tau',y_1,y_2)d\tau' }F\big(\tau,\mathcal{X}(\tau,y_j)\big)\Big|_{j=1}^{j=2}d\tau.
\end{align*}
Under the assumption \eqref{sec2:eq25}, we have  (recall that $c_0(y_1)=c_0(y_2)$)
\[
 g(t,y_1,y_2)\in \big[a_*(c_0(y_1)),a^*(c_0(y_1))\big],
\]
and hence
\begin{align}
    \abs{G\big(\varrho(t,y_j),c_0(y_1)\big)\Big|_{j=1}^{j=2}}&\leqslant a^*(c_0(y_1)) e^{- t a_*(c_0(y_1)) }\abs{f(\rho_0,c_0)}_{\dot \cC^\alpha_{\pw,\Sigma_0}(\R^3)}\abs{y_2-y_1}^\alpha\notag\\
    &+a^*(c_0(y_1))\int_0^t e^{-(t-\tau)a_*(c_0(y_1))}\abs{F(\tau)}_{\dot\cC^\alpha_{\pw,\Sigma(\tau)}(\R^3)}\abs{\mathcal{X}(\tau,y_2)-\mathcal{X}(\tau,y_1)}^{\alpha}d\tau.\label{sec2:eq5}
\end{align}
Recalling that the velocity flow map satisfies 
\[
\mathcal{X}(t,y_j)=y_j +\int_0^{t} u\big(s,\mathcal{X}(s,y_j)\big) d s,
\]
we deduce that for all $0\leqslant t'\leqslant t$
\[
\abs{\mathcal{X}(t',y_2)-\mathcal{X}(t',y_1)}\leqslant \abs{\mathcal{X}(t,y_2)-\mathcal{X}(t,y_1)}+\int_{t'}^{t} \norm{\nabla u(s)}_{L^\infty(\R^3)}\abs{\mathcal{X}(s,y_2)-\mathcal{X}(s,y_1)} d s,
\]
and Gr\"onwall's Lemma yields
\begin{equation*}
    \abs{\mathcal{X}(t',y_2)-\mathcal{X}(t',y_1)}\leqslant e^{\int_{t'}^t\norm{\nabla u(s)}_{L^\infty(\R^3)}ds} \abs{\mathcal{X}(t,y_2)-\mathcal{X}(t,y_1)}.
\end{equation*}
Returning to \eqref{sec2:eq5}, we find that
\begin{align}
    &\abs{G\big(\varrho(t,y_j),c_0(y_1)\big)\Big|_{j=1}^{j=2}}\notag\\
    &\leqslant a^*(c_0(y_1)) e^{- t a_*(c_0(y_1)) +\alpha\int_0^t\norm{\nabla u(s)}_{L^\infty(\R^3)}ds}\abs{f(\rho_0,c_0)}_{\dot \cC^\alpha_{\pw,\Sigma_0}(\R^3)}\abs{\mathcal{X}(t,y_2)-\mathcal{X}(t,y_1)}^\alpha\notag\\
    &+a^*(c_0(y_1))\abs{\mathcal{X}(t,y_2)-\mathcal{X}(t,y_1)}^\alpha\int_0^t e^{-(t-\tau) a_*(c_0(y_1))+\alpha\int_\tau^t\norm{\nabla u(s)}_{L^\infty(\R^3)}ds}\abs{F(\tau)}_{\dot\cC^\alpha_{\pw,\Sigma(\tau)}(\R^3)}d\tau,\notag
\end{align}
and hence
\begin{align*}
\abs{G(\rho(t),c(t))}_{\dot \cC^\alpha(\overline D(t))}&\leqslant \kappa^*(0) e^{- a_*(0) t+\alpha\int_0^t\norm{\nabla u(s)}_{L^\infty(\R^3)}ds}\abs{f(\rho_0,c_0)}_{\dot \cC^\alpha_{\pw,\Sigma_0}(\R^3)}\notag\\
    &+a^*(0)\int_0^t e^{-(t-\tau)a_*(0)+\alpha\int_\tau^t\norm{\nabla u(s)}_{L^\infty(\R^3)}ds}\abs{F(\tau)}_{\dot\cC^\alpha_{\pw,\Sigma(\tau)}(\R^3)}d\tau,
\end{align*}
and 
\begin{align*}
\abs{G(\rho(t),c(t))}_{\dot \cC^\alpha(\R^3\setminus D(t))}&\leqslant a^*(1) e^{- a_*(1) t+\alpha\int_0^t\norm{\nabla u(s)}_{L^\infty(\R^3)}ds}\abs{f(\rho_0,c_0)}_{\dot \cC^\alpha_{\pw,\Sigma_0}(\R^3)}\notag\\
    &+a^*(1)\int_0^t e^{-(t-\tau)a_*(1)+\alpha\int_\tau^t\norm{\nabla u(s)}_{L^\infty(\R^3)}ds}\abs{F(\tau)}_{\dot\cC^\alpha_{\pw,\Sigma(\tau)}(\R^3)}d\tau.
\end{align*}
Then, we deduce (recalling \eqref{sec2:eq12})
\begin{align}
\abs{f(\rho(t),c(t))}_{\dot \cC^\alpha_{\pw,\Sigma(t)}(\R^3)}&+\int_0^t\abs{G(\rho,c)}_{\dot \cC^\alpha_{\pw,\Sigma}(\R^3)}\notag\\
&\leqslant \overline{a} e^{\alpha\int_0^t\norm{\nabla u}_{L^\infty(\R^3)}}\Bigg(\abs{f(\rho_0,c_0)}_{\dot \cC^\alpha_{\pw,\Sigma_0}(\R^3)}+\int_0^t\abs{F(\tau)}_{\dot\cC^\alpha_{\pw,\Sigma}(\R^3)}\Bigg).\label{sec2:eq14}
\end{align}

To obtain an estimate for the piecewise H\"older semi-norm of 
\begin{equation}\label{sec2:eq15}
F=-(-\Delta)^{-1}\dvg (\rho \dot u)+\big[\check{\mathcal{K}}, \mu(\rho,c)-\widetilde\mu\big](2\D u),
\end{equation}
we use the embedding $\dot W^{1,3/(1-\alpha)}(\R^3)\hookrightarrow \dot\cC^\alpha(\R^3)$, and we follow the computations leading to  \eqref{sec2:eq6} and we arrive at 
{\small
\begin{align}
  &\abs{f(\rho(t),c(t))}_{\dot \cC^\alpha_{\pw,\Sigma(t)}(\R^3)}+\int_0^t\abs{G(\rho,c)}_{\dot \cC^\alpha_{\pw,\Sigma}(\R^3)}\notag\\
  &\leqslant  \overline a e^{\alpha \int_0^t\norm{\nabla u}_{L^\infty(\R^3)}}\Bigg(\abs{f(\rho_0,c_0)}_{\dot \cC^\alpha_{\pw,\Sigma_0}(\R^3)}+C\int_0^t\norm{\rho \dot u}_{L^{3/(1-\alpha)}(\R^3)}\Bigg)\notag\\
  &+C\overline a e^{\alpha\int_0^t\norm{\nabla u}_{L^\infty(\R^3)}}\int_0^t\Big(\abs{\mu(\rho,c)}_{\dot \cC^{\alpha}_{\pw,\Sigma}(\R^3)}+ \big(\ell^{-\alpha}_{\Sigma}+\mathfrak{P}_{\Sigma} \abs{\Sigma}_{\dot \cC^{1+\alpha}}\big)\norm{\llbracket \mu(\rho,c)\rrbracket}_{L^\infty(\Sigma)}\Big)\norm{ \nabla u}_{L^\infty(\R^3)}\notag\\
  &+C\overline a e^{\alpha\int_0^t\norm{\nabla u}_{L^\infty(\R^3)}}\int_0^t\Big( \abs{\mu(\rho,c)}_{\dot \cC^{\alpha/2}_{\pw,\Sigma}(\R^3)}+ \ell^{-\alpha/2}_{\Sigma}\norm{\llbracket \mu(\rho,c)\rrbracket}_{L^\infty(\Sigma)}\Big)\abs{\nabla u}_{\dot \cC^{\alpha/2}_{\pw,\Sigma}(\R^3)}\notag\\
  &+C\overline a e^{\alpha\int_0^t\norm{\nabla u}_{L^\infty(\R^3)}}\sup_{[0,t]}\norm{\llbracket \mu(\rho,c)\rrbracket}_{L^\infty(\Sigma)}\int_0^t\abs{\nabla u}_{\dot \cC^\alpha_{\pw,\Sigma}(\R^3)}. \label{sec3:eq1}
\end{align}
}
Now, we combine the above estimate with \eqref{sec2:eq8}, and we use the smallness condition on the shear viscosity jump \eqref{sec2:eq7} to derive (recall that $\absb{ \mu(\rho,c), \nu(\rho,c)}_{\dot \cC^\beta_{\pw,\Sigma}(\R^3)}\leqslant C_*\absb{f(\rho,c)}_{\dot \cC^\beta_{\pw,\Sigma}(\R^3)}$):
\begin{align}
    &\abs{f(\rho(t),c(t))}_{\dot \cC^\alpha_{\pw,\Sigma(t)}(\R^3)}+\int_0^t\abs{\nabla u, G(\rho,c)}_{\dot \cC^\alpha_{\pw,\Sigma}(\R^3)}\notag\\
    &\leqslant C_* e^{\alpha \int_0^t\norm{\nabla u}_{L^\infty(\R^3)}}\Bigg(\abs{f(\rho_0,c_0)}_{\dot \cC^\alpha_{\pw,\Sigma_0}(\R^3)}+\int_0^t\norm{\rho \dot u}_{L^{3/(1-\alpha)}(\R^3)}\Bigg)\notag\\
    &+C_* e^{\alpha\int_0^t\norm{\nabla u}_{L^\infty(\R^3)}} \int_0^t\Big(\abs{f(\rho,c)}_{\dot \cC^{\alpha}_{\pw,\Sigma}(\R^3)}+ \ell^{-\alpha}_{\Sigma}+\mathfrak{P}_{\Sigma} \abs{\Sigma}_{\dot \cC^{1+\alpha}}\Big)\Big( \psi+\norm{\nabla \mathbb P u,\, \nabla u}_{L^\infty(\R^3)}\Big)\notag\\
  &+C_* e^{\alpha\int_0^t\norm{\nabla u}_{L^\infty(\R^3)}}\int_0^t\Big( \abs{f(\rho,c)}_{\dot \cC^{\alpha/2}_{\pw,\Sigma}(\R^3)}+ \ell^{-\alpha/2}_{\Sigma}\norm{\llbracket \mu(\rho,c)\rrbracket}_{L^\infty(\Sigma)}\Big)\abs{\nabla u}_{\dot \cC^{\alpha/2}_{\pw,\Sigma}(\R^3)}.\notag
\end{align}
With the help of Young's and interpolation inequalities, we estimate the last term above as follows ($\eta>0$):
\begin{align*}
    &C_* e^{\alpha\int_0^t\norm{\nabla u}_{L^\infty(\R^3)}}\int_0^t\Big( \abs{f(\rho,c)}_{\dot \cC^{\alpha/2}_{\pw,\Sigma}(\R^3)}+ \ell^{-\alpha/2}_{\Sigma}\norm{\llbracket \mu(\rho,c)\rrbracket}_{L^\infty(\Sigma)}\Big)\abs{\nabla u}_{\dot \cC^{\alpha/2}_{\pw,\Sigma}(\R^3)}\\
    &\leqslant C_* e^{\alpha\int_0^t\norm{\nabla u}_{L^\infty(\R^3)}}\int_0^t\Big(\abs{f(\rho,c)}_{\dot \cC^{\alpha/2}_{\pw,\Sigma}(\R^3)}+ \ell^{-\alpha/2}_{\Sigma}\norm{\llbracket \mu(\rho,c)\rrbracket}_{L^\infty(\Sigma)}\Big)\norm{\nabla u}_{L^\infty(\R^3)}^{\frac{1}{2}}\abs{\nabla u}_{\dot \cC^{\alpha}_{\pw,\Sigma}(\R^3)}^{\frac{1}{2}}\\
    &\leqslant \eta \int_0^t \abs{\nabla u}_{\dot \cC^{\alpha}_{\pw,\Sigma}(\R^3)} +\frac{C_*}{\eta} e^{2\alpha\int_0^t\norm{\nabla u}_{L^\infty(\R^3)}}\int_0^t\Big( \abs{f(\rho,c)}_{\dot \cC^{\alpha/2}_{\pw,\Sigma}(\R^3)}^2+ \ell^{-\alpha}_{\Sigma}\norm{\llbracket \mu(\rho,c)\rrbracket}_{L^\infty(\Sigma)}^2\Big)\norm{\nabla u}_{L^\infty(\R^3)}\\
    &\leqslant \eta \int_0^t \abs{\nabla u}_{\dot \cC^{\alpha}_{\pw,\Sigma}(\R^3)} +\frac{C_*}{\eta} e^{2\alpha\int_0^t\norm{\nabla u}_{L^\infty(\R^3)}}\int_0^t\Big( \abs{f(\rho,c)}_{\dot \cC^{\alpha}_{\pw,\Sigma}(\R^3)}+ \ell^{-\alpha}_{\Sigma}\Big)\norm{\nabla u}_{L^\infty(\R^3)}.
\end{align*}
Choosing $\eta$ small, we finally obtain:
{\small
\begin{align}
    \abs{f(\rho(t),c(t))}_{\dot \cC^\alpha_{\pw,\Sigma(t)}(\R^3)}&+\int_0^t\abs{\nabla u, P(\rho,c)}_{\dot \cC^\alpha_{\pw,\Sigma}(\R^3)}\notag\\
    &\leqslant C_* e^{2\alpha \int_0^t\norm{\nabla u}_{L^\infty(\R^3)}}\Bigg[\abs{f(\rho_0,c_0)}_{\dot \cC^\alpha_{\pw,\Sigma_0}(\R^3)}+\int_0^t\norm{\rho \dot u}_{L^{3/(1-\alpha)}(\R^3)}\notag\\
    &+ \int_0^t\Big(\abs{f(\rho,c)}_{\dot \cC^{\alpha}_{\pw,\Sigma}(\R^3)}+ \ell^{-\alpha}_{\Sigma}+\mathfrak{P}_{\Sigma} \abs{\Sigma}_{\dot \cC^{1+\alpha}}\Big)\Big( \psi+\norm{\nabla \mathbb P u, \nabla u}_{L^\infty(\R^3)}\Big)\Bigg].\label{sec2:eq9}
\end{align}
}
We now proceed to estimate $\abs{\Sigma(t)}_{\dot \cC^{1+\alpha}}$, $\ell_{\Sigma(t)}^{-\alpha}$ and $\mathfrak{P}_{\Sigma(t)}$ as they appear in the above bound.

First, let $\big(\tau_{0,j}, B_j\big)$, $j = 1, \dots, J$, denote the charts introduced in Definition \ref{def:interface} corresponding to the initial interface $\Sigma_0$. Since $\Sigma(t)$ is a free surface, it is covered by the family of charts $\big(\tau_j(t), B_j\big)$, where 
\[
\tau_j(t,s)= \tau_{0,j}(s)+\int_0^t u\big(t',\tau_j(t',s)\big) dt'.
\]
From this, it follows that 
\[
\norm{\nabla \tau_j(t)}_{L^\infty}\leqslant \norm{\nabla \tau_{0,j}}_{L^\infty} e^{\int_0^t\norm{\nabla u}_{L^\infty(\R^3)}}, \quad \abs{\Sigma(t)}_{\text{inf}}\geqslant  \abs{\Sigma_0}_{\text{inf}} e^{-\int_0^t \norm{\nabla u}_{L^\infty(\R^3)}}
\]
and
\[
\abs{\nabla\tau_j(t)}_{\dot \cC^\alpha}\leqslant  \abs{\nabla \tau_{0,j}}_{\dot \cC^\alpha}+e^{\alpha\int_0^t\norm{\nabla u}_{L^\infty(\R^3)}}\int_0^t\Big(\norm{\nabla u}_{\dot \cC^\alpha_{\pw,\Sigma}(\R^3)}+\norm{\nabla u}_{L^\infty(\R^3)}\abs{\nabla\tau_j}_{\dot \cC^\alpha}\Big).
\]
We conclude that 
\[
\mathfrak{P}_{\Sigma(t)}\leqslant \mathfrak{P}_{\Sigma_0} e^{C\int_0^t \norm{\nabla u}_{L^\infty(\R^3)}}
\]
and 
\[
\big|\Sigma(t)\big|_{\dot \cC^{1+\alpha}}\leqslant  \big|\Sigma_0\big|_{\dot \cC^{1+\alpha}}+e^{\alpha\int_0^t\norm{\nabla u}_{L^\infty(\R^3)}}\int_0^t\Big(\norm{\nabla u}_{\dot \cC^\alpha_{\pw,\Sigma}(\R^3)}+\norm{\nabla u}_{L^\infty(\R^3)}\big|\Sigma\big|_{\dot \cC^{1+\alpha}}\Big).
\]

On the other hand,  we have  
\[
D(t)=\big\{x\in \R^3\colon \varphi(t,x)>0\big\}, \text{ where } \varphi(t)=  \varphi_0\circ \mathcal{X}^{-1}(t).
\]
It is not difficult to prove
\[
\abs{\nabla \varphi(t)}_{\text{inf}}\geqslant \abs{\nabla \varphi_0}_{\text{inf}} e^{-\int_0^t\norm{\nabla u}_{L^\infty(\R^3)}},
\]
and 
\[
\abs{\nabla \varphi(t)}_{\dot\cC^\alpha}\leqslant \abs{\nabla \varphi_{0}}_{\dot \cC^\alpha}+e^{\alpha\int_0^t\norm{\nabla u}_{L^\infty(\R^3)}}\int_0^t\Big(\norm{\nabla u}_{\dot \cC^\alpha_{\pw,\Sigma}(\R^3)}+\norm{\nabla u}_{L^\infty(\R^3)}\abs{\nabla\varphi}_{\dot \cC^\alpha}\Big).
\]

 Finally, we come back to  \eqref{sec2:eq9}  and we obtain
\begin{align}
\check\vartheta(t) &\leqslant C_* e^{C \int_0^t\norm{\nabla u}_{L^\infty(\R^3)}}\Bigg(\big|\Sigma_0\big|_{\dot \cC^{1+\alpha}}+\abs{\nabla \varphi_{0}}_{\dot \cC^\alpha}+\abs{f(\rho_0,c_0)}_{\dot \cC^\alpha_{\pw,\Sigma_0}(\R^3)}+\int_0^t\norm{\rho \dot u}_{L^{3/(1-\alpha)}(\R^3)}\Bigg)\notag\\
    &+C_* e^{C\int_0^t\norm{\nabla u}_{L^\infty(\R^3)}} \int_0^t\check\psi\Big(\abs{f(\rho,c)}_{\dot \cC^{\alpha}_{\pw,\Sigma}(\R^3)}+ \abs{\nabla \varphi}_{\dot\cC^\alpha}+ \abs{\Sigma}_{\dot \cC^{1+\alpha}}\Big),\label{sec2:eq10}
\end{align}
where 
\begin{equation*}
\check\psi= K_0 \Big( \psi+\norm{\nabla\mathbb P u, \nabla u}_{L^\infty(\R^3)}\Big),\;\text{ and }\; K_0=1+\abs{\nabla\varphi_0}_{\text{inf}}^{-1}+\mathfrak{P}_{\Sigma_0}.
\end{equation*}
 Gr\"onwall's Lemma yields
\begin{align}
    \check\vartheta(t)
    &\leqslant C_*\Bigg( \abs{ \Sigma_{0}}_{\dot \cC^{1+\alpha}}+\abs{\nabla \varphi_0}_{\dot\cC^\alpha}+\abs{f(\rho_0,c_0)}_{\dot \cC^\alpha_{\pw,\Sigma_0}(\R^3)}+\int_0^t\norm{\rho \dot u}_{L^{3/(1-\alpha)}(\R^3)}\Bigg)\notag\\
    &\times \exp\Bigg(C\int_0^t\norm{\nabla u}_{L^\infty(\R^3)}+C_* e^{C\int_0^t\norm{\nabla u}_{L^\infty(\R^3)}}\int_0^t \check\psi(s)ds\Bigg).\notag
\end{align}
Estimate \eqref{sec2:pwholder} finally follows from the fact that  
\begin{align}
\int_0^t \psi(s)ds\leqslant C_*\vartheta(t)+C_*\int_0^t\norm{ \rho\dot u}_{L^6(\R^3)}^{1/2}\norm{\rho \dot u}_{L^2(\R^3)}^{1/2}.\label{sec2:eq13}
\end{align}
This concludes the first part of the proof of Lemma \ref{sec2:lem1}.
\end{proof}
\begin{proof}[Proof of \eqref{sec2:Lipschtz}]
In this second part of the proof of Lemma \ref{sec2:lem1}, we aim to obtain a Lipschitz estimate for the velocity field. To this end, we proceed by estimating each term appearing in \eqref{sec2:eq4}. For clarity, our estimates will be expressed in terms of the functional $\phi_\Sigma$ and $\check{\phi}_\Sigma$ defined in \eqref{sec2:eq27} above.
\begin{itemize}
\item \textbf{First, second and fourth terms}: Similarly to \eqref{sec2:eq13}, we apply the Gagliardo–Nirenberg's, Sobolev's and Young's inequalities to obtain:
\begin{gather}\label{sec2:eq31}
\int_0^t\norm{\frac{1}{\mu(\rho,c)}\nabla(-\Delta)^{-1}\mathbb P (\rho \dot u)}_{L^\infty(\R^3)}\leqslant C_*\int_0^t\Big(\norm{\sqrt\rho \dot u}_{L^2(\R^3)}+\norm{\rho \dot u}_{L^{6/(2-\alpha)}(\R^3)}\Big).
\end{gather}
Next, we express the second term in \eqref{sec2:eq4} as follows:
\begin{align*}
\nabla (-\Delta)^{-1}\nabla\bigg(\frac{1}{\nu(\rho,c)}(-\Delta)^{-1}\dvg (\rho \dot u)\bigg)&=\Big[\nabla (-\Delta)^{-1}\nabla,\frac{1}{\nu(\rho,c)}\Big](-\Delta)^{-1}\dvg (\rho \dot u)\\
&+\frac{1}{\nu(\rho,c)}\nabla (-\Delta)^{-1}\nabla(-\Delta)^{-1}\dvg (\rho \dot u).
\end{align*}
Estimate \eqref{sec2:eq31} also applies to the last term above. For the remaining term, we use the commutator estimate \eqref{ap:eq10} with $\alpha'=\alpha/2$, $p=r\in (6/\alpha,\infty)$, and we obtain:
\begin{align}
    &\int_0^t\norm{\Big[\nabla (-\Delta)^{-1}\nabla,\frac{1}{\nu(\rho,c)}\Big](-\Delta)^{-1}\dvg (\rho \dot u)}_{L^\infty(\R^3)}\notag\\
    &\leqslant \frac{C}{\mu_*^2}\int_0^t\Big(\abs{\nu(\rho,c)}_{\dot \cC^{\alpha/2}_{\pw,\Sigma}(\R^3)}+ \ell_{\Sigma}^{-\alpha/2}\norm{\nu(\rho,c)-\widetilde\nu}_{L^\infty(\R^3)}\Big)\norm{(-\Delta)^{-1}\dvg (\rho \dot u)}_{L^{p}(\R^3)}\notag\\
    &+\frac{C}{\mu_*^2}\int_0^t\norm{\llbracket\nu(\rho,c)\rrbracket}_{L^\infty(\Sigma)}\norm{(-\Delta)^{-1}\dvg(\rho \dot u)}_{\cC^{\alpha/2}(\R^3)}\notag\\
    &\leqslant C_* \int_0^t \phi_\Sigma\norm{(-\Delta)^{-1}\dvg (\rho \dot u)}_{L^{p}(\R^3)}+C_*\int_0^t\Big(\norm{\sqrt\rho \dot u}_{L^2(\R^3)}+\norm{\rho \dot u}_{L^{6/(2-\alpha)}(\R^3)}\Big).\label{sec2:eq24}
\end{align}
We also apply \eqref{ap:eq10} to estimate the second-to-last term in \eqref{sec2:eq4} as follows:
    \begin{align}
        \int_0^t\norm{\frac{1}{\mu(\rho,c)}\big[\mathcal{K},\mu(\rho,c)-\widetilde \mu\big](2\D u)}_{L^\infty(\R^3)}&\leqslant C_*\int_0^t\phi_\Sigma\norm{\nabla u}_{L^{p}(\R^3)}\notag\\
        &+ \frac{C}{\mu_*}\sup_{[0,t]}\norm{\llbracket \mu(\rho,c)\rrbracket}_{L^\infty(\Sigma)}\int_0^t\norm{\nabla u}_{\cC^{\alpha/2}_{\pw,\Sigma}(\R^3)}.\label{sec2:eq19}
    \end{align}
\item \textbf{Remaining terms}:
We express the third and last terms in \eqref{sec2:eq4} as follows:
    \begin{align}\label{sec2:eq49}
-\nabla (-\Delta)^{-1}\nabla &\Bigg(\frac{1}{\nu(\rho,c)}\Big(G(\rho,c)+\big[\check{\mathcal{K}},\mu(\rho,c)-\widetilde \mu\big](2\D u)\Big)\Bigg)\notag\\
&=-\Big[\nabla (-\Delta)^{-1}\nabla, \frac{1}{\nu(\rho,c)}\Big]\Big(G(\rho,c)+\big[\check{\mathcal{K}},\mu(\rho,c)-\widetilde \mu\big](2\D u)\Big) \notag\\
&-\frac{1}{\nu(\rho,c)}\nabla (-\Delta)^{-1}\nabla \Big(G(\rho,c)+\big[\check{\mathcal{K}},\mu(\rho,c)-\widetilde \mu\big](2\D u)\Big).
    \end{align}
Similarly as in the preceding step, we apply \eqref{ap:eq10} and we obtain 
\begin{align*}
    &\int_0^t\norm{\Big[\nabla (-\Delta)^{-1}\nabla,\frac{1}{\nu(\rho,c)}\Big]\Big(G(\rho,c)+\big[\check{\mathcal{K}},\mu(\rho,c)-\widetilde \mu\big](2\D u)\Big)}_{L^\infty(\R^3)}\\
    &\leqslant \frac{C}{\mu_*}\int_0^t\phi_\Sigma\Big(\norm{G(\rho,c)}_{L^p(\R^3)}+\norm{\mu(\rho,c)-\widetilde\mu}_{L^\infty(\R^3)}\norm{\nabla u}_{L^{p}(\R^3)}\Big)\\
    &+\frac{C}{\mu_*^2}\sup_{[0,t]}\norm{\llbracket\nu(\rho,c)\rrbracket}_{L^\infty(\Sigma)}\int_0^t\norm{G(\rho,c)+\big[\check{\mathcal{K}},\mu(\rho,c)-\widetilde \mu\big](2\D u)}_{\cC^{\alpha/2}_{\pw,\Sigma}(\R^3)}\\
    &\leqslant \frac{C}{\mu_*}\int_0^t\phi_\Sigma\Big(\norm{G(\rho,c)}_{L^p(\R^3)}+\norm{\mu(\rho,c)-\widetilde\mu}_{L^\infty(\R^3)}\norm{\nabla u}_{L^{p}(\R^3)}\Big)+\frac{C\mu^*}{\mu_*^2}\vartheta(t)\sup_{[0,t]}\norm{\llbracket\nu(\rho,c)\rrbracket}_{L^\infty(\Sigma)}\\
    &+\frac{C}{\mu_*^2}\sup_{[0,t]}\norm{\llbracket\nu(\rho,c)\rrbracket}_{L^\infty(\Sigma)}\int_0^t\abs{G(\rho,c)+\big[\check{\mathcal{K}},\mu(\rho,c)-\widetilde \mu\big](2\D u)}_{\dot \cC^{\alpha/2}_{\pw,\Sigma}(\R^3)}.
\end{align*}
Applying \eqref{ap:eq41} with $\beta=\alpha/2$, we obtain:
\begin{align}\label{sec2:eq51}
\int_0^t \abs{\big[\check{\mathcal{K}},\mu(\rho,c)-\widetilde \mu\big](2\D u)}_{\dot \cC^{\alpha/2}_{\pw,\Sigma}(\R^3)}&\leqslant C\Big(\abs{\mu(\rho,c)}_{\dot\cC^\alpha(\R^3)}+ \ell_{\Sigma}^{-\alpha}\norm{\llbracket \mu(\rho,c)\rrbracket}_{L^\infty(\Sigma)}\Big)\norm{\nabla u}_{L^{6/\alpha}(\R^3)}\notag\\
&+C\int_0^t\Big(\abs{\llbracket \mu(\rho,c)\rrbracket }_{\dot \cC^{\alpha/2}(\Sigma)}+ \big(\ell^{-\alpha/2}_{\Sigma}+\mathfrak{P}_{\Sigma} \abs{\Sigma}_{\dot \cC^{1+\alpha/2}}\big)\norm{\llbracket \mu(\rho,c)\rrbracket}_{L^\infty(\Sigma)}\Big)\norm{\nabla u}_{L^\infty(\R^3)}\notag\\
&+ C\int_0^t\norm{\llbracket\mu(\rho,c)\rrbracket}_{L^\infty(\Sigma)}\abs{\nabla u}_{\dot \cC^{\alpha/2}_{\pw,\Sigma}(\R^3)}, 
\end{align}
and following the computations leading to \eqref{sec2:eq14}, we arrive at: 
\begin{align}
    &\abs{f(\rho(t),c(t))}_{\dot \cC^{\alpha/2}_{\pw,\Sigma(t)}(\R^3)}+\int_0^t\abs{G(\rho,c)}_{\dot \cC^{\alpha/2}_{\pw,\Sigma}(\R^3)}\notag\\
&\leqslant \overline{a} e^{\alpha\int_0^t\norm{\nabla u}_{L^\infty(\R^3)}}\Bigg(\abs{f(\rho_0,c_0)}_{\dot \cC^{\alpha/2}_{\pw,\Sigma_0}(\R^3)}+\int_0^t\abs{F(\tau)}_{\dot\cC^{\alpha/2}_{\pw,\Sigma}(\R^3)}\Bigg)\notag\\
&\leqslant C\overline{a} e^{\alpha\int_0^t\norm{\nabla u}_{L^\infty(\R^3)}}\Bigg(\abs{f(\rho_0,c_0)}_{\dot \cC^{\alpha/2}_{\pw,\Sigma_0}(\R^3)}+\int_0^t\big(\norm{\rho \dot u}_{L^{6/(2-\alpha)}(\R^3)}+\norm{\llbracket \mu(\rho,c)\rrbracket}_{L^\infty(\Sigma)}\abs{\nabla u}_{\dot \cC^{\alpha/2}_{\pw,\Sigma}(\R^3)}\big)\Bigg)\notag\\
&+C\overline{a} e^{\alpha\int_0^t\norm{\nabla u}_{L^\infty(\R^3)}}\sup_{[0,t]}\Big(\abs{\mu(\rho,c)}_{\dot\cC^\alpha(\R^3)}+ \ell_{\Sigma}^{-\alpha}\norm{\llbracket \mu(\rho,c)\rrbracket}_{L^\infty(\Sigma)}\Big)\int_0^t\norm{\nabla u}_{L^{6/\alpha}(\R^3)}\notag\\
&+C\overline{a} e^{\alpha\int_0^t\norm{\nabla u}_{L^\infty(\R^3)}}\int_0^t\Big(\abs{\llbracket \mu(\rho,c)\rrbracket }_{\dot \cC^{\alpha/2}(\Sigma)}+ \big(\ell^{-\alpha/2}_{\Sigma}+\mathfrak{P}_{\Sigma} \abs{\Sigma}_{\dot \cC^{1+\alpha/2}}\big)\norm{\llbracket \mu(\rho,c)\rrbracket}_{L^\infty(\Sigma)}\Big)\norm{\nabla u}_{L^\infty(\R^3)}.\label{sec2:eq50}
\end{align}
Gathering all the above estimates, we deduce the following bound for the first term in the expression \eqref{sec2:eq49}:
\begin{align}
    &\int_0^t\norm{\Big[\nabla (-\Delta)^{-1}\nabla,\frac{1}{\nu(\rho,c)}\Big]\Big(G(\rho,c)+\big[\check{\mathcal{K}},\mu(\rho,c)-\widetilde \mu\big](2\D u)\Big)}_{L^\infty(\R^3)}\notag\\
    &\leqslant \frac{C\mu^*}{\mu_*^2}\vartheta(t)\sup_{[0,t]}\norm{\llbracket\nu(\rho,c)\rrbracket}_{L^\infty(\Sigma)}+\frac{C}{\mu_*}\int_0^t\phi_\Sigma\Big(\norm{G(\rho,c)}_{L^p(\R^3)}+\norm{\mu(\rho,c)-\widetilde\mu}_{L^\infty(\R^3)}\norm{\nabla u}_{L^{p}(\R^3)}\Big)\notag\\
    &+\frac{C\overline{a}}{\mu_*^2}\sup_{[0,t]}\norm{\llbracket\nu(\rho,c)\rrbracket}_{L^\infty(\Sigma)}e^{\alpha\int_0^t\norm{\nabla u}_{L^\infty(\R^3)}}\Bigg\{\abs{f(\rho_0,c_0)}_{\dot \cC^{\alpha/2}_{\pw,\Sigma_0}(\R^3)}+\int_0^t\norm{\rho \dot u}_{L^{6/(2-\alpha)}(\R^3)}\notag\\
    &+\sup_{[0,t]}\norm{\llbracket \mu(\rho,c)\rrbracket}_{L^\infty(\Sigma)}\int_0^t\abs{\nabla u}_{\dot \cC^{\alpha/2}_{\pw,\Sigma}(\R^3)}+\sup_{[0,t]}\Big(\abs{\mu(\rho,c)}_{\dot\cC^\alpha(\R^3)}+ \ell_{\Sigma}^{-\alpha}\norm{\llbracket \mu(\rho,c)\rrbracket}_{L^\infty(\Sigma)}\Big)\int_0^t\norm{\nabla u}_{L^{6/\alpha}(\R^3)}\notag\\
    &+\sup_{[0,t]}\Big(\abs{\llbracket \mu(\rho,c)\rrbracket }_{\dot \cC^{\alpha/2}(\Sigma)}+ \big(\ell^{-\alpha/2}_{\Sigma}+\mathfrak{P}_{\Sigma} \abs{\Sigma}_{\dot \cC^{1+\alpha/2}}\big)\norm{\llbracket \mu(\rho,c)\rrbracket}_{L^\infty(\Sigma)}\Big)\int_0^t\norm{\nabla u}_{L^\infty(\R^3)}\Bigg\}.\label{sec2:eq17}
\end{align}

    For the remaining term of \eqref{sec2:eq49}, we shall write 
    \[
    w:= G(\rho,c)+\big[\check{\mathcal{K}},\mu(\rho,c)-\widetilde \mu\big](2\D u)=w^e+ (w-w^e),
    \] 
    where $w^e$ denotes a $\cC^\alpha$-extension of  $w$ that verifies \eqref{ap:eq11}-\eqref{ap:eq5}-\eqref{ap:eq15}-\eqref{ap:eq9}. With the help of \eqref{ap:eq17}, we estimate
    \begin{align}
        \int_0^t \norm{\frac{1}{\nu(\rho,c)}\nabla (-\Delta)^{-1}\nabla\big(w-w^e\big)}_{L^\infty(\R^3)}&\leqslant \frac{C}{\mu_*} \int_0^t \Big(\ell_{\Sigma}^{-\frac{3}{p}}\norm{w-w^e}_{L^{p}(\R^3)}+\norm{w-w^e}_{\cC^{\alpha/2}_{\pw,\Sigma}(\R^3)} \Big)\notag\\
        &\leqslant \frac{C}{\mu_*} \int_0^t \Big(\ell_{\Sigma}^{-\frac{\alpha}{2}}\norm{\llbracket w\rrbracket}_{L^\infty\cap L^{p}(\Sigma)}
        +\abs{\llbracket w\rrbracket}_{\dot \cC^{\alpha/2}(\Sigma)}\Big).\label{sec2:eq11}
    \end{align}
Then, we use the logarithmic interpolation inequality 
\begin{gather}\label{sec2:eq33}
\norm{\nabla (-\Delta)^{-1}\nabla w^e}_{L^\infty(\R^3)}\leqslant C \norm{w^e}_{L^\infty(\R^3)}\Bigg(1+\log^+\Bigg(\frac{\abs{w^e}_{\dot \cC^{\alpha/2}(\R^3)}}{\norm{ w^e}_{L^\infty(\R^3)}}\Bigg)\Bigg)
\end{gather}
and we obtain  (with the help of \eqref{ap:eq11}-\eqref{ap:eq9})
\begin{align*}
 \norm{\frac{1}{\nu(\rho,c)}\nabla (-\Delta)^{-1}\nabla w^e}_{L^\infty(\R^3)}\leqslant \frac{C}{\mu_*}\norm{w}_{L^\infty(\R^3)} \Bigg(1+\log^+\Bigg(\frac{\abs{w}_{\dot \cC^{\alpha/2}_{\pw,\Sigma}(\R^3)}+\ell_{\Sigma}^{-\alpha/2}\norm{\llbracket w\rrbracket}_{L^\infty(\Sigma)}}{\norm{w}_{L^\infty(\R^3)}}\Bigg)\Bigg).
\end{align*}
Following  the computations that lead to \eqref{sec2:eq14}, we arrive at (here $\underline{a}_*=\min_c a_*(c)$)
\begin{align*}
    &\abs{G(\rho,c)(t)}_{\dot \cC^{\alpha/2}_{\pw,\Sigma(t)}(\R^3)}\\
    &\leqslant C_*e^{\alpha\int_0^t\norm{\nabla u}_{L^\infty(\R^3)}}\Bigg(e^{-\underline{a}_* t}\abs{f(\rho_0,c_0)}_{\dot \cC^{\alpha/2}_{\pw,\Sigma_0}(\R^3)}+\int_0^t e^{-\underline{a}_* (t-s)}\abs{F(s)}_{\dot \cC^{\alpha/2}_{\pw,\Sigma(s)}(\R^3)}ds\bigg),
\end{align*}
and \eqref{sec2:eq51} implies 
\begin{align*}
    &\abs{w(t)}_{\dot \cC^{\alpha/2}_{\pw,\Sigma(t)}(\R^3)}\leqslant C_*e^{\alpha\int_0^t\norm{\nabla u}_{L^\infty(\R^3)}}\Bigg\{ e^{-\underline{a}_* t}\abs{f(\rho_0,c_0)}_{\dot \cC^{\alpha/2}_{\pw,\Sigma_0}(\R^3)}+\int_0^te^{-\underline{a}_* (t-s)}\norm{\rho \dot u}_{L^{6/(2-\alpha)}(\R^3)}\\
    &+\abs{\big[\check{\mathcal{K}},\mu(\rho,c)-\widetilde \mu\big](2\D u)(t)}_{\dot \cC^{\alpha/2}_{\pw,\Sigma(t)}(\R^3)}+\int_0^t e^{-\underline{a}_* (t-s)}\abs{\big[\check{\mathcal{K}},\mu(\rho,c)-\widetilde \mu\big](2\D u)(s)}_{\dot \cC^{\alpha/2}_{\pw,\Sigma(s)}(\R^3)}ds\Bigg\}.
\end{align*}
It follows that 
\begin{align}
    &\int_0^t\norm{\frac{1}{\nu(\rho,c)}\nabla (-\Delta)^{-1}\nabla w^e}_{L^\infty(\R^3)}\leqslant C_*\int_0^t \norm{w(s)}_{L^\infty(\R^3)}\Bigg\{1+ \int_0^s \norm{\nabla u}_{L^\infty(\R^3)}\notag\\
    &+\log^+\Bigg[\frac{C_*}{\norm{w(s)}_{L^\infty(\R^3)}}\Bigg(e^{-\underline{a}_* s}\abs{f(\rho_0,c_0)}_{\dot \cC^{\alpha/2}_{\pw,\Sigma_0}(\R^3)}+\ell_{\Sigma(s)}^{-\alpha/2}\norm{\llbracket w(s)\rrbracket}_{L^\infty(\Sigma(s))}+\int_0^se^{-\underline{a}_* (s-s')}\norm{\dot u(s')}_{L^{6/(2-\alpha)}(\R^3)}ds'\notag\\
    &+\abs{\big[\check{\mathcal{K}},\mu(\rho,c)-\widetilde \mu\big](2\D u)(s)}_{\dot \cC^{\alpha/2}_{\pw,\Sigma(s)}(\R^3)}+ \int_0^se^{-\underline{a}_* (s-s')}\abs{\big[\check{\mathcal{K}},\mu(\rho,c)-\widetilde \mu\big](2\D u)}_{\dot \cC^{\alpha/2}_{\pw,\Sigma}(\R^3)}(s')ds'\Bigg)\Bigg]\Bigg\}\notag.
\end{align}
Next, using the elementary bound $\log^+(x)\leqslant \frac{2}{e}\sqrt{x}$, H\"older's inequality and \eqref{sec2:eq51} we infer that
\begin{align}
    &\int_0^t\norm{\frac{1}{\nu(\rho,c)}\nabla (-\Delta)^{-1}\nabla w^e}_{L^\infty(\R^3)}\notag\\
 &\leqslant C_*\abs{f(\rho_0,c_0)}_{\dot \cC^{\alpha/2}_{\pw,\Sigma_0}(\R^3)}+C_*\int_0^t\Big(\norm{\rho \dot u}_{L^{6/(2-\alpha)}(\R^3)}+\norm{w}_{L^\infty(\R^3)} \big(1+\norm{\nabla u}_{L^1((0,s),L^\infty(\R^3))}\big)\Big)ds\notag\\
 &+C_*\Bigg(\int_0^t\norm{w}_{L^\infty(\R^3)}\Bigg)^{1/2} \Bigg(\int_0^t\bigg(\ell_{\Sigma}^{-\alpha/2}\norm{\llbracket w\rrbracket}_{L^\infty(\Sigma)} + \abs{\big[\check{\mathcal{K}},\mu(\rho,c)-\widetilde \mu\big](2\D u)}_{\dot \cC^{\alpha/2}_{\pw,\Sigma}(\R^3)}\bigg)\Bigg)^{1/2}\notag\\
 &\leqslant  C_*\abs{f(\rho_0,c_0)}_{\dot \cC^{\alpha/2}_{\pw,\Sigma_0}(\R^3)}+C_*\int_0^t\Big(\norm{\rho \dot u}_{L^{6/(2-\alpha)}(\R^3)}+\norm{w}_{L^\infty(\R^3)} \big(1+\norm{\nabla u}_{L^1((0,s),L^\infty(\R^3))}\big)\Big)ds\notag\\
 &+C_*\Bigg(\int_0^t\norm{w}_{L^\infty(\R^3)}\Bigg)^{1/2} \Bigg\{\int_0^t\Big(\ell_{\Sigma}^{-\alpha/2}\norm{\llbracket w\rrbracket}_{L^\infty(\Sigma)}+\norm{\llbracket \mu(\rho,c)\rrbracket}_{L^\infty(\Sigma)}\abs{\nabla u}_{\dot \cC^{\alpha/2}_{\pw,\Sigma}(\R^3)}\Big)\notag\\
 &+\vartheta(t)\sup_{[0,t]}\Big(\abs{\llbracket \mu(\rho,c)\rrbracket }_{\dot \cC^{\alpha/2}(\Sigma)}+ \big(\ell^{-\alpha/2}_{\Sigma}+\mathfrak{P}_{\Sigma} \abs{\Sigma}_{\dot \cC^{1+\alpha/2}}\big)\norm{\llbracket \mu(\rho,c)\rrbracket}_{L^\infty(\Sigma)}\Big)\notag\\
 &+\sup_{[0,t]}\Big(\abs{\mu(\rho,c)}_{\dot\cC^\alpha(\R^3)}+ \ell_{\Sigma}^{-\alpha}\norm{\llbracket \mu(\rho,c)\rrbracket}_{L^\infty(\Sigma)}\Big)\int_0^t\norm{\nabla u}_{L^{6/\alpha}(\R^3)}\Bigg\}^{1/2}.\label{sec2:eq21}
\end{align}

We will now estimate the $L^\infty$-bound of $w$ and $\llbracket w\rrbracket$ appearing in the above estimate and \eqref{sec2:eq11}. To achieve this, we go back to \eqref{sec2:eq18} and we write
\begin{equation}\label{sec2:eq28}
    f\big(\varrho(t,y),c_0(y)\big)= e^{-\int_0^t \check{g}(t',y)dt'} f\big(\rho_0(y),c_0(y)\big) -\int_0^t e^{-\int_{t'}^t \check{g}(\tau,y)d\tau}  F\big(t', \mathcal{X}(t',y)\big)d t',
\end{equation}
where (from \eqref{sec2:eq25})
\[
\check{g}(t,y)=\frac{G\big(\varrho(t,y),c_0(y)\big)}{f\big(\varrho(t,y),c_0(y)\big)}\in \big[a_*(c_0(y)), a^*(c_0(y))\big].
\]
We can now readily deduce that
\begin{align*}
\norm{f(\rho,c)(t)}_{L^\infty(\R^3)}+\int_0^t \norm{G(\rho,c)}_{L^\infty(\R^3)}\leqslant \overline{a}\norm{f(\rho_0,c_0)}_{L^\infty(\R^3)}+\overline{a}\int_0^t \norm{F(t')}_{L^\infty(\R^3)}dt'.
\end{align*}
Recalling the expression of $F$ in \eqref{sec2:eq15} and applying the Gagliardo-Nirenberg's inequality together 
with \eqref{sec2:eq19}, we obtain
\begin{align}
    \norm{f(\rho,c)(t)}_{L^\infty(\R^3)}&+\int_0^t \norm{G(\rho,c),[\check{\mathcal{K}}, \mu(\rho,c)-\widetilde\mu]\D u}_{L^\infty(\R^3)}\leqslant \overline{a}\norm{f(\rho_0,c_0)}_{L^\infty(\R^3)}\notag\\
    &+C\overline{a}\int_0^t \big(\norm{\rho \dot u}_{L^2(\R^3)}+\norm{\rho \dot u}_{L^{6/(2-\alpha)}(\R^3)}+\phi_\Sigma\norm{\nabla u}_{L^{p}(\R^3)}\big)\notag\\
        &+ C\overline{a}\sup_{[0,t]}\norm{\llbracket \mu(\rho,c)\rrbracket}_{L^\infty(\Sigma)}\int_0^t\norm{\nabla u}_{\cC^{\alpha/2}_{\pw,\Sigma}(\R^3)}.\label{sec2:eq30}
\end{align}
 
Next, we take the jump across the discontinuity surface in \eqref{sec2:eq2} and we obtain:
\begin{gather}\label{sec2:eq35}
\left\llbracket \nu(\rho,c)\dvg u -G(\rho,c)\right\rrbracket=\Big\llbracket \big[\check{\mathcal{K}},\mu(\rho,c)-\widetilde \mu\big](2\D u)\Big\rrbracket.
\end{gather}
We observe that the balance of forces at the discontinuity surface suggests that the normal component of the stress tensor is continuous, namely: 
\begin{gather}\label{sec2:eq45}
\llbracket 2\mu(\rho,c)\D u\cdot n_x+\big(\lambda(\rho,c)\dvg u- G(\rho,c)\big) n_x\rrbracket=0.
\end{gather}
where $n_x$ denotes the outward normal vector of the surface. Scalar product with $n_x$ leads to: 
\begin{gather}\label{sec2:eq34}
\llbracket 2\mu(\rho,c) n_x\cdot \D u\cdot n_x + \lambda(\rho,c)\dvg u - G(\rho,c)\rrbracket=0.
\end{gather}
Let $\tau_x$ and $\overline\tau_x$ be two  orthonormal vector fields spanning the tangent plane to $\Sigma$, which can be obtained from $n_x\times \vec{e}_j$, where $\vec e_j$, $j=1,2,3$ are the vectors of the canonical basis of $\R^3$.  Then we have (here we use the notation $\partial_w g= w\cdot \nabla g$):
\[
n_x\cdot \D u\cdot n_x= n_x\cdot \partial_{n_x} u=\dvg u-\tau_x\cdot \partial_{\tau_x} u-\overline\tau_x\cdot \partial_{\overline\tau_x} u.
\]
Hence \eqref{sec2:eq34} implies
\begin{equation}\label{sec2:eq36}
\llbracket \nu(\rho,c)\dvg u- G(\rho,c)\rrbracket= 2\llbracket \mu(\rho,c)\rrbracket \big(\tau\cdot \partial_{\tau} u+\overline\tau\cdot \partial_{\overline\tau} u\big),
\end{equation}
and consequently (see \eqref{sec2:eq35})
\[
\llbracket [\check{\mathcal{K}},\mu(\rho,c)-\widetilde \mu ](2\D u)\rrbracket= 2\llbracket \mu(\rho,c)\rrbracket  \big(\tau\cdot \partial_{\tau} u+\overline\tau\cdot \partial_{\overline\tau} u\big).
\]
It immediately follows that for all $q\in [1,\infty]$,
\begin{gather*}
    \norm{\llbracket [\check{\mathcal{K}},\mu(\rho,c)-\widetilde \mu](2\D u)\rrbracket}_{L^q(\Sigma)}\leqslant C\norm{\llbracket \mu(\rho,c)\rrbracket}_{L^q(\Sigma)}\norm{\nabla u}_{L^\infty(\R^3)}.
\end{gather*}
From \eqref{sec2:eq16}, we readily obtain
\begin{gather}\label{sec2:eq52}
\abs{n_x,\,\tau_x,\,\overline{\tau}_x}_{\dot \cC^{\alpha}(\Sigma)}\leqslant C \ell_\Sigma^{-\alpha},
\end{gather}
and therefore 
\begin{align*}
\abs{\llbracket [\check{\mathcal{K}},\mu(\rho,c)-\widetilde \mu](2\D u)\rrbracket}_{\dot \cC^\beta(\Sigma)}&\leqslant C\Big(\abs{\llbracket \mu(\rho,c)\rrbracket}_{\dot \cC^\beta(\Sigma)}+\ell_\Sigma^{-\beta}\norm{\llbracket\mu(\rho,c)\rrbracket}_{L^\infty(\Sigma)}\Big)\norm{\nabla u}_{L^\infty(\R^3)}\\
&+C\norm{\llbracket\mu(\rho,c)\rrbracket}_{L^\infty(\Sigma)}\abs{\nabla u}_{\dot\cC^\beta_{\pw,\Sigma}(\R^3)}.
\end{align*}
Gathering \eqref{sec2:eq17}-\eqref{sec2:eq11}-\eqref{sec2:eq21}-\eqref{sec2:eq30} together with the computations above, we infer:
{\small
\begin{align}
    &\int_0^t\norm{\nabla (-\Delta)^{-1}\nabla\Bigg(\frac{1}{\nu(\rho,c)}\Big(G(\rho,c)+\big[\check{\mathcal{K}},\mu(\rho,c)-\widetilde \mu\big](2\D u)\Big)\Bigg)}_{L^\infty(\R^3)}\notag\\
    &\leqslant C_* \norm{f(\rho_0,c_0)}_{L^\infty(\R^3)}+\int_0^t\norm{\rho \dot u}_{L^{2}(\R^3)}+C_*\int_0^t\phi_\Sigma\norm{G(\rho,c),\,\nabla u}_{L^{p}(\R^3)}+C_*\int_0^t\norm{w(s)}_{L^\infty(\R^3)}\vartheta(s)ds\notag\\
&+C_*\Big(1+e^{\alpha\vartheta(t)}\sup_{[0,t]}\norm{\llbracket\nu(\rho,c)\rrbracket}_{L^\infty(\Sigma)}\Big)\Bigg\{\abs{f(\rho_0,c_0)}_{\dot \cC^{\alpha/2}_{\pw,\Sigma_0}(\R^3)}+\int_0^t\norm{\rho \dot u}_{L^{6/(2-\alpha)}(\R^3)}\notag\\ 
&+\sup_{[0,t]}\Big(\abs{\mu(\rho,c)}_{\dot\cC^\alpha(\R^3)}+ \ell_{\Sigma}^{-\alpha}\norm{\llbracket \mu(\rho,c)\rrbracket}_{L^\infty(\Sigma)}\Big)\int_0^t\norm{\nabla u}_{L^{6/\alpha}(\R^3)}\Bigg\}+C_* \int_0^t \Big(\ell_{\Sigma}^{-\frac{\alpha}{2}}\norm{\llbracket G(\rho,c)\rrbracket}_{L^\infty\cap L^{p}(\Sigma)}
        +\abs{\llbracket G(\rho,c)\rrbracket}_{\dot \cC^{\alpha/2}(\Sigma)}\Big)\notag\\
    &+C_*\check{\vartheta}(t)\sup_{[0,t]} \norm{\llbracket \mu(\rho,c)\rrbracket}_{\cC^{\alpha/2}(\Sigma)}\Big(1+e^{\alpha\vartheta(t)}\sup_{[0,t]}\norm{\llbracket\nu(\rho,c)\rrbracket}_{L^\infty(\Sigma)}\Big)\notag\\
    &+C_*\vartheta(t)\Bigg( \check{\phi}_\Sigma e^{\alpha \vartheta(t)} \sup_{[0,t]}\norm{\llbracket\nu(\rho,c)\rrbracket}_{L^\infty(\Sigma)} +\big(\ell^{-\alpha/2}_{\Sigma}+\mathfrak{P}_{\Sigma} \abs{\Sigma}_{\dot \cC^{1+\alpha/2}}\big) \sup_{[0,t]} \norm{\llbracket \mu(\rho,c)\rrbracket}_{L^\infty\cap L^{p}(\Sigma)}\Bigg)\notag\\
    &+\frac{C}{\mu_*}\vartheta(t)\Bigg(\sup_{[0,t]} \abs{\llbracket \mu(\rho,c)\rrbracket}_{\dot \cC^{\alpha/2}(\Sigma)}+\frac{1}{\mu_*}\sup_{[0,t]}\norm{\llbracket\nu(\rho,c)\rrbracket}_{L^\infty(\Sigma)}\Big(\mu^*+\overline{a}e^{\alpha\int_0^t\norm{\nabla u}_{L^\infty(\R^3)}}\sup_{[0,t]}\norm{\llbracket \mu(\rho,c)\rrbracket}_{L^\infty(\Sigma)}\Big)\Bigg).\label{sec2:eq20}
\end{align}
}
We have thus completed the estimate for the third and last terms of \eqref{sec2:eq4}.
\end{itemize}
{\small 
Gathering \eqref{sec2:eq20}-\eqref{sec2:eq31}-\eqref{sec2:eq24}-\eqref{sec2:eq19} we deduce the bound below:
\begin{align}
    &\int_0^t\norm{\nabla u}_{L^\infty(\R^3)}+\frac{1}{\mu_*}\int_0^t \norm{G(\rho,c),\, \big[\mathcal{K},\mu(\rho,c)-\widetilde \mu\big](\D u),\, \big[\check{\mathcal{K}},\mu(\rho,c)-\widetilde \mu\big](\D u}_{L^\infty(\R^3)}\notag\\
    &\leqslant C_* \norm{f(\rho_0,c_0)}_{L^\infty(\R^3)}+C_*\int_0^t\norm{\dot u}_{L^{2}(\R^3)}+C_*\int_0^t\phi_\Sigma\norm{G(\rho,c),\,\nabla u\, (-\Delta)^{-1}\dvg (\rho \dot u)}_{L^{p}(\R^3)}+C_*\int_0^t\norm{w(s)}_{L^\infty(\R^3)}\vartheta(s)ds\notag\\
&+C_*\Big(1+e^{\alpha\vartheta(t)}\sup_{[0,t]}\norm{\llbracket\nu(\rho,c)\rrbracket}_{L^\infty(\Sigma)}\Big)\Bigg\{\abs{f(\rho_0,c_0)}_{\dot \cC^{\alpha/2}_{\pw,\Sigma_0}(\R^3)}+\int_0^t\norm{\rho \dot u}_{L^{6/(2-\alpha)}(\R^3)}\notag\\ 
&+\sup_{[0,t]}\Big(\abs{\mu(\rho,c)}_{\dot\cC^\alpha(\R^3)}+ \ell_{\Sigma}^{-\alpha}\norm{\llbracket \mu(\rho,c)\rrbracket}_{L^\infty(\Sigma)}\Big)\int_0^t\norm{\nabla u}_{L^{6/\alpha}(\R^3)}\Bigg\}+C_* \int_0^t \Big(\ell_{\Sigma}^{-\frac{\alpha}{2}}\norm{\llbracket G(\rho,c)\rrbracket}_{L^\infty\cap L^{p}(\Sigma)}
        +\abs{\llbracket G(\rho,c)\rrbracket}_{\dot \cC^{\alpha/2}(\Sigma)}\Big)\notag\\
    &+C_*\check{\vartheta}(t)\sup_{[0,t]} \norm{\llbracket \mu(\rho,c)\rrbracket}_{\cC^{\alpha/2}(\Sigma)}\Big(1+e^{\alpha\vartheta(t)}\sup_{[0,t]}\norm{\llbracket\nu(\rho,c)\rrbracket}_{L^\infty(\Sigma)}\Big)\notag\\
    &+C_*\vartheta(t)\Bigg( \check{\phi}_\Sigma e^{\alpha \vartheta(t)} \sup_{[0,t]}\norm{\llbracket\nu(\rho,c)\rrbracket}_{L^\infty(\Sigma)} +\big(\ell^{-\alpha/2}_{\Sigma}+\mathfrak{P}_{\Sigma} \abs{\Sigma}_{\dot \cC^{1+\alpha/2}}\big) \sup_{[0,t]} \norm{\llbracket \mu(\rho,c)\rrbracket}_{L^\infty\cap L^{p}(\Sigma)}\Bigg)\notag\\
    &+\frac{C}{\mu_*}\vartheta(t)\Bigg(\sup_{[0,t]} \abs{\llbracket \mu(\rho,c)\rrbracket}_{\dot \cC^{\alpha/2}(\Sigma)}+\frac{1}{\mu_*}\sup_{[0,t]}\norm{\llbracket\nu(\rho,c)\rrbracket}_{L^\infty(\Sigma)}\Big(\mu^*+\overline{a}e^{\alpha\int_0^t\norm{\nabla u}_{L^\infty(\R^3)}}\sup_{[0,t]}\norm{\llbracket \mu(\rho,c)\rrbracket}_{L^\infty(\Sigma)}\Big)\Bigg). \label{sec2:eq32}
\end{align}
}
By applying the Leray projector $\mathbb P$ to \eqref{sec2:eq4} and repeating the same computations as above, we recover estimate \eqref{sec2:eq32} for $\nabla \mathbb P u$, so that we can include  $\normb{\nabla\mathbb P u)}_{L^\infty(\R^3)}$  on the LHS of \eqref{sec2:eq32}. Then, 
we use the smallness of
\[
\frac{C}{\mu_*}\Bigg[\sup_{[0,t]} \abs{\llbracket \mu(\rho,c)\rrbracket}_{\dot \cC^{\alpha/2}(\Sigma)}+\frac{1}{\mu_*}\sup_{[0,t]}\norm{\llbracket\nu(\rho,c)\rrbracket}_{L^\infty(\Sigma)}\Big(\mu^*+\overline{a}e^{\alpha\int_0^t\norm{\nabla u}_{L^\infty(\R^3)}}\sup_{[0,t]}\norm{\llbracket \mu(\rho,c)\rrbracket}_{L^\infty(\Sigma)}\Big)\Bigg]
\]
to absorb the last term above into the LHS. Finally, we apply Gr\"onwall's Lemma to obtain
{\small
\begin{align}
    \vartheta(t)&\leqslant C_* \Bigg\{\norm{f(\rho_0,c_0)}_{L^\infty(\R^3)}+\int_0^t\norm{\dot u}_{L^{2}(\R^3)}+\int_0^t\phi_\Sigma\norm{G(\rho,c),\,\nabla u\, (-\Delta)^{-1}\dvg (\rho \dot u)}_{L^{p}(\R^3)}\notag\\
&+\Big(1+\delta(t) e^{\alpha\vartheta(t)}\Big)\Bigg[\abs{f(\rho_0,c_0)}_{\dot \cC^{\alpha/2}_{\pw,\Sigma_0}(\R^3)}+\int_0^t\norm{\rho \dot u}_{L^{6/(2-\alpha)}(\R^3)}+\Big(\sup_{[0,t]}\abs{\mu(\rho,c)}_{\dot\cC^\alpha(\R^3)}+ \delta(t)\ell_{\Sigma}^{-\alpha}\Big)\int_0^t\norm{\nabla u}_{L^{6/\alpha}(\R^3)}\Bigg]\notag\\
&+\check{\vartheta}(t)\sup_{[0,t]} \norm{\llbracket \mu(\rho,c)\rrbracket}_{\cC^{\alpha/2}(\Sigma)}\Big(1+\delta(t)e^{\alpha\vartheta(t)}\Big)+ \int_0^t \Big(\ell_{\Sigma}^{-\frac{\alpha}{2}}\norm{\llbracket G(\rho,c)\rrbracket}_{L^\infty\cap L^{p}(\Sigma)}
        +\abs{\llbracket G(\rho,c)\rrbracket}_{\dot \cC^{\alpha/2}(\Sigma)}\Big)\notag\\
    &+\vartheta(t)\delta(t)\Bigg( \check{\phi}_\Sigma e^{\alpha \vartheta(t)}  +\big(\ell^{-\alpha/2}_{\Sigma}+\mathfrak{P}_{\Sigma} \abs{\Sigma}_{\dot \cC^{1+\alpha/2}}\big)\Bigg)\Bigg\}e^{C_*\int_0^t\normb{w}_{L^\infty(\R^3)}},\notag
\end{align}
}
where 
\[
\delta(t)= \sup_{[0,t]} \norm{\llbracket \mu(\rho,c),\, \nu(\rho,c)\rrbracket}_{L^\infty\cap L^{p}(\Sigma)}. 
\]
The estimate \eqref{sec2:Lipschtz} follows upon recalling \eqref{sec2:eq30}.
\end{proof}

\subsection{Proof of Lemma \ref{sec2:lem3}}\label{prooflem3}
\begin{proof}
    We first recall that the $(\rho,c)$-dependent function $f=f(\rho,c)$ satisfies 
    \[
    \dpt f(\rho,c)+ u\cdot\nabla f(\rho,c)+ G(\rho,c)= -F.
    \]
    We write this equation along the flow map of the velocity  and we then take jump at the discontinuity surface to obtain:
    \[
    \dpt \llbracket f(\varrho,c_0)\rrbracket + \llbracket G(\varrho,c_0)\rrbracket=-\llbracket F\circ \mathcal{X}\rrbracket.
    \]
    Here, $\mathcal{X}$ denotes the flow map of the fluid velocity and $\varrho=\rho\circ \mathcal{X}$. The expression \eqref{sec2:eq36} then implies
    \begin{gather*}
    \dpt \llbracket f(\varrho,c_0)\rrbracket + \llbracket G(\varrho,c_0)\rrbracket=-2\llbracket \mu(\varrho,c_0)\rrbracket\Big( \tau_x \cdot\partial_{\tau_x} u+\overline{\tau}_x\cdot\partial_{\overline{\tau}_x} u\Big)\circ \mathcal{X}.
    \end{gather*}
 Under the assumption \eqref{sec1:jumpcondition}, we write the pressure jump in term of the jump of $f(\rho,c)$, and we  obtain 
    \begin{gather}\label{sec3:eq2}
    \dpt \llbracket f(\varrho,c_0)\rrbracket + g_P\llbracket f(\varrho,c_0)\rrbracket= -2\llbracket \mu(\varrho,c_0)\rrbracket\big( \tau_x \cdot\partial_{\tau_x} u+\overline{\tau}_x\cdot\partial_{\overline{\tau}_x} u\big)\circ \mathcal{X},
    \end{gather}
    where 
    \[
    g_P=\frac{\llbracket \Psi_P\big(f(\varrho,c_0)\big)\rrbracket}{\llbracket f(\varrho,c_0)\rrbracket}=\int_0^1\Psi_P'\big(rf(\varrho^{+},1)+(1-r)f(\varrho^{-},0)\big)dr.
    \]
From \eqref{psiderivative}-\eqref{sec2:eq25}, we have $g_P\in  [a_{*},a^{*}]\subset (0,\infty)$; and it follows that 
    \begin{align}\label{sec2:eq39}
      \abs{\llbracket f(\varrho(t),c_0)\rrbracket}&\leqslant e^{- a_{*} t} \abs{\llbracket f(\rho_0,c_0)\rrbracket}+C\int_0^te^{- a_{*} (t-t')} \abs{\llbracket \mu(\varrho(t'),c_0)\rrbracket}\norm{\nabla u(t')}_{L^\infty(\R^3)}dt'.
    \end{align}
    From this, it is straightforward to deduce \eqref{sec2:eq37}.

Taking difference in \eqref{sec3:eq2}, we obtain
\begin{align*}
    \dpt \llbracket f(\varrho,c_0)\rrbracket\Big|_{s_1}^{s_2} + g_P(s_2)\llbracket f(\varrho,c_0)\rrbracket\Big|_{s_1}^{s_2}&= -2\llbracket \mu(\varrho,c_0)\rrbracket\big( \tau_x \cdot\partial_{\tau_x} u+\overline{\tau}_x\cdot\partial_{\overline{\tau}_x} u\big)\circ \mathcal{X}\Big|_{s_1}^{s_2}\\
    &-\big(g_P(s_2)-g_P(s_1)\big)\llbracket f(\varrho(s_1),c_0)\rrbracket,
\end{align*}
and whence, 
\begin{align}
     \abs{\llbracket f(\varrho(t),c_0)\rrbracket\Big|_{s_1}^{s_2}} &\leqslant  e^{- a_{*}t}\abs{\llbracket f(\varrho_0,c_0)\rrbracket\Big|_{s_1}^{s_2}}+\int_0^te^{-a_{*}(t-t')}\abs{g_P(s_2)-g_P(s_1)}\abs{\llbracket f(\varrho,c_0)(s_1)\rrbracket} dt'\notag\\
    &+C\int_0^te^{-a_{*}(t-t')}\abs{\llbracket \mu(\varrho,c_0)\rrbracket\big( \tau_x \cdot\partial_{\tau_x} u+\overline{\tau}_x\cdot\partial_{\overline{\tau}_x} u\big)\circ \mathcal{X}\Big|_{s_1}^{s_2}}dt'.\label{sec2:eq43}
\end{align}
Since we have 
\begin{align}
g_P(s_2)-g_P(s_1)&=\int_0^1\int_0^1 \Big( r f\big(\varrho^+(s),1\big)\Big|_{s_1}^{s_2}+(1-r) f\big(\varrho^-(s),0\big)\Big|_{s_1}^{s_2}\Big)\Psi_P''\big( r'\big(rf(\varrho^{+}(s_2),1)+(1-r)f(\varrho^{-}(s_2),0)\big)\notag\\
& +(1-r')\big(r f(\varrho^{+}(s_1),1)+(1-r)f(\varrho^{-}(s_1),0)\big)\big)drdr',\label{sec2:eq42}
\end{align}
it is not difficult to derive (we use \eqref{sec2:eq52}) 
\begin{align*}
     \abs{\llbracket f(\rho,c)(t)\rrbracket}_{\dot\cC^\alpha(\Sigma(t))}&\leqslant  e^{-a_{*}t+ \alpha\int_0^t\normb{\nabla u}_{L^\infty(\R^3)}} \abs{\llbracket f(\rho_0,c_0)\rrbracket}_{\dot\cC^\alpha(\Sigma_0)}\\
     &+C\int_0^te^{-a_{*}(t-t')+\alpha\int_0^t\normb{\nabla u}_{L^\infty(\R^3)}}\abs{\llbracket \mu(\rho,c)\rrbracket}_{\dot\cC^\alpha(\Sigma)}\norm{\nabla u}_{L^\infty(\R^3)}dt'\notag\\
     &+C\int_0^te^{-a_{*}(t-t')+\alpha\int_0^t\normb{\nabla u}_{L^\infty(\R^3)}}\norm{\llbracket \mu(\rho,c)\rrbracket}_{L^\infty(\Sigma)}\big( \ell_{\Sigma}^{-\alpha}\norm{\nabla u}_{L^\infty(\R^3)}+\abs{\nabla u}_{\dot\cC^\alpha_{\pw,\Sigma}(\R^3)}\big)dt'\notag\\
    &+C_*\int_0^te^{-a_{*}(t-t')+\alpha\int_0^t\normb{\nabla u}_{L^\infty(\R^3)}}\abs{f(\rho,c)}_{\dot \cC^\alpha_{\pw,\Sigma}(\R^3)}\norm{\llbracket f(\rho,c)\rrbracket}_{L^\infty(\R^3)} dt',
\end{align*}
and hence 
\begin{align*}
     \sup_{[0,t]}\abs{\llbracket f(\rho,c)\rrbracket}_{\dot\cC^\alpha(\Sigma)}&+\int_0^t\abs{\llbracket G(\rho,c)\rrbracket}_{\dot\cC^\alpha(\Sigma)}\\
     &\leqslant  \frac{a^*}{a_*}e^{\alpha\int_0^t\normb{\nabla u}_{L^\infty(\R^3)}} \Bigg[\abs{\llbracket f(\varrho_0,c_0)\rrbracket}_{\dot\cC^\alpha(\Sigma_0)}+\sup_{[0,t]}\abs{\llbracket \mu(\rho,c)\rrbracket}_{\dot\cC^\alpha(\Sigma)}\int_0^t\norm{\nabla u}_{L^\infty(\R^3)}\notag\\
     &+\sup_{[0,t]}\norm{\llbracket \mu(\rho,c)\rrbracket}_{L^\infty(\Sigma)}\Bigg(\int_0^t\ell_{\Sigma}^{-\alpha}\norm{\nabla u}_{L^\infty(\R^3)}+\int_0^t\abs{\nabla u}_{\dot \cC^\alpha_{\pw,\Sigma}(\R^3)}\Bigg)\\
     &+C_*\sup_{[0,t]}\abs{f(\rho,c)}_{\dot \cC^\alpha_{\pw,\Sigma}(\R^3)}\int_0^t\norm{\llbracket f(\rho,c)\rrbracket}_{L^\infty(\R^3)}\Bigg].
\end{align*}
This concludes the proof of \eqref{sec2:eq38}.

Under the assumption \eqref{sec1:jumpconditionbis}, we have  
    \[
    \abs{\llbracket \mu(\varrho,c_0)\rrbracket}\leqslant C_*\abs{\llbracket f(\varrho,c_0)\rrbracket},
    \]
    and by applying Gr\"onwall's lemma to \eqref{sec2:eq39}, we find:
    \begin{align}\label{sec2:eq44}
      \abs{\llbracket f(\varrho(t),c_0)\rrbracket}&\leqslant e^{- a_{*} t+C_*\int_0^t\norm{\nabla u}_{L^\infty(\R^3)}} \abs{\llbracket f(\rho_0,c_0)\rrbracket}.
    \end{align}
    It is then straightforward to deduce \eqref{sec2:eq40}. 
    
To prove \eqref{sec2:eq41}, we express
\[
\llbracket \mu(\varrho,c_0)(s)\rrbracket\Big|_{s_1}^{s_2}=g_\mu(s_2)\llbracket f\big(\varrho,c_0\big)(s)\rrbracket\Big|_{s_1}^{s_2}+ \big(g_\mu(s_2)-g_{\mu}(s_1)\big)\llbracket f(\varrho,c_0)(s_1)\rrbracket,
\]
where 
\[
g_\mu=\int_0^1\Psi_\mu'\big(rf(\varrho^{+},1)+(1-r)f(\varrho^{-},0)\big)dr.
\]
Hence, proceeding as in \eqref{sec2:eq42}, we find that \eqref{sec2:eq43} implies: 
\begin{align*}
     \abs{\llbracket f(\rho,c)(t)\rrbracket}_{\dot\cC^\alpha(\Sigma)}&\leqslant  e^{-a_{*}t+ \alpha \int_0^t\norm{\nabla u}_{L^\infty(\R^3)}} \abs{\llbracket f(\rho_0,c_0)\rrbracket}_{\dot\cC^\alpha(\Sigma_0)}\\
     &+C_*\int_0^te^{- a_{*}(t-t')+\alpha\int_0^t\norm{\nabla u}_{L^\infty(\R^3)}}\abs{\llbracket f(\rho,c)\rrbracket}_{\dot\cC^\alpha(\Sigma)}\norm{\nabla u}_{L^\infty(\R^3)}dt'\notag\\
     &+C_*\int_0^te^{-a_{*}(t-t')+\alpha \int_0^t\norm{\nabla u}_{L^\infty(\R^3)}}\norm{\llbracket f(\rho,c)\rrbracket}_{L^\infty(\Sigma)}\big( \ell_{\Sigma}^{-\alpha}\norm{\nabla u}_{L^\infty(\R^3)}+\abs{\nabla u}_{\dot\cC^\alpha_{\pw,\Sigma}(\R^3)}\big)dt'\notag\\
    &+C_*\int_0^te^{- a_{*}(t-t')+\alpha\int_0^t\norm{\nabla u}_{L^\infty(\R^3)}}\abs{f(\rho,c)}_{\dot \cC^\alpha_{\pw,\Sigma}(\R^3)}\norm{\llbracket f(\rho,c)\rrbracket}_{L^\infty(\Sigma)}\big(1+\norm{\nabla u}_{L^\infty(\R^3)}\big) dt'.
\end{align*}
Whence, Gr\"onwall's lemma and \eqref{sec2:eq44} yields 
\begin{align*}
     \abs{\llbracket f(\rho,c)(t)\rrbracket}_{\dot\cC^\alpha(\Sigma(t))}&\leqslant  C_*e^{-a_{*}t+ C_*\int_0^t\norm{\nabla u}_{L^\infty(\R^3)}e^{ \alpha\int_0^t\norm{\nabla u}_{L^\infty(\R^3)}}}\Bigg[\abs{\llbracket f(\varrho_0,c_0)\rrbracket}_{\dot\cC^\alpha(\Sigma_0)}\\
     &+\norm{\llbracket f(\rho_0,c_0)\rrbracket}_{L^\infty(\Sigma_0)}\int_0^t\Big(\big(\ell_{\Sigma}^{-\alpha}+\abs{f(\rho,c)}_{\dot \cC^\alpha(\Sigma)}\big)\norm{\nabla u}_{L^\infty(\R^3)}+\abs{f(\rho,c),\,\nabla u}_{\dot\cC^\alpha_{\pw,\Sigma}(\R^3)}\Big)\Bigg].
\end{align*}
This proves \eqref{sec2:eq41} and  we now turn to the decay estimate for the velocity gradient jump.

To this end, we first express 
\[
\nabla u^j = \partial_{\tau_x} u^j \tau_x + \partial_{\overline\tau_x} u^j \overline{\tau}_x +\partial_{n_x} u^j n_x,
\]
and since  $\partial_{\tau_x} u^j$ and $\partial_{\overline\tau_x} u^j$ are continuous across $\Sigma$, we have
\begin{align}\label{sec2:eq47}
\llbracket \nabla u^j\rrbracket&= \llbracket\partial_{n_x} u^j\rrbracket n_x.
\end{align}
From \eqref{sec2:eq34}, we have
\[
\llbracket 2\mu(\rho,c)n_x\cdot \partial_{n_x} u+\lambda(\rho,c)\dvg u -P(\rho,c)\rrbracket=0,
\]
and since 
\[
\dvg u=\tau_x\cdot \partial_{\tau_x} u+\overline\tau_x\cdot \partial_{\overline\tau_x} u+n_x\cdot \partial_{n_x} u,
\]
we deduce that 
\[
\llbracket \nu(\rho,c)n_x\cdot \partial_{n_x} u\rrbracket =\llbracket P(\rho,c)\rrbracket-\llbracket \lambda(\rho,c)\rrbracket\big(\tau_x\cdot \partial_{\tau_x} u+\overline{\tau}_x\partial_{\overline{\tau}_x} u\big).
\]
and hence
\begin{equation*}
   n_x\cdot \llbracket \partial_{n_x} u\rrbracket=\frac{1}{\scalar{\nu(\rho,c)}_{\text{avg}}}\llbracket P(\rho,c)\rrbracket-\llbracket \lambda(\rho,c)\rrbracket\frac{\scalar{\dvg u}_{\text{avg}}}{\scalar{\nu(\rho,c)}_{\text{avg}}}-2\llbracket\mu(\rho,c)\rrbracket\frac{n_x\cdot \scalar{\partial_{n_x} u}_{\text{avg}}}{\scalar{\nu(\rho,c)}_{\text{avg}}}.
\end{equation*}
Similarly, scalar product of \eqref{sec2:eq45} with the tangential vector fields $\tau_x$ and $\overline{\tau}_x$ yields:
\begin{gather*}
    \tau_x \cdot \llbracket \partial_{n_x} u\rrbracket =-\frac{\llbracket\mu(\rho,c)\rrbracket}{\scalar{\mu(\rho,c)}_{\text{avg}}}\big( n_x \cdot \partial_{\tau_x} u+ \tau_x\cdot\partial_{n_x} u\big), 
\end{gather*}
and 
\begin{gather*}
    \overline\tau_x \cdot \llbracket \partial_{n_x} u\rrbracket =-\frac{\llbracket\mu(\rho,c)\rrbracket}{\scalar{\mu(\rho,c)}_{\text{avg}}}\big( n_x \cdot \partial_{\overline\tau_x} u+ \overline\tau_x\cdot\partial_{n_x} u\big), 
\end{gather*}
Estimate \eqref{sec2:eq46} simply follows from \eqref{sec2:eq47} and these expressions. This completes the proof of Lemma \ref{sec2:lem3}.
\end{proof}
\subsection{Basic a priori estimates}\label{basic:estimate}
In this subsection, we derive some basic a priori bounds for the solutions to the system \eqref{ep2.1}. 
These estimates consist of mass conservation and energy balance, and follow from the renormalized form of the equations 
for the volume fraction and the density (see $\eqref{ep2.1}_{1,2}$):
\begin{gather}\label{ep2.2}
    \dpt b(\rho,c)+\dvg \big(b(\rho,c)u\big)+\big(\rho \partial_\rho b(\rho,c)-b(\rho,c)\big)\dvg u=0.
\end{gather}
Above, $b$ is any $W^{1,\infty}_{\loc}$-regular function of the density and volume fraction. 

\noindent
\textbf{Conservation of mass.} Let $\mathcal{H}=\mathcal{H}(c)$ be a $W^{1,\infty}_\loc$-regular function. We observe that 
$b(\rho,c)=\rho \mathcal{H}(c)$ satisfies 
\[
\rho \partial_{\rho} b(\rho,c)-b(\rho,c)=0,
\]
and  \eqref{ep2.2} becomes
\[
\dpt \big( \rho \mathcal{H}(c)\big)+\dvg \big(\rho \mathcal{H}(c) u\big)=0.
\]
As a result, we have
\begin{gather*}
    \int_{\R^3}\rho(t,x) \mathcal{H}(c(t,x))dx=    \int_{\R^3}\rho_0(x) \mathcal{H}(c_0(x))dx,
\end{gather*}
as soon as $\rho_0 \mathcal{H}(c_0)\in L^1(\R^3)$. This bound, as well as those in the following paragraph, holds true even when the volume fraction is not an indicator function.

\noindent
\textbf{Energy balance.} Here, we consider $b$ as solution of the ODE  
\begin{gather}\label{ep2.3}
\rho \partial_\rho b(\rho,c)-b(\rho,c)=P(\rho,c)-\widetilde P. 
\end{gather}
Since this is a first-order  ODE, an additional condition is required to uniquely determine $b$.  
In order to work within a framework that allows for densities close to two constants ($\widetilde\rho(1)$ in $D^c(t)$, and  $\widetilde\rho(0)$ in $D(t)$)  we shall impose the following condition:
\[
b(\widetilde\rho(c),c)=0.
\]
Such a function $b$, now denoted by $H_1(\rho,c))$, is therefore given by
\[
H_1(\rho,c)= \rho \int_{\widetilde\rho(c)}^\rho\dfrac{1}{\varrho^2}\Big(P(\varrho,c)-\widetilde P\Big)d\varrho.
\]
Moreover, in view of \eqref{ep2.2}, it satisfies the equation
\[
\dpt H_1(\rho,c)+\dvg \big(H_1(\rho,c) u\big)+ \big( P(\rho,c)-\widetilde P\big)\dvg u=0,
\]
and consequently
\begin{gather}\label{ep2.4}
    \dfrac{d}{dt}\int_{\R^3} H_1(\rho,c)=-\int_{\R^3}\big( P(\rho,c)-\widetilde P\big)\dvg u.
\end{gather}
Differentiating \eqref{ep2.3} with respect to the density, and using the fact that the map $\rho \mapsto P(\rho,c)$ is non-decreasing, we deduce that $\rho \mapsto H_1(\rho,c)$ is convex. The non-negativity of $(\rho,c)\mapsto H_1(\rho,c)$ then follows from  $H_1(\widetilde\rho(c),c)=\partial_\rho H_1(\widetilde\rho(c),c)=0$.

As is standard, to compute the classical energy balance, we use the velocity field $u$ as a test function in  $\eqref{ep2.1}_3$ and we obtain
\begin{gather}\label{ep2.5}
\dfrac{d}{dt}\int_{\R^3}\rho\dfrac{\abs{u}^2}{2}-\int_{\R^3}\dvg u\big(P(\rho,c)-\widetilde P\big)+\int_{\R^2}\Big(2\mu(\rho,c)\D^{jk} u\partial_k u^j+\lambda(\rho,c)(\dvg u)^2\Big)=0.
\end{gather}
We have already computed the second term above in \eqref{ep2.4}, and now, using the symmetry of  $\D u$  we have
\[
\sum_{j,k}\D^{jk} u\partial_k u^j=\sum_{j,k}\D^{jk} u\partial_j u^k\implies\sum_{j,k}\D^{jk} u\partial_k u^j=\dfrac{1}{2}\sum_{j,k}\D^{jk} u\big(\partial_k u^j+\partial_j u^k\big)= \abs{\D u}^2,
\]
and therefore \eqref{ep2.5} leads to 
\begin{gather}\label{sec3:eq4}
    \dfrac{d}{dt}E+\int_{\R^3}\Big(2\mu(\rho,c)\abs{\D u}^2+\lambda(\rho,c)(\dvg u)^2\Big)=0,
\end{gather}
where the energy functional $E$ is given by
\[
E(t)=\int_{\R^3}\bigg(\rho\dfrac{\abs{u}^2}{2}+H_1\big(\rho,c\big)\bigg)(t,x)dx, \text{ and we set } E(0)= E_0.
\]
 
 The term contributing to the energy potential in \eqref{ep2.5} can be computed as follows: 
\[
-\int_{\R^3} \big(P(\rho,c)-\widetilde P\big)\dvg u= -\int_{\R^3} (P(\rho,c)-c\widetilde P)\dvg u- \widetilde P\int_{\R^3}(c-1)\dvg u.
\]
From the equation on the volume fraction (see $\eqref{ep2.1}_1$), we express the last term above as:
\[
- \widetilde P\int_{\R^3}(c-1)\dvg u=\widetilde P\dfrac{d}{dt}\int_{\R^3}(1-c),
\]
 and thanks to \eqref{ep2.2}, the previous one is
 \[
 -\int_{\R^3} (P(\rho,c)-c\widetilde P)\dvg u=\dfrac{d}{dt}\int_{\R^3}\check H_1(\rho,c),\text{ where } \check H_1(\rho,c)=\rho \int_{c\widetilde\rho}^\rho\dfrac{1}{\varrho^2}\Big(P(\varrho,c)-c\widetilde P\Big)d\varrho\geqslant 0.
 \]
 As a result, we have 
 \begin{gather}\label{sec3:eq3}
 \check E(t)+\int_0^t\int_{\R^3}\Big(2\mu(\rho,c)\abs{\D u}^2+\lambda(\rho,c)(\dvg u)^2\Big)= \check E(0)=: \check E_{0},
\end{gather}
where
\[
\check E(t)=\int_{\R^3}\bigg(\rho\dfrac{\abs{u}^2}{2}+\check H_1\big(\rho,c\big)+\widetilde P\big(1-c\big)\bigg)(t,x)dx.
\]

All the computations above remain valid even when the volume fraction is not an indicator function. In the indicator-function setting, we have  $1-c(t)=\mathbb{1}_{D(t)}$ and \eqref{sec3:eq3} yields 
\begin{gather}\label{ep2.11}
\abs{D(t)}=\int_{\R^3}\big(1-c\big)(t,x)dx\leqslant \widetilde P^{-1} \check E_{0}.
\end{gather}
This provides a sharper estimate of the volume of the moving domain than the one  obtained
using the flow of the velocity field.

\noindent
\textbf{The potential energies $H_l$.} In the spirit of \eqref{ep2.3}, for $l\in [1,\infty)$ we define the 
potential energy $H_l=H_l(\rho,c)$ as the solution of
\begin{gather}\label{ep2.6}
\begin{cases}
\rho \partial_\rho H_l(\rho,c)-H_l(\rho,c)=\abs{P(\rho,c)-\widetilde P}^{l-1}\big(P(\rho,c)-\widetilde P\big),\\
H_l(\widetilde\rho(c),c)=0.
\end{cases}
\end{gather}
It reads
\[
H_l(\rho,c)=\rho \int_{\widetilde\rho(c)}^\rho \dfrac{1}{\varrho^2}\abs{P(\varrho,c)-\widetilde P}^{l-1}\big(P(\varrho,c)-\widetilde P\big)d\varrho,
\]
and is clearly nonnegative as
\[
\rho \partial^2_{\rho} H_l(\rho,c)=l\partial_\rho P(\rho,c)\abs{P(\rho,c)-\widetilde P}^{l-1}\geqslant 0,\text{ and } \partial_{\rho} H_l(\widetilde \rho(c),c)= H_l(\widetilde\rho(c),c)=0.
\]
In particular, the map $\rho \mapsto H_l(\rho,c)$ is convex, which leads to the inequality
\[
0=H_l(\widetilde\rho(c),c)\geqslant H_l(\rho,c)+\partial_\rho H_l(\rho,c)\big(\widetilde\rho(c)-\rho\big),
\]
and combined with \eqref{ep2.6}, this implies 
\begin{gather}\label{ep2.7}
\widetilde\rho(c) H_l (\rho,c)\leqslant \abs{P(\rho,c)-\widetilde P}^{l-1}\big(P(\rho,c)-\widetilde P\big)\big(\rho-\widetilde \rho(c)\big).
\end{gather}

For some functions $g=g(\rho,c)$, and $a_l=a_l(c)$ to be specified later, let us consider 
\[
g_l(\rho,c)=\dfrac{H_l(\rho,c)}{\rho}-\dfrac{a_l(c)}{\rho}\abs{P(\rho,c)-\widetilde P}^{l-1} \big( P(\rho,c)-\widetilde P\big) g(\rho,c),
\]
and compute
\begin{align*}
&\rho^2\partial_\rho g_l(\rho,c)=\big(P(\rho,c)-\widetilde P\big)\abs{P(\rho,c)-\widetilde P}^{l-1}-a_l(c)\rho \partial_\rho g(\rho,c)
\abs{P(\rho,c)-\widetilde P}^{l-1} \big( P(\rho,c)-\widetilde P\big)\\
&-a_l(c)g(\rho,c)\bigg(l \rho \partial_\rho P(\rho,c) \abs{P(\rho,c)-\widetilde P}^{l-1}-\abs{P(\rho,c)-\widetilde P}^{l-1} \big( P(\rho,c)-\widetilde P\big)\bigg)\\
&=\big(P(\rho,c)-\widetilde P\big)\abs{P(\rho,c)-\widetilde P}^{l-1}\Bigg[1-a_l(c)\big(\rho \partial_\rho g(\rho,c)-g(\rho,c)\big)-la_l(c) \dfrac{\rho g(\rho,c)\partial_\rho P(\rho,c)}{P(\rho,c)-\widetilde P}\Bigg].
\end{align*}
\begin{itemize}
    \item Choosing $g(\rho,c)=\rho-\widetilde\rho(c)$ and $a_l(c)=1/\widetilde\rho(c)$ ensures that the expression within the square bracket is non-positive, as the map $\rho \mapsto P(\rho,c)$ is non-decreasing. As a result, the map $\rho \mapsto g_l(\rho,c)$ is non-decreasing on $[0,\widetilde\rho(c)]$, non-increasing on $[\widetilde\rho(c),\infty)$, with 
    $g_l(\widetilde\rho(c),c)=0$. This implies that $g\leq 0$, which in turn leads us again to \eqref{ep2.7}.
\item Setting 
\[
g(\rho,c)=1-\frac{2\widetilde\rho(c)}{\rho+\widetilde\rho(c)},
\]
the expression inside the square bracket above becomes
\[
1-a_l(c)\Bigg[1-\Bigg(\frac{\rho}{\rho+\widetilde\rho(c)}\Bigg)^2+l\frac{\rho}{\rho+\widetilde\rho(c)}\partial_\rho P(\rho,c) \dfrac{\rho-\widetilde\rho(c)}{P(\rho,c)-\widetilde P}\Bigg].
\]
We now choose $a_l(c)>0$ such that 
\[
a_l(c)\Bigg(1+l \partial_\rho P(\rho,c)\dfrac{\rho-\widetilde \rho(c)}{P(\rho,c)-\widetilde P}\Bigg)\leqslant 1,
\]
and, following the same reasoning as before, we obtain that $g_l\geqslant 0$, which in turn implies 
\begin{gather}\label{ep2.8}
a_l(c)\abs{P(\rho,c)-\widetilde P}^{l-1}\big(P(\rho,c)-\widetilde P\big)\big(\rho-\widetilde \rho(c)\big)\leqslant \big(\rho+\widetilde\rho(c)\big)H_l(\rho,c).
\end{gather}
In general, $a_l(c)$ depends only on $l$ and the upper bound of the density. However, in specific cases, such as when 
the pressure law of the phase "+" (resp. "-") is polynomial, $a_l(1)$ (resp. $a_l(0)$, can be chosen independently 
of the density's upper bound.
\end{itemize}

The potential energies $H_l$ will serve to derive an estimate for the $L^{l+1}_{t,x}$-norm of the pressure fluctuation. Indeed, from \eqref{ep2.6}-\eqref{ep2.2}, one readily obtains
\begin{gather}\label{ep2.9}
    \dfrac{d}{dt} \int_{\R^3}H_l(\rho,c)+\int_{\R^3}\abs{P(\rho,c)-\widetilde P}^{l-1}\big(P(\rho,c)-\widetilde P\big)\dvg u=0.
\end{gather}
Next, we express the velocity divergence in terms of the flux $F$ given by
\begin{gather}\label{ep2.12}
F=\nu(\rho,c)\dvg u-P(\rho,c)+\widetilde P,
\end{gather}
and apply H\"older's and Young's inequalities, and we arrive at
\begin{gather}\label{sec3:eq12}
    \dfrac{d}{dt} \int_{\R^3}H_l(\rho,c)+\dfrac{1}{l+1}\int_{\R^3}\dfrac{1}{\nu(\rho,c)}\abs{P(\rho,c)-\widetilde P}^{l+1}\leqslant \dfrac{1}{l+1}\int_{\R^3}\dfrac{1}{\nu(\rho,c)}\abs{F}^{l+1}.
\end{gather}
We have already known (see \eqref{sec2:eq2}) that this flux can be rewritten as
\begin{gather}\label{sec3:eq5}
F = -(-\Delta)^{-1} \dvg(\rho \dot u) + [(-\Delta)^{-1}\dvg \dvg , \mu(\rho,c) - \widetilde\mu] (2 \D u),
\end{gather}
and its $L_{t,x}^{l+1}$-bound will follow from energy estimates for the material acceleration $\dot u$.

\subsection{Proof of Lemma \ref{sec2:lem4}}\label{proof:sec2:lem4}
The goal of this section is to derive estimates for the functionals $\mathcal{A}_1$, $\mathcal{A}_2$, and $\mathcal{A}_3$.
The starting point (see Section \ref{sec3.1:decay}) will be to prove that, under the low frequency assumption \eqref{init:lowfrequency}, we gain a time decay for the classical energy.  

\subsubsection{Time-decay of the classical energy}\label{sec3.1:decay} 
This paragraph is dedicated to proving the following estimate.
\begin{lemma}\label{lem3.1}
    There exists a constant $c_*>0$, depending only on the 
lower and upper bounds of the density, such that if $E_0 \leqslant c_*$, then the following holds:
\begin{align}\label{ep3.8}
    \sup_{[0,t]}\scalar{\cdot}^{1+r_0} E&+ \int_0^t\scalar{t'}^{r_0} E(t')dt'+\int_0^t \scalar{t'}^{1+r_0}\int_{\R^3}\big(2\mu(\rho,c)\abs{\D u}^2+\lambda(\rho,c)(\dvg u)^2\big)dt'
    \leqslant C_*  C_0.
\end{align}
\end{lemma}
\begin{proof}[Proof of Lemma \ref{lem3.1}] 
The proof of \eqref{ep3.8} proceeds as follows: we establish a uniform-in-time bound on the \\ $L^2((0,t)\times \R^3)$-norm of both the high- and low-frequency components of the velocity and pressure fluctuations. These in turn lead to \eqref{ep3.8}. Our method is similar to that  in \cite{hu2020optimal},but with major differences. 

We shall introduce the following notations: $\cF$ will denote the Fourier transform:
\[
\cF f(\xi)=\int_{\R^3} f(x) e^{-i x\cdot \xi} dx.
\]
For a   smooth real function $\psi$  satisfying,  $\psi(\xi)=1$ for $\abs{\xi}\leqslant 1/2$, and $\psi(\xi)=0$ for $\abs{\xi}\geqslant 1$,  we define the low-frequency part of $f$, denoted $f^l$, and high-frequency part of $f$, denoted  $f^h$,  as follows:
\[
f^l=\cF ^{-1}\big( \xi \mapsto \psi(\xi) \cF (f)(\xi)\big),\quad f^h=f-f^l. 
\]
\paragraph{\textbf{Step 1: $L^2_{t,x}$-bound for the high-frequency components}}
By applying the Plancherel–Parseval identity, we obtain:
\begin{align*}
\normb{u^h}_{L^2(\R^3)}^2&= \int_{\R^3}\big(1-\psi(\xi)\big)^2\abs{\cF u(\xi)}^2d\xi\\
                         &\leqslant \int_{\abs{\xi}\geqslant 1/2}\abs{\cF u(\xi)}^2d\xi\\
                         &\leqslant 4\int_{\abs{\xi}\geqslant 1/2}\abs{\xi}^2\abs{\cF u(\xi)}^2d\xi\\
                         &\leqslant 4 \norm{\nabla u}_{L^2(\R^3)}^2.
\end{align*}
It immediately follows that (see \eqref{sec3:eq4})
\begin{equation}\label{sec3:eq9}
    \int_0^t\scalar{t'}^{r_0}\normb{u^h(t')}^2_{L^2(\R^3)}dt'\leqslant C_* \int_0^t\scalar{t'}^{r_0}\norm{\nabla u(t')}_{L^2(\R^3)}^2dt'.
\end{equation}
Next, from the expressions of the effective flux (see \eqref{ep2.12}-\eqref{sec3:eq5}), we have
    \begin{gather*}
    G(\rho,c) =(-\Delta)^{-1}\dvg (\rho \dot u)+\nu(\rho,c)\dvg u-[(-\Delta)^{-1}\dvg \dvg ,\mu(\rho,c)-\widetilde\mu](2\D u)
\end{gather*}
and it follows that 
\begin{align*}
\int_{\R^3} G(\rho,c)\varrho^h &=\int_{\R^3}\big((-\Delta)^{-1}\dvg (\rho \dot{u})\big)\varrho ^h\\
&+\int_{\R^3}\big(\nu(\rho,c)\dvg u-[(-\Delta)^{-1}\dvg \dvg ,\mu(\rho,c)-\widetilde\mu](2\D u)\big)\varrho^h, 
\end{align*}
where $\varrho$ denotes the density fluctuation around its equilibrium states
\[
\varrho=\partial_{\rho} P(\widetilde\rho(c),c)\big(\rho-\widetilde\rho(c)\big).
\]
We estimate the last term above as: 
\begin{align*}
    &\Bigg|\int_0^t\scalar{t'}^{r_0}\int_{\R^3}\big(\nu(\rho,c)\dvg u-[(-\Delta)^{-1}\dvg \dvg ,\mu(\rho,c)-\widetilde\mu](2\D u)\big)\varrho^hdt'\Bigg|\\
    &\leqslant \eta \int_0^t\scalar{t'}^{r_0}\normb{\varrho}_{L^2(\R^3)}^2dt'+ \frac{C_*}{\eta}\int_0^t\scalar{t'}^{r_0}\norm{\nabla u}_{L^2(\R^3)}^2dt'.
\end{align*}
We express $\rho\dot u=\dpt (\rho u)+\dvg (\rho u\otimes u)$, and the remaining term is computed as
\begin{align}
&\int_{\R^3}\big((-\Delta)^{-1}\dvg (\rho \dot{u})\big)\varrho^h=\dfrac{d}{dt}\int_{\R^3}\big((-\Delta)^{-1}\dvg (\rho u)\big)\varrho^h\notag\\
&-\int_{\R^3}\big((-\Delta)^{-1}\dvg (\rho u)\big)\dpt \varrho^h+\int_{\R^3}\big((-\Delta)^{-1}\dvg\dvg  (\rho u\otimes u)\big)\varrho^h.\label{sec3:eq6}
\end{align}
We bound the last term as: 
\begin{align*}
\Bigg|\int_0^t\scalar{t'}^{r_0}\int_{\R^3}\big((-\Delta)^{-1}\dvg\dvg  (\rho u\otimes u)\big)\varrho^h\Bigg|
&\leqslant C_*\int_0^t\scalar{t'}^{r_0}\normb{\varrho}_{L^{2}(\R^3)}\norm{u}_{L^{4}(\R^3)}^2.
\end{align*}
Observing that $\varrho$ solves $\dpt \varrho+\dvg (\varrho u)+\widetilde\rho(c)\partial_{\rho} P(\widetilde\rho(c),c)\dvg u=0$, writing $\rho=(\rho-\widetilde\rho)+ \widetilde\rho$, and applying H\"older and Sobolev's inequalities, we obtain :
\begin{align*}
    &\Bigg|\int_0^t\scalar{t'}^{r_0}\int_{\R^3}\big((-\Delta)^{-1}\dvg (\rho u)\big)\dpt \varrho^h\Bigg|
    \leqslant \Bigg|\int_0^t\scalar{t'}^{r_0}\int_{\R^3}\big((-\Delta)^{-1}\dvg (\rho u)\big) \dvg (\varrho u)^h\Bigg|\\
    &+\Bigg|\int_0^t\scalar{t'}^{r_0}\int_{\R^3}\big((-\Delta)^{-1}\dvg (\rho u)\big)\big(\widetilde\rho(c)\partial_{\rho} P(\widetilde\rho(c),c)\dvg u\big)^h\Bigg|\\
    &\leqslant \Bigg|\int_0^t\scalar{t'}^{r_0}\int_{\R^3}\big((-\Delta)^{-1}\dvg \big((\rho-\widetilde\rho) u\big)\big) \dvg (\varrho u)^h\Bigg|+\widetilde\rho\Bigg|\int_0^t\scalar{t'}^{r_0}\int_{\R^3}\big((-\Delta)^{-1}\dvg u\big) \dvg (\varrho u)^h\Bigg|\\
    &+\Bigg|\int_0^t\scalar{t'}^{r_0}\int_{\R^3}\big((-\Delta)^{-1}\dvg \big((\rho-\widetilde\rho) u\big)\big)\big(\widetilde\rho(c)\partial_{\rho} P(\widetilde\rho(c),c)\dvg u\big)^h\Bigg|\\
    &+\widetilde\rho\Bigg|\int_0^t\scalar{t'}^{r_0}\int_{\R^3}\big((-\Delta)^{-1}\dvg  u\big)\big(\widetilde\rho(c)\partial_{\rho} P(\widetilde\rho(c),c)\dvg u\big)^h\Bigg|\\
    &\leqslant C\sup_{[0,t]}\big(1+\normb{\rho-\widetilde\rho}_{L^3(\R^3)}\big)\big(1+ \normb{\varrho}_{L^3(\R^3)} \big)\int_0^t\scalar{t'}^{r_0}\norm{\nabla u}_{L^2(\R^3)}^2.
\end{align*}
The time integral of the remaining term in \eqref{sec3:eq6} is
\begin{align*}
    \Bigg|\int_0^t\scalar{t'}^{r_0}\dfrac{d}{dt}\int_{\R^3}\big((-\Delta)^{-1}\dvg (\rho u)\big)\varrho^hdt'\Bigg|&\leqslant C\sup_{[0,t]}\big(\scalar{\cdot}^{r_0}\normb{\rho u}_{L^2(\R^3)}\normb{\varrho}_{L^2(\R^3)}\big)\\
    &+C\int_0^t\scalar{t'}^{r_0-1}\normb{\rho u}_{L^2(\R^3)}\normb{\varrho}_{L^2(\R^3)}dt'.
\end{align*}
Summing everything together, we obtain:
\begin{align*}
\int_0^t\int_{\R^3} G(\rho,c)\varrho^h&\leqslant \eta \int_0^t\scalar{t'}^{r_0}\norm{\varrho}_{L^2(\R^3)}^2dt'+C\sup_{[0,t]}\big(\scalar{\cdot}^{r_0}\normb{\rho u}_{L^2(\R^3)}\normb{\varrho}_{L^2(\R^3)}\big)\\
&+C\int_0^t\scalar{t'}^{r_0-1}\normb{\rho u}_{L^2(\R^3)}\normb{\varrho}_{L^2(\R^3)}dt'+C_*\int_0^t\scalar{t'}^{r_0}\normb{\varrho}_{L^{2}(\R^3)}\norm{u}_{L^{4}(\R^3)}^2dt'\\
&+C_*\sup_{[0,t]}\Big(\frac{1}{\eta}+\normb{\rho-\widetilde\rho}_{L^3(\R^3)}\Big)\big(1+ \normb{\varrho}_{L^3(\R^3)} \big)\int_0^t\scalar{t'}^{r_0}\norm{\nabla u}_{L^2(\R^3)}^2dt'.
\end{align*}
Recalling  that $\widetilde P= P(\widetilde\rho(c),c)$, we infer
\begin{gather}\label{sec3:eq7}
G(\rho,c)= \big(\rho-\widetilde\rho(c)\big)\int_0^1 \partial_{\rho} P\big(r \rho+ (1-r)\widetilde\rho(c),c\big)d  r,
\end{gather}
and using the monotonicity of the map $\rho \mapsto P(\rho,c)$, we obtain
\begin{align*}
&\frac{1}{C_*}\int_0^t\scalar{t'}^{r_0} \normb{\varrho}_{L^2(\R^3)}dt'\leqslant \int_0^t\scalar{t'}^{r_0}\int_{\R^3} G(\rho,c)\varrho dt'\\
&= \int_0^t\scalar{t'}^{r_0}\int_{\R^3} G(\rho,c)\varrho^ldt' +\int_0^t\scalar{t'}^{r_0}\int_{\R^3} G(\rho,c)\varrho^hdt'\\
&\leqslant \eta \int_0^t\scalar{t'}^{r_0}\normb{\varrho}_{L^2(\R^3)}^2dt'+ \int_0^t\scalar{t'}^{r_0}\int_{\R^3} G(\rho,c)\varrho^l+C\sup_{[0,t]}\big(\scalar{\cdot}^{r_0}\normb{\rho u}_{L^2(\R^3)}\normb{\varrho}_{L^2(\R^3)}\big)\\
&+C\int_0^t\scalar{t'}^{r_0-1}\normb{\rho u}_{L^2(\R^3)}\normb{\varrho}_{L^2(\R^3)}dt'+C_*\int_0^t\scalar{t'}^{r_0}\normb{\varrho}_{L^{2}(\R^3)}\norm{u}_{L^{4}(\R^3)}^2dt'\\
&+C_*\sup_{[0,t]}\Big(\frac{1}{\eta}+\normb{\rho-\widetilde\rho}_{L^3(\R^3)}\Big)\big(1+ \normb{\varrho}_{L^3(\R^3)} \big)\int_0^t\scalar{t'}^{r_0}\norm{\nabla u}_{L^2(\R^3)}^2dt'.
\end{align*}
From \eqref{sec3:eq7}, we estimate the second term of the RHS above as:
\[
\abs{\int_0^t\scalar{t'}^{r_0}\int_{\R^3} G(\rho,c)\varrho^l}dt'\leqslant  \eta\int_0^t\scalar{t'}^{r_0}\normb{\varrho}_{L^2(\R^3)}^2dt'+\frac{C_*}{\eta}\int_0^t\scalar{t'}^{r_0}\normb{\varrho^l}_{L^2(\R^3)}^2dt'.
\]
Holder's and Young's inequalities imply (since $r_0\in (0,1)$)
\[
C\int_0^t\scalar{t'}^{r_0-1}\normb{\rho u}_{L^2(\R^3)}\normb{\varrho}_{L^2(\R^3)}dt'\leqslant \eta \int_0^t\scalar{t'}^{r_0}
\normb{\varrho}_{L^2(\R^3)}^2dt'+ \frac{C}{\eta} E_0.
\]
Interpolation inequality yields
\begin{align*}
C_*\int_0^t\scalar{t'}^{r_0}&\normb{\varrho}_{L^{2}(\R^3)}\norm{u}_{L^{4}(\R^3)}^2dt'\leqslant C_*\int_0^t\scalar{t'}^{r_0}\normb{\varrho}_{L^{2}(\R^3)}\norm{u}_{L^{2}(\R^3)}^{\frac{1}{2}}\normb{\nabla u}_{L^2(\R^3)}^{\frac{3}{2}}\\
&\leqslant C_*\Bigg(\int_0^t\scalar{t'}^{r_0}\normb{\varrho}_{L^{2}(\R^3)}^4\norm{u}_{L^{2}(\R^3)}^2dt'\Bigg)^{1/4}\Bigg(\int_0^t\scalar{t'}^{r_0}\normb{\nabla u}_{L^2(\R^3)}^2dt'\Bigg)^{3/4}\\
&\leqslant \eta \int_0^t\scalar{t'}^{r_0} \normb{\varrho}_{L^2(\R^3)}^2dt'+ \frac{C_*}{\eta} E_0^{\frac{2}{3}}\int_0^t\scalar{t'}^{r_0}\normb{\nabla u}_{L^2(\R^3)}^2dt'.
\end{align*}
It follows that 
\begin{align}\label{sec3:eq8}
\int_0^t\scalar{t'}^{r_0} \normb{\varrho}_{L^2(\R^3)}^2dt'&\leqslant C_* \int_0^t\scalar{t'}^{r_0}\normb{\varrho^l}_{L^2(\R^3)}^2dt'+C\sup_{[0,t]}\big(\scalar{\cdot}^{r_0}\normb{\rho u}_{L^2(\R^3)}\normb{\varrho}_{L^2(\R^3)}\big)\notag\\
&+C_* E_0+ C_*\big(1+E_0^{2/3}+\sup_{[0,t]}\normb{\rho-\widetilde\rho}_{L^3(\R^3)}\big(1+ \normb{\varrho}_{L^3(\R^3)} \big)\big)\int_0^t\scalar{t'}^{r_0}\normb{\nabla u}_{L^2(\R^3)}^2dt'.
\end{align}
Estimates \eqref{sec3:eq8}-\eqref{sec3:eq9} conclude this paragraph, and we will now derive bounds for the low-frequency component of the velocity and density fluctuations. 
\paragraph{\textbf{Step 2: $L^2_{t,x}$-bound for the low-frequency components}} In order to derive an estimate for the $L^2_{t,x}$-bound of the low-frequency component of the velocity and density fluctuations, we first write the following equations from \eqref{ep2.1}:
\begin{gather*}
    \begin{cases}
        \dpt \varrho + \partial_{\rho}P(\widetilde\rho,1)\dvg (\rho u)&=\dvg \big(\partial_\rho P(\widetilde\rho,1)(\rho-\widetilde\rho) u-\partial_{\rho} P(\widetilde\rho(c),c))(\rho-\widetilde\rho(c)) u\big)\\
        &+\big(\widetilde\rho \partial_{\rho} P(\widetilde\rho,1)-\widetilde\rho(c) \partial_{\rho} P(\widetilde\rho(c),c)\big)\dvg u\\
        &=:\mathbb F,\\
        \dpt (\rho u)+\nabla \varrho &-\frac{\widetilde\mu}{\widetilde\rho}\Delta (\rho u)-\frac{2\widetilde\mu+\widetilde\lambda}{\widetilde\rho}\nabla\dvg (\rho u)\\
        &=-\dvg (\rho u\otimes u)-\nabla\big( P(\rho,c)-\partial_{\rho} P(\widetilde\rho(c),c)\varrho-\widetilde P\big)+\dvg \big(2(\mu(\rho,c)-\widetilde\mu)\D u\big)\\
        &+\nabla\big((\lambda(\rho,c)-\widetilde\lambda)\dvg u\big)+ \frac{\widetilde\mu}{\widetilde\rho}\Delta \big(\widetilde\rho-\rho) u\big)+\frac{2\widetilde\mu+\widetilde\lambda}{\widetilde\rho}\nabla\dvg \big(\widetilde\rho-\rho)u\big)\\
        &=: \mathbb G. 
    \end{cases}
\end{gather*}
Defining the differential operator $\mathbb{A}$ by
\[
\mathbb A= \begin{pmatrix}
    0 & \partial_{\rho} P(\widetilde\rho,1)\dvg \\
     \nabla & -\frac{\widetilde\mu}{\widetilde\rho}\Delta-\frac{2\widetilde\mu+\widetilde\lambda}{\widetilde\rho}\nabla\dvg
\end{pmatrix},
\]
the pair $(\varrho, \rho u)$ can then be written as:
\begin{gather}\label{ep3.2}
    \begin{pmatrix}
        \varrho\\
        \rho u
    \end{pmatrix}(t)=e^{-t \mathbb A } \begin{pmatrix}
        \varrho\\
        \rho u
    \end{pmatrix}(0)+\int_0^t e^{-\mathbb (t-s)\mathbb A} \begin{pmatrix}
        \mathbb F\\
        \mathbb G
    \end{pmatrix}(s)ds.
\end{gather}
The semi-group $\big(e^{-t \mathbb{A}}\big)_{t\geqslant 0}$ is known (see \cite{charve2010global,LiLargetime}) to satisfy the following property: there exist positive constants $C$ and $k_0$ depending only on the constants coefficients of $\mathbb A$, such that: 
\begin{gather}\label{ep3.9}
\abs{\cF\big( e^{-t \mathbb A } U\big)(\xi)}\leqslant C e^{-k_0 t\abs{\xi}^2}\abs{\cF \big(U\big)(\xi)},\; \text{ for all } \abs{\xi}\leqslant 1.
\end{gather}
It immediately follows that for all $m\in \N$, and $U\in L^q(\R^3)$, with $q\in [1,2]$, we have
\begin{gather*}
 \norm{\big(e^{-t \mathbb A } \nabla^m U\big)^l}_{L^2(\R^3)}\leqslant C\scalar{t}^{-\frac{3}{2}(\frac{1}{q}-\frac{1}{2}+\frac{m}{3})}\norm{U}_{L^q(\R^3)},\; \text{ for all } t>0.
\end{gather*}
As a result, we have the following bounds.
\begin{itemize}
    \item With $p_0\in (1, 6/5)$ and  $r_0\in (0,3/p_0-5/2)$, we have $r_0-\frac{3}{p_0}+\frac{3}{2}<-1$, and  we can estimate the first term on the RHS of \eqref{ep3.2} as:
    \begin{align*}
        \int_0^t\scalar{t'}^{r_0}\norm{ e^{-t' \mathbb A } 
        \begin{pmatrix}
            \varrho\\
        \rho u
        \end{pmatrix}_{|t=0}^l}^2_{L^2(\R^3)}d t'&\leqslant C \norm{(\varrho, \rho u)_{|t=0}}_{L^{p_0}(\R^3)}^2\int_0^t\scalar{t'}^{r_0-\frac{3}{p_0}+\frac{3}{2}}d\tau\\
        &\leqslant C \norm{(\varrho, \rho u)_{|t=0}}_{L^{p_0}(\R^3)}^2.
    \end{align*}
    \item The source term $\mathbb{G}$ is expressed as 
    \[
    \mathbb G= \mathbb G_1+\mathbb G_2, \;\text{ where }\; \mathbb G_2=\frac{\widetilde\mu}{\widetilde\rho}\Delta \big(\widetilde\rho-\rho) u\big)+\frac{2\widetilde\mu+\widetilde\lambda}{\widetilde\rho}\nabla\dvg \big(\widetilde\rho-\rho)u\big),
    \]
    and we estimate
    \begin{equation*}
        \norm{\int_0^{t'} e^{-\mathbb (t'-s)\mathbb A} \begin{pmatrix}
        0\\
        \mathbb G_2^l
    \end{pmatrix}(s)ds}_{L^2(\R^3)} \leqslant C\int_0^{t'}\scalar{t'-s}^{-\frac{5}{4}}\normb{\rho-\widetilde\rho}_{L^2(\R^3)}\norm{u}_{L^6(\R^3)}(s)ds.
    \end{equation*}
   Given that for all $s\in(0,t')$, we have $\scalar{t'}^{\frac{r_0}{2}}\leqslant \scalar{s}^{\frac{r_0}{2}}+\scalar{t-s}^{\frac{r_0}{2}}$, we have 
    \begin{align*}
        &\norm{\scalar{t'}^{\frac{r_0}{2}}\int_0^{t'} e^{-\mathbb (t'-s)\mathbb A} \begin{pmatrix}
        0\\
        \mathbb G_2^l
    \end{pmatrix}(s)ds}_{L^2(0,t),L^2(\R^3))} \\
    &\leqslant C\sup_{[0,t]}\normb{\rho-\widetilde\rho}_{L^2(\R^3)}\norm{\int_0^{t'}\scalar{t'-s}^{-\frac{5}{4}+\frac{r_0}{2}}\norm{u}_{L^6(\R^3)}(s)ds}_{L^2(0,t)}\\
    &+C\sup_{[0,t]}\normb{\rho-\widetilde\rho}_{L^2(\R^3)}\norm{\int_0^{t'}\scalar{t'-s}^{-\frac{5}{4}}\scalar{s}^{\frac{r_0}{2}}\norm{u}_{L^6(\R^3)}(s)ds}_{L^2((0,t))},
    \end{align*}
    and Young's inequality for convolution and Sobolev embedding yield:
    \begin{align*}
        &\norm{\scalar{t'}^{\frac{r_0}{2}}\int_0^{t'} e^{-\mathbb (t'-s)\mathbb A} \begin{pmatrix}
        0\\
        \mathbb G_2^l
    \end{pmatrix}(s)ds}_{L^2(0,t),L^2(\R^3))} \\
    &\leqslant C\sup_{[0,t]}\normb{\rho-\widetilde\rho}_{L^2(\R^3)}\Big(\norm{\nabla u}_{L^{1}((0,t),L^2(\R^3))}+\norm{\scalar{\cdot}^{\frac{r_0}{2}}\nabla u}_{L^{5/4}((0,t)\times\R^3)}\Big).
    \end{align*}

 For the remaining term, we first observe that (see \eqref{ep2.11}) 
 \[
 \norm{\widetilde\rho(c(t))-\widetilde\rho,\; \partial_\rho P(\widetilde\rho(t),c(t))-\partial_\rho P(\widetilde \rho,1)}_{L^p(\R^3)}\leqslant C \abs{D(t)}^{1/p}\leqslant C \check E_0^{1/p},
 \]
 and via Taylor expansion
  \[
 \norm{P(\rho,c)-\varrho-\widetilde P}_{L^1(\R^3)}\leqslant C_*\normb{\varrho}_{L^2(\R^3)}^2.
 \]
 It then follows that (recall that $r_0<1/2$)
 \begin{align*}
        &\norm{\scalar{t'}^{\frac{r_0}{2}}\int_0^{t'} e^{-\mathbb (t'-s)\mathbb A} \begin{pmatrix}
        \mathbb F\\
        \mathbb G_1
    \end{pmatrix}^l(s)ds}_{L^2((0,t)\times \R^3)} \\
    &\leqslant C_*\norm{\int_0^{t'}\scalar{t-s}^{-\frac{5}{4}+\frac{r_0}{2}}\normb{\varrho}_{L^2(\R^3)}^2(s)ds}_{L^2(0,t)}+C_*\norm{\int_0^{t'}\scalar{t'-s}^{-\frac{5}{4}}\scalar{s}^{\frac{r_0}{2}}\normb{\varrho}_{L^2(\R^3)}^2(s)ds}_{L^2(0,t)}\\
    &+ C_*\Big(E_0^\frac{1}{2}+\check{E}_0^\frac{1}{2}\Big)\norm{\int_0^{t'}\scalar{t'-s}^{-\frac{3}{4}+\frac{r_0}{2}}\norm{\nabla u}_{L^2(\R^3)}(s)ds}_{L^2(0,t)}\\
    &+ C_*\Big(E_0^\frac{1}{2}+\check{E}_0^\frac{1}{2}\Big)\norm{\int_0^{t'}\scalar{t'-s}^{-\frac{3}{4}}\scalar{s}^{\frac{r_0}{2}}\norm{\nabla u}_{L^2(\R^3)}(s)ds}_{L^2(0,t)}\\
    &+C_*\Big(E_0^\frac{1}{2}+\check{E}_0^\frac{1}{2}\Big)\norm{\int_0^{t'}\scalar{t'-s}^{-\frac{5}{4}+\frac{r_0}{2}}\norm{\nabla u}_{L^2(\R^3)}(s)ds}_{L^2(0,t)}\\
    &+C_*\Big(E_0^\frac{1}{2}+\check{E}_0^\frac{1}{2}\Big)\norm{\int_0^{t'}\scalar{t'-s}^{-\frac{5}{4}}\scalar{s}^{\frac{r_0}{2}}\norm{\nabla u}_{L^2(\R^3)}(s)ds}_{L^2(0,t)}\\
    &\leqslant C_* \big( E_0^{\frac{1}{2}}+\check E_0^{\frac{1}{2}}\big)\Big(\norm{\nabla u}_{L^1((0,t), L^2(\R^3))}+\norm{\scalar{\cdot}^{\frac{r_0}{2}}\nabla u}_{L^{5/4}((0,t), L^2(\R^3))}\Big)+C_*\norm{\scalar{\cdot}^{\frac{r_0}{4}}\varrho}_{L^4((0,t), L^2(\R^3))}^2.
 \end{align*}
 
\end{itemize}
Summing all these computations, we arrive at
\begin{align}
\int_0^t\scalar{t'}^{r_0}\norm{\varrho^l, (\rho u)^l}_{L^2(\R^3)}^2(t')dt'&\leqslant C \norm{(\varrho, \rho u)_{|t=0}}_{L^{p_0}(\R^3)}^2+C_*E_0\int_0^t\scalar{t'}^{r_0}\norm{\varrho}_{L^2(\R^3))}^2\notag\\
&+C_* \big( E_0+\check E_0\big)\Big(\norm{\nabla u}_{L^1((0,t), L^2(\R^3))}+\norm{\scalar{\cdot}^{\frac{r_0}{2}}\nabla u}_{L^{5/4}((0,t), L^2(\R^3))}\Big).\label{sec3:eq13}
\end{align}
Applying H\"older's and Sobolev's inequalities, we have 
\[
\normb{\widetilde\rho u^l}_{L^2(\R^3)}^2\leqslant \normb{(\rho u)^l}_{L^2(\R^3)}^2+ \normb{\rho-\widetilde\rho}_{L^3(\R^3)}^2\norm{\nabla u}_{L^2(\R^3)}^2.
\]
Combining these with \eqref{sec3:eq9} and \eqref{sec3:eq8}, and the smallness of $E_0$ (to absorb the second term on the RHS of \eqref{sec3:eq13} into the LHS) we derive:
\begin{align*}
\int_0^t\scalar{t'}^{r_0}\normb{\varrho, u}_{L^2(\R^3)}^2dt'&\leqslant C_* \Big(E_0+  \norm{(\varrho, \rho u)_{|t=0}}_{L^{p_0}(\R^3)}^2+\sup_{[0,t]}\big(\scalar{\cdot}^{r_0}\normb{\rho u}_{L^2(\R^3)}\normb{\varrho}_{L^2(\R^3)}\big)\Big)\\
&+C_* \big( E_0+\check E_0\big)\Big(\norm{\nabla u}_{L^1((0,t), L^2(\R^3))}+\norm{\scalar{\cdot}^{\frac{r_0}{2}}\nabla u}_{L^{5/4}((0,t), L^2(\R^3))}\Big)\\
&+ C_*\big(1+E_0^{2/3}+\check E_0^{2/3}\big)\int_0^t\scalar{t'}^{r_0}\normb{\nabla u}_{L^2(\R^3)}^2dt'.
\end{align*}
With this estimate in hand, we multiply \eqref{sec3:eq4} by $\scalar{t'}^{1+r_0}$ and integrate in time to obtain:
\begin{align*}
    \scalar{t}^{1+r_0} E(t)+\int_0^t \scalar{t'}^{r_0}E(t')dt'&+\int_0^t \scalar{t'}^{1+r_0}\int_{\R^3}\big(2\mu(\rho,c)\abs{\D u}^2+\lambda(\rho,c)(\dvg u)^2\big)(t')dt'\\
    &\leqslant  E_0+C\int_0^t \scalar{t'}^{r_0}E(t')dt'\\
    &\leqslant C_* \Big(E_0+  \norm{(\varrho, \rho u)_{|t=0}}_{L^{p_0}(\R^3)}^2+\sup_{[0,t]}\big(\scalar{\cdot}^{r_0}\normb{\rho u}_{L^2(\R^3)}\normb{\varrho}_{L^2(\R^3)}\big)\Big)\\
&+C_* \big( E_0+\check E_0\big)\Big(\norm{\nabla u}_{L^1((0,t), L^2(\R^3))}+\norm{\scalar{\cdot}^{\frac{r_0}{2}}\nabla u}_{L^{5/4}((0,t), L^2(\R^3))}\Big)\\
&+ C_*\big(1+E_0^{2/3}+\check E_0^{2/3}\big)\int_0^t\scalar{t'}^{r_0}\normb{\nabla u}_{L^2(\R^3)}^2dt'.
\end{align*}
It is straightforward to obtain the following inequalities:
\begin{align*}
\int_0^t\norm{\nabla u(t')}_{L^2(\R^3)}dt'&\leqslant \Bigg(\int_0^t\scalar{t'}^{-2r(1+r_0)}dt'\Bigg)^{\frac{1}{2}}\Bigg(\int_0^t\norm{\nabla u(t')}_{L^2(\R^3)}^2\Bigg)^{\frac{1}{2}-r}\Bigg(\int_0^t\scalar{t'}^{1+r_0}\norm{\nabla u(t')}_{L^2(\R^3)}^2dt'\Bigg)^r\\
&\leqslant C(r)E_0^{\frac{1}{2}-r}\Bigg(\int_0^t\scalar{t'}^{1+r_0}\norm{\nabla u(t')}_{L^2(\R^3)}^2dt'\Bigg)^r,\;  \text{ with }\; \frac{1}{2(1+r_0)}<r<\frac{1}{2};\\
\norm{\scalar{\cdot}^{\frac{r_0}{2}}\nabla u}_{L^{5/4}((0,t), L^2(\R^3))}^{5/4}&\leqslant C(r) E_0^{\frac{5}{8}-\check r}\Bigg(\int_0^t\scalar{t'}^{1+r_0}\norm{\nabla u(s)}_{L^2(\R^3)}^2ds\Bigg)^{\check r},\; \text{ with }\;  \frac{5}{8}-\frac{1/4}{1+r_0}<\check r<\frac{5}{8};\\
\int_0^t\scalar{t'}^{r_0}\normb{\nabla u}_{L^2(\R^3)}^2dt'&\leqslant E_0^{\frac{1}{1+r_0}} \Bigg(\int_0^t\scalar{t'}^{1+r_0}\norm{\nabla u(t')}_{L^2(\R^3)}^2dt'\Bigg)^{\frac{r_0}{1+r_0}}.
\end{align*}
We choose 
\[
\check{r}=\frac{5}{4}r,\;\text{ with }\; \max\Bigg\{1-\frac{r_0}{1+r_0};1-\frac{2/5}{1+r_0}\Bigg\}< 2r<1,
\]
and apply Young's inequality to obtain
\begin{align*}
    \sup_{[0,t]}\scalar{\cdot}^{1+r_0} E&+ \int_0^t\scalar{t'}^{r_0} E(t')dt'+\int_0^t \scalar{t'}^{1+r_0}\int_{\R^3}\big(2\mu(\rho,c)\abs{\D u}^2+\lambda(\rho,c)(\dvg u)^2\big)dt'\\
    &\leqslant C_* \Big( E_0\big(1+E_0+\check{E}_0\big)+E_0^{\frac{r}{1-r}}\big(E_0+\check{E}_0\big)^{\frac{1}{1-r}}+\norm{(\varrho, \rho u)_{|t=0}}_{L^{p_0}(\R^3)}^2\Big).
\end{align*}
This completes the proof of \eqref{ep3.8}, and we now proceed to the bound of the functionals $\mathcal{A}_1$, $\mathcal{A}_2$, and $\mathcal{A}_3$ defined  in \eqref{sec2:eq48}.
\end{proof}

\subsubsection{Bounds for the functional $\mathcal{A}_1$}\label{sec3.1:funcA1} This subsection aim to provide a bound for the functional $\mathcal{A}_1$ defined in \eqref{sec2:eq48}. 
\begin{proof}
To achieve this, we use the material acceleration $\dot u$
as a test function in the momentum equation $\eqref{ep2.1}_3$ and we obtain 
\begin{align}
     \dfrac{d}{dt}\int_{\R^3}\Big[\mu(\rho,c)\abs{\D u}^2&+\dfrac{1}{2}\lambda(\rho,c)\abs{\dvg u}^2\Big]+\int_{\R^3}\rho\abs{\dot u}^2=- \int_{\R^3} \left(\rho \partial_\rho \mu(\rho,c)-\mu(\rho,c)\right) \abs{\D u}^2 \dvg u\nonumber\\
    &-2 \int_{\R^3}\mu(\rho,c) \D^{jk} u \partial_k u^l \partial_l u^j-\dfrac{1}{2} \int_{\R^3}\left(\rho \partial_\rho \lambda(\rho,c)-\lambda(\rho,c)\right)(\dvg u)^3\nonumber\\
    &-\int_{\R^3}\lambda(\rho,c)\dvg u \nabla u^k \partial_k u+\dfrac{d}{dt}\int_{\R^3} G(\rho,c)\dvg u\nonumber\\
    &+ \int_{\R^3} (\rho \partial_{\rho} P(\rho,c)- G(\rho,c))(\dvg u)^2+\int_{\R^3} G(\rho,c)\nabla u^k \cdot\partial_k u.\label{sec3:eq11}
\end{align}
We multiply \eqref{sec3:eq4} by a constant $L\geqslant 1$ and add it 
to the above equality, then multiply the resulting expression by $\kappa_1(t)=\scalar{t}^{1+r_0}$ before integrating in 
time:
\begin{align}
    &\frac{\kappa_1(t)}{2}\int_{\R^3}\Big[2\mu(\rho,c)\abs{\D u}^2+\lambda(\rho,c)\abs{\dvg u}^2\Big](t,x)dx+\int_0^t \kappa_1\int_{\R^3}\rho\abs{\dot u}^2\notag\\
    &+ L\kappa_1(t) E(t)+\int_{0}^t\kappa_1(t')\int_{\R^3}\Big[2L\mu(\rho,c)\abs{\D u}^2+ \big(L\lambda(\rho,c)-\partial_{\rho} P(\rho,c)\big)(\dvg u)^2\Big](t')dt'\notag\\
    &=L E_0+\int_{\R^3}\Big[\mu(\rho_0,c_0)\abs{\D u_0}^2+\dfrac{1}{2}\lambda(\rho_0,c_0)\abs{\dvg u_0}^2\Big]+L\int_0^t\kappa_1'(t') E(t')dt'\notag\\
    &+\int_0^t\kappa_1'(t')\int_{\R^3}\Big[\mu(\rho,c)\abs{\D u}^2+\dfrac{1}{2}\lambda(\rho,c)\abs{\dvg u}^2\Big]dt'-\int_{\R^3} G(\rho_0,c_0)\dvg u_0\notag\\
    &+\kappa_1(t)\int_{\R^3} G(\rho,c)\dvg u-\int_0^t\kappa_1'(t')\int_{\R^3} G(\rho,c)\dvg u-2 \int_0^t\kappa_1(t')\int_{\R^3}\mu(\rho,c) \D^{jk} u \partial_k u^l \partial_l u^j\notag\\
    &- \int_0^t\kappa_1(t')\int_{\R^3} \left(\rho \partial_\rho \mu(\rho,c)-\mu(\rho,c)\right) \abs{\D u}^2 \dvg u-\dfrac{1}{2} \int_0^t\kappa_1(t')\int_{\R^3}\left(\rho \partial_\rho \lambda(\rho,c)-\lambda(\rho,c)\right)(\dvg u)^3\notag\\
    &-\int_0^t\kappa_1(t')\int_{\R^3}\lambda(\rho,c)\dvg u \nabla u^k \partial_k u- \int_0^t\kappa_1(t')\int_{\R^3} G(\rho,c)(\dvg u)^2+\int_0^t\kappa_1(t')\int_{\R^3} G(\rho,c)\nabla u^k \cdot\partial_k u.\label{sec3:eq10}
\end{align}
 In view of the classical energy balance \eqref{sec3:eq4} and \eqref{ep3.8}, the terms appearing in the second and third lines above are all bounded by
    \[
     C_*\big(\norm{\nabla u_0}_{L^2(\R^3)}^2+ L C_0\big).
    \]
    With the help of H\"older and Young's inequalities the term that follows can be bounded as follows: 
    \[
    \Bigg|\kappa_1(t)\int_{\R^3} G(\rho,c)\dvg u\Bigg|\leqslant \frac{\kappa_1(t)}{\eta L}\norm{\dvg u(t)}_{L^2(\R^3)}^2+L\eta \kappa_1(t)\norm{G(\rho,c)}_{L^2(\R^3)}^2.
    \]
Next, using \eqref{ep2.4} we have:
\begin{align*}
    -\int_0^t\kappa_1'(t')\int_{\R^3} G(\rho,c)\dvg u&=\kappa_1'(t) H_1(\rho(t),c(t))-\int_0^t\kappa_1''(t')H_1(\rho,c)(t')dt',
\end{align*}
and hence
\[
\Bigg|\int_0^t\kappa_1'(t')\int_{\R^3} G(\rho,c)\dvg u\Bigg|\leqslant C_*  C_0.
\]
All the remaining terms in the RHS of \eqref{sec3:eq10} are bounded by 
\[
C_*\int_0^t\kappa_1(t') \big(\norm{\nabla u}_{L^3(\R^3)}^3+\norm{G(\rho,c)}_{L^3(\R^3)}^3\big)(t')dt'.
\]
In sum, we obtain 
\begin{align*}
    \mathcal{A}_1(t)+ L\kappa_1(t) E(t)+&\int_{0}^t\kappa_1(t')\int_{\R^3}\Big[2L\mu(\rho,c)\abs{\D u}^2+ \big(L\lambda(\rho,c)-\partial_{\rho} P(\rho,c)\big)(\dvg u)^2\Big](t')dt'\\
    &\leqslant L\eta \kappa_1(t)\normb{G(\rho,c)}_{L^2(\R^3)}^2+
\frac{1}{\eta L}\kappa_1(t)\norm{\dvg u(t)}_{L^2(\R^3)}^2\\
    &+C_*\Bigg(\norm{\nabla u_0}_{L^2(\R^3)}^2+ L C_0+\int_0^t\kappa_1(t') \big(\norm{\nabla u}_{L^3(\R^3)}^3+\norm{G(\rho,c)}_{L^3(\R^3)}^3\big)(t')dt'\Bigg).
\end{align*}
Given that 
\[
\normb{G(\rho,c)}_{L^2(\R^3)}^2\leqslant C_* \int_{\R^3} H_1(\rho,c),
\]
we first fix $\eta$ small to absorb the first RHS term into the left, and then choose $L$ large, depending on the shear viscosity lower bound, to ensure that the last LHS term is nonnegative and the second RHS term is absorbed into the left. It then follows that 
\begin{align}\label{sec3:eq14}
    \mathcal{A}_1(t)&\leqslant C_*\Bigg(\norm{\nabla u_0}_{L^2(\R^3)}^2+C_0+\int_0^t\kappa_1(t')  \big(\norm{\nabla u}_{L^3(\R^3)}^3+\normb{G(\rho,c)}_{L^3(\R^3)}^3\big)(t')dt'\Bigg).
\end{align}
What remains is to estimate the last two terms above. 

To achieve this, we recall the expression of the velocity gradient \eqref{sec2:eq3}, from which we deduce
\begin{align}
    \norm{\nabla u}_{L^{l+1}(\R^3)}&\leqslant \frac{1}{\mu_*}C_l^{1/(l+1)}\Big(\norm{(-\Delta)^{-1}\nabla \mathbb P(\rho \dot u)}_{L^{l+1}(\R^3)}+\norm{(-\Delta)^{-1}\dvg (\rho \dot u)}_{L^{l+1}(\R^3)}+\norm{ G(\rho,c)}_{L^{l+1}(\R^3)}\Big)\notag\\
    &+\frac{1}{\mu_*}C_l^{1/(l+1)}\Big(\norm{\big[\mathcal{K}, \mu(\rho,c)-\widetilde\mu\big]\D u}_{L^{l+1}(\R^3)}+\norm{\big[\check{\mathcal{K}}, \mu(\rho,c)-\widetilde\mu\big]\D u}_{L^{l+1}(\R^3)}\Big).\label{sec3:eq20}
\end{align}
Next, we recall \eqref{sec3:eq12}, from which we deduce:
\begin{align*}
\kappa_1(t)  \int_{\R^3}H_2(\rho,c)(t)&+\dfrac{1}{3\mu^*}\int_0^t\kappa_1(t') \norm{G(\rho,c)}^{3}_{L^3(\R^3)}(t')dt' \\
&\leqslant \int_{\R^3}H_2(\rho_0,c_0)+\int_0^t\kappa_1'(t')\int_{\R^3}H_2(\rho,c)dt'+\dfrac{1}{3\mu_*}\int_0^t\kappa_1(t') \norm{F}^{3}_{L^3(\R^3)}(t')dt'.
\end{align*}
It then follows that (see the expression of $F$ in \eqref{sec3:eq5})
\begin{align}
   \kappa_1(t) \int_{\R^3}H_2(\rho,c)(t,x) dx&+ \dfrac{1}{3\mu^*}\int_0^t\kappa_1(t')\Big(\norm{G(\rho,c)}^{3}_{L^3(\R^3)}+\mu_*^3\norm{\nabla u}_{L^3(\R^3)}^3\Big)(t')dt'\notag\\
   &\leqslant C\Bigg(\int_{\R^3}H_2(\rho_0,c_0)+\int_0^t\kappa_1'(t') \int_{\R^3}H_2(\rho,c)\Bigg)\notag\\
   &+\frac{C}{\mu_*}\int_0^t\kappa_1(t')\Big(\norm{(-\Delta)^{-1}\nabla \mathbb P(\rho \dot u)}_{L^{3}(\R^3)}^3+\norm{(-\Delta)^{-1}\dvg (\rho \dot u)}_{L^{3}(\R^3)}^3\Big)(t')dt'\notag\\
    &+\frac{C}{\mu_*}\int_0^t\kappa_1(t')\Big(\norm{\big[\mathcal{K}, \mu(\rho,c)-\widetilde\mu\big]\D u}_{L^{3}(\R^3)}^3+\norm{\big[\check{\mathcal{K}}, \mu(\rho,c)-\widetilde\mu\big]\D u}_{L^{3}(\R^3)}^3\Big)(t')dt'.\notag
\end{align}
We will now estimate the terms appearing in the RHS above. 
\begin{itemize}
    \item From \eqref{ep2.7}-\eqref{ep2.8}, we have:
    \begin{align}
        \widetilde\rho(c) H_{l+1} (\rho,c)&\leqslant \abs{P(\rho,c)-\widetilde P}^{l}\big(P(\rho,c)-\widetilde P\big)\big(\rho-\widetilde \rho(c)\big)\notag\\
        &\leqslant C_*\abs{P(\rho,c)-\widetilde P}^{l-1}\big(P(\rho,c)-\widetilde P\big)\big(\rho-\widetilde \rho(c)\big)\notag\\
        &\leqslant C_*\frac{\rho+\widetilde\rho(c)}{a_l(c)} H_l(\rho,c).\label{sec3:eq17}
    \end{align}
    In particular, we have (see \eqref{ep3.8})
    \[
    \int_{\R^3}H_2(\rho_0,c_0)+\int_0^t\kappa_1'(t')\int_{\R^3}H_2(\rho,c)\leqslant C_* C_0.
    \]
    \item Owing to the expressions \eqref{sec2:eq1}-\eqref{sec2:eq2} and Gagliardo-Nirenberg's inequality, we have:
    \begin{align}
        &\frac{C}{\mu_*}\int_0^t\kappa_1(t')\Big(\norm{(-\Delta)^{-1}\nabla \mathbb P(\rho \dot u)}_{L^{3}(\R^3)}^3+\norm{(-\Delta)^{-1}\dvg (\rho \dot u)}_{L^{3}(\R^3)}^3\Big)(t')dt'\notag\\
        &\leqslant \frac{C}{\mu_*}\int_0^t \kappa_1(t')\norm{\rho \dot u}_{L^2(\R^3)}^{\frac{3}{2}} \norm{(-\Delta)^{-1}\nabla \mathbb P(\rho \dot u),\,(-\Delta)^{-1}\dvg (\rho \dot u)}_{L^{2}(\R^3)}^{\frac{3}{2}}(t')dt'\notag\\
        &\leqslant C_* \Bigg(\int_0^t \kappa_1(t')\norm{\rho \dot u}_{L^2(\R^3)}^2\Bigg)^{3/4}\Bigg(\int_0^t \kappa_1(t')\bigg(\norm{\nabla u}_{L^2(\R^3)}^6+\norm{G(\rho,c)}_{L^2(\R^3)}^6\bigg)(t')dt'\Bigg)^{1/4}\notag\\
        &\leqslant C_* E^{\frac{1}{4}}_0\big(\mathcal{A}_1(t)^{\frac{5}{4}}+C_0^{\frac{5}{4}}\big).\notag
    \end{align}
    \item Applying \eqref{ap:eq6} with $p=3$ and $\alpha'=\frac{\alpha}{2}$, we obtain
          \begin{align}
              \norm{\big[\mathcal{K}, \mu(\rho,c)-\widetilde\mu\big]\D u}_{L^{3}(\R^3)}&\leqslant C\Big(\abs{\mu(\rho,c)}_{\dot \cC^{\alpha/2}_{\pw,\Sigma}(\R^3)}+ \norm{\llbracket \mu(\rho,c)\rrbracket}_{L^\infty(\Sigma)}\ell^{-\alpha/2}_{\Sigma}\Big)\norm{\nabla u}_{L^{6/(2+\alpha)}(\R^3)}\notag\\
              &+C\norm{\llbracket \mu(\rho,c)\rrbracket}_{L^\infty(\Sigma)}\norm{\nabla u}_{L^3(\R^3)}\notag\\
              &\leqslant C\Big( \check \phi_\Sigma\norm{\nabla u}_{L^{6/(2+\alpha)}(\R^3)}+\norm{\llbracket \mu(\rho,c)\rrbracket}_{L^\infty(\Sigma)}\norm{\nabla u}_{L^3(\R^3)}\Big),\label{sec3:eq35}
          \end{align}
    and hence 
    \begin{align}
        \int_0^t\kappa_1(t') \norm{\big[\mathcal{K}, \mu(\rho,c)-\widetilde\mu\big]\D u}_{L^{3}(\R^3)}^3dt'
        &\leqslant  C\sup_{[0,t]}\norm{\llbracket \mu(\rho,c)\rrbracket}_{L^\infty(\Sigma)}^3\int_0^t\kappa_1(t') \norm{\nabla u}_{L^{3}(\R^3)}^3dt'\notag\\
        &+C\int_0^t\kappa_1(t')\check\phi_\Sigma^3\norm{\nabla u}_{L^{6/(2+\alpha)}(\R^3)}^3dt'.\notag
    \end{align}
    The same estimate holds for 
    \[
    \int_0^t\kappa_1(t') \norm{\big[\check{\mathcal{K}}, \mu(\rho,c)-\widetilde\mu\big]\D u}_{L^{3}(\R^3)}^3dt'.
    \]
\end{itemize}
 Combining all these estimates, we infer 
\begin{align*}
   \kappa_1(t) \int_{\R^3}H_2(\rho,c)(t) &+ \dfrac{1}{3\mu^*}\int_0^t\kappa_1(t')\Big(\norm{G(\rho,c)}^{3}_{L^3(\R^3)}+\mu_*^3\norm{\nabla u}_{L^3(\R^3)}^3\Big)(t')dt'\\
   &\leqslant C_*  C_0+\frac{C}{\mu_*}\sup_{[0,t]}\norm{\llbracket \mu(\rho,c)\rrbracket}_{L^\infty(\Sigma)}^3\int_0^t\kappa_1(t') \norm{\nabla u}_{L^{3}(\R^3)}^3dt'\\
   &+C_* E^{\frac{1}{4}}_0\big(\mathcal{A}_1(t)^{\frac{5}{4}}+C_0^{\frac{5}{4}}\big)+C_*\int_0^t\kappa_1(t')\check\phi_\Sigma^3\norm{\nabla u}_{L^{6/(2+\alpha)}(\R^3)}^3dt'.
\end{align*}
Therefore, if 
\begin{gather}\label{sec3:eq27}
C\frac{\mu^*}{\mu_*}\Bigg(\frac{1}{\mu_*}\sup_{[0,t]}\norm{\llbracket \mu(\rho,c)\rrbracket}_{L^\infty(\Sigma)}\Bigg)^3\leqslant \frac{1}{6} ,
\end{gather}
the second term on the RHS above can be absorbed into the LHS, and applying interpolation and Young's inequalities to the last term we obtain:
\begin{align}\label{sec3:eq42}
   \kappa_1(t) \int_{\R^3}H_2(\rho,c)(t) &+ \int_0^t\kappa_1(t')\Big(\norm{G(\rho,c)}^{3}_{L^3(\R^3)}+\norm{\nabla u}_{L^3(\R^3)}^3\Big)(t')dt'\notag\\
    &\leqslant C_* \Big( C_0+E^{\frac{1}{4}}_0\big(\mathcal{A}_1(t)^{\frac{5}{4}}+C_0^{\frac{5}{4}}\big)\Big)+C_*\int_0^t\kappa_1(t')\check\phi_\Sigma^{\frac{3}{\alpha}}\norm{\nabla u}_{L^{2}(\R^3)}^3dt'.
\end{align}
Returning to \eqref{sec3:eq14}, we deduce:
\begin{align}
    \mathcal{A}_1(t)+ \kappa_1(t) \int_{\R^3}H_2(\rho,c)(t) &+ \int_0^t\kappa_1(t')\Big(\norm{G(\rho,c)}^{3}_{L^3(\R^3)}+\norm{\nabla u}_{L^3(\R^3)}^3\Big)(t')dt'\notag\\
    &\leqslant C_*\Big(C_0+ \norm{\nabla u_0}_{L^2(\R^3)}^2+E^{\frac{1}{4}}_0\big(\mathcal{A}_1(t)^{\frac{5}{4}}+C_0^{\frac{5}{4}}\big)\Big)+C_*\int_0^t\kappa_1(t')\check \phi_\Sigma^{\frac{3}{\alpha}}\norm{\nabla u}_{L^{2}(\R^3)}^3dt'\notag.
\end{align}
We estimate the last term above as 
\[
\int_0^t\kappa_1(t')\check \phi_\Sigma^{\frac{3}{\alpha}}\norm{\nabla u}_{L^{2}(\R^3)}^3dt'\leqslant \check \phi_\Sigma^{\frac{3}{\alpha}} \sup_{[0,t]} \big(\kappa_1\norm{\nabla u}_{L^2(\R^3)}^2\big)^{1/2}\Bigg(\int_0^t\kappa_1(t')\norm{\nabla u}_{L^2(\R^3)}^2\Bigg)^{1/2}\Bigg(\int_0^t\norm{\nabla u}_{L^2(\R^3)}^2\Bigg)^{1/2}
\]
and  we apply  Young's inequality to obtain:
\begin{align}
\mathcal{A}_1(t)+ \sup_{[0,t]}\kappa_1 \int_{\R^3}H_2(\rho,c) &+ \int_0^t\kappa_1(t')\Big(\norm{G(\rho,c)}^{3}_{L^3(\R^3)}+\norm{\nabla u}_{L^3(\R^3)}^3\Big)(t')dt'\notag\\
    &\leqslant C_*\Big(C_0\big(1+\check \phi_\Sigma^{\frac{6}{\alpha}} E_0\big)+ \norm{\nabla u_0}_{L^2(\R^3)}^2+E^{\frac{1}{4}}_0\big(\mathcal{A}_1(t)^{\frac{5}{4}}+C_0^{\frac{5}{4}}\big)\Big).\notag
\end{align}
This completes the proof of \eqref{sec2:eq53}.
\end{proof}
\subsubsection{Bound for the functional $\mathcal{A}_2$}\label{sec3.1:funcE2} This section is devoted to deriving a bound for the functional $\mathcal{A}_2$ defined in \eqref{sec2:eq48} above.
\begin{proof}
The functional $\mathcal{A}_2$ arises when applying the material derivative $\dpt \cdot + \dvg (\cdot\, u)$ to the momentum equation $\eqref{ep2.1}_3$, followed by a testing against $\dot u$. This yields
    \begin{align}
    \dfrac{1}{2}\dfrac{d}{dt}\int_{\R^3} \rho\abs{\dot u}^2&+\int_{\R^3} \Big(2\mu(\rho,c)\abs{\D \dot u}^2+\lambda(\rho,c)(\dvg \dot u)^2\Big)= \int_{\R^3} (\rho \partial_\rho \lambda(\rho,c)-\lambda(\rho,c))(\dvg u)^2 \dvg \dot u \nonumber\\
    &+ \int_{\R^3} \lambda(\rho,c) \nabla u^k \cdot \partial_k u \dvg \dot u+ \int_{\R^3} \lambda(\rho,c)\dvg u \partial_j u^k \partial_k \dot u^j+2\int_{\R^3}\mu(\rho,c)\partial_k u\cdot \nabla \dot u^j \D^{jk} u\nonumber\\
    &+\int_{\R^3}\mu(\rho,c) \big(\partial_j u\cdot \nabla u^k + \partial_k u\cdot \nabla u^j\big)\partial_k \dot u^j+2\int_{\R^3}\left(\rho \partial_\rho \mu(\rho,c)-\mu(\rho,c)\right)\dvg u\D^{jk} u \partial_k \dot u^j\notag\\
    &+\int_{\R^3}G(\rho,c)\big(\dvg u\dvg \dot u-\partial_j u^k \partial_k \dot u^j\big)-\int_{\R^3} \rho \partial_\rho P(\rho,c)\dvg u\dvg \dot u.\notag
\end{align}
The computation is similar to that of the one-phase model (see e.g. \cite[Appendix A]{zodji2023discontinuous}), and we do not repeat it in this paper. The last term is
\begin{align*}
    -\int_{\R^3} \rho \partial_\rho P(\rho,c)\dvg u\dvg \dot u&=-\frac{1}{2}\dfrac{d}{dt}\int_{\R^3}\rho \partial_{\rho} P(\rho,c)(\dvg u)^2-\frac{1}{2}\int_{\R^3}\rho^2 \partial_{\rho}^2 P(\rho,c) (\dvg u)^3\\
    &-\int_{\R^3}\rho \partial_{\rho} P(\rho,c)\dvg u \nabla u^j \cdot \partial_j u.
\end{align*}
By combining the above expressions and multiplying by $\kappa_2(t')=\sigma(t')\scalar{t'}^{\frac{3}{2}(1+r_0)}$ we obtain (since $\kappa_2(0)=0$):
\begin{align}
    &\dfrac{\kappa_2(t)}{2}\int_{\R^3} \rho\Big(\abs{\dot u}^2+\partial_{\rho} P(\rho,c)(\dvg u)^2\Big)(t,x)dx+\int_0^t\int_{\R^3} \kappa_2(t')\Big(2\mu(\rho,c)\abs{\D \dot u}^2+\lambda(\rho,c)(\dvg \dot u)^2\Big)dt' \notag\\
    &\leqslant\dfrac{1}{2}\int_0^t\kappa_2'\int_{\R^3} \rho\Big(\abs{\dot u}^2+\partial_{\rho} P(\rho,c)(\dvg u)^2\Big)+C_*\int_0^t \kappa_2\big(\norm{\nabla u,\,G(\rho,c)}_{L^4(\R^3)}^4+\norm{\nabla u}_{L^3(\R^3)}^3\big).\label{sec3:eq15}
\end{align}
In the following items, we estimate the terms on the RHS  above.
\begin{itemize}
    \item Since $\kappa_2'(t')\leqslant C\scalarb{t'}^{\frac{1+3r_0}{2}}\leqslant C\kappa_1(t')$,  it follows that the first term on the RHS of \eqref{sec3:eq15}  can be estimated as follows: 
    \begin{gather}\label{sec3:eq23}
    \dfrac{1}{2}\int_0^t\kappa_2'(t')\int_{\R^3} \rho\Big(\abs{\dot u}^2+\partial_{\rho} P(\rho,c)(\dvg u)^2\Big)(t',x)dxdt' \leqslant C_*\big(C_0+\mathcal{A}_1(t)\big).
    \end{gather}
    \item Multiplying \eqref{sec3:eq12} by a time weight $\check\kappa$, and integrating in time yields
    \begin{align*}
        \check\kappa(t)\int_{\R^3} H_l(\rho,c)(t,x)dx&+\dfrac{1}{(l+1)\mu^*}\int_0^t\check\kappa\norm{G(\rho,c)}_{L^{l+1}(\R^3)}^{l+1}\leqslant \check\kappa(0)\int_{\R^3} H_l(\rho_0,c_0)(x)dx+\int_0^t\check\kappa'\int_{\R^3} H_l(\rho,c)\\
        &+\dfrac{1}{(l+1)\mu_*}\int_0^t\check\kappa\Big(\norm{(-\Delta)^{-1}\dvg (\rho \dot u)}_{L^{l+1}(\R^3)}^{l+1}+\norm{\big[\check{\mathcal{K}}, \mu(\rho,c)-\widetilde\mu\big]\D u}_{L^{l+1}(\R^3)}^{l+1}\Big).
    \end{align*}
    Combining this with \eqref{sec3:eq20}, we obtain 
\begin{align}
   \check\kappa(t) \int_{\R^3}&H_l(\rho,c)(t)+ \dfrac{1}{(l+1)\mu^*}\int_0^t\check\kappa(t')\Big(\norm{G(\rho,c)}^{l+1}_{L^{l+1}(\R^3)}+\mu_*^{l+1}\norm{\nabla u}_{L^{l+1}(\R^3)}^{l+1}\Big)(t')dt'\notag\\
   &\leqslant \check\kappa(0)\int_{\R^3} H_l(\rho_0,c_0)(x)dx+C\int_0^t\check\kappa'\int_{\R^3} H_l(\rho,c)\notag\\
   &+\frac{C_l}{(l+1)\mu_*}\int_0^t\check\kappa(t')\Big(\norm{(-\Delta)^{-1}\nabla \mathbb P(\rho \dot u)}_{L^{l+1}(\R^3)}^{l+1}+\norm{(-\Delta)^{-1}\dvg (\rho \dot u)}_{L^{l+1}(\R^3)}^{l+1}\Big)(t')dt'\notag\\
    &+\frac{C_l}{(l+1)\mu_*}\int_0^t\check\kappa(t')\Big(\norm{\big[\mathcal{K}, \mu(\rho,c)-\widetilde\mu\big]\D u}_{L^{l+1}(\R^3)}^{l+1}+\norm{\big[\check{\mathcal{K}}, \mu(\rho,c)-\widetilde\mu\big]\D u}_{L^{l+1}(\R^3)}^{l+1}\Big)(t')dt'.\notag
\end{align}
Applying the commutator estimate \eqref{ap:eq6} with $\alpha'=\alpha/2$, and $p=l+1$,  we obtain the following bound for the last term above:
\begin{align*}
    &\int_0^t\check\kappa(t')\Big(\norm{\big[\mathcal{K}, \mu(\rho,c)-\widetilde\mu\big]\D u}_{L^{l+1}(\R^3)}^{l+1}+\norm{\big[\check{\mathcal{K}}, \mu(\rho,c)-\widetilde\mu\big]\D u}_{L^{l+1}(\R^3)}^{l+1}\Big)(t')dt'\\
    &\leqslant C_l\sup_{[0,t]}\norm{\llbracket \mu(\rho,c)\rrbracket}_{L^\infty(\Sigma)}^{l+1}\int_0^t\check \kappa(t') \norm{\nabla u}_{L^{l+1}(\R^3)}^{l+1}dt'+C_l\int_0^t\check\kappa(t')\check\phi_\Sigma^{l+1}\norm{\nabla u}_{L^{r}(\R^3)}^{l+1}dt',
\end{align*}
where $\frac{1}{r}=\frac{\alpha}{6}+\frac{1}{l+1}$. Hence, under the smallness of the shear viscosity jump 
\[
C_l\frac{\mu^*}{\mu_*}\Bigg(\frac{1}{\mu_*}\sup_{[0,t]}\norm{\llbracket \mu(\rho,c)\rrbracket}_{L^\infty(\Sigma)}\Bigg)^{l+1}\leqslant \frac{1}{2(l+1)}, 
\]
we have 
\begin{align}\label{sec3:eq16}
   \check\kappa(t) \int_{\R^3}H_l(\rho,c)(t) &+ \int_0^t\check\kappa(t')\Big(\norm{G(\rho,c)}^{l+1}_{L^{l+1}(\R^3)}+\norm{\nabla u}_{L^{l+1}(\R^3)}^{l+1}\Big)(t')dt'\notag\\
   &\leqslant C_{*,l}\Bigg(\check\kappa(0)\int_{\R^3} H_l(\rho_0,c_0)(x)dx+\int_0^t\check\kappa'\int_{\R^3} H_l(\rho,c)+\int_0^t\check\kappa(t')\check\phi_\Sigma^{l+1}\norm{\nabla u}_{L^{r}(\R^3)}^{l+1}dt'\Bigg)\notag\\
   &+C_{*,l}\int_0^t\check\kappa(t')\Big(\norm{(-\Delta)^{-1}\nabla \mathbb P(\rho \dot u)}_{L^{l+1}(\R^3)}^{l+1}+\norm{(-\Delta)^{-1}\dvg (\rho \dot u)}_{L^{l+1}(\R^3)}^{l+1}\Big)(t')dt'.
\end{align}
Next, we interpolate
\[
\int_0^t\check\kappa(t')\check\phi_\Sigma^{l+1}\norm{\nabla u}_{L^{r}(\R^3)}^{l+1}dt'\leqslant \Bigg(\int_0^t\check\kappa(t')\norm{\nabla u}_{L^{l+1}(\R^3)}^{l+1}dt'\Bigg)^{1-\frac{\alpha(l+1)}{3(l-1)}}\Bigg(\int_0^t\kappa_3(t')\check\phi_\Sigma^{\frac{3}{\alpha}(l-1)}\norm{\nabla u}_{L^{2}(\R^3)}^{l+1}dt'\Bigg)^{\frac{\alpha(l+1)}{3(l-1)}};
\]
and we express
\[
\int_0^t\check\kappa'\int_{\R^3} H_l(\rho,c)= \int_0^{\sigma(t)}\check\kappa'\int_{\R^3} H_l(\rho,c)+\int_{\sigma(t)}^t\check\kappa'\int_{\R^3} H_l(\rho,c).
\]
By iterating \eqref{sec3:eq17}, it is straightforward to obtain $H_l(\rho,c)\leqslant C_{*,l} H_1(\rho,c)$ and therefore
\[
\int_0^{\sigma(t)}\check\kappa'\int_{\R^3} H_l(\rho,c)\leqslant C_{*,l} C_0 \int_0^{\sigma(t)} \check\kappa',
\]
while for the second term, we obtain
\begin{align*}
    \int_{\sigma(t)}^t\check\kappa'\int_{\R^3} H_l(\rho,c)\leqslant \Bigg(\int_{\sigma(t)}^t\check\kappa\int_{\R^3} H_l(\rho,c)\Bigg)^{\eta}\Bigg(C_{*,l}\int_{\sigma(t)}^t\big(\check\kappa'\check\kappa^{-\eta}\big)^{\frac{1}{1-\eta}}\int_{\R^3} H_1(\rho,c)\Bigg)^{1-\eta}, \;\text{ for }\;\eta\in (0,1).
\end{align*}
We now refer back to \eqref{sec3:eq16}, from which Young's inequality provides
\begin{align}\label{sec3:eq18}
   &\check\kappa(t) \int_{\R^3}H_l(\rho,c)(t) + \int_0^t\check\kappa(t')\Big(\norm{G(\rho,c)}^{l+1}_{L^{l+1}(\R^3)}+\norm{\nabla u}_{L^{l+1}(\R^3)}^{l+1}\Big)(t')dt'\notag\\
   &\leqslant C_{*,l}\Bigg[C_0 \Bigg(\check\kappa(0)+\int_0^{\sigma(t)} \check\kappa'\Bigg)+\int_0^t\kappa_3(t')\check\phi_\Sigma^{\frac{3}{\alpha}(l-1)}\norm{\nabla u}_{L^{2}(\R^3)}^{l+1}dt'+\int_{\sigma(t)}^t\big(\check\kappa'\check\kappa^{-\eta}\big)^{\frac{1}{1-\eta}}\int_{\R^3} H_1(\rho,c)\Bigg]\notag\\
   &+C_{*,l}\int_0^t\check\kappa(t')\Big(\norm{(-\Delta)^{-1}\nabla \mathbb P(\rho \dot u)}_{L^{l+1}(\R^3)}^{l+1}+\norm{(-\Delta)^{-1}\dvg (\rho \dot u)}_{L^{l+1}(\R^3)}^{l+1}\Big)(t')dt'.
\end{align}
In particular for $l=3$, and $\check\kappa=\kappa_3=\sigma(t)^{\frac{1}{2}}\scalar{t}^{2(1+r_0)}$, we infer
\[
\int_0^t\kappa_3(t')\check\phi_\Sigma^{\frac{6}{\alpha}}\norm{\nabla u}_{L^{2}(\R^3)}^{4}dt'\leqslant C_*\check\phi_\Sigma^{\frac{6}{\alpha}} C_0\mathcal{A}_1(t),
\]
and for $\eta=\frac{1+r_0}{2+r_0}$,
\[
\int_{\sigma(t)}^t\big(\kappa_3'\kappa_3^{-\eta}\big)^{\frac{1}{1-\eta}}\int_{\R^3} H_1(\rho,c)\leqslant C_*\int_0^t\scalar{t'}^{r_0} \int_{\R^3}H_1(\rho,c)dt'\leqslant  C_*C_0.
\]
Next, Gagliardo-Nirenberg's inequality yields 
\begin{align*}
    &\int_0^t\kappa_3(t')\Big(\norm{(-\Delta)^{-1}\nabla \mathbb P(\rho \dot u)}_{L^{4}(\R^3)}^4+\norm{(-\Delta)^{-1}\dvg (\rho \dot u)}_{L^{4}(\R^3)}^4\Big)(t')dt'\\
    &\leqslant C \sup_{[0,t]}\bigg[\sigma^\frac{1}{2}\kappa_1\norm{\rho \dot u}_{L^2(\R^3)}\norm{\nabla u,\, G(\rho,c)}_{L^2(\R^3)}\bigg]\int_0^t\kappa_1 \norm{\rho \dot u}_{L^2(\R^3)}^2\\
    &\leqslant C_* \mathcal{A}_2(t)^{1/2}\big(C_0^{1/2}+\mathcal{A}_1(t)^{1/2}\big)\mathcal{A}_1(t).
\end{align*}
From these, we  deduce the following bound for the next-to-last term in \eqref{sec3:eq15}:
\begin{align}
   \kappa_3(t) \int_{\R^3}H_3(\rho,c)(t) &+ \int_0^t\kappa_3(t')\Big(\norm{G(\rho,c)}^{4}_{L^4(\R^3)}+\norm{\nabla u}_{L^4(\R^3)}^4\Big)(t')dt'\notag\\
   &\leqslant C_*\bigg( C_0\big(1+\check\phi_\Sigma^{\frac{6}{\alpha}} \mathcal{A}_1(t)\big)+\mathcal{A}_2(t)^{1/2}\big(C_0^{1/2}+\mathcal{A}_1(t)^{1/2}\big)\mathcal{A}_1(t)\bigg).\label{sec3:eq22}
\end{align}
To estimate the last term, we start with 
\begin{align}
\int_0^t\sigma^{\frac{1}{2}} \scalar{}^{\frac{3}{2}(1+r_0)}\norm{\nabla u}_{L^3(\R^3)}^3&\leqslant  C\Bigg(\int_0^t \scalar{t'}^{(1+r_0)}\norm{\nabla u}_{L^2(\R^3)}^2dt'\Bigg)^{\frac{3}{4}}\Bigg(\int_0^t \sigma(t')^{2}\scalar{t'}^{3(1+r_0)}\norm{\nabla u}_{L^6(\R^3)}^6dt'\Bigg)^{\frac{1}{4}}\notag\\
&\leqslant C_* C_0^{3/4}\Bigg(\int_0^t \sigma(t')^{2}\scalar{t'}^{3(1+r_0)}\norm{\nabla u}_{L^6(\R^3)}^6dt'\Bigg)^{\frac{1}{4}}.\label{sec3:eq21}
\end{align}
In \eqref{sec3:eq18}, we take $l=5$, 
$\check\kappa=\kappa_4=\sigma^2\scalar{}^{3(1+r_0)}$ to obtain 
\begin{align}
   \kappa_4(t) \int_{\R^3}&H_5(\rho,c)(t) + \int_0^t\kappa_4(t')\Big(\norm{G(\rho,c)}^{6}_{L^{6}(\R^3)}+\norm{\nabla u}_{L^{6}(\R^3)}^{6}\Big)(t')dt'\notag\\
   &\leqslant C_{*}C_0 \big(1+\check\phi_\Sigma^{\frac{12}{\alpha}}\mathcal{A}_1(t)^2\big)+C_{*}\int_0^t\kappa_4(t')\norm{(-\Delta)^{-1}\nabla \mathbb P(\rho \dot u,\,(-\Delta)^{-1}\dvg (\rho \dot u)}_{L^{6}(\R^3)}^{6}dt'\notag\\
   &\leqslant C_{*}C_0 \big(1+\check\phi_\Sigma^{\frac{12}{\alpha}}\mathcal{A}_1(t)^2\big)+C_{*}\int_0^t\kappa_4(t')\norm{\rho \dot u}_{L^{2}(\R^3)}^{6}dt'\notag\\
   &\leqslant C_{*}C_0 \big(1+\check\phi_\Sigma^{\frac{12}{\alpha}}\mathcal{A}_1(t)^2\big)+C_{*}\mathcal{A}_1(t)\mathcal{A}_2(t)^2.\label{sec3:eq19}
\end{align}
\end{itemize}
Finally, we  combine \eqref{sec3:eq19}-\eqref{sec3:eq21}-\eqref{sec3:eq22}-\eqref{sec3:eq23}-\eqref{sec3:eq15} and we obtain:
\begin{align*}
\mathcal{A}_2(t)+\sup_{[0,t]}\kappa_3 \int_{\R^3}H_3(\rho,c) &+ \int_0^t\kappa_3(t')\Big(\norm{G(\rho,c)}^{4}_{L^4(\R^3)}+\norm{\nabla u}_{L^4(\R^3)}^4\Big)(t')dt'\\
&\leqslant C_*C_0\Big(1+\mathcal{A}_1(t)^2\Big)+ C_*\mathcal{A}_1(t)\Big(1+\mathcal{A}_1(t)^2+\check\phi_\Sigma^{\frac{6}{\alpha}} C_0\Big).
\end{align*}
This completes the proof of \eqref{sec2:eq54}.
\end{proof}
\subsubsection{Bound for the functional $\mathcal{A}_3$}
This section aims at obtaining bound for the energy functional $\mathcal{A}_3$. 
\begin{proof}
   To obtain the functional $\mathcal{A}_3$, we first apply the material derivative $\dpt \,\cdot +\dvg (\,\cdot\, u)$ to the momentum equations \eqref{ep2.1} and obtain
\begin{gather}\label{sec3:eq36}
 \dpt (\rho \dot u^j)+\dvg (\rho \dot u^j u)=\partial_k \{\dot \Pi^{jk}\}+\partial_k\{\Pi^{jk}\dvg u\}-\dvg \{\partial_k u\Pi^{jk}\},
 \end{gather}
 where 
 \begin{align}
 \dot \Pi^{jk}&=2\mu(\rho,c)\D \dot u-\mu(\rho,c)\partial_j u^l \partial_l u^k -\mu(\rho,c)\partial_k u^l\partial_l u^j-2\rho \partial_\rho\mu(\rho,c)\D ^{jk} u\dvg u\nonumber\\
 &+\left(\lambda(\rho,c)\dvg \dot u-\lambda(\rho,c)\nabla u^l\cdot \partial_l u -\rho \partial_\rho \lambda(\rho,c)(\dvg u)^2+\rho \partial_\rho P(\rho,c)\dvg u  \right)\delta^{jk}.\notag
 \end{align}
We use $\kappa_5\ddot u$, with $\kappa_5=\sigma^{1+\check{\alpha}}\scalar{}^{\frac{3}{2}(1+r_0)}$ as a test function to obtain
\begin{align}
    \int_0^t\kappa_5\norm{\sqrt{\rho}\ddot u}_{L^2(\R^3)}^2&+\frac{1}{2}\kappa_5(t)\int_{\R^3}\big(2\mu(\rho,c)\abs{\D \dot u}^2+\lambda(\rho,c)(\dvg \dot u)^2\big)(t,x)dx\notag\\
    &=\frac{1}{2}\int_0^t\kappa_5'\int_{\R^3}\big(2\mu(\rho,c)\abs{\D \dot u}^2+\lambda(\rho,c)(\dvg \dot u)^2\big)-\kappa_5(t)\int_{\R^3} \rho \partial_\rho P(\rho,c)\dvg \dot u\dvg u \notag\\
    &+\int_0^{t}\kappa_5'\int_{\R^3} \rho  \partial_\rho P(\rho,c)\dvg \dot u\dvg u+\int_0^t\kappa_5\int_{\R^3} \rho  \partial_\rho P(\rho,c)(\dvg \dot u)^2\notag\\
        &+\kappa_5(t)J_{1}(t)-\int_0^{t} \kappa_5' J_{1}(s)ds+ \int_0^t\kappa_5J_{2}(s)ds+ \int_0^t\kappa_5J_{3}(s)ds.\label{sec3:eq25}
\end{align}
Above, $J_{1},\; J_{2},\; J_{3}$ are of the form
\begin{gather}\label{sec3:eq24}
\begin{cases}
J_{1}&= \displaystyle  \int_{\R^3} \phi(\rho,c)\partial_{j_1} \dot u^{j_2}\partial_{j_3} u^{j_4} \partial_{j_5} u^{j_6}
+\int_{\R^3} \partial_{j_1} \dot u^{j_2}\partial_{j_3} u^{j_4} G(\rho,c),\\
J_{2}&= \displaystyle \int_{\R^3} \phi(\rho,c)\partial_{j_1} \dot u^{j_2}\partial_{j_3} u^{j_4} \partial_{j_5} u^{j_6}\partial_{j_7} u^{j_8}
+\int_{\R^3} \partial_{j_1} \dot u^{j_2}\partial_{j_3} u^{j_4} \partial_{j_5} u^{j_6} G(\rho,c)+\int_{\R^3}  \psi(\rho,c) \partial_{j_1} \dot u ^{j_2}  \partial_{j_3}  u ^{j_4} \partial_{j_5} u ^{j_6},\\
J_{3}&= \displaystyle  \int_{\R^3} \phi(\rho,c)\partial_{j_1} \dot u^{j_2}\partial_{j_3}\dot u^{j_4} \partial_{j_5} u^{j_6}
+\int_{\R^3} \partial_{j_1} \dot u^{j_2}\partial_{j_3}\dot u^{j_4} G(\rho,c),
\end{cases}
\end{gather}
 where $\phi$ is either the viscosities $\mu$, $\lambda$ or $\rho  \partial_\rho \mu$ , $\rho \partial_\rho  \lambda$, $\rho^2 \partial_\rho^2 \mu$ or $\rho^2 \partial_\rho^2 \lambda$, while $\psi$ corresponds to either  $\rho \partial_\rho  P$ or $\rho^2\partial_\rho^2 P$.  The computations are analogous to those in \cite[Section B.2]{zodji2023discontinuous} and we do not repeat it here. In the following steps, we provide estimates for the RHS terms of \eqref{sec3:eq25}.
 \paragraph{\textbf{Step 1}} In this paragraph, we will estimate the first term of the RHS of \eqref{sec3:eq25}. For $\check{\alpha}=1$, it is clear that the underlying term is controlled as follows
 \[
 \frac{1}{2}\int_0^t\kappa_5'\int_{\R^3}\big(2\mu(\rho,c)\abs{\D \dot u}^2+\lambda(\rho,c)(\dvg \dot u)^2\big)\leqslant C \mathcal{A}_2(t).
 \]
Even if  $\check\alpha\in [0,1)$, we still have 
\[
\frac{1}{2}\int_{\sigma(t)}^t\kappa_5'\int_{\R^3}\big(2\mu(\rho,c)\abs{\D \dot u}^2+\lambda(\rho,c)(\dvg \dot u)^2\big)\leqslant C \mathcal{A}_2(t), 
\]
and to address the case of small time, we shall make use of a duality method. Hence, we write $\dot u=v+\check{v}$ where $v$ and $\check{v}$ satisfy:
\begin{gather}\label{sec3:eq29}
 \begin{cases}
  \dpt (\rho v)+\dvg (\rho v\otimes u)=\dvg \big(2\mu(\rho,c) \D v\big)+\nabla \big(\lambda(\rho,c)\dvg v\big),\\
  (\rho v)_{|t=0}= (\rho \dot u)_{|t=0},
  \end{cases}
 \end{gather}
 and 
 \begin{gather}\label{sec3:eq33}
 \begin{cases}
     \dpt (\rho \check{v})+\dvg(\rho \check{v}\otimes u)-\dvg \big(2\mu(\rho,c)\D \check{v}\big)-\nabla\big(\lambda(\rho,c)\dvg \check{v}\big)=\dvg (\mathscr{R}),\\
     (\rho \check{v})_{|t=0}=0,
 \end{cases}
 \end{gather}
 with
 \begin{align*}
     (\dvg \mathscr R)^j&= -\partial_k\big(\mu(\rho,c)\partial_j u^l \partial_l u^k +\mu(\rho,c)\partial_k u^l\partial_l u^j+2\rho \partial_\rho\mu(\rho,c)\D ^{jk} u\dvg u\big)\\
     &-\partial_j\big(\lambda(\rho,c)\nabla u^l\cdot \partial_l u +\rho \partial_\rho \lambda(\rho,c)(\dvg u)^2-\rho \partial_\rho P(\rho,c)\dvg u \big)+\partial_k\big(\Pi^{jk}\dvg u\big)-\dvg \big(\partial_k u\Pi^{jk}\big).
 \end{align*}
Recalling that $(\rho \dot u)_{|t=0}\in L^2(\R^3)$ (see \eqref{sec2:compa}), one can establish the existence and uniqueness of $v$, and thus of $\check{v}$, through a regularizing argument. We now aim at deriving bounds for $v$ and $\check{v}$, which will then imply the desired bound for $\dot u$.

For $\tau\in (0,\sigma(t)]$, and $\check \varphi\in L^2((0,\tau)\times \R^3)$, we consider 
the linear nonhomogeneous problem 
 \begin{gather}\label{sec3:eq26}
 \begin{cases}
 \rho \dpt  w+ \rho (u\cdot\nabla w)+\dvg (2\mu(\rho,c)\D w)+\nabla (\lambda(\rho,c)\dvg w)=\rho\check \vph, \\
 w_{|t=\tau}=0.
 \end{cases}
 \end{gather}
 The existence of $w$ can be established by reversing the arrow of time, and we readily have (recall that $\tau\in [0,1]$)
 \begin{gather}
     \frac{1}{2}\sup_{[0,\tau]}\norm{\sqrt{\rho} w}_{L^2(\R^3)}^2+\int_0^\tau\int_{\R^3}\bigg[2\mu(\rho,c)\abs{\D w}^2+\lambda(\rho,c)(\dvg w)^2\Big]\leqslant C \int_0^\tau\norm{\sqrt{\rho} \check \vph}_{L^2(\R^3)}^2.\notag
 \end{gather}
 Next, using $\dot w$ as a test function, we obtain
  \begin{align}
 \int_0^{\tau'}\norm{\sqrt{\rho}\dot {w}}^2_{L^2(\R^3)}&+ \dfrac{1}{2}\int_{\R^3}\big[2\mu(\rho,c)\abs{\D w}^2+\lambda(\rho,c)\abs{\dvg w}^2\big](\tau',x)dx\notag\\
 &\leqslant \int_0^{\tau'}\norm{\sqrt{\rho}\dot w}_{L^2(\R^3)}\norm{\sqrt\rho \check \vph}_{L^2(\R^3)}+C\int_0^{\tau'}\norm{\nabla u}_{L^3(\R^3)}\norm{\nabla w}_{L^3(\R^3)}^2.\label{sec3:eq28}
 \end{align}
Repeating the computations that led to \eqref{sec2:eq3}, we find
\begin{align*}
    \nabla w&=-\frac{1}{\mu(\rho,c)}\nabla(-\Delta)^{-1}\mathbb P \big(\rho (\dot w-\check{\vph})\big)+\nabla (-\Delta)^{-1}\nabla\bigg(\frac{1}{\nu(\rho,c)}(-\Delta)^{-1}\dvg \big(\rho (\dot w-\check{\vph})\big)\bigg)\notag\\
            &+\frac{1}{\mu(\rho,c)}\big[\mathcal{K}, \mu(\rho,c)-\widetilde \mu\big](2\D w)-\nabla (-\Delta)^{-1}\nabla\bigg(\frac{1}{\nu(\rho,c)}\big[\check{\mathcal{K}},\mu(\rho,c)-\widetilde \mu\big](2\D w)\bigg);
\end{align*}
and by means of the commutator estimate \eqref{ap:eq6}, it follows, under the assumption of a small viscosity jump (see \eqref{sec3:eq27})
\begin{align*}
\norm{\nabla w}_{L^3(\R^3)}&\leqslant C_*\Big(\norm{(-\Delta)^{-1}\nabla \big(\rho(\dot w-\check{\vph})\big)}_{L^3(\R^3)}+\check \phi_\Sigma^{\frac{1}{\alpha}}\norm{\nabla w}_{L^{2}(\R^3)}\Big) \\
                           &\leqslant C_*\Big(\norm{\nabla w}_{L^2(\R^3)}^{1/2}\norm{\rho(\dot w-\check{\vph})}_{L^2(\R^3)}^{1/2}+\check \phi_\Sigma^{\frac{1}{\alpha}}\norm{\nabla w}_{L^{2}(\R^3)}\Big).
\end{align*}
Therefore, by Young's inequality, \eqref{sec3:eq28} leads to:
\begin{align*}
    \int_0^{\tau'}\norm{\sqrt{\rho}\dot {w}}^2_{L^2(\R^3)}&+\int_{\R^3}\big[2\mu(\rho,c)\abs{\D w}^2+\lambda(\rho,c)\abs{\dvg w}^2\big](\tau',x)dx\notag\\
 &\leqslant C_*\int_0^{\tau'}\norm{\sqrt\rho \check \vph}_{L^2(\R^3)}^2+C_*\int_0^{\tau'}\Big(\norm{\nabla u}_{L^3(\R^3)}^2+\check \phi_\Sigma^{\frac{2}{\alpha}}\norm{\nabla u}_{L^3(\R^3)}\Big)\norm{\nabla w}_{L^2(\R^3)}^2,
\end{align*}
from which Gr\"onwall's Lemma yields
 \begin{gather}\label{sec3:eq32}
 \int_0^{\tau}\norm{\sqrt{\rho}\dot {w}}_{L^2(\R^3)}^2+ \sup_{[0,\tau]}\norm{\nabla w}_{L^2(\R^3)}^2
 \leqslant \Bigg(C_*\int_0^{\tau}\norm{\sqrt{\rho}\check \vph}_{L^2(\R^3)}^2\Bigg)\exp{\Bigg[C_*\int_0^{\tau}\Big(\norm{\nabla u}_{L^3(\R^3)}^2+\check \phi_\Sigma^{\frac{2}{\alpha}}\norm{\nabla u}_{L^3(\R^3)}\Big)\Bigg]}.
 \end{gather}
 
We now use $w$ as a test function in \eqref{sec3:eq29} and we obtain
\begin{align*}
 \Bigg|\int_0^{\tau}\int_{\R^3} (\rho v)(\tau',x) \check\vph(\tau',x) d\tau' dx \Bigg|&=\Bigg|\int_{\R^3}(\rho v)_{|t=0}(x) w(0,x)dx\Bigg|\\
 &\leqslant \norm{(\rho v)_{|t=0}}_{\dot H^{-1}(\R^3)}\norm{w(0)}_{\dot H^{1}(\R^3)}\\
 &\leqslant C_* \norm{(\rho v)_{|t=0}}_{\dot H^{-1}(\R^3)}\norm{\check\vph}_{L^2((0,\tau)\times \R^3)}\exp{\Bigg[C_*\int_0^{\tau}\Big(\norm{\nabla u}_{L^3(\R^3)}^2+\check \phi_\Sigma^{\frac{2}{\alpha}}\norm{\nabla u}_{L^3(\R^3)}\Big)\Bigg]}.
 \end{align*}
By duality, it follows that 
 \begin{gather*}
 \int_0^{\tau}\norm{\rho v(\tau')}_{L^2(\R^3)}^2d\tau' \leqslant C_*\norm{(\rho v)_{|t=0}}_{\dot H^{-1}(\R^3)}^2\exp{\Bigg[C_*\int_0^{\tau}\Big(\norm{\nabla u}_{L^3(\R^3)}^2+\check \phi_\Sigma^{\frac{2}{\alpha}}\norm{\nabla u}_{L^3(\R^3)}\Big)\Bigg]},
 \end{gather*}
and  using $\sigma v$ as a testing function in \eqref{sec3:eq29} yields  
 \begin{align*}
 \sup_{[0,\tau]} \sigma \norm{\sqrt\rho v}_{L^2(\R^3)}^2 +\int_0^{\tau}\sigma \norm{\nabla v}_{L^2(\R^3)}^2 &\leqslant C_*\int_0^{\tau}\norm{\sqrt{\rho}v}_{L^2(\R^3)}^2\\
 &\leqslant C_* \norm{(\rho  v)_{|t=0}}_{\dot H^{-1}(\R^3)}^2\exp{\Bigg[C_*\int_0^{\tau}\Big(\norm{\nabla u}_{L^3(\R^3)}^2+\check \phi_\Sigma^{\frac{2}{\alpha}}\norm{\nabla u}_{L^3(\R^3)}\Big)\Bigg]}.
 \end{align*}
As well, by using  $v$ as a test function in \eqref{sec3:eq29} it is straightforward to obtain 
 \[
 \sup_{[0,\tau]} \norm{\sqrt{\rho}v}_{L^2(\R^3)}^2+\int_0^{\tau}\norm{\nabla v}_{L^2(\R^3)}^2 \leqslant C_* \norm{(\rho v)_{|t=0}}_{L^2(\R^3)}^2.
 \]
 Finally, we apply a Riesz-Thorin theorem  to conclude that 
 \begin{gather}\label{sec3:eq30}
 \sup_{[0,\tau]} \sigma^{\check{\alpha}}\norm{\sqrt{\rho}v}_{L^2(\R^3)}^2+\int_0^{\tau}\sigma^{\check{\alpha}}\norm{\nabla v}_{L^2(\R^3)}^2 \leqslant C_*\norm{(\rho \dot u)_{|t=0}}_{\dot H^{-\check{\alpha}}(\R^3)}^2\exp{\Bigg[C_*\int_0^{\tau}\Big(\norm{\nabla u}_{L^3(\R^3)}^2+\check \phi_\Sigma^{\frac{2}{\alpha}}\norm{\nabla u}_{L^3(\R^3)}\Big)\Bigg]}.
 \end{gather}
This ends the estimate for $v$ and we now turn to the one of $\check{v}$.

To achieve this, we rather consider the following Cauchy problem 
\begin{gather}\label{sec3:eq31}
 \begin{cases}
 \rho \dpt \check w+\rho (u\cdot\nabla)\check w+\dvg \big(2\mu(\rho,c)\D \check w\big)+\nabla\big(\lambda(\rho,c)\dvg \check w\big)=0,\\
 \check w_{|t=\tau}=\check\vph,
 \end{cases}
 \end{gather}
where  $0<\tau\leqslant \sigma(t)$ and $\check\vph=\check{\vph}(x)\in L^2(\R^3)$ is a given function. A basic energy estimate yields 
 \begin{gather*}
 \frac{1}{2}\sup_{[0,\tau]}\norm{\sqrt{\rho}\check w}_{L^2(\R^3)}^2+ \int_0^{\tau} \int_{\R^3}\left[2\mu(\rho,c)\abs{\D \check w}^2+\lambda(\rho,c)(\dvg \check w)^2\right]\leqslant \frac{1}{2}\norm{\sqrt{\rho}(\tau)\check\vph}_{L^2(\R^3)}^2.
 \end{gather*}
Let $\sigma_\tau$ denotes the time weight $\sigma_\tau=\tau-s$, $s\in [0,\tau]$. Using $\sigma_\tau \dot{\check{w}}$ as a test function in \eqref{sec3:eq31} and following the same computations that lead to \eqref{sec3:eq32}, we infer
  \begin{gather}
 \int_0^\tau  \sigma_\tau \norm{ \sqrt{\rho}\dot {\check w}}_{L^2(\R^3)}^2+\sup_{[0,\tau]}\sigma_\tau\norm{\nabla \check w}^2_{L^2(\R^3)}
 \leqslant C_*\norm{\sqrt{\rho}(\tau)\check\vph}_{L^2(\R^3)}^2\exp{\Bigg[C_*\int_0^{\tau}\Big(\norm{\nabla u}_{L^3(\R^3)}^2+\check \phi_\Sigma^{\frac{2}{\alpha}}\norm{\nabla u}_{L^3(\R^3)}\Big)\Bigg]}.\notag
 \end{gather}
 We now use $\check{w}$ as a test function in \eqref{sec3:eq33} and we obtain
 \begin{align*}
     \Bigg|\int_{\R^3} (\rho \check{v})(\tau,x) \check\vph(x)dx\Bigg| &=\Bigg|\int_0^\tau\int_{\R^3}\mathscr{R}^{jk}(\tau',x) \partial_k \check{w}^j(\tau',x)dxd\tau'\Bigg|\\
     &\leqslant C_*\int_0^\tau\bigg(\norm{\nabla u}_{L^4(\R^3)}^2+\norm{G(\rho,c)}_{L^4(\R^3)}^2+ \norm{\nabla u}_{L^2(\R^3)}\bigg)(\tau')\norm{\nabla \check w(\tau')}_{L^2(\R^3)}d\tau'\\
     &\leqslant C_*\norm{\sqrt{\rho}(\tau)\check\vph}_{L^2(\R^3)} \Bigg[E_0^{\frac{1}{2}}+\int_0^\tau \sigma_{\tau}^{-\tfrac{1}{2}}\Big(\norm{\nabla u}_{L^4(\R^3)}^2+\norm{ G(\rho,c)}_{L^4(\R^3)}^2\Big)\Bigg]\\
     &\times \exp{\Bigg[C_*\int_0^{\tau}\Big(\norm{\nabla u}_{L^3(\R^3)}^2+\check \phi_\Sigma^{\frac{2}{\alpha}}\norm{\nabla u}_{L^3(\R^3)}\Big)\Bigg]}.
 \end{align*}
  Hence, by duality we obtain 
\begin{gather*}
     \norm{\sqrt{\rho}\check{v}(\tau)}_{L^2(\R^3)}\leqslant C_*\Bigg[E_0^{\tfrac{1}{2}}+\int_0^\tau \sigma_{\tau}^{-\tfrac{1}{2}}\norm{\nabla u,\, G(\rho,c)}_{L^4(\R^3)}^2\Bigg]\exp{\Bigg[C_*\int_0^{\tau}\Big(\norm{\nabla u}_{L^3(\R^3)}^2+\check \phi_\Sigma^{\frac{2}{\alpha}}\norm{\nabla u}_{L^3(\R^3)}\Big)\Bigg]}.
\end{gather*}
Note that for all $r\in \big(0,4/(2+\check{\alpha})\big)$,  
we have 
\[
\norm{\norm{\nabla u,\, G(\rho,c)}_{L^4(\R^3)}^2}_{L^r(0,\sigma(t))}^{2}\leqslant C_r \int_0^{\sigma(t)}\sigma^{\frac{\check{\alpha}}{2}} \norm{\nabla u,\, G(\rho,c)}_{L^4(\R^3)}^4,
\]
and Hardy-Littlewood-Sobolev's inequality (see \cite[Theorem 1.7]{bahouri2011fourier}) yields 
\[
\norm{\tau \mapsto \int_0^\tau \sigma_{\tau}^{-\tfrac{1}{2}}\left(\norm{\nabla u}_{L^4(\R^3)}^2+\norm{ G(\rho,c)}_{L^4(\R^3)}^2\right)}_{L^{2r/(2-r)}(0,\sigma(t))}^2\leqslant C\int_0^t\sigma^{\frac{\check\alpha}{2}}\norm{\nabla u,\, G(\rho,c)}^4_{L^4(\R^3)}.
\]
By H\"older's inequality, it follows that for $r\in \big(2/(1+\check{\alpha}), 4/(2+\check{\alpha})\big)$
\begin{align*}
    \int_{0}^{\sigma(t)}\tau^{\check{\alpha}-1}&\norm{\sqrt{\rho}\check{v}(\tau)}_{L^2(\R^3)}^2d\tau \leqslant  C_r\norm{\sqrt{\rho} \check{v}}_{L^{2r/(2-r)}(\R^3)}^2\\
    &\leqslant C_{*,r}\Bigg(E_0+\int_0^{\sigma(t)} \sigma^{\frac{\check\alpha}{2}}\norm{\nabla u,\, G(\rho,c)}^4_{L^4(\R^3)}\Bigg)\exp{\Bigg[C_*\int_0^{\tau}\Big(\norm{\nabla u}_{L^3(\R^3)}^2+\check \phi_\Sigma^{\frac{2}{\alpha}}\norm{\nabla u}_{L^3(\R^3)}\Big)\Bigg]}.
\end{align*}
We  now use $\sigma^{\check{\alpha}} \check{v}$ as a test function in \eqref{sec3:eq33} 
and we obtain 
\begin{align}
    \sup_{[0,\sigma(t)]} \sigma^{\check{\alpha}}\norm{\sqrt{\rho}\check v}_{L^2(\R^3)}^2&+\int_0^{\sigma(t)}\sigma^{\check{\alpha}}\norm{\nabla\check v}_{L^2(\R^3)}^2 \leqslant C_*\Bigg(E_0+\int_{0}^{\sigma(t)}\sigma^{\check{\alpha}-1}\norm{\sqrt{\rho}\check{v}}_{L^2(\R^3)}^2 +\int_0^{\sigma(t)}\sigma^{\check{\alpha}} \norm{\nabla u,\, G(\rho,c)}_{L^4(\R^3)}^4\Bigg)\notag\\
    &\leqslant C_*\Bigg(E_0+\int_0^{\sigma(t)} \sigma^{\frac{\check\alpha}{2}}\norm{\nabla u,\, G(\rho,c)}^4_{L^4(\R^3)}\Bigg)\exp{\Bigg[C_*\int_0^{\tau}\Big(\norm{\nabla u}_{L^3(\R^3)}^2+\check \phi_\Sigma^{\frac{2}{\alpha}}\norm{\nabla u}_{L^3(\R^3)}\Big)\Bigg]}.\label{sec3:eq34}
\end{align}

Finally, from \eqref{sec3:eq30}-\eqref{sec3:eq34} and  $\dot u=v+\check{v}$,
we obtain that
\begin{align}
    &\sup_{[0,\sigma(t)]} \sigma^{\check{\alpha}}\norm{\sqrt{\rho}\dot u}_{L^2(\R^3)}^2+\int_0^{\sigma(t)}\sigma^{\check{\alpha}}\norm{\nabla\dot u
    }_{L^2(\R^3)}^2 \notag\\
    &\leqslant C_*\Bigg(E_0+\norm{(\rho \dot u)_{|t=0}}_{\dot H^{-\check{\alpha}}(\R^3)}^2+\int_0^{\sigma(t)} \sigma^{\frac{\check\alpha}{2}}\norm{\nabla u,\, G(\rho,c)}^4_{L^4(\R^3)}\Bigg)\exp{\Bigg[C_*\int_0^{\tau}\Big(\norm{\nabla u}_{L^3(\R^3)}^2+\check \phi_\Sigma^{\frac{2}{\alpha}}\norm{\nabla u}_{L^3(\R^3)}\Big)\Bigg]}.\notag
\end{align}
Following the computations that go from \eqref{sec3:eq18} to \eqref{sec3:eq22}, we obtain
\begin{align}\label{sec3:eq44}
        \int_0^{\sigma(t)}\tau^{\frac{\check{\alpha}}{2}}\norm{G(\rho,c),\,\nabla u}_{L^4(\R^3)}^4(\tau)d\tau
   &\leqslant C_*\bigg( C_0\big(1+\check\phi_\Sigma^{\frac{6}{\alpha}} \mathcal{A}_1(t)\big)+\mathcal{A}_1(t)\big( C_0^{1/2}+ \mathcal{A}_1(t)^{1/2}\big)\sup_{[0,\sigma(t)]} \sigma^{\frac{\check{\alpha}}{2}}\norm{\sqrt{\rho}\dot u}_{L^2(\R^3)}\bigg),
\end{align}
and  Young's inequality yields 
\begin{align*}
    &\sup_{[0,\sigma(t)]} \sigma^{\check{\alpha}}\norm{\sqrt{\rho}\dot u}_{L^2(\R^3)}^2+\int_0^{\sigma(t)}\sigma^{\check{\alpha}}\norm{\nabla\dot u
    }_{L^2(\R^3)}^2\\
    &\leqslant  C_*\Bigg(C_0\big(1+\check\phi_\Sigma^{\frac{6}{\alpha}} \mathcal{A}_1(t)\big)+\norm{(\rho \dot u)_{|t=0}}_{\dot H^{-\check{\alpha}}(\R^3)}^2+\mathcal{A}_1(t)^2\big( C_0+ \mathcal{A}_1(t)\big)\Bigg)\exp{\Bigg[C_*\int_0^{\tau}\Big(\norm{\nabla u}_{L^3(\R^3)}^2+\check \phi_\Sigma^{\frac{2}{\alpha}}\norm{\nabla u}_{L^3(\R^3)}\Big)\Bigg]},
\end{align*}
and we conclude this step  with the following bound for the first term on the RHS of \eqref{sec3:eq25}
\begin{align}
    \check{\mathcal A}_2(t)&:=\sup_{[0,t]}\sigma^{\check{\alpha}}\scalar{}^{\frac{3}{2}(1+r_0)}\norm{\sqrt{\rho}\dot u}_{L^2(\R^3)}^2+\int_0^t\sigma^{\check{\alpha}}\scalar{}^{\frac{3}{2}(1+r_0)}\int_{\R^3}\big(2\mu(\rho,c)\abs{\D \dot u}^2+\lambda(\rho,c)(\dvg \dot u)^2\big)\notag\\
    &\leqslant C_*\Bigg(\mathcal{A}_2(t)+C_0\big(1+\check\phi_\Sigma^{\frac{6}{\alpha}} \mathcal{A}_1(t)\big)+\norm{(\rho \dot u)_{|t=0}}_{\dot H^{-\check{\alpha}}(\R^3)}^2+\mathcal{A}_1(t)^2\big( C_0+ \mathcal{A}_1(t)\big)\Bigg)\notag\\
    &\times \exp{\Bigg[C_*\int_0^{\tau}\Big(\norm{\nabla u}_{L^3(\R^3)}^2+\check \phi_\Sigma^{\frac{2}{\alpha}}\norm{\nabla u}_{L^3(\R^3)}\Big)\Bigg]}.\label{sec3:eq41}
\end{align}
The others terms are estimated in the following steps. 
\paragraph{\textbf{Step 2}}
In what follows, we estimate all the remaining RHS terms in \eqref{sec3:eq25}, 
except for the last one,  which is addressed next.
\begin{itemize}
    \item We readily bound the second term on the RHS of \eqref{sec3:eq25} as follows:
    \begin{align*}
        \Bigg|\kappa_5(t)\int_{\R^3} \rho \partial_\rho P(\rho,c)\dvg \dot u\dvg u\Bigg|&\leqslant  \eta \kappa_5t)\norm{\dvg \dot u(t)}_{L^2(\R^3)}^2+\frac{C_*}{\eta}\norm{\dvg u(t)}_{L^2(\R^3)}^2\\
        &\leqslant \eta \kappa_5t)\norm{\dvg \dot u(t)}_{L^2(\R^3)}^2+\frac{C_*}{\eta}\mathcal{A}_2(t).
    \end{align*}
    The two  terms that follow  are
    \begin{align*}
        \Bigg|\int_0^{t}\kappa_5'\int_{\R^3} \rho  \partial_\rho P(\rho,c)\dvg \dot u\dvg u\Bigg|+\Bigg|\int_0^t\kappa_5\int_{\R^3} \rho  \partial_\rho P(\rho,c)(\dvg \dot u)^2\Bigg|&\leqslant C_*\Big(E_0+ \check{\mathcal{A}}_2(t)\Big).
    \end{align*}
    \item Next, using \eqref{sec3:eq22}-\eqref{sec3:eq44}, we estimate the sixth and final terms in the expression of $J_2$ (see \eqref{sec3:eq24}) as follows:
    \begin{align*}
        \Bigg|\int_0^{t} \kappa_5' J_{1}(s)ds\Bigg|&+\Bigg|\int_0^t\kappa_5\int_{\R^3}  \psi(\rho,c) \partial_{j_1} \dot u ^{j_2}  \partial_{j_3}  u ^{j_4} \partial_{j_5} u ^{j_6}\Bigg|\notag\\
        &\leqslant C_*\int_0^t \sigma^{\check\alpha}\scalar{}^{\frac{3}{2}(1+r_0)} \Big(\norm{\nabla \dot u}_{L^2(\R^3)}^2+\norm{\nabla u}_{L^4(\R^3)}^4\Big)\\
        &\leqslant C_*\bigg(\check{\mathcal{A}}_2(t)+ C_0\big(1+\check\phi_\Sigma^{\frac{6}{\alpha}} \mathcal{A}_1(t)\big)+\big(C_0+\mathcal{A}_1(t)\big)\mathcal{A}_1(t)^2\bigg).
    \end{align*}
    Following the computations leading to \eqref{sec3:eq19}, we estimate the remaining terms of $J_2$ as follows:
    \begin{align}
        \Bigg|\int_0^t\kappa_5\int_{\R^3} \phi(\rho,c)\partial_{j_1} \dot u^{j_2}\partial_{j_3} u^{j_4} \partial_{j_5} u^{j_6}\partial_{j_7} u^{j_8}\Bigg|
&+\Bigg|\int_0^t\kappa_5\int_{\R^3} \partial_{j_1} \dot u^{j_2}\partial_{j_3} u^{j_4} \partial_{j_5} u^{j_6} G(\rho,c)\Bigg|\notag\\
&\leqslant  C_*\int_0^t \kappa_5 \Big(\norm{\nabla \dot u}_{L^2(\R^3)}^2+\norm{\nabla u, G(\rho,c)}_{L^6(\R^3)}^6\Big)\notag\\
&\leqslant C_{*}C_0 \big(1+\check\phi_\Sigma^{\frac{12}{\alpha}}\mathcal{A}_1(t)^2\big)+C_{*}\mathcal{A}_2(t)\big(1+\mathcal{A}_1(t)\check{\mathcal{A}}_2(t)\big).\label{sec3:eq45}
    \end{align}
    \item For the first term in the last line of \eqref{sec3:eq25}, we have 
    we estimate 
    \[
    \abs{\kappa_5(t) J_1(t)}\leqslant \eta\kappa_5(t) \norm{\nabla\dot u(t)}_{L^2(\R^3)}^2+\frac{C_*}{\eta}\kappa_5(t)\norm{\nabla u(t), G(\rho,c)(t)}_{L^4(\R^3)}^4.
    \]
    From \eqref{ep2.8}, we have $\absb{G(\rho,c)}^4\leqslant C_* H_3(\rho,c)$; and \eqref{sec3:eq22} yields: 
    \[
    \sigma(t)^{\frac{1}{2}}\scalar{t}^{2(1+r_0)}\norm{ G(\rho,c)(t)}_{L^4(\R^3)}^4\leqslant C_*\bigg( C_0\big(1+\check\phi_\Sigma^{\frac{6}{\alpha}} \mathcal{A}_1(t)\big)+\mathcal{A}_2(t)^{1/2}\big(C_0^{1/2}+\mathcal{A}_1(t)^{1/2}\big)\mathcal{A}_1(t)\bigg).
    \]
    From \eqref{sec3:eq20} and the commutator estimate \eqref{ap:eq6}, it follows that 
    \begin{align*}
        \norm{\nabla u}_{L^{4}(\R^3)}&\leqslant \frac{C}{\mu_*}\Big(\norm{(-\Delta)^{-1}\nabla \mathbb P(\rho \dot u)}_{L^{4}(\R^3)}+\norm{(-\Delta)^{-1}\dvg (\rho \dot u)}_{L^4(\R^3)}+\norm{ G(\rho,c)}_{L^4(\R^3)}\Big)\notag\\
    &+\frac{C}{\mu_*}\Big(\norm{\big[\mathcal{K}, \mu(\rho,c)-\widetilde\mu\big]\D u}_{L^4(\R^3)}+\norm{\big[\check{\mathcal{K}}, \mu(\rho,c)-\widetilde\mu\big]\D u}_{L^4(\R^3)}\Big)\\
    &\leqslant \frac{C}{\mu_*}\Big(\norm{(-\Delta)^{-1}\nabla \mathbb P(\rho \dot u)}_{L^{4}(\R^3)}+\norm{(-\Delta)^{-1}\dvg (\rho \dot u)}_{L^4(\R^3)}+\norm{ G(\rho,c)}_{L^4(\R^3)}\Big)\notag\\
    &+\frac{C}{\mu_*}\Big( \check \phi_\Sigma\norm{\nabla u}_{L^{12/(2\alpha+3)}(\R^3)}+\norm{\llbracket \mu(\rho,c)\rrbracket}_{L^\infty(\Sigma)}\norm{\nabla u}_{L^4(\R^3)}\Big)
    \end{align*}
    and, using the smallness of the shear viscosity jump together with 
    interpolation and Young's inequality, we obtain 
    \begin{align*}
        \norm{\nabla u}_{L^{4}(\R^3)}&\leqslant C_*\norm{(-\Delta)^{-1}\nabla \mathbb P(\rho \dot u),\,(-\Delta)^{-1}\dvg (\rho \dot u),\, G(\rho,c)}_{L^4(\R^3)}+ C_* \check \phi_\Sigma^{\frac{3}{2\alpha}}\norm{\nabla u}_{L^{2}(\R^3)}.
    \end{align*}
    Then, Gagliardo–Nirenberg's inequality yields
    \begin{align*}
        &\sigma^{\frac{1}{2}+\check{\alpha}}\scalar{}^{\frac{11}{4}(1+r_0)}\norm{(-\Delta)^{-1}\nabla \mathbb P(\rho \dot u),\,(-\Delta)^{-1}\dvg (\rho \dot u)}_{L^4(\R^3)}^4\\
        &\leqslant C_*\Big[\sigma \scalar{}^{\frac{3}{2}(1+r_0)}\norm{\sqrt\rho \dot u}_{L^2(\R^3)}^2\Big]^{\frac{1}{2}}\Big[\sigma^{\check{\alpha}} \scalar{}^{\frac{3}{2}(1+r_0)}\norm{\sqrt\rho \dot u}_{L^2(\R^3)}^2\Big]\Big[\scalar{}^{1+r_0}\norm{\nabla u,\,G(\rho,c)}_{L^2(\R^3)}^2\Big]^{\frac{1}{2}}\\
        &\leqslant C_* \check{\mathcal{A}}_2(t)\mathcal{A}_1(t)^{1/2}\big(C_0^{1/2}+\mathcal{A}_1(t)^{1/2}\big),
    \end{align*}
    and hence 
    \begin{align}
        \sup_{[0,t]}\sigma^{\frac{1}{2}+\check{\alpha}}\scalar{}^{\frac{3}{2}(1+r_0)}\norm{\nabla u,\, G(\rho,c)}_{L^{4}(\R^3)}^4&\leqslant C_* \big(C_0+ \check\phi_\Sigma^{\frac{6}{\alpha}} \mathcal{A}_1(t)\big(C_0+\mathcal{A}_1(t)\big)\big)\notag\\
        &+ C_* \mathcal{A}_1(t)^{1/2}\big(\mathcal{A}_1(t)+\mathcal{A}_2(t)+\check{\mathcal{A}}_2(t)\big) \big(C_0^{1/2}+\mathcal{A}_1(t)^{1/2}\big).\label{sec3:eq40}
    \end{align}
    Finally, we have the following estimate for the first term in the last line of \eqref{sec3:eq25}: 
    \begin{align*}
        \abs{\kappa_5(t) J_1(t)}&\leqslant \eta\kappa_5(t) \norm{\nabla\dot u(t)}_{L^2(\R^3)}^2\\
        &+\frac{C_*}{\eta}\Big(C_0+ \mathcal{A}_1(t)\check \phi_\Sigma^{\frac{6}{\alpha}}\big(C_0+\mathcal{A}_1(t)\big)+ \mathcal{A}_1(t)^{1/2}\big(\mathcal{A}_1(t)+\mathcal{A}_2(t)+\check{\mathcal{A}}_2(t)\big) \big(C_0^{1/2}+\mathcal{A}_1(t)^{1/2}\big)\Big).
    \end{align*}
    The remaining term will be addressed in the following item.
\end{itemize}
 \paragraph{\textbf{Step 3}} 
We estimate the remaining term in \eqref{sec3:eq25} as follows:
 \[
 \abs{\int_0^t\kappa_5J_{3}(s)ds}\leqslant C_*\mathcal{A}_2(t)+C_* \Bigg(\int_0^t \sigma^{\frac{5}{4}+\frac{3}{2}\check{\alpha}}\scalar{}^{\frac{3}{2}(1+r_0)}\norm{\nabla \dot u}_{L^3(\R^3)}^3\Bigg)^{2/3}\Bigg(\int_0^t \sigma^{\frac{1}{2}}\scalar{}^{\frac{3}{2}(1+r_0)}\norm{\nabla u}_{L^3(\R^3)}^3\Bigg)^{1/3}.
 \]
 We now express $\nabla \dot u$ in order to estimate its $L^3(\R^3)$-norm. To achieve this, we apply to \eqref{sec3:eq36} the computations that lead to \eqref{sec2:eq3}. Indeed, applying the Leray projector $\mathbb P$ to \eqref{sec3:eq36}, we obtain
\[
\mathbb P(\rho \ddot u)=\dvg \mathbb P\big(2\mu(\rho,c)\D \dot u\big)+ \dvg \mathbb P\bigg(\big(\Pi-2\rho \partial_{\rho}\mu(\rho,c)\D u\big) \dvg u-\Pi\cdot \nabla u-\mu(\rho,c)\big(\nabla u\cdot \nabla u +(\nabla u\cdot\nabla u)^t\big)\bigg)
\]
and since 
\begin{align*}
\dvg \mathbb P\big(2\mu(\rho,c)\D \dot u\big)&= \dvg \mathbb P\big(\big(\mu(\rho,c)-\widetilde\mu\big)(2\D \dot u)\big)+\widetilde\mu \Delta \mathbb{P} \dot u,
\end{align*}
it follows that 
\begin{align}\label{sec3:eq38}
\mu(\rho,c)\nabla \mathbb P \dot u&= -(-\Delta)^{-1}\nabla \mathbb P(\rho \ddot u)+\big[\mathcal{K}, \mu(\rho,c)-\widetilde\mu\big](2\D \dot u)\notag\\
&+\nabla (-\Delta)^{-1}\dvg \mathbb P\bigg(\big(\Pi-2\rho \partial_{\rho}\mu(\rho,c)\D u\big) \dvg u-\Pi\cdot \nabla u-\mu(\rho,c)\big(\nabla u\cdot \nabla u +(\nabla u\cdot\nabla u)^t\big)\bigg).
\end{align}
Next, applying the divergence operator to \eqref{sec3:eq36}, we obtain
\begin{align*}
    \dvg (\rho \ddot u)&=\dvg \dvg \big(2\mu(\rho,c)\D\dot u\big)+\Delta \big(\lambda(\rho,c)\dvg \dot u-\lambda(\rho,c)\nabla u^l\cdot \partial_l u -\rho \partial_\rho \lambda(\rho,c)(\dvg u)^2+\rho \partial_\rho P(\rho,c)\dvg u\big)\\
    &+ \dvg \dvg\bigg(\big(\Pi-2\rho \partial_{\rho}\mu(\rho,c)\D u\big) \dvg u-\Pi\cdot \nabla u-\mu(\rho,c)\big(\nabla u\cdot \nabla u +(\nabla u\cdot\nabla u)^t\big)\bigg),
\end{align*}
and therefore, $\check{F}$ given by 
\[
\check{F}=\nu(\rho,c)\dvg \dot u-\lambda(\rho,c)\nabla u^l\cdot \partial_l u -\rho \partial_\rho \lambda(\rho,c)(\dvg u)^2+\rho \partial_\rho P(\rho,c)\dvg u
\]
can be expressed as 
\begin{align}
    \check{F}&=-(-\Delta)^{-1}\dvg (\rho \ddot u)+\big[\mathcal{K}', \mu(\rho,c)-\widetilde\mu\big](2\D u)\notag\\
    &+(-\Delta)^{-1}\dvg \dvg\bigg(\big(\Pi-2\rho \partial_{\rho}\mu(\rho,c)\D u\big) \dvg u-\Pi\cdot \nabla u-\mu(\rho,c)\big(\nabla u\cdot \nabla u +(\nabla u\cdot\nabla u)^t\big)\bigg).\label{sec3:eq39}
\end{align}
In turn, we can express $\dvg \dot u$ as
\begin{gather}\label{sec3:eq37}
    \nu(\rho,c)\dvg \dot u=\check F+\lambda(\rho,c)\nabla u^l\cdot \partial_l u +\rho \partial_\rho \lambda(\rho,c)(\dvg u)^2-\rho \partial_\rho P(\rho,c)\dvg u.
\end{gather}
From \eqref{sec3:eq37}-\eqref{sec3:eq38} it follows that
\begin{align*}
    \norm{\nabla \dot u}_{L^3(\R^3)}&=\norm{\nabla \mathbb P \dot u-\nabla (-\Delta)^{-1}\nabla\dvg \dot u}_{L^3(\R^3)}\\
    &\leqslant C_*\Big(\norm{(-\Delta)^{-1}\nabla \mathbb P(\rho \ddot u),\,(-\Delta)^{-1}\dvg (\rho \ddot u)}_{L^3(\R^3)}+\norm{\nabla u,\, G(\rho,c)}_{L^6(\R^3)}^2+C_*\norm{\dvg u}_{L^3(\R^3)}\Big)\\
    &+\frac{C}{\mu_*}\norm{\big[\mathcal{K}', \mu(\rho,c)-\widetilde\mu\big](2\D u),\,\big[\mathcal{K}, \mu(\rho,c)-\widetilde\mu\big](2\D \dot u)}_{L^3(\R^3)},
\end{align*}
and the commutator estimate \eqref{ap:eq6},  together with the smallness of the shear viscosity jump yields
\begin{align}\label{sec3:eq43}
\norm{\nabla \dot u}_{L^3(\R^3)}&\leqslant C_*\norm{(-\Delta)^{-1}\nabla \mathbb P(\rho \ddot u),\,(-\Delta)^{-1}\dvg (\rho \ddot u)}_{L^3(\R^3)}\notag\\
&+C_*\Big(\norm{\nabla u,\, G(\rho,c)}_{L^6(\R^3)}^2+\norm{\dvg u}_{L^3(\R^3)}+\check \phi_\Sigma^{\frac{1}{\alpha}}\norm{\nabla \dot u}_{L^{2}(\R^3)}\Big),
\end{align}
and therefore 
\begin{align}
    &\int_0^t\sigma^{\frac{5}{4}+\frac{3}{2}\check\alpha}\scalar{}^{\frac{3}{2}(1+r_0)}\norm{\nabla \dot u}_{L^3(\R^3)}^3\notag\\
    &\leqslant C_*\Bigg(\int_0^t \kappa_5\norm{\sqrt{\rho}\ddot u}_{L^2(\R^3)}^2\bigg)^{3/4}\Bigg(\int_0^t \sigma^{2+3\check{\alpha}}\scalar{}^{\frac{3}{2}(1+r_0)}\norm{(-\Delta)^{-1}\nabla \mathbb P(\rho \ddot u),\,(-\Delta)^{-1}\dvg (\rho \ddot u)}_{L^2(\R^3)}^{6}\Bigg)^{1/4}\notag\\
    &+C_*\int_0^t\sigma^{\frac{5}{4}+\frac{3}{2}\check\alpha}\scalar{}^{\frac{3}{2}(1+r_0)}\big(\norm{\nabla u,\, G(\rho,c)}_{L^6(\R^3)}^6+\norm{\dvg u}_{L^3(\R^3)}^3+\check \phi_\Sigma^{\frac{3}{\alpha}}\norm{\nabla \dot u}_{L^{2}(\R^3)}^3\big).\notag
\end{align}
With the help of the expressions of $(-\Delta)^{-1}\nabla \mathbb P(\rho \ddot u)$ and $(-\Delta)^{-1}\dvg (\rho \ddot u)$ in \eqref{sec3:eq38}-\eqref{sec3:eq39}, we have
\begin{align*}
    &\int_0^t \sigma^{2+3\check{\alpha}}\scalar{}^{\frac{3}{2}(1+r_0)}\norm{(-\Delta)^{-1}\nabla \mathbb P(\rho \ddot u),\,(-\Delta)^{-1}\dvg (\rho \ddot u)}_{L^2(\R^3)}^{6}\\
    &\leqslant C_*\int_0^t\sigma^{2+3\check{\alpha}}\scalar{}^{\frac{3}{2}(1+r_0)}\Big(\norm{\nabla \dot u}_{L^2(\R^3)}^6+\norm{\nabla u,\, G(\rho,c)}_{L^4(\R^3)}^{12}+\norm{\dvg u}_{L^2(\R^3)}^6\Big)\\
    &\leqslant C_*\Bigg( C_0 \mathcal{A}_1(t)^2+ \mathcal{A}_3(t)^2\int_0^t\sigma^{\check{\alpha}}\norm{\nabla \dot u}_{L^2(\R^3)}^2+\Big[\sup_{[0,t]}\sigma^{\frac{3}{4}+\frac{3\check{\alpha}}{2}}\norm{\nabla u, G(\rho,c)}_{L^4(\R^3)}^4\Big]^2\int_0^t\sigma^{\frac{1}{2}}\norm{\nabla u, G(\rho,c)}_{L^4(\R^3)}^4\Bigg).
\end{align*}
It follows that 
\begin{align*}
    \int_0^t\sigma^{\frac{5}{4}+\frac{3}{2}\check\alpha}\scalar{}^{\frac{3}{2}(1+r_0)}\norm{\nabla \dot u}_{L^3(\R^3)}^3&\leqslant  C_*\mathcal{A}_3(t)^{3/4}\Big( C_0 \mathcal{A}_1(t)^2+ \mathcal{A}_3(t)^2\check{\mathcal{A}}_2(t)\Big)^{1/4}\\
    &+C_* \mathcal{A}_3(t)^{3/4}\Big[\sup_{[0,t]}\sigma^{\frac{3}{4}+\frac{3\check{\alpha}}{2}}\norm{\nabla u, G(\rho,c)}_{L^4(\R^3)}^4\Big]^{1/2}\Bigg(\int_0^t\sigma^{\frac{1}{2}}\norm{\nabla u, G(\rho,c)}_{L^4(\R^3)}^4\Bigg)^{1/4}\\
    &+C_*\int_0^t\sigma^{\frac{5}{4}+\frac{3}{2}\check\alpha}\scalar{}^{\frac{3}{2}(1+r_0)}\big(\norm{\nabla u,\, G(\rho,c)}_{L^6(\R^3)}^6+\norm{\dvg u}_{L^3(\R^3)}^3+\check \phi_\Sigma^{\frac{3}{\alpha}}\norm{\nabla \dot u}_{L^{2}(\R^3)}^3\big),
\end{align*}
and we recall the estimates \eqref{sec3:eq40}-\eqref{sec3:eq22}-\eqref{sec3:eq21}-\eqref{sec3:eq19}.

Summing all the estimates above and applying Young’s inequality, we conclude that 
\begin{align*}
    \mathcal{A}_3(t)+\int_0^t\sigma^{\frac{5}{4}+\frac{3}{2}\check\alpha}\scalar{}^{\frac{3}{2}(1+r_0)}\norm{\nabla \dot u}_{L^3(\R^3)}^3&\leqslant C_*\Big( \check C_0+ \mathcal{L}\bigg(\check\phi_\Sigma^{\frac{2}{\alpha}}C_0,\,\mathcal{A}_1(t)\bigg)\Big)\exp{\Bigg[C_*
 \int_0^{\sigma(t)} \Big(\norm{\nabla u}_{L^3(\R^3)}^2+\check \phi_\Sigma^{\frac{2}{\alpha}}\norm{\nabla u}_{L^3(\R^3)}\Big)\Bigg]},
\end{align*}
where $\mathcal{L}$ is a polynomial, whose exact expression is not needed for the remaining steps in the proof of Theorem \ref{th1}. Estimate \eqref{sec2:eq55} finally follows from (see \eqref{sec3:eq42})
\[
\int_0^{\sigma(t)}\norm{\nabla u}_{L^3(\R^3)}^3\leqslant C_* C_0\big(1+\check{\phi}_\Sigma^{\frac{3}{\alpha}}E_0^{\frac{1}{2}}\mathcal{A}_1(t)^{\frac{1}{2}}\big)+ C_* E_0^{\frac{1}{4}}\big(\mathcal{A}_1(t)^{\frac{5}{4}}+ C_0^{\frac{5}{4}} \big).
\]
\end{proof}
\subsection{Proof of Lemma \ref{sec2:lem5}}\label{proof:sec2:lem5}
\begin{proof}
From\eqref{sec3:eq43}, we have
\begin{align}
    &\int_0^t\sigma^{\frac{1}{2}+\check{\alpha}}\scalar{}^{1+r_0}\norm{\nabla \dot u}_{L^3(\R^3)}^2\notag\\
    &\leqslant C_*\Bigg(\int_0^t \sigma^{1+\check{\alpha}}\scalar{}^{\frac{3}{2}(1+r_0)}\norm{\rho\ddot u}_{L^2(\R^3)}^2\Bigg)^{1/2}\Bigg(\int_0^t\sigma^{\check{\alpha}}\scalar{}^{1+r_0}\norm{(-\Delta)^{-1}\nabla \mathbb P(\rho \ddot u),\,(-\Delta)^{-1}\dvg (\rho \ddot u)}_{L^2(\R^3)}^2\Bigg)^{1/2}\notag\\
    &+C_*\int_0^t\sigma^{\frac{1}{2}+\check{\alpha}}\scalar{}^{1+r_0}\big(\norm{\nabla u,\, G(\rho,c)}_{L^6(\R^3)}^4+\norm{\dvg u}_{L^3(\R^3)}^2+\check \phi_\Sigma^{\frac{2}{\alpha}}\norm{\nabla \dot u}_{L^{2}(\R^3)}^2\big)\notag\\
    &\leqslant C_*\mathcal{A}_3(t)^{1/2}\Bigg(\int_0^t\sigma^{\check{\alpha}}\scalar{}^{1+r_0}\Big( \norm{\nabla\dot u}_{L^2(\R^3)}^2+\norm{\nabla u}_{L^4(\R^3)}^4+\norm{\dvg u}_{L^2(\R^3)}^2\Big)\Bigg)^{1/2}\notag\\
    &+C_*\int_0^t\sigma^{\frac{1}{2}+\check{\alpha}}\scalar{}^{1+r_0}\big(\norm{\nabla u,\, G(\rho,c)}_{L^6(\R^3)}^4+\norm{\dvg u}_{L^3(\R^3)}^2+\check \phi_\Sigma^{\frac{2}{\alpha}}\norm{\nabla \dot u}_{L^{2}(\R^3)}^2\big).\label{sec3:eq47}
\end{align}
Using \eqref{sec3:eq44}, we deduce 
\begin{align*}
\int_0^t\sigma^{\check{\alpha}}\scalar{}^{1+r_0}\Big( \norm{\nabla\dot u}_{L^2(\R^3)}^2+\norm{\nabla u}_{L^4(\R^3)}^4&+\norm{\dvg u}_{L^2(\R^3)}^2\Big)\\
&\leqslant C_*\Big(\check{\mathcal{A}}_2(t)+C_0\big(1+\check\phi_\Sigma^{\frac{6}{\alpha}} \mathcal{A}_1(t)\big)+\big(C_0+\mathcal{A}_1(t)\big)\mathcal{A}_1(t)^2\Big),
\end{align*}
and with the help of \eqref{sec3:eq45}, the remaining terms can be estimated as follows:
\begin{align*}
    &\int_0^t\sigma^{\frac{1}{2}+\check{\alpha}}\scalar{}^{1+r_0}\big(\norm{\nabla u,\, G(\rho,c)}_{L^6(\R^3)}^4+\norm{\dvg u}_{L^3(\R^3)}^2+\check \phi_\Sigma^{\frac{2}{\alpha}}\norm{\nabla \dot u}_{L^{2}(\R^3)}^2\big)\\
    &\leqslant \Bigg(\int_0^t\scalar{}^{1+r_0}\norm{\nabla u, G(\rho,c)}_{L^6(\R^3)}^2\Bigg)^{1/2}\Bigg(\int_0^t\sigma^{1+2\check{\alpha}}\scalar{}^{1+r_0}\big(\norm{\nabla u,\, G(\rho,c)}_{L^6(\R^3)}^6+\norm{\dvg u}_{L^2(\R^3)}^2\big)\Bigg)^{1/2}+C_*\check \phi_\Sigma^{\frac{2}{\alpha}}\check{\mathcal{A}}_2(t)\\
    &\leqslant \Bigg(\int_0^t\scalar{}^{1+r_0}\norm{\nabla u, G(\rho,c)}_{L^6(\R^3)}^2\Bigg)^{1/2}\Bigg(C_{*}C_0 \big(1+\check\phi_\Sigma^{\frac{12}{\alpha}}\mathcal{A}_1(t)^2\big)+C_{*}\mathcal{A}_2(t)\big(1+\mathcal{A}_1(t)\check{\mathcal{A}}_2(t)\big)\Bigg)^{1/2}+C_*\check \phi_\Sigma^{\frac{2}{\alpha}}\check{\mathcal{A}}_2(t).
\end{align*}
The $L^2((0,t),L^6(\R^3))$-norm of $\scalar{}^{(1+r_0)/2}\big(\nabla u, G(\rho,c)\big)$ will be estimated in what follows.

To achieve this, we first deduce from \eqref{ep2.9}-\eqref{ep2.12} to obtain 
\begin{align*}
\dfrac{d}{dt} \int_{\R^3}H_5(\rho,c)+\int_{\R^3}\dfrac{1}{\nu(\rho,c)}\abs{G(\rho,c)}^{6}&=-\int_{\R^3}\dfrac{1}{\nu(\rho,c)} F G(\rho,c)\abs{G(\rho,c)}^4,
\end{align*}
and from \eqref{ep2.7}-\eqref{ep2.8}, we have 
\begin{gather}\label{sec3:eq46}
\widetilde\rho(c) g(\rho,c)H_l(\rho,c)\leqslant \abs{G(\rho,c)}^{l+1}\leqslant \frac{\rho+\widetilde\rho(c)}{a_l(c)}g(\rho,c)H_{l}(\rho,c),\;\text{ where }\; g(\rho,c)=\frac{G(\rho,c)}{\rho-\widetilde\rho(c)}.
\end{gather}
We will use the following notations: 
\begin{gather}
    a_{l,*}=\inf_c a_{l}(c),\;\; g^*=\sup_{\rho,c} g(\rho,c),\;\; g_*=\inf_{\rho,c} g(\rho,c),\;\; k_*=\dfrac{\rho_* g_*}{2\rho^* g^*}\;\text{ and }\;  k^*_l=\frac{\nu^*}{a_{l,*}\nu_*}.
\end{gather}
It follows  that 
\begin{align*}
\dfrac{d}{dt} \int_{\R^3}H_5(\rho,c)+\int_{\R^3}\widetilde\rho(c)\dfrac{g(\rho,c)}{\nu(\rho,c)}H_5(\rho,c)&\leqslant \Bigg(\int_{\R^3}\dfrac{1}{\nu(\rho,c)} \abs{F}^6\Bigg)^{1/6}\Bigg(\int_{\R^3} \frac{\rho+\widetilde\rho(c)}{a_5(c)}\frac{g(\rho,c)}{\nu(\rho,c)}H_5(\rho,c)\Bigg)^{5/6}\\
&\leqslant \frac{1}{\nu_*}\Bigg(\frac{2\rho^* g^* }{a_{5,*}}\Bigg)^{5/6} \norm{F}_{L^6(\R^3)}\Bigg(\int_{\R^3}H_5(\rho,c)\Bigg)^{5/6}.
\end{align*}
Hence
\begin{gather}\label{sec3:eq48}
    3\dfrac{d}{dt}\Bigg(\int_{\R^3}H_5(\rho,c)\Bigg)^{1/3}+\dfrac{\rho_* g_*}{\nu^*}\Bigg(\int_{\R^3}H_5(\rho,c)\Bigg)^{1/3}\leqslant \frac{1}{\nu_*}\Bigg(\frac{2\rho^* g^* }{a_{5,*}}\Bigg)^{5/6} \norm{F}_{L^6(\R^3)}\Bigg(\int_{\R^3}H_5(\rho,c)\Bigg)^{1/6},
\end{gather}
and  Young's inequality yields 
\begin{gather*}
    3\dfrac{d}{dt}\Bigg(\int_{\R^3}H_5(\rho,c)\Bigg)^{1/3}+\dfrac{\rho_* g_*}{2\nu^*}\Bigg(\int_{\R^3}H_5(\rho,c)\Bigg)^{1/3}\leqslant \frac{\nu^*\nu_*^{-2}}{2\rho_* g_*}\Bigg(\frac{2\rho^* g^*}{a_{5,*}}\Bigg)^{5/3} \norm{F}_{L^6(\R^3)}^{2}.
\end{gather*}
Upon multiplying by the time weight $\kappa_1=\scalar{}^{1+r_0}$ and integrating in time, we obtain:
\begin{align*}
 \kappa_1(t)\Bigg(\int_{\R^3}H_5(\rho,c)(t)\Bigg)^{1/3}+\dfrac{\rho_* g_*}{6\nu^*}\int_0^t\kappa_1\Bigg(\int_{\R^3}H_5(\rho,c)\Bigg)^{1/3}&\leqslant \Bigg(\int_{\R^3}H_5(\rho_0,c_0)\Bigg)^{1/3}+\int_0^t\kappa_1'\Bigg(\int_{\R^3}H_5(\rho,c)\Bigg)^{1/3}\\
        &+\frac{\nu^*\nu_*^{-2}}{6\rho_* g_*}\Bigg(\frac{2\rho^* g^*}{a_{5,*}}\Bigg)^{5/3}\int_0^t\kappa_1\norm{F}_{L^6(\R^3)}^{2}.
\end{align*}
From  \eqref{sec3:eq17}-\eqref{ep3.8} we infer
\[
\Bigg(\int_{\R^3}H_5(\rho_0,c_0)\Bigg)^{1/3}+\int_0^t\kappa_1'\Bigg(\int_{\R^3}H_5(\rho,c)\Bigg)^{1/3}\leqslant C_*  C_0,
\]
and  \eqref{sec3:eq46}, yields
\[
\int_0^t\kappa_1\Bigg(\int_{\R^3}H_5(\rho,c)\Bigg)^{1/3}\geqslant \Big(\frac{a_{5,*}}{2\rho^* g^*}\Big)^{1/3}\int_0^t\kappa_1\norm{G(\rho,c)}_{L^6(\R^3)}^2.
\]
It then follows that 
\begin{gather*}
    k_*^2\int_0^t\kappa_1\norm{G(\rho,c)}_{L^6(\R^3)}^2\leqslant C_* C_0+{k_5^*}^2\int_0^t\kappa_1\norm{F}_{L^6(\R^3)}^{2}.
\end{gather*}
Once again, we combine \eqref{sec3:eq20}–\eqref{ap:eq6} to deduce that 
\begin{align*}
        k_*^2\int_0^t\kappa_1\Big(\norm{G(\rho,c)}_{L^6(\R^3)}^2+\mu_*^2\norm{\nabla u}_{L^6(\R^3)}^2\Big)&\leqslant C_* C_0 +C_*\int_0^t\kappa_1\norm{(-\Delta)^{-1}\nabla \mathbb P(\rho \dot u),\,(-\Delta)^{-1}\dvg (\rho \dot u)}_{L^6(\R^3)}^2\notag\\
    &+C\big(k_*^2+{k_5^*}^2\big)\int_0^t\kappa_1\norm{\big[\mathcal{K}, \mu(\rho,c)-\widetilde\mu\big]\D u,\,\big[\check{\mathcal{K}}, \mu(\rho,c)-\widetilde\mu\big]\D u}_{L^6(\R^3)}^2\\
    &\leqslant C_* C_0 +C_*\int_0^t\kappa_1\Big(\norm{\sqrt\rho \dot u}_{L^2(\R^3)}^2+\check{\phi}_\Sigma^2\norm{\nabla u}_{L^{6/(1+\alpha)}(\R^3)}^2\Big)\notag\\
    &+C\big(k_*^2+{k_5^*}^2\big)\sup_{[0,t]}\norm{\llbracket\mu(\rho,c)\rrbracket}_{L^\infty(\Sigma)}^2\int_0^t\kappa_1\norm{\nabla u}_{L^6(\R^3)}^2.
\end{align*}
Under the smallness assumption on the shear-viscosity jump 
\[
 \frac{C}{\mu_*^2}\Bigg(1+\frac{{k_5^*}^2}{k_*^2}\Bigg)\sup_{[0,t]}\norm{\llbracket\mu(\rho,c)\rrbracket}_{L^\infty(\Sigma)}^2\leqslant \frac{1}{2},
\]
the last term can be absorbed into the LHS, and Young’s inequality gives
\[
\int_0^t\kappa_1\norm{G(\rho,c),\,\nabla u}_{L^6(\R^3)}^2
\leqslant C_* \Big(C_0 \big(1+\check{\phi}_{\Sigma}^{\frac{2}{\alpha}}\big)+\mathcal{A}_1(t)\Big).
\]
This completes the estimate of the last term in \eqref{sec3:eq47}.

Using an interpolation inequality, we have:
\begin{align*}
\int_0^t \norm{\dot u}_{L^{6/(2-\alpha)}(\R^3)}\leqslant C_*\Bigg(\int_0^t\kappa_1\norm{\dot u}_{L^2(\R^3)}^2\Bigg)^{\frac{1-\alpha}{4}}\Bigg(\int_0^t\sigma\kappa_1\norm{\nabla\dot u}_{L^2(\R^3)}^2\Bigg)^{\frac{1+\alpha}{4}}\Bigg(\int_0^t\kappa^{-1}_1\sigma^{-\frac{1+\alpha}{2}}\Bigg)^{\frac{1}{2}},
\end{align*}
and 
\begin{itemize}
    \item if $\alpha\in (0, 1/2)$, 
    \begin{align*}
         \int_0^t \norm{\dot u}_{L^{3/(1-\alpha)}(\R^3)}&\leqslant C_*\int_0^t \norm{\rho\dot u}_{L^2(\R^3)}^{\tfrac{1-2\alpha}{2}}\norm{\nabla \dot u}_{L^2(\R^3)}^{\tfrac{1+2\alpha}{2}}\\
         &\leqslant C_*\Bigg(\int_0^t \kappa_1\norm{\rho\dot u}_{L^2(\R^3)}^2\Bigg)^{\tfrac{1-2\alpha}{4}}\Bigg(\int_0^t\sigma\kappa_1\norm{\nabla \dot u}_{L^2(\R^3)}^2\Bigg)^{\tfrac{1+2\alpha}{4}}\Bigg(\int_0^t\kappa_1^{-1}\sigma^{-\frac{1+2\alpha}{2}}\Bigg)^{1/2};
    \end{align*}
    \item  If $\alpha\in [1/2,1)$, we have $\check\alpha\in (0,3/2-\alpha)$ and whence 
    \begin{align*}
\int_0^t \norm{\dot u}_{L^{3/(1-\alpha)}(\R^3)}&\leqslant C_*\int_0^t \norm{\nabla \dot u}_{L^2(\R^3)}^{2(1-\alpha)}\norm{\nabla\dot u}_{L^3(\R^3)}^{2\alpha-1}\\
&\leqslant C_*\Bigg(\int_0^t\sigma^{\check{\alpha}}\kappa_1\norm{\nabla\dot u}_{L^2(\R^3)}^2\Bigg)^{1-\alpha}\Bigg(\int_0^t\sigma^{\frac{1}{2}+\check{\alpha}}\kappa_1\norm{\nabla\dot u}_{L^3(\R^3)}^2\Bigg)^{\alpha-\frac{1}{2}}\Bigg(\int_0^t\kappa_1^{-1}\sigma^{\frac{1}{2}-(\alpha+\check{\alpha})}\Bigg)^{\frac{1}{2}},
\end{align*}
\end{itemize}

The final step is devoted to obtaining an $L^1((0,t),L^p(\R^3))$ estimate for $\nabla u$ and $G(\rho,c)$, with $p\in [6/\alpha,\infty)$.

To achieve this, we follow the computations that lead to \eqref{sec3:eq48} and we obtain  
\[
p\dfrac{d}{dt}\Bigg[\int_{\R^3} H_{p-1}(\rho,c)\Bigg]^{\frac{1}{p}}+\dfrac{\rho_{*} g_*}{\nu^*}\Bigg[\int_{\R^3} H_{p-1}(\rho,c)\Bigg]^{\frac{1}{p}}\leqslant \dfrac{1}{\nu_*}\Bigg(\frac{2\rho^* g^*}{a_{p,*}}\Bigg)^{\frac{1}{p'}}\norm{F}_{L^{p}(\R^3)},
\]
and time integration, together with \eqref{sec3:eq46}, yields:
\[
k_* \int_0^t\norm{G(\rho,c)}_{L^p(\R^3)}\leqslant C_* C_0+k_p^*\int_0^t\norm{F}_{L^{p}(\R^3)}.
\]
It follows that (using again \eqref{sec3:eq20}–\eqref{ap:eq6}):
\begin{align*}
            k_*\int_0^t\norm{G(\rho,c),\,\mu_*\nabla u}_{L^p(\R^3)}&\leqslant C_* C_0 +C_*\int_0^t\Big(\norm{\rho \dot u}_{L^{3p/(3+p)}(\R^3)}+\check{\phi}_\Sigma\norm{\nabla u}_{L^{6p/(1+p\alpha)}(\R^3)}\Big)\notag\\
    &+C\big(k_*+{k_p^*}\big)\sup_{[0,t]}\norm{\llbracket\mu(\rho,c)\rrbracket}_{L^\infty(\Sigma)}\int_0^t\norm{\nabla u}_{L^p(\R^3)}.
\end{align*}
Under the smallness assumption on the shear-viscosity jump 
\[
 \frac{C}{\mu_*}\Bigg(1+\frac{k_p^*}{k_*}\Bigg)\sup_{[0,t]}\norm{\llbracket\mu(\rho,c)\rrbracket}_{L^\infty(\Sigma)}\leqslant \frac{1}{2},
\]
the last term can be absorbed into the LHS, and Young’s inequality gives
\begin{align*}
&\int_0^t\norm{G(\rho,c),\,\nabla u}_{L^p(\R^3)}
\leqslant C_* C_0 +C_*\int_0^t\Big(\norm{\rho \dot u}_{L^{3p/(3+p)}(\R^3)}+\check{\phi}_\Sigma^{\frac{3}{\alpha}\big(1-\frac{2}{p}\big)}\norm{\nabla u}_{L^{2}(\R^3)}\Big)\\
&\leqslant C_*C_0\Bigg(1+\check{\phi}_\Sigma^{\frac{3}{\alpha}\big(1-\frac{2}{p}\big)}\Bigg)+C_*\Bigg(\int_0^t\kappa_1\norm{\dot u}_{L^{2}(\R^3)}^2\Bigg)^{\frac{1}{4}+\frac{3}{2p}}\Bigg(\int_0^t\sigma\kappa_1\norm{\nabla \dot u}_{L^{2}(\R^3)}^2\Bigg)^{\frac{1}{4}-\frac{3}{2p}}\Bigg(\int_0^t\kappa^{-1}_1\sigma^{-\frac{1}{2}+\frac{3}{p}}\Bigg)^{\frac{1}{2}}.
\end{align*}
This completes the proof of Lemma \ref{sec2:lem5}.
\end{proof}
\section{Proof of the main result}\label{sec:theorem}
This section is devoted to the final step in the proof of the main result, Theorem \ref{th1}. It is divided into three steps. First, (see Section \ref{part4:sec1}) we construct an approximate sequence $(c^n,\rho^n,u^n)_{n\in \N}$ for the Cauchy problem \eqref{ep2.1}.  This sequence comes with the a priori estimates summarized in Section \ref{part2}. In Section \ref{part4:sec2}, we close these estimates using a bootstrap argument.  
\subsection{Construction of approximate solutions}\label{part4:sec1}
Usually, the approximate sequence $(c^n,\rho^n,u^n)$ is constructed as the solution of the Cauchy problem \eqref{intro:twofluid} with initial data $(c_0^n, \rho_0^n, u_0^n)$ obtained by mollifying $(c_0,\rho_0,u_0)$ with a smooth kernel.
This approach is unsuited here, because smoothing the initial data would remove the density discontinuity. Furthermore, classical arguments \textit{à la Feireisl-Novotn\'y-Petzeltov\'a}, or \textit{à la P.–L. Lions} fail in the density-dependent 
viscosity case; in particular, the well-known effective flux and vorticity degenerate completely. Motivated by this, we establish in  \cite{zodji2023well} a local-in-time existence of weak solutions to the system \eqref{ep2.1} within a functional framework tailored to our setting.
In particular, beyond the initial data assumptions  \eqref{initialvolum}-\eqref{init:density-velocity}-\eqref{density:piecewise}, it is enough to require  $(\rho \dot u)_{|t=0}\in L^2(\R^3)$ to construct the approximate sequence $(c^n, \rho^n, u^n)$ to \eqref{ep2.1}. In other terms, we shall consider the Cauchy problem:
\begin{gather}\label{sec4:eq3}
\begin{cases}
    \dpt c^n+u^n\cdot\nabla c^n=0,\\
    \dpt \rho^n +\dvg (\rho^n u^n)=0,\\
    \dpt (\rho^n u^n)+\dvg (\rho^n u^n\otimes u^n)+\nabla P(\rho^n,c^n)=\dvg \big(2\mu(\rho^n,c^n)\D u^n\big)+\nabla\big(\lambda(\rho^n,c^n)\dvg u^n\big),\\
    (c^n,\rho^n,u^n)_{|t=0}= (c_0,\rho_0, u_0^n),
\end{cases}
\end{gather}
with $u_0^n\in H^1(\R^3)\cap L^{p_0}(\R^3)$, fulfilling the so-called compatibility condition
\[
    (\rho^n \dot u^n)_{|t=0}=\dvg \big(2\mu(\rho_0,c_0)\D u_0^n+\big(\lambda(\rho_0,c_0)\dvg u_0^n-P(\rho_0,c_0)+\widetilde P\big) I_3\big)\in L^2(\R^3).
\]
We construct the sequence of initial velocity $(u_0^n)_n$  as follows. 
\begin{lemma}\label{sec4:lemma:1} There exists a sequence $(u_0^n)_n$ of initial velocities such that
\[
\dvg \big(2\mu(\rho_0,c_0)\D u_0^n\big)+\nabla\big(\lambda(\rho_0,c_0)\dvg u_0^n\big)-\nabla P(\rho_0,c_0)\in L^2(\R^3);
\]
and 
\begin{gather}\label{sec4:eq8}
\normb{u_0^n-u_0}_{H^1(\R^3)}^2+\normb{u_0^n-u_0}_{L^{p_0}(\R^3)}^2\xrightarrow{n\to \infty} 0.
\end{gather}
\end{lemma}
We provide the proof of Lemma \ref{sec4:lemma:1} below. With the sequence $(u_0^n)_n$ in place, we construct $(c^n,\rho^n, u^n)_n$ by applying the Theorem \ref{sec2:thlocal}. This theorem  refines our recent work in \cite{zodji2023well}, where local-in-time well-posedness for \eqref{sec4:eq3} was obtained assuming 
\[
\norm{\mu(\rho_0,c_0)-\widetilde\mu}_{\cC^\alpha_{\pw,\Sigma_0}(\R^3)}\ll 1.
\]
Indeed, here we only require smallness of the shear viscosity \emph{jump}. In particular, when the shear viscosity is continuous, the condition holds trivially. Thus, our result aligns with the well-posedness theory for smooth data.

\begin{proof}[Proof of Lemma \ref{sec4:lemma:1}]
    Following Section 3.5 of \cite{zodji2023discontinuous}, we consider a sequence of 
    nonnegative and smooth mollifier $(\omega_n)_{n\in \N}$, and we denote by $\Pi_0$ the stress tensor at the initial time 
    \[
\Pi_0=2\mu(\rho_0,c_0)\D u_0+\big(\lambda(\rho_0,c_0)\dvg u_0-P(\rho_0,c_0)+\widetilde P\big) I_3.
   \]
We now consider the elliptic equation
   \begin{equation}\label{sec4:eq1}
    -\dvg\Big(2\mu(\rho_0,c_0)\D u_0^n+\big(\lambda(\rho_0,c_0)\dvg u^n_0- P(\rho_0,c_0)+\widetilde P\big) I_3\Big)+ d_n u_0^n=-\dvg\big(\omega_n*\Pi_0\big),
\end{equation}
where $(d_n)_{n\in \N}$ is a sequence of positive real numbers is a sequence of positive real numbers to be specified later. 

For each $n\in \N$, the existence and uniqueness of $u_0^n\in H^1(\R^3)$ follow from the Lax–Milgram theorem. We are now left to prove \eqref{sec4:eq8}. We add 
$\dvg \big(\Pi_0\big)$ to both sides of \eqref{sec4:eq1}
\begin{gather}\label{sec4:eq2}
-\dvg \big(2\mu(\rho_0,c_0)\D \big( u_0^n- u_0\big)\big)-\nabla\big(\lambda(\rho_0,c_0)\dvg \big( u_0^n- u_0\big)\big)+d_n u^n_0 =-\dvg\big(\omega_n*\Pi_0-\Pi_0\big),
\end{gather}
and basic energy estimate yields 
\begin{align*}
\int_{\R^3}\big( 2\mu(\rho_0,c_0)\absb{\D (u_0^n-u_0)}^2&+\lambda(\rho_0,c_0)\absb{\dvg (u_0^n-u_0)}^2\big)(x)dx+d_n \int_{\R^3}\absb{u_0^n}^2(x)dx\\
&=\int_{\R^3} \big(\omega_n * \Pi_0 -\Pi_0\big)(x) \D (u_0^n-u_0)(x)dx+d_n \int_{\R^3}u_0^n(x) u_0(x)dx.
\end{align*}
Applying Young’s and Korn’s inequalities, we get
\begin{align}
\int_{\R^3}\big[\mu(\rho_0,c_0)\absb{\D(u_0^n-u_0)}^2&+\lambda(\rho_0,c_0)\big(\dvg (u_0^n-u_0)\big)^2\big](x)dx+\frac{d_n}{2} \normb{u_0^n}^2_{L^2(\R^3)}\notag\\
&\leqslant \frac{1}{\mu_*}\norm{\omega_n* \Pi_0-\Pi_0}_{L^2(\R^3)}^2+\frac{d_n}{2}\normb{u_0}_{L^2(\R^3)}^2,\label{sec4:eq5}
\end{align}
and therefore 
\[
\lim_{n\to \infty}\normb{\nabla u_0^n-\nabla u_0}_{L^2(\R^3)}= 0\; \text{ and }\; \limsup_{\delta\to 0} \normb{u_0^n}_{L^2(\R^3)}\leqslant \normb{u_0}_{L^2(\R^3)},
\]
as soon as $(d_n)_n$ satisfies 
\begin{gather}\label{sec4:eq4}
d_n\xrightarrow{n\to \infty}0\;\text{ and }\; d_n^{-1}\norm{\omega_n*\Pi_0-\Pi_0}_{L^2(\R^3)}^2\xrightarrow{n\to \infty} 0.
\end{gather}
As a consequence, $u_0^n \xrightarrow{n \to \infty} u_0$ strongly in $H^1(\R^3)$. It now remains to prove $\normb{u_0^n-u_0}_{L^{p_0}(\R^3)}\xrightarrow{n\to \infty}0$.

To achieve this, we set $v_0^n= u_0^n- u_0$, and we rewrite \eqref{sec4:eq2} as 
\begin{align*}
-\widetilde\mu \Delta v_0^n-\big(\widetilde\mu+\widetilde\lambda\big)\nabla\dvg v_0^n +d_n v_0^n&=-d_n u_0-\dvg \big( \omega_n*\Pi_0- \Pi_0\big)\\
&+\dvg \big( 2\big(\mu(\rho_0,c_0)-\widetilde\mu\big)\D v^n_0\big)+\nabla\big(\big(\lambda(\rho_0,c_0)-\widetilde\lambda\big)\dvg v^n_0\big)=:h_0^n.
\end{align*}
From this, we deduce (recall that $\mathbb P$ denotes the Leray projector)
\begin{align*}
\mathbb P v_0^n= \mathscr Z_n* \mathbb P h_0^n,\;\text{ and }\;
\big( \mathbb I-\mathbb P\big) v_0^n= \check{\mathscr Z}_n* \big(\mathbb I-\mathbb P\big) h_0^n;
\end{align*}
where 
\[
\big(\mathscr{F}\mathscr Z_n\big)(\xi)=\frac{1}{d_n+ \widetilde\mu \abs{\xi}^2}, \; \text{ and }\; \big(\mathscr{F}\check{\mathscr Z}_n\big)(\xi)=\frac{1}{d_n+ (2\widetilde\mu+\widetilde\lambda) \abs{\xi}^2}.
\]
For any $a>0$, we have 
\[
\frac{1}{1+a\abs{\xi}^2}=\int_0^\infty e^{-t(1+a\abs{\xi}^2)}dt,
\]
and whence 
\[
\check{\mathscr{K}}_{a}(x):=\mathscr{F}^{-1}\Big(\xi\mapsto\frac{1}{1+a\abs{\xi}^2}\Big)(x)=\int_0^\infty e^{-t} \mathscr{F}^{-1}\Big(\xi\mapsto e^{-a t\abs{ \xi}^2}\Big)dt=a^{-3/2} \int_0^\infty e^{-t} \mathscr{K}_{t}\Big(\frac{x}{\sqrt a}\Big)dt,
\]
where  $(\mathscr{K}_t)_{t>0}$ denotes the heat kernel. As a result, for all $q\in [1,\infty]$ and $p\in [1,3/2)$, we have

\[
\norm{\check{\mathscr{K}}_{a}}_{L^q(\R^3)}\leqslant C a^{-\frac{3}{2}(1-1/q)},\;\text{ and }\; \norm{\nabla \check{\mathscr{K}}_{a}}_{L^p(\R^3)}\leqslant C a^{-2+3/(2p)}.
\]
Since $\mathscr{Z}_{n}={d_n}^{-1}\check{\mathscr{K}}_{\widetilde\mu/d_n}$,
it is straightforward to derive the following bound (for some $p\in (1,p_0)$, and $1/\check{p}+1/p=1/2+1/p_0$):
\begin{align*}
    \norm{\mathscr{Z}_{n}*\dvg \mathbb P\big( 2\big(\mu(\rho_0,c_0)-\widetilde\mu\big)\D v^n_0\big)}_{L^{p_0}(\R^3)}&+\norm{\mathscr{Z}_{n}*\mathbb P\nabla\big(\big(\lambda(\rho_0,c_0)-\widetilde\lambda\big)\dvg v^n_0\big)}_{L^{p_0}(\R^3)}\\
    &\leqslant C\norm{\nabla \mathscr{Z}_{n}}_{L^p(\R^3)}\normb{\nabla v_0^n}_{L^{2}(\R^3)}\normb{\mu(\rho_0,c_0)-\widetilde\mu,\, \lambda(\rho_0,c_0)-\widetilde\lambda}_{L^{\check p}(\R^3)}\\
    &\leqslant C_* d_n^{1-\frac{3}{2p}}\normb{\nabla v_0^n}_{L^{2}(\R^3)}\normb{\mu(\rho_0,c_0)-\widetilde\mu,\, \lambda(\rho_0,c_0)-\widetilde\lambda}_{L^{\check p}(\R^3)}.
\end{align*}
Hence, from \eqref{sec4:eq4}–\eqref{sec4:eq5}, we obtain
\begin{gather}\label{sec4:eq6}
\norm{\mathscr{Z}_{n}*\dvg \mathbb P\big( 2\big(\mu(\rho_0,c_0)-\widetilde\mu\big)\D v^n_0\big)}_{L^{p_0}(\R^3)}+\norm{\mathscr{Z}_{n}*\mathbb P\nabla\big(\big(\lambda(\rho_0,c_0)-\widetilde\lambda\big)\dvg v^n_0\big)}_{L^{p_0}(\R^3)} \xrightarrow{n\to \infty} 0.
\end{gather}

Next, for $R>0$ we write 
\begin{align*}
d_n \mathscr{Z}_n* \mathbb P u_0&=\check{\mathscr{K}}_{\widetilde\mu/d_n}* \mathbb P u_0\\
&=\check{\mathscr{K}}_{\widetilde\mu/d_n}* \mathbb P \big( \mathbb 1_{B_R}u_0\big)+\check{\mathscr{K}}_{\widetilde\mu/d_n}* \mathbb P \big( \mathbb 1_{B_R^c}u_0\big),
\end{align*}
and it follows that 
\[
\norm{d_n \mathscr{Z}_n* \mathbb P u_0}_{L^{p_0}(\R^3)}\leqslant C_*\Big(d_n^{\frac{3}{2}(1/p-1/p_0)} R^{3(1/p-1/p_0)}\normb{u_0}_{L^{p_0}(\R^3)}+\normb{\mathbb 1_{B_R^c} u_0}_{L^{p_0}(\R^3)}\Big).
\]
Choosing $R= d_n^{-\frac{1}{4}}$, it is immediate that the first term converges to zero as $n\to \infty$; the same holds for the second term by the dominated convergence theorem.
In the following will now  prove 
\begin{gather}\label{sec4:eq7}
\norm{\mathscr{Z}_{n}*\dvg \mathbb P\big( \omega_n*\Pi_0- \Pi_0\big)}_{L^{p_0}(\R^3)}\xrightarrow{n\to \infty} 0.
\end{gather}

Following the computations that led to \eqref{sec4:eq6}, we obtain \eqref{sec4:eq7}
provided that 
\[
d_n^{-\frac{1}{2}} \norm{\omega_n*\Pi_0-\Pi_0}_{L^{p_0}(\R^3)}\xrightarrow{n\to \infty} 0.
\]
However, this would require assuming $\Pi_0\in L^{p_0}(\R^3)$, and hence  $\nabla u_0\in L^{p_0}(\R^3)$ which we aim to avoid. To this end, we express 
\begin{align*}
\Pi_0&=2\big(\mu(\rho_0,c_0)-\widetilde\mu\big)\D u_0+\big(\lambda(\rho_0,c_0)-\widetilde\lambda\big)\dvg u_0+\widetilde\mu\Delta u_0+\big(2\widetilde\mu+\widetilde\lambda\big)\nabla\dvg u_0-P(\rho_0,c_0)+\widetilde P\\
&=: \Pi_{0,1}+\widetilde\mu\Delta u_0+\big(2\widetilde\mu+\widetilde\lambda\big)\nabla\dvg u_0,
\end{align*}
and assuming that 
\[
d_n^{-\frac{1}{2}} \norm{\omega_n*\Pi_{0,1}-\Pi_{0,1}}_{L^{p_0}(\R^3)}\xrightarrow{n \to 0} 0, \;\text{ we readily obtain } \;
\norm{\mathscr{Z}_{n}*\dvg \mathbb P\big( \omega_n*\Pi_{0,1}- \Pi_{0,1}\big)}_{L^{p_0}(\R^3)}\xrightarrow{n\to \infty} 0.
\]
Then we write the remaining term as 
\[
-\widetilde\mu\mathscr{F}^{-1}\Bigg(\xi \mapsto \frac{\abs{\xi}^2}{d_n+\widetilde\mu \abs{\xi}^2}\mathscr{F}\mathbb P\big( \omega_n* u_0-u_0\big)\Bigg)-\big(2\widetilde\mu+\widetilde\lambda\big)\mathscr{F}^{-1}\Bigg(\xi \mapsto \frac{\xi }{d_n+\widetilde\mu \abs{\xi}^2}\xi\cdot \mathscr{F}\mathbb P\big( \omega_n* u_0-u_0\big)\Bigg),
\]
and a Fourier multiplier theorem (see e.g. Theorem 6.2.7 of \cite{GrafakosClassical}) yields
\begin{align*}
    &\norm{\widetilde\mu\mathscr{F}^{-1}\Bigg(\xi \mapsto \frac{\abs{\xi}^2}{d_n+\widetilde\mu \abs{\xi}^2}\mathscr{F}\mathbb P\big( \omega_n* u_0-u_0\big)\Bigg)+\big(2\widetilde\mu+\widetilde\lambda\big)\mathscr{F}^{-1}\Bigg(\xi \mapsto \frac{\xi }{d_n+\widetilde\mu \abs{\xi}^2}\xi\cdot \mathscr{F}\mathbb P\big( \omega_n* u_0-u_0\big)\Bigg)}_{L^{p_0}(\R^3)}\\
    &\leqslant  C_*\norm{\omega_n* u_0 - u_0}_{L^{p_0}(\R^3)}\xrightarrow{n\to  \infty} 0.
\end{align*}
Putting everything together, we obtain 
\[
\normb{\mathbb P\big(u_0^n-u_0\big)}_{L^{p_0}(\R^3)}\xrightarrow{n\to \infty} 0.
\]
Following the same computations, we also get 
\[
\normb{\big(\mathbb I-\mathbb P\big)\big(u_0^n-u_0\big)}_{L^{p_0}(\R^3)}\xrightarrow{n\to \infty} 0.
\]
This completes the proof of Lemma \ref{sec4:lemma:1}. In addition, from \eqref{sec4:eq2} one readily deduces the strong convergence 
\[
\dvg \big(2\mu(\rho_0,c_0)\D u_0^n\big)+\nabla\big(\lambda(\rho_0,c_0)\dvg u_0^n\big)-\nabla P(\rho_0,c_0)\xrightarrow{n\to \infty} \dvg\big(\Pi_0\big),
\]
 in $\dot H^{-\check\alpha}(\R^3)$, as soon as $\dvg\big(\Pi_0\big)\in \dot H^{-\check\alpha}(\R^3)$.
\end{proof}
\subsection{The bootstrap argument and final steps}\label{part4:sec2}
This section is devoted to closing all the a priori estimates derived in Section \ref{part2}.  Since the initial density 
is bounded away from zero and, in view of the assumptions on the shear viscosity coefficient $\mu=\mu(\rho,c)$,
we have the following bounds 
\begin{equation*}
    0<\rho_{0,*}=\inf_{\R^3}\rho_0\leqslant \sup_{\R^3}\rho_0=\rho^*_0<\infty;
\end{equation*}
\begin{gather*}
    0<\mu_{0,*}:=\inf_{\R^3}\mu(\rho_0,c_0)\leqslant \sup_{\R^3}\mu(\rho_0,c_0)=\mu^*_0<\infty.
    \end{gather*}
Let us define 
\begin{align*}
    L_0= \abs{ f(\rho_0,c_0)}_{\dot \cC^{\alpha}_{\pw,\Sigma_0}(\R^3)}+\norm{\Sigma_0}_{\text{Lip}}+\abs{\Sigma_0}_{\dot \cC^{1+\alpha}}+ \abs{\nabla\vph_0}_{\text{inf}}^{-1}+\abs{\Sigma_0}_{\text{inf}}^{-1},
\end{align*}
and let us consider the constants $\eps_1$, and $\eps_2$  given in Theorem \ref{sec2:thlocal} below, for $\alpha'=\alpha/2$. It is obvious that 
there exists \\
$\delta_0=\delta_0(\mu_{0,*},\mu^*_0,L_0)>0$ such that if 
\[
 \norm{\llbracket\mu(\rho_0,c_0) \rrbracket }_{L^\infty(\Sigma_0)}\leqslant \delta_0, 
\]
 then 
 \[
 \frac{1}{\mu_{0,*}}  \norm{\llbracket\mu(\rho_0,c_0) \rrbracket }_{L^\infty(\Sigma_0)}\leqslant \frac{\eps_1}{10},\quad\text{and} \quad \Psi\big(\Sigma_0, \mu(\rho_0,c_0), \nu(\rho_0,c_0)\big)\leqslant \frac{\eps_2}{10}. 
 \]
As a consequence, Theorem \ref{sec2:thlocal} ensures the existence of a unique solution $(c^n,\rho^n, u^n)$ to \eqref{sec4:eq3}, defined on a maximal time interval $[0,T_n)$.  On this time interval,  the solution has 
enough regularity such that the computations performed in the Section \ref{part2}-\ref{lemmasProof} make sense. We now 
consider the functionals $\vartheta$, $\check{\vartheta}$, $\delta$, $\mathcal{A}_{1}$, $\mathcal{A}_{2}$ and $\mathcal{A}_{3}$, 
 with $(c,\rho,u)$ replaced by $(c^n,\rho^n,u^n)$.  Let $T\in (0, T_n)$ be such that, for every $t\in [0,T]$ we have  
\begin{align*}
\begin{cases}
    \displaystyle 0<\frac{\rho_{0,*}}{10}\leqslant \inf_{x\in \R^3}\rho^n(t,x)\leqslant \sup_{x\in \R^3}\rho^n(t,x)\leqslant 10\rho_{0}^*,\\ 
    \vartheta(t)\leqslant L_1, \quad
    \check{\vartheta}(t)\leqslant L_2,\quad
    \delta (t)\leqslant 10\delta_0,\quad
    \mathcal{A}_{1}(t)\leqslant L_3;
\end{cases}
\end{align*}
where $L_1$, $L_2$, and $L_3$ will be fixed later. On the time interval $[0,T]$ the density and the viscosity are lower and upper bounded as in \eqref{sec2:eq29}. In particular, whenever $\delta_0$,  $E_0$ are below thresholds depending only 
on the upper and lower bounds of density (and viscosity), Lemma \ref{sec2:lem4} ensures that  
\[
\mathcal{A}_{1}(t)\leqslant C_*\Big(C_0\big(1+\check \phi_{\Sigma_n}^{\frac{6}{\alpha}} E_0\big)+ \norm{\nabla u_0}_{L^2(\R^3)}^2+E^{\frac{1}{4}}_0\big(\mathcal{A}_1(t)^{\frac{5}{4}}+C_0^{\frac{5}{4}}\big)\Big), 
\]
so that we can choose 
\[
L_3= 5 C_*\Big(C_0\big(1+\check \phi_{\Sigma_n}^{\frac{6}{\alpha}} E_0\big)+ \norm{\nabla u_0}_{L^2(\R^3)}^2\Big)
\]
and deduce that on the time intervale $[0,T]$, we have
\[
\mathcal{A}_1(t)\leqslant \frac{3}{5} L_3.
\]
Recalling the expression of $\check{\phi}_{\Sigma_n}$ from \eqref{sec2:eq27}, it follows that
\[
\check{\phi}_{\Sigma_n}\leqslant C_* L_2^{\frac{1}{2}}\big(1+ \delta_0 K_0e^{C' L_1}\big), 
\]
and we shall assume that
\[
C_* L_2^{\frac{1}{2}}\big(1+ \delta_0 K_0e^{C' L_1}\big) \big(C_0+\norm{\nabla u_0}_{L^2(\R^3)}^2\big)^\beta \leqslant C, \quad\text{ for some }  \beta\in (0, \alpha).
\]
This implies that 
\[
 \check{\phi}_{\Sigma_n}^{\frac{6}{\alpha}}\big(C_0+\norm{\nabla u_0}_{L^2(\R^3)}^2\big)\leqslant  C,
\]
and  Lemma \ref{sec2:lem5} provides us with 
\begin{gather*}
    \int_0^t\Big(\norm{\nabla u^n, \, G(\rho^n,c^n)}_{L^p(\R^3)}+\norm{\dot u^n}_{L^{6/(2-\alpha)}(\R^3)}\Big)\leqslant C_*L_3^{1/2}\big(1+ L_3\big),\; \text{ and }\;\int_0^t \norm{\rho^n\dot u^n}_{L^{3/(1-\alpha)}(\R^3)}\leqslant C_*\check{L}_3,
\end{gather*}
where  $\check{L}_3$ depends nonlinearly on $\rho_{*,0}$, $\rho_{0}^*$, $\check C_0$.  

Let us denote by $\varrho^n$ the density $\rho^n$ along the flow map $\mathcal{X}^n$ of $u^n$. It is straightforward to deduce from  the mass equation \eqref{sec4:eq3} that: 
\[
\dpt \log \big(\varrho^n/\widetilde\rho(c_0)\big) +\frac{1}{\nu(\varrho^n,c_0)}\frac{P(\varrho^n,c_0)-\widetilde P}{\log \big(\varrho^n/\widetilde\rho(c_0)\big)}\log \big( \varrho^n/\widetilde\rho(c_0)\big)=-\frac{1}{\nu(\varrho^n,c_0)} F^n\circ \mathcal{X}^n,
\]
and hence 
\[
\log \big( \varrho^n/ \widetilde\rho(c_0)\big)= e^{-\int_0^t g^n(t')dt'} \log \big( \rho_0/ \widetilde\rho(c_0)\big)-\int_0^t\frac{e^{-\int_{t'}^t g^n(t'')dt''}}{\nu(\varrho^n(t'),c_0)} (F^n\circ \mathcal{X}^n)(t')dt',
\]
where 
\[
g^n=\frac{1}{\nu(\varrho^n,c_0)}\frac{P(\varrho^n,c_0)-\widetilde P}{\log \big(\varrho^n/\widetilde\rho(c_0)\big)}.
\]
It follows that 
\[
\sup_{x\in \R^3}\rho^n(t,x)\leqslant \rho_0^*e^{C_*\int_0^t\norm{F^n}_{L^\infty(\R^3)}},\;\text{ and }\;  \inf_{x\in \R^3}\rho^n(t,x)\geqslant \rho_{0,*}e^{-C_*\int_0^t\norm{F^n}_{L^\infty(\R^3)}},
\]
where the $L^1((0,t),L^\infty(\R^3))$ norm of $F^n$ is estimated as 
\begin{align*}
\int_0^t\norm{F^n}_{L^\infty(\R^3)}&\leqslant C_*\int_0^t \big(\norm{\rho^n \dot u^n}_{L^2(\R^3)}+\norm{\rho^n \dot u^n}_{L^{6/(2-\alpha)}(\R^3)}+\phi_{\Sigma_n}\norm{\nabla u^n}_{L^{p}(\R^3)}\big)\\
&+ C_*\sup_{[0,t]}\norm{\llbracket \mu(\rho^n,c^n)\rrbracket}_{L^\infty(\Sigma_n)}\int_0^t\norm{\nabla u^n}_{\cC^{\alpha/2}_{\pw,\Sigma_n}(\R^3)}.
\end{align*}
From the expression of $\phi_{\Sigma_n}$ given \eqref{sec2:eq27}, we obtain 
\[
\int_0^t \norm{F^n}_{L^\infty(\R^3)}\leqslant  C_* L_3^{1/2}\big(1+ K_0 e^{C'L_1}L_2^{1/2}\big)\big(1+L_3\big)+ C_*\delta_0 L_1^{1/2}\big(L_1^{1/2}+ L_2^{1/2}\big).
\]
We assume that  $C_0$, $\norm{\nabla u_0}_{L^2(\R^3)}$ and $\delta_0$ are small so that the RHS above is less than $2\log (3)$. This implies 
\[
\sup_{x\in \R^3}\rho^n(t,x)\leqslant 9 \rho_0^*,\;\text{ and }\;  \inf_{x\in \R^3}\rho^n(t,x)\geqslant \frac{\rho_{0,*}}{9}.
\]
Under the assumption that 
\[
C_*\delta_0^{1/2} \big(\delta_0^{1/2} e^{C'L_1} +L_2^{1/2}\big)\leqslant 1, 
\]
\eqref{sec2:eq7} holds and and the bounds \eqref{sec2:pwholder}–\eqref{sec2:Lipschtz} follow. In particular, one can choose 
\[
L_2=3C_*\big( \check{\vartheta}(0)+\check{L}_3\big) e^{C_* K_0(L_1+L_3^{1/2}(1+ L_3)) e^{C L_1}} . 
\]
What remains is to close the estimates for $\vartheta$ and $\delta$. Recall that 
\begin{align*}
  \vartheta(t)&\leqslant e^{C_*\int_0^t\normb{w^n}_{L^\infty(\R^3)}}\Bigg\{C_*\abs{f(\rho_0,c_0)}_{\cC^{\alpha/2}_{\pw,\Sigma_0}(\R^3)}+C_*L_3^{1/2}\big(1+L_3\big)\big(1+ L_2\big)+C_*\delta_0^{1/2}\big(1+\delta_0^{1/2} L_2\big)K_0 e^{C' L_1}L_3^{1/2}\big(1+ L_3\big)\notag\\
&+C_*\delta_0^{1/2}L_2^{1/2}\big( L_2+\delta_0^{1/2} L_1\big(1+  K_0e^{C' L_1}\big)\big) + C_*\int_0^t \Big[\ell_{\Sigma_n}^{-\frac{\alpha}{2}}\norm{\llbracket G(\rho^n,c^n)\rrbracket}_{L^\infty\cap L^{p}(\Sigma_n)}
        +\abs{\llbracket G(\rho^n,c^n)\rrbracket}_{\dot \cC^{\alpha/2}(\Sigma_n)}\Big]\Bigg\},
\end{align*}
with 
\begin{align}\label{sec4:eq10}
   \sup_{[0,t]}\norm{f(\rho^n,c^n)}_{L^\infty(\R^3)}&+\int_0^t\norm{w^n}_{L^\infty(\R^3)}\notag\\
   &\leqslant C_*\Bigg(\norm{f(\rho_0,c_0)}_{L^\infty(\R^3)}+\big(1+L_2^{1/2} \big(1+ K_0 e^{C' L_1}\big)\big)L_3^{1/2}\big(1+L_3\big)+\delta_0L^{1/2}_1\big(L^{1/2}_1+L^{1/2}_2\big)\Bigg). 
\end{align}
From the rough bounds \eqref{sec2:eq22}–\eqref{sec2:eq50} for the pressure jumps, we derive
\begin{align}
    &\int_0^t \Big[\ell_{\Sigma_n}^{-\frac{\alpha}{2}}\norm{\llbracket G(\rho^n,c^n)\rrbracket}_{L^\infty\cap L^{p}(\Sigma_n)}
        +\abs{\llbracket G(\rho^n,c^n)\rrbracket}_{\dot \cC^{\alpha/2}(\Sigma_n)}\Big]\notag\\
        &\leqslant C_*e^{C' L_1}\Bigg(L_2^{1/2} K_0\norm{f(\rho_0,c_0)}_{L^\infty(\R^3)}+\abs{f(\rho_0,c_0)}_{\dot \cC^{\alpha/2}_{\pw,\Sigma_0}(\R^3)}+L_3^{1/2}\big(1+ L_3\big) \big(1+ L_2+ K_0e^{C' L_1}\big)\notag\\
        &+\big(1+L_2^{1/2} \big(1+ K_0 e^{C' L_1}\big)\big)L_3^{1/2}\big(1+L_3\big)+\delta_0^{1/2}L^{1/2}_1L_2^{1/2}\big(L^{1/2}_1+K_0 \delta_0^{1/2}\big(L_1 e^{C' L_1}+L^{1/2}_2\big)\big)\Bigg). \label{sec4:eq9}
\end{align}
One  can first choose 
\[
L_1= \Big[10^5 C_* \Big(\big( \check{\vartheta}(0)+\check{L}_3\big)^{1/2} K_0 \norm{f(\rho_0,c_0)}_{L^\infty(\R^3)}+\abs{f(\rho_0,c_0)}_{\dot \cC^{\alpha/2}_{\pw,\Sigma_0}(\R^3)}+ L_3^{1/4}\Big)\Big] e^{C_*\big(1+ L_3^{1/2}+\norm{f(\rho_0,c_0)}_{L^\infty(\R^3)}\big)},
\]
and then $L_3$ and $\delta_0$ small enough so that 
\[
\vartheta(t)\leqslant \Big(10^{-4} L_1\Big) e^{C_* K_0(L_1^{1/2}+L_3^{1/2}(1+ L_3)) e^{C L_1} }.
\]
Therefore, under the smallness assumption on  $\norm{f(\rho_0,c_0)}_{\cC^{\alpha/2}_{\pw,\Sigma_0}(\R^3)}$ and $L_3$, we have
\[
\vartheta(t)\leqslant \frac{9}{10} L_1. 
\]
Observe that when the constitutive coefficients satisfy \eqref{sec1:jumpcondition} or \eqref{sec1:jumpconditionbis}, we obtain  sharper bounds (see Lemma \ref{sec2:lem3}) for the pressure jump, which we shall use instead of \eqref{sec4:eq9}. This sharper bounds allow us to 
require smallness only on $\delta_0$ and $L_3$, rather than  $\norm{f(\rho_0,c_0)}_{\cC^{\alpha/2}_{\pw,\Sigma_0}(\R^3)}$. Furthermore, one readily proves that
\[
\delta(t)\leqslant 9 \delta_0,
\]
assuming that $\normb{\llbracket f(\rho_0,c_0)\rrbracket}_{\cC^{\alpha/2}(\Sigma_0)}$ (and also $\normb{\llbracket \mu(c_0)\rrbracket}_{\cC^{\alpha/2}(\Sigma_0)}=\text{cnst}$, if \eqref{sec1:jumpcondition} holds) is small enough.  If neither \eqref{sec1:jumpcondition} nor \eqref{sec1:jumpconditionbis} holds, we should 
make use of the rough bound \eqref{sec4:eq10} and deduce the above conclusion  assuming  that $L_1$ and $L_3$ are small.  A continuation argument yields $T=T_n$. 

To complete Theorem \ref{th1}, we shall show that $T_n=\infty$ using the continuation criterion  (see item 3 of Theorem \ref{sec2:thlocal}).
The convergence of the sequence $(c^n,\rho^n,u^n)_n$  to a solution $(c,\rho,u)$ of  \eqref{intro:twofluid} follows by arguments 
analogous those of \cite[Section 3.2]{zodji2023discontinuous}. The solution $(c,\rho,u)$ has the regularity stated in Theorem \ref{th1}. In particular the velocity $u$ is Lipschitz, which allows a change of variables to Lagrangian coordinates and, by adapting Proposition 3.1 of the above reference, yields the uniqueness of $(c,\rho,u)$. Let us observe that we will need  $\sigma \nabla u\in L^2_\loc ([0,\infty), L^\infty(\R^3))$. This improved time integrability follows by repeating the computations leading to \eqref{sec2:Lipschtz}. 
This proves Theorem \ref{th1}.

\section*{Acknowledgment}
This work  was carried out during my PhD studies; and I thank my supervisors Cosmin Burtea and David Gérard-Varet, for fruitful discussions and careful reading of this manuscript. 

I was funded by the European 
Union’s Horizon 2020 research and innovation program under the Marie Skłodowska-Curie grant agreement No 945332.

\appendix
\section{Local-in-time well-posedness}\label{appendix:sec1}
This section aims to outline how the smallness condition on the shear viscosity fluctuation in Theorem 1.1 and Corollary 1.5 of \cite{zodji2023well} can be relaxed to a smallness condition on the viscosity jump, as stated in Theorem \ref{sec2:thlocal} below. 

\begin{theorem}\label{sec2:thlocal} Consider the system \eqref{intro:twofluid} with initial data $(c_0,\rho_0,u_0)$ satisfying \eqref{initialvolum}-\eqref{init:density-velocity}-\eqref{density:piecewise}-\eqref{sec2:compa}, and let $\alpha'\in (0,\alpha)$. There exist positive constants 
$\eps_1$, and $\eps_2$ depending only on $\alpha$ and $\alpha'$ such that if 
\begin{gather} \label{ap:eq54}
\frac{1}{\inf_{\R^3}\mu(\rho_0,c_0)} \norm{\llbracket \mu(\rho_0,c_0)\rrbracket}_{L^\infty(\Sigma_0)}\leqslant \eps_1,
\end{gather}
and \footnote{Recall that  $\mathfrak{P}_{\Sigma}$ is a polynomial in $\abs{\Sigma}$, $\norm{\Sigma}_{\text{Lip}}$, and $\abs{\Sigma}_{\text{inf}}^{-1}$.}
\begin{align}
    \Psi\big(\Sigma_0, \mu(\rho_0,c_0), \nu(\rho_0,c_0)\big)&:= \frac{1}{\inf_{\R^3}\mu(\rho_0,c_0)}\Big(\abs{\llbracket \mu(\rho_0,c_0) \rrbracket }_{\dot \cC^{\alpha'}(\Sigma_0)}+ \big(\ell^{-\alpha'}_{\Sigma_0}+\mathfrak{P}_{\Sigma_0} \abs{\Sigma_0}_{\dot \cC^{1+\alpha'}}\big)\norm{\llbracket \mu(\rho_0,c_0)\rrbracket}_{L^\infty(\Sigma_0)}\Big)\notag\\
    &+\Bigg(\frac{1}{\inf_{\R^3} \mu(\rho_0,c_0)}\norm{\llbracket \mu(\rho_0,c_0)\rrbracket}_{L^\infty(\Sigma_0)}\Bigg)^{\frac{\alpha}{\alpha'}}\Bigg\{\frac{1}{\inf_{\R^3}\mu(\rho_0,c_0)}\Bigg[\abs{\mu(\rho_0,c_0)}_{\dot \cC^{\alpha}_{\pw,\Sigma_0}(\R^3)}\notag\\
    &+ \norm{\llbracket \mu(\rho_0,c_0)\rrbracket}_{L^\infty(\Sigma_0)}\Bigg(\frac{1}{\inf_{\R^3}\mu(\rho_0,c_0)}\abs{\mu(\rho_0,c_0),\,\nu(\rho_0,c_0)}_{\dot \cC^\alpha_{\pw,\Sigma_0}(\R^3)}+\mathfrak{P}_{\Sigma_0} \big|\Sigma_0\big|_{\dot \cC^{1+\alpha}}+\ell^{-\alpha}_{\Sigma_0}\Bigg)\Bigg]\notag\\
    &+\Bigg(\frac{1}{\inf_{\R^3}\mu(\rho_0,c_0)}\Bigg)^{\frac{\alpha}{\alpha-\alpha'}}\Bigg[\abs{ \mu(\rho_0,c_0)}_{\dot \cC^{\alpha-\alpha'}_{\pw,\Sigma_0}(\R^3)}+\norm{\llbracket \mu(\rho_0,c_0)\rrbracket}_{L^\infty(\Sigma_0)}\notag\\
    &\times \Big(\frac{1}{\inf_{\R^3}\mu(\rho_0,c_0)}\abs{\mu(\rho_0,c_0),\,\nu(\rho_0,c_0)}_{\dot \cC^\alpha_{\pw,\Sigma_0}(\R^3)}+\mathfrak{P}_{\Sigma_0} \big|\Sigma_0\big|_{\dot \cC^{1+\alpha}}+ \ell^{-\alpha+\alpha'}_{\Sigma_0}\Big)\Bigg]^{\frac{\alpha}{\alpha-\alpha'}}\Bigg\}\notag\\
    &<\eps_2,\label{ap:eq53}
\end{align}
then there exist a time $T > 0$ and a unique solution $(c, \rho, u)$ satisfying the following properties.
\begin{enumerate}
    \item \textbf{Energy bounds}: 
    \begin{align*}
        &u\in \cC([0,T], H^1(\R^3)),\;\, \dot u\in \cC([0,T], L^2(\R^3));\\
        &\sqrt{\sigma}\nabla \dot u, \sigma \ddot u\in L^\infty((0,T), L^2(\R^3));\\
        &\nabla \dot u, \sqrt{\sigma}\ddot u, \sigma \nabla \dot u\in L^2((0,T)\times  \R^3).
    \end{align*}

    \item \textbf{Piecewise H\"older bounds}: There exists $r_\alpha \in (0, 8/3)$ such that 
    \[
    \sigma^{r_\alpha/4} \nabla u \in L^4((0, T); L^\infty(\R^3)).
    \]
   As a consequence, the velocity field $u$ admits a flow map $\mathscr{X}$ that transports the initial domain $D_0$ to $D(t) = \mathscr{X}(t) D_0$, and $c(t)=\mathbb{1}_{D(t)}$, $t \in [0, T]$. Moreover, we have: for all $t\in [0,T]$,
    \[
    \rho(t)\in \cC^\alpha_{\pw,\partial D(t)}(\R^3),\; \text{ and }\;\int_0^{T}\sigma^{r_\alpha}\norm{\nabla u(t')}_{\dot \cC^\alpha_{\pw,\partial D(t')}(\R^3)}^4dt'<\infty.
    \]
    \item \textbf{Continuation criterion}: Let $T^*$ be the lifespan of $(c,\rho,u)$. If $T^*<\infty$, then we have either
    \begin{equation*}
        \limsup_{t\to T^*}\Bigg(\norm{\rho(t),\,\frac{1}{\rho(t)}, \frac{1}{\mu(\rho(t),c(t))}}_{L^\infty(\R^3)}+\norm{\rho(t)}_{\dot \cC^\alpha_{\pw,\partial D(t)}(\R^3)}+\norm{\nabla u(t), \dot u(t)}_{L^2(\R^3)}\Bigg)=\infty,
    \end{equation*}
    or
    \begin{equation*}
        \limsup_{t\to T^*}\Bigg( \norm{\partial D(t)}_{\text{Lip}}+\abs{\partial D(t)}_{\dot \cC^{1+\alpha}}+\frac{1}{\abs{\partial D(t)}_{\text{inf}}}\Bigg)=\infty,
    \end{equation*}
    or 
    \begin{equation*}
        \limsup_{t\to T^*}\Bigg(\frac{1}{\inf_{\R^3}\mu(\rho(t),c(t))} \norm{\llbracket \mu(\rho(t),c(t))\rrbracket}_{L^\infty(\partial D(t))}\Bigg)\geqslant 10 \eps_1,
    \end{equation*}
    or 
    \begin{equation*}
        \limsup_{t\to T^*} \Psi\big(\partial D(t), \mu(\rho(t),c(t)), \nu(\rho(t),c(t))\big)\geqslant 10\eps_2.
    \end{equation*}
\end{enumerate}
\end{theorem}

As pointed out in Remark 1.2 of \cite{zodji2023well}, the smallness assumption of the shear viscosity fluctuation 
is imposed to derive an estimate for the $L^p(\R^3)\cap \cC^\alpha_{\pw,\Sigma_0}(\R^3)$-norm (with $p\in (2,16]$)  for the gradient 
of the solution $v=v(x)\in \R^3$ of the following (stationary) elliptic equation:
\begin{equation} \label{ap:eq2}
    -\dvg \big(2\mu\D v+\lambda\dvg v I_3\big)=\dvg G.
\end{equation}
Here $\mu=\mu(x);\lambda=\lambda(x)\in \cC^\alpha_{\pw,\Sigma_0}(\R^3)$, and  $G=\big(G^{jk}(x)\big)_{1\leqslant j,k\leqslant 3}\in L^p(\R^3)\cap \cC^\alpha_{\pw,\Sigma_0}(\R^3)$ are given. The argument was to perturb the shear viscosity around the constant state $\widetilde\mu$ 
\begin{equation*}
-\widetilde\mu \Delta v- \nabla \big((\widetilde\mu+\lambda)\dvg v\big)=\dvg G+\dvg \big( 2(\mu-\widetilde\mu )\D v\big). 
\end{equation*}
From this we deduce  
\begin{equation}\label{ap:eq4}
\begin{cases}
-\widetilde\mu \Delta\mathbb{P} v=  \dvg \mathbb{P} G   +\dvg\mathbb{P}\big(2(\mu-\widetilde\mu )\D v\big),\\
-\Delta\big((2\widetilde\mu+\lambda)\dvg v\big)=\dvg \dvg G +\dvg \dvg \big(2(\mu-\widetilde\mu )\D v\big),
\end{cases}
\end{equation}
and hence
\begin{align}
    \widetilde\mu \nabla v&= \widetilde\mu \nabla \mathbb{P} v+\widetilde\mu \nabla \big( \mathbb{I}-\mathbb{P}\big) v\notag\\
    &=\nabla (-\Delta)^{-1} \dvg \mathbb{P} G-\widetilde\mu \nabla(-\Delta)^{-1}\nabla\dvg v+\nabla (-\Delta)^{-1}\dvg \mathbb{P}\big(2(\mu-\widetilde\mu )\D v\big)\notag\\
    &=\nabla (-\Delta)^{-1}\dvg \mathbb{P} G-\nabla (-\Delta)^{-1}\nabla\Big(\frac{\widetilde\mu}{2\widetilde\mu+\lambda}(-\Delta)^{-1}\dvg \dvg G\Big)\notag\\
    &+\nabla (-\Delta)^{-1}\dvg\mathbb{P}\big(2(\mu-\widetilde\mu )\D v\big)-\nabla (-\Delta)^{-1}\nabla\Big(\frac{\widetilde\mu}{2\widetilde\mu+\lambda}(-\Delta)^{-1}\dvg \dvg \big(2(\mu-\widetilde\mu )\D v\big)\Big).\label{ap:eq1}
\end{align}
Since the linear operators $\partial_j (-\Delta)^{-1}\partial_k$
are continuous on $L^p(\R^3)$, with $p\in (1,\infty)$, it follows that there exists a constant $C(p)>0$ such that 
\[
\widetilde\mu \norm{\nabla v}_{L^p(\R^3)}\leqslant C(p)\norm{G}_{L^p(\R^3)}+  C (p)\norm{\mu-\widetilde\mu}_{L^\infty(\R^3)}\norm{\nabla v}_{L^p(\R^3)}.
\]
Then we assume that the shear viscosity fluctuation is small
\begin{equation}\label{ap:eq25}
C(p)\norm{\mu-\widetilde\mu }_{L^\infty(\R^3)}< \widetilde\mu,
\end{equation}
in order to absorb the last term on the LHS, and whence derive an estimate for the $L^p(\R^3)$-norm of $\nabla v$. By a similar (yet more technical) perturbation argument, we derived piecewise H\"older bound for the velocity gradient. 

To pass from the smallness condition of the viscosity fluctuation to that of its jump across the discontinuity surface, we first derive the following expression from $\eqref{ap:eq4}_1$ (see the computations that led to \eqref{sec2:eq3}):
\begin{align}
    \nabla v&=\nabla\mathbb{P} v-\nabla (-\Delta)^{-1}\nabla \dvg v\notag\\
    &= \frac{1}{\mu}\nabla (-\Delta)^{-1}\dvg \mathbb{P} G-\nabla (-\Delta)^{-1}\nabla\bigg(\frac{1}{2\mu+\lambda}(-\Delta)^{-1}\dvg \dvg G\bigg)\notag\\
    &+\frac{1}{\mu}\big[\nabla(-\Delta)^{-1}\dvg \mathbb{P},\mu-\widetilde\mu\big]\big(2\D v\big)-\nabla (-\Delta)^{-1}\nabla\bigg( \frac{1}{2\mu+\lambda}\big[(-\Delta)^{-1}\dvg \dvg, \mu-\widetilde\mu\big] \big(2\D v\big)\bigg).\label{ap:eq13}
\end{align}
We estimate the $L^p(\R^3)$ and piecewise H\"older norms of the last two terms above using the following lemma, whose proof is deferred to the end of this section.
\begin{lemma}\label{ap:commutator} Let $\mathcal{T}$ be a singular integral operator of Calder\'on-Zygmung type of even-order, associated with a homogeneous kernel $K$. Assume that the assumptions and notations in Definition \ref{def:interface} hold for  $\Sigma$.
Let $p\in (1,\infty)$, $w,g\in \cC^\alpha_{\pw,\Sigma}(\R^3)$, and $g\in L^p(\R^3)$.
    \begin{enumerate}
        \item $L^p(\R^3)$- \textbf{\upshape estimate}. Let $\alpha'\in (0,\alpha] \cap (0,\frac{3}{p'})$ ($p'$ denotes the H\"older  conjugate exponent of $p$), and assume that $g\in L^r(\R^3)$, with $\frac{1}{r}=\frac{\alpha'}{3}+\frac{1}{p}<1$. We have 
        \begin{equation}\label{ap:eq6}
            \norm{\big[\mathcal{T}, w\big]g}_{L^p(\R^3)}\leqslant C\abs{w}_{\dot \cC^{\alpha'}_{\pw,\Sigma}(\R^3)}\norm{g}_{L^r(\R^3)}+ C\norm{\llbracket w\rrbracket}_{L^\infty(\Sigma)}\big(\ell^{-\alpha'}_{\Sigma}\norm{g}_{L^r(\R^3)}+ \norm{g}_{L^p(\R^3)}\big).
        \end{equation}
        \item $L^\infty(\R^3)$- \textbf{\upshape estimate}. Let $\alpha'\in (0,\alpha]$, $r\in (\frac{3}{\alpha'},\infty]$, and assume that $g\in L^r(\R^3)$. We have  
        \begin{align}\label{ap:eq10}
            \norm{\big[\mathcal{T}, w\big] g}_{L^\infty(\R^3)}&\leqslant  C\Big(\abs{ w}_{\dot \cC^{\alpha'}_{\pw,\Sigma}(\R^3)}+\ell^{-\alpha'}_{\Sigma}\norm{\llbracket w\rrbracket }_{L^\infty(\Sigma)}\Big)\norm{g}_{L^r(\R^3)}\notag\\
    &+C\Big(\norm{\llbracket w\rrbracket}_{L^\infty(\Sigma)}\norm{g}_{\cC^{\alpha'}_{\pw,\Sigma}(\R^3)}+\ell_{\Sigma}^{-\alpha'}\norm{ w}_{L^\infty(\R^3)}\norm{g}_{L^p(\R^3)}\Big).
        \end{align}
        \item $\dot \cC^\beta_{\pw,\Sigma}(\R^3)$- \textbf{\upshape estimate}. Let $\beta\in (0,\alpha]$ and $\alpha'\in (0,\min\{\alpha,\beta\}]$. Then, we have
        \begin{align}\label{ap:eq19}
            \abs{\big[\mathcal{T}, w\big]g}_{\dot \cC^\beta_{\pw,\Sigma}(\R^3)}&\leqslant C\Big(\abs{ w}_{\dot \cC^{\beta}_{\pw,\Sigma}(\R^3)}+ \big(\ell^{-\beta}_{\Sigma}+\mathfrak{P}_{\Sigma} \abs{\Sigma}_{\dot \cC^{1+\beta}}\big)\norm{\llbracket w\rrbracket}_{L^\infty(\Sigma)}\Big)\norm{\mathcal{T} g, g}_{L^\infty(\R^3)}\notag\\
    &+C\Big( \abs{ w }_{\dot \cC^{\beta-\alpha'}_{\pw,\Sigma}(\R^3)}+ \ell^{-\beta+\alpha'}_{\Sigma}\norm{\llbracket w\rrbracket}_{L^\infty(\Sigma)}\Big)\abs{g}_{\dot \cC^{\alpha'}_{\pw,\Sigma}(\R^3)}+C\norm{\llbracket w\rrbracket}_{L^\infty(\Sigma)}\abs{g}_{\dot \cC^\beta_{\pw,\Sigma}(\R^3)}.
        \end{align}
        Above, $\mathfrak{P}_\Sigma$ is a polynomial in $\abs{\Sigma}$, $\norm{\Sigma}_{\text{Lip}}$, and $\abs{\Sigma}_{\text{inf}}^{-1}$.
    \end{enumerate}
    The constant $C$ is independent of $g$, $w$ and $\Sigma$.
\end{lemma}
 The proof of Lemma \ref{ap:commutator} combines, in a nontrivial manner, the extension theorem for piecewise H\"older regular functions (see \cite{hoff2002dynamics}), and the continuity of Calder\'on-Zygmung operators on $L^q(\R^3)\cap \cC^\beta_{\pw,\Sigma}(\R^3)$, with $q\in (1,\infty)$ (see \cite{gancedo2022quantitative,gancedo2023global}).

 We now apply it to the last terms of \eqref{ap:eq13}.
\paragraph{\textbf{Estimate of the $L^p(\R^3)$ norm}}
 We consider $p\in (2,\infty)$, and  $r\in [2, p)$ such that $\frac{1}{r}=\frac{\alpha'}{3}+\frac{1}{p}$, for some $\alpha'\in (0,\alpha]$. 
By applying \eqref{ap:eq6} to the last terms in \eqref{ap:eq13}, we obtain
\begin{align}
    &\norm{\frac{1}{\mu}\big[\nabla(-\Delta)^{-1}\dvg \mathbb{P},\mu-\widetilde\mu\big]\big(2\D v\big);\nabla (-\Delta)^{-1}\nabla\bigg( \frac{1}{2\mu+\lambda}\big[(-\Delta)^{-1}\dvg \dvg, \mu-\widetilde\mu\big] \big(2\D v\big)\bigg)}_{L^p(\R^3)}\notag\\
    &\leqslant \frac{C}{\mu_*}\norm{\llbracket \mu \rrbracket}_{L^\infty(\Sigma_0)} \norm{\nabla v}_{L^p(\R^3)}+C_*\norm{\nabla v}_{L^r(\R^3)}\Big(\abs{\mu}_{\dot \cC^{\alpha'}_{\pw,\Sigma_0}(\R^3)}+ \norm{\llbracket \mu \rrbracket}_{L^\infty(\Sigma_0)}\ell^{-\alpha'}_{\Sigma_0}\Big),
\end{align}
and it follows that 
\begin{align}\label{ap:eq29}
\norm{\nabla v}_{L^p(\R^3)}&\leqslant \frac{C}{\mu_*}\norm{\llbracket \mu \rrbracket}_{L^\infty(\Sigma_0)} \norm{\nabla v}_{L^p(\R^3)}
\notag\\
&+C_*\bigg(\norm{F}_{L^p(\R^3)}+\norm{\nabla v}_{L^r(\R^3)}\Big(\abs{\mu}_{\dot \cC^{\alpha'}_{\pw,\Sigma_0}(\R^3)}+ \norm{\llbracket \mu \rrbracket}_{L^\infty(\Sigma_0)}\ell^{-\alpha'}_{\Sigma_0}\Big)\bigg).
\end{align}
Let us fix an arbitrary $\varepsilon\in (0,1)$ and assume that 
\begin{equation}\label{ap:eq32}
    \frac{C}{\mu_*} \norm{\llbracket \mu \rrbracket}_{L^\infty(\Sigma_0)} \leqslant \varepsilon.
\end{equation}
The first term on the RHS of \eqref{ap:eq29} can then be absorbed into the LHS, yielding:
\begin{equation*}
\norm{\nabla v}_{L^p(\R^3)}\leqslant \frac{C_*}{1-\varepsilon}\bigg(\norm{F}_{L^p(\R^3)}+\norm{\nabla v}_{L^r(\R^3)}\Big(\abs{\mu}_{\dot \cC^{\alpha'}_{\pw,\Sigma_0}(\R^3)}+ \norm{\llbracket \mu \rrbracket}_{L^\infty(\Sigma_0)}\ell^{-\alpha'}_{\Sigma_0}\Big)\bigg).
\end{equation*}
An energy estimate on \eqref{ap:eq2} yields
\[
\norm{\nabla v}_{L^2(\R^3)}\leqslant C_*\norm{F}_{L^2(\R^3)},
\]
and interpolation inequality gives
\begin{equation*}
\norm{\nabla v}_{L^r(\R^3)}\leqslant C\norm{\nabla u}_{L^p(\R^3)}^{1-\theta}\norm{\nabla v}_{L^2(\R^3)}^{\theta}\\
\leqslant C_*\norm{\nabla u}_{L^p(\R^3)}^{1-\theta}\norm{F}_{L^2(\R^3)}^\theta,\text{ with}\quad \frac{1}{2}-\frac{1}{p}=\frac{\alpha'}{3\theta}.
\end{equation*}
Finally,  applying Young’s inequality, we obtain: 
\begin{equation}\label{ap:eq31}
\norm{\nabla v}_{L^p(\R^3)}\leqslant \frac{C_*}{1-\varepsilon}\norm{F}_{L^p(\R^3)}+C_*(1-\varepsilon)^{-\frac{1}{\theta}}\norm{F}_{L^2(\R^3)}\Big(\abs{\mu}_{\dot \cC^{\alpha'}_{\pw,\Sigma_0}(\R^3)}+ \norm{\llbracket \mu \rrbracket}_{L^\infty(\Sigma_0)}\ell^{-\alpha'}_{\Sigma_0}\Big)^{\frac{1}{\theta}}.
\end{equation}
We emphasize that only the smallness of the shear viscosity jump is required (see \eqref{ap:eq32}), rather than that of the 
overall shear viscosity fluctuation (see \eqref{ap:eq25}). We now turn to deriving an estimate for the piecewise H\"older norm of $\nabla v$.
\paragraph{\textbf{Estimate of the $\cC^\alpha_{\pw,\Sigma_0}(\R^3)$ norm}}
We fix some $\alpha'\in (0,\alpha)$ and $r\in (3/\alpha',\infty)$. First, we apply \eqref{ap:eq10} and \eqref{ap:eq19}  to obtain
\begin{align}
    \norm{\big[\nabla(-\Delta)^{-1}\dvg \mathbb{P},\mu-\widetilde\mu\big]\big(2\D v\big)}_{L^\infty(\R^3)}
    &\leqslant C\Big(\abs{\mu}_{\dot \cC^{\alpha'}_{\pw,\Sigma_0}(\R^3)}+\ell^{-\alpha'}_{\Sigma_0}\norm{\mu-\widetilde\mu}_{L^\infty(\R^3)}\Big)\norm{\nabla v}_{L^r(\R^3)} \notag\\
    &+C\norm{\llbracket \mu\rrbracket}_{L^\infty(\Sigma_0)}\norm{\nabla v}_{\cC^{\alpha'}_{\pw,\Sigma_0}(\R^3)},\label{ap:eq35}
\end{align}
and 
\begin{align}
    &\abs{\big[\nabla(-\Delta)^{-1}\dvg \mathbb{P},\mu-\widetilde\mu\big]\big(2\D v\big)}_{\dot \cC^{\alpha}_{\pw,\Sigma_0}(\R^3)}\notag\\
    &\leqslant C\Big(\abs{\mu }_{\dot \cC^{\alpha}_{\pw,\Sigma_0}(\R^3)}+ \big(\ell^{-\alpha}_{\Sigma_0}+\mathfrak{P}_{\Sigma_0} \abs{\Sigma_0}_{\dot \cC^{1+\alpha}}\big)\norm{\llbracket \mu\rrbracket}_{L^\infty(\Sigma_0)}\Big)\norm{\nabla(-\Delta)^{-1}\dvg \mathbb{P} \D v,\, \D  v}_{L^\infty(\R^3)}\notag\\
    &+C\Big( \abs{ \mu }_{\dot \cC^{\alpha-\alpha'}_{\pw,\Sigma_0}(\R^3)}+ \ell^{-\alpha+\alpha'}_{\Sigma_0}\norm{\llbracket \mu \rrbracket}_{L^\infty(\Sigma_0)}\Big)\abs{\nabla v}_{\dot \cC^{\alpha'}_{\pw,\Sigma_0}(\R^3)}+C\norm{\llbracket \mu \rrbracket}_{L^\infty(\Sigma_0)}\abs{\nabla v}_{\dot \cC^\alpha_{\pw,\Sigma_0}(\R^3)}.\label{ap:eq36}
\end{align}

Next, we estimate the second-to-last term in \eqref{ap:eq13} as follows:
\begin{align}
    &\abs{\frac{1}{\mu}\big[\nabla(-\Delta)^{-1}\dvg \mathbb{P},\mu-\widetilde\mu\big]\big(2\D v\big)}_{\dot \cC^{\alpha}_{\pw,\Sigma_0}(\R^3)}\notag\\
    &\leqslant \dfrac{1}{\mu_*^2}\abs{\mu}_{\dot \cC^\alpha_{\pw,\Sigma_0}(\R^3)}\norm{\big[\nabla(-\Delta)^{-1}\dvg \mathbb{P},\mu-\widetilde\mu\big]\big(2\D v\big)}_{L^\infty(\R^3)}+\frac{1}{\mu_*}\abs{\big[\nabla(-\Delta)^{-1}\dvg \mathbb{P},\mu-\widetilde\mu\big]\big(2\D v\big)}_{\dot \cC^{\alpha}_{\pw,\Sigma_0}(\R^3)}\notag\\
    &\leqslant \dfrac{C}{\mu_*^2}\abs{\mu}_{\dot \cC^\alpha_{\pw,\Sigma_0}(\R^3)}\norm{\nabla v}_{L^r(\R^3)}\Big(\abs{\mu}_{\dot \cC^{\alpha'}_{\pw,\Sigma_0}(\R^3)}+\ell^{-\alpha'}_{\Sigma_0}\norm{\mu-\widetilde\mu}_{L^\infty(\R^3)}\Big)\notag\\
    &+\frac{C}{\mu_*}\norm{\nabla u,\, \nabla \mathbb P v}_{L^\infty(\R^3)}\Bigg[\abs{\mu}_{\dot \cC^\alpha_{\pw,\Sigma_0}(\R^3)}\Big(1+\frac{1}{\mu_*}\norm{\llbracket \mu\rrbracket}_{L^\infty(\Sigma_0)}\Big)+ \big(\ell^{-\alpha}_{\Sigma_0}+\mathfrak{P}_{\Sigma_0} \abs{\Sigma_0}_{\dot \cC^{1+\alpha}}\big)\norm{\llbracket \mu\rrbracket}_{L^\infty(\Sigma_0)}\Bigg]\notag\\
    &+\frac{C}{\mu_*}\norm{\nabla v}_{\dot \cC^{\alpha'}_{\pw,\Sigma_0}(\R^3)}\Bigg[\abs{ \mu }_{\dot \cC^{\alpha-\alpha'}_{\pw,\Sigma_0}(\R^3)}+\norm{\llbracket \mu \rrbracket}_{L^\infty(\Sigma_0)}\Big(\ell^{-\alpha+\alpha'}_{\Sigma_0}+\frac{1}{\mu_*}\abs{\mu}_{\dot \cC^\alpha_{\pw,\Sigma_0}(\R^3)}\Big)\Bigg]\notag\\
    &+\frac{C}{\mu_*}\norm{\llbracket \mu \rrbracket}_{L^\infty(\Sigma_0)}\abs{\nabla v}_{\dot \cC^\alpha_{\pw,\Sigma_0}(\R^3)}.\label{ap:eq33}
\end{align}
Applying \eqref{ap:eq23}, the last term in \eqref{ap:eq13} can be estimated as 
\begin{align}
    &\abs{\nabla (-\Delta)^{-1}\nabla\bigg( \frac{1}{2\mu+\lambda}\big[(-\Delta)^{-1}\dvg \dvg, \mu-\widetilde\mu\big] \big(2\D v\big)\bigg)}_{\dot \cC^\alpha_{\pw,\Sigma_0}(\R^3)}\notag\\
    &\leqslant \frac{C}{\mu_*}\Big(\frac{1}{\mu_*}\abs{2\mu+\lambda}_{\dot \cC^\alpha_{\pw,\Sigma_0}(\R^3)}+\mathfrak{P}_{\Sigma_0} \big|\Sigma_0\big|_{\dot \cC^{1+\alpha}}\Big)\norm{\big[(-\Delta)^{-1}\dvg\dvg ,\mu-\widetilde\mu\big](2\D v)}_{L^\infty(\R^3)}\notag\\
    &+\frac{C}{\mu_*}\abs{\big[(-\Delta)^{-1}\dvg \dvg, \mu-\widetilde\mu\big] \big(2\D v\big)}_{\dot \cC^\alpha_{\pw,\Sigma_0}(\R^3)},\notag
\end{align}
and hence (similar estimates to \eqref{ap:eq35}-\eqref{ap:eq36} also hold for $[(-\Delta)^{-1}\dvg \dvg, \mu-\widetilde\mu] (2\D v)$)
\begin{align}
     &\abs{\nabla (-\Delta)^{-1}\nabla\bigg( \frac{1}{2\mu+\lambda}\big[(-\Delta)^{-1}\dvg \dvg, \mu-\widetilde\mu\big] \big(2\D v\big)\bigg)}_{\dot \cC^\alpha_{\pw,\Sigma_0}(\R^3)}\notag\\
     &\leqslant \frac{C}{\mu_*}\norm{\nabla v}_{L^r(\R^3)}\Big(\frac{1}{\mu_*}\abs{\nu}_{\dot \cC^\alpha_{\pw,\Sigma_0}(\R^3)}+\mathfrak{P}_{\Sigma_0} \big|\Sigma_0\big|_{\dot \cC^{1+\alpha}}\Big)\Big(\abs{\mu}_{\dot \cC^{\alpha'}_{\pw,\Sigma_0}(\R^3)}+\ell^{-\alpha'}_{\Sigma_0}\norm{\mu-\widetilde\mu}_{L^\infty(\R^3)}\Big)\notag\\
     &+\frac{C}{\mu_*}\norm{\nabla v}_{L^\infty(\R^3)}\Bigg[\abs{\mu }_{\dot \cC^{\alpha}_{\pw,\Sigma_0}(\R^3)}+ \norm{\llbracket \mu\rrbracket}_{L^\infty(\Sigma_0)}\Big(\frac{1}{\mu_*}\abs{\nu}_{\dot \cC^\alpha_{\pw,\Sigma_0}(\R^3)}+\mathfrak{P}_{\Sigma_0} \big|\Sigma_0\big|_{\dot \cC^{1+\alpha}}+\ell^{-\alpha}_{\Sigma_0}\Big)\Bigg]\notag\\
     &+\frac{C}{\mu_*}\norm{\nabla v}_{\dot \cC^{\alpha'}_{\pw,\Sigma_0}(\R^3)}\Bigg[\abs{ \mu }_{\dot \cC^{\alpha-\alpha'}_{\pw,\Sigma_0}(\R^3)}+\norm{\llbracket \mu\rrbracket}_{L^\infty(\Sigma_0)}\Big(\frac{1}{\mu_*}\abs{\nu}_{\dot \cC^\alpha_{\pw,\Sigma_0}(\R^3)}+\mathfrak{P}_{\Sigma_0} \big|\Sigma_0\big|_{\dot \cC^{1+\alpha}}+ \ell^{-\alpha+\alpha'}_{\Sigma_0}\Big)\Bigg]\notag\\
     &+\frac{C}{\mu_*}\norm{\llbracket \mu \rrbracket}_{L^\infty(\Sigma_0)}\abs{\nabla v}_{\dot \cC^\alpha_{\pw,\Sigma_0}(\R^3)}.\label{ap:eq34}
\end{align}
Combining \eqref{ap:eq13}, \eqref{ap:eq33}, and \eqref{ap:eq34}, we deduce that 
\begin{align*}
    \abs{\nabla v}_{\dot \cC^\alpha_{\pw,\Sigma_0}(\R^3)}&\leqslant \abs{\nabla u_G}_{\dot \cC^\alpha_{\pw,\Sigma_0}(\R^3)}+ \frac{C}{\mu_*}\norm{\nabla v}_{L^r(\R^3)}\Big(\frac{1}{\mu_*}\abs{\mu,\,\nu}_{\dot \cC^\alpha_{\pw,\Sigma_0}(\R^3)}+\mathfrak{P}_{\Sigma_0} \big|\Sigma_0\big|_{\dot \cC^{1+\alpha}}\Big)\Big(\abs{\mu}_{\dot \cC^{\alpha'}_{\pw,\Sigma_0}(\R^3)}+\ell^{-\alpha'}_{\Sigma_0}\norm{\mu-\widetilde\mu}_{L^\infty(\R^3)}\Big)\notag\\
     &+\frac{C}{\mu_*}\norm{\nabla v,\, \nabla \mathbb P v}_{L^\infty(\R^3)}\Bigg[\abs{\mu }_{\dot \cC^{\alpha}_{\pw,\Sigma_0}(\R^3)}+ \norm{\llbracket \mu\rrbracket}_{L^\infty(\Sigma_0)}\Big(\frac{1}{\mu_*}\abs{\mu,\,\nu}_{\dot \cC^\alpha_{\pw,\Sigma_0}(\R^3)}+\mathfrak{P}_{\Sigma_0} \big|\Sigma_0\big|_{\dot \cC^{1+\alpha}}+\ell^{-\alpha}_{\Sigma_0}\Big)\Bigg]\notag\\
     &+\frac{C}{\mu_*}\norm{\nabla v}_{\dot \cC^{\alpha'}_{\pw,\Sigma_0}(\R^3)}\Bigg[\abs{ \mu }_{\dot \cC^{\alpha-\alpha'}_{\pw,\Sigma_0}(\R^3)}+\norm{\llbracket \mu\rrbracket}_{L^\infty(\Sigma_0)}\Big(\frac{1}{\mu_*}\abs{\mu,\,\nu}_{\dot \cC^\alpha_{\pw,\Sigma_0}(\R^3)}+\mathfrak{P}_{\Sigma_0} \big|\Sigma_0\big|_{\dot \cC^{1+\alpha}}+ \ell^{-\alpha+\alpha'}_{\Sigma_0}\Big)\Bigg]\notag\\
     &+\frac{C}{\mu_*}\norm{\llbracket \mu \rrbracket}_{L^\infty(\Sigma_0)}\abs{\nabla v}_{\dot \cC^\alpha_{\pw,\Sigma_0}(\R^3)}.
\end{align*}
Hereafter, $\nabla u_G$ denotes 
\[
\nabla u_G=\frac{1}{\mu}\nabla (-\Delta)^{-1}\dvg \mathbb{P} G-\nabla (-\Delta)^{-1}\nabla\bigg(\frac{1}{2\mu+\lambda}(-\Delta)^{-1}\dvg \dvg G\bigg).
\]
The smallness of the shear viscosity jump (see \eqref{ap:eq32}) helps absorb the last term above into the LHS. 
Next, we use the interpolation inequality
\begin{equation}\label{ap:eq48}
\abs{\nabla v}_{\dot \cC^{\alpha'}_{\pw,\Sigma_0}(\R^3)}\leqslant 2 \norm{\nabla v}_{L^\infty(\R^3}^{1-\frac{\alpha'}{\alpha}}\abs{\nabla v}_{\dot \cC^{\alpha}_{\pw,\Sigma_0}(\R^3)}^\frac{\alpha'}{\alpha}
\end{equation}
and Young's inequality to obtain
\begin{align*}
    \abs{\nabla v}_{\dot \cC^\alpha_{\pw,\Sigma_0}(\R^3)} &\leqslant  \frac{C}{1-\varepsilon}\abs{\nabla u_G}_{\dot \cC^\alpha_{\pw,\Sigma_0}(\R^3)}  +C\Psi_\eps\norm{\nabla v,\,\nabla \mathbb P v}_{L^\infty(\R^3)} \notag\\
    &+\dfrac{C}{\mu_*(1-\eps)}\norm{\nabla v}_{L^r(\R^3)}\Big(\frac{1}{\mu_*}\abs{\mu,\,\nu}_{\dot \cC^\alpha_{\pw,\Sigma_0}(\R^3)}+\mathfrak{P}_{\Sigma_0} \big|\Sigma_0\big|_{\dot \cC^{1+\alpha}}\Big)\Big(\abs{\mu}_{\dot \cC^{\alpha'}_{\pw,\Sigma_0}(\R^3)}+\ell^{-\alpha'}_{\Sigma_0}\norm{\mu-\widetilde\mu}_{L^\infty(\R^3)}\Big), 
\end{align*}
where, 
\begin{align*}
 \Psi_\eps &=\frac{1}{\mu_*(1-\eps)}\Bigg[\abs{\mu }_{\dot \cC^{\alpha}_{\pw,\Sigma_0}(\R^3)}+ \norm{\llbracket \mu\rrbracket}_{L^\infty(\Sigma_0)}\Big(\frac{1}{\mu_*}\abs{\mu,\,\nu}_{\dot \cC^\alpha_{\pw,\Sigma_0}(\R^3)}+\mathfrak{P}_{\Sigma_0} \big|\Sigma_0\big|_{\dot \cC^{1+\alpha}}+\ell^{-\alpha}_{\Sigma_0}\Big)\Bigg]\\
&+\frac{1}{\big(\mu_* (1-\eps)\big)^{\alpha/(\alpha-\alpha')}}\Bigg[\abs{ \mu }_{\dot \cC^{\alpha-\alpha'}_{\pw,\Sigma_0}(\R^3)}+\norm{\llbracket \mu\rrbracket}_{L^\infty(\Sigma_0)}\Big(\frac{1}{\mu_*}\abs{\mu,\,\nu}_{\dot \cC^\alpha_{\pw,\Sigma_0}(\R^3)}+\mathfrak{P}_{\Sigma_0} \big|\Sigma_0\big|_{\dot \cC^{1+\alpha}}+ \ell^{-\alpha+\alpha'}_{\Sigma_0}\Big)\Bigg]^{\frac{\alpha}{\alpha-\alpha'}}.\notag
\end{align*}
For the reader convenience, we shall denotes by $C_{**}$ a generic constant that depends on the piecewise regularity of $\mu$, $\lambda$, as well $\ell_{\Sigma_0}$,  $\mathfrak{P}_{\Sigma_0}$, $\big|\Sigma_0\big|_{\dot \cC^{1+\alpha}}$.  For instance the above estimate can be written as 
\begin{gather}\label{ap:eq45}
\abs{\nabla v}_{\dot \cC^\alpha_{\pw,\Sigma_0}(\R^3)}\leqslant \frac{C}{1-\varepsilon}\abs{\nabla u_G}_{\dot \cC^\alpha_{\pw,\Sigma_0}(\R^3)} +\frac{C_{**}}{1-\eps}\norm{\nabla v}_{L^r(\R^3)}+C\Psi_\eps\norm{\nabla v,\,\nabla \mathbb P v}_{L^\infty(\R^3)}.
\end{gather}

We will now  derive an estimate for the $L^\infty$-norm of $\nabla v, \nabla\mathbb P v$.

Observe that \eqref{ap:eq35} already provides a bound for the second term in the expression \eqref{ap:eq13}. We express the last term as 
\begin{align}
    -\nabla (-\Delta)^{-1}\nabla\bigg( \frac{1}{2\mu+\lambda}\big[(-\Delta)^{-1}\dvg \dvg, \mu-\widetilde\mu\big] \big(2\D v\big)\bigg)&=-\Big[\nabla (-\Delta)^{-1}\nabla,\, \frac{1}{2\mu+\lambda}\Big]\big[(-\Delta)^{-1}\dvg \dvg, \mu-\widetilde\mu\big] \big(2\D v\big)\notag\\
    &- \frac{1}{2\mu+\lambda}\nabla (-\Delta)^{-1}\nabla\bigg(\big[(-\Delta)^{-1}\dvg \dvg, \mu-\widetilde\mu\big] \big(2\D v\big)\bigg),\label{ap:eq52}
\end{align}
and applying \eqref{ap:eq10}-\eqref{ap:eq41}, we obtain 
\begin{align*}
    &\norm{\Big[\nabla (-\Delta)^{-1}\nabla,\, \frac{1}{2\mu+\lambda}\Big]\big[(-\Delta)^{-1}\dvg \dvg, \mu-\widetilde\mu\big] \big(2\D v\big)}_{L^\infty(\R^3)}\\
    &\leqslant \frac{C}{\mu_*}\Big(\frac{1}{\mu_*}\abs{\nu}_{\dot \cC^{\alpha'}_{\pw,\Sigma_0}(\R^3)}+\ell^{-\alpha'}_{\Sigma_0}\Big)\norm{\big[(-\Delta)^{-1}\dvg \dvg, \mu-\widetilde\mu\big] \big(2\D v\big)}_{L^r(\R^3)} \notag\\
    &+\frac{C}{\mu_*^2}\norm{\llbracket \nu\rrbracket}_{L^\infty(\Sigma_0)}\norm{\big[(-\Delta)^{-1}\dvg \dvg, \mu-\widetilde\mu\big] \big(2\D v\big)}_{\cC^{\alpha'}_{\pw,\Sigma_0}(\R^3)}\\
    &\leqslant C_{**} \Big(\norm{\nabla v}_{L^r(\R^3)}+\norm{\nabla v}_{L^{3/(\alpha-\alpha')}(\R^3)}\Big)+\frac{C}{\mu_*^2}\norm{\llbracket \nu\rrbracket}_{L^\infty(\Sigma_0)}\norm{\llbracket \mu\rrbracket}_{L^\infty(\Sigma_0)}\norm{\nabla v}_{\dot \cC^{\alpha'}_{\pw,\Sigma_0}(\R^3)}\\
    &+\frac{C}{\mu_*^2}\norm{\llbracket \nu\rrbracket}_{L^\infty(\Sigma_0)}\Big(\abs{\llbracket \mu \rrbracket }_{\dot \cC^{\alpha'}(\Sigma_0)}+ \big(\ell^{-\alpha'}_{\Sigma_0}+\mathfrak{P}_{\Sigma_0} \abs{\Sigma_0}_{\dot \cC^{1+\alpha'}}\big)\norm{\llbracket \mu\rrbracket}_{L^\infty(\Sigma_0)}\Big)\norm{\nabla v}_{L^\infty(\R^3)}.
\end{align*}
Using \eqref{ap:eq17}-\eqref{ap:eq10}-\eqref{ap:eq41} we can estimate the last term of \eqref{ap:eq52} as follows
\begin{align*}
    &\norm{\frac{1}{2\mu+\lambda}\nabla (-\Delta)^{-1}\nabla\bigg(\big[(-\Delta)^{-1}\dvg \dvg, \mu-\widetilde\mu\big] \big(2\D v\big)\bigg)}_{L^\infty(\R^3)}\\
    &\leqslant \frac{C}{\mu_*} \ell_{\Sigma_0}^{-\frac{3}{r}}\norm{\big[(-\Delta)^{-1}\dvg \dvg, \mu-\widetilde\mu\big] \big(2\D v\big)}_{L^r(\R^3)}
    +\frac{C}{\mu_*}\norm{\big[(-\Delta)^{-1}\dvg \dvg, \mu-\widetilde\mu\big] \big(2\D v\big)}_{\cC^{\alpha'}_{\pw,\Sigma_0}(\R^3)}\\
    &\leqslant C_{**} \Big(\norm{\nabla v}_{L^r(\R^3)}+\norm{\nabla v}_{L^{3/(\alpha-\alpha')}(\R^3)}\Big)+\frac{C}{\mu_*}\norm{\llbracket \mu\rrbracket}_{L^\infty(\Sigma_0)}\norm{\nabla v}_{\dot \cC^{\alpha'}_{\pw,\Sigma_0}(\R^3)}\\
    &+\frac{C}{\mu_*}\Big(\abs{\llbracket \mu \rrbracket }_{\dot \cC^{\alpha'}(\Sigma_0)}+ \big(\ell^{-\alpha'}_{\Sigma_0}+\mathfrak{P}_{\Sigma_0} \abs{\Sigma_0}_{\dot \cC^{1+\alpha'}}\big)\norm{\llbracket \mu\rrbracket}_{L^\infty(\Sigma_0)}\Big)\norm{\nabla v}_{L^\infty(\R^3)}.
\end{align*}
By combining all these computations, we obtain the following estimate (we freely use  $\frac{1}{\mu_*}\normb{\llbracket\nu\rrbracket}_{L^\infty(\Sigma_0)}\leqslant C$):
\begin{align}
    \norm{\nabla v}_{L^\infty(\R^3)}&\leqslant \norm{\nabla u_G}_{L^\infty(\R^3)}+ C_{**} \Big(\norm{\nabla v}_{L^r(\R^3)}+\norm{\nabla v}_{L^{3/(\alpha-\alpha')}(\R^3)}\Big)+\frac{C}{\mu_*}\norm{\llbracket \mu\rrbracket}_{L^\infty(\Sigma_0)}\norm{\nabla v}_{\dot \cC^{\alpha'}_{\pw,\Sigma_0}(\R^3)}\notag\\
    &+\frac{C}{\mu_*}\Big(\abs{\llbracket \mu \rrbracket }_{\dot \cC^{\alpha'}(\Sigma_0)}+ \big(\ell^{-\alpha'}_{\Sigma_0}+\mathfrak{P}_{\Sigma_0} \abs{\Sigma_0}_{\dot \cC^{1+\alpha'}}\big)\norm{\llbracket \mu\rrbracket}_{L^\infty(\Sigma_0)}\Big)\norm{\nabla v}_{L^\infty(\R^3)}.\notag
\end{align}
Applying the Leray projector $\mathbb P$ to \eqref{ap:eq5} and repeating the above computations, we observe that a similar estimate also holds  for $\nabla \mathbb P$. Next, \eqref{ap:eq48} and  Young's inequality yield
\begin{align}
    \norm{\nabla v,\, \nabla \mathbb P v}_{L^\infty(\R^3)}&\leqslant C\norm{\nabla u_G,\, \mathbb P \nabla u_G}_{L^\infty(\R^3)}+ C_{**} \Big(\norm{\nabla v}_{L^r(\R^3)}+\norm{\nabla v}_{L^{3/(\alpha-\alpha')}(\R^3)}\Big)+\Big(\frac{C}{\mu_*}\norm{\llbracket \mu\rrbracket}_{L^\infty(\Sigma_0)}\Big)^{\frac{\alpha}{\alpha'}}\norm{\nabla v}_{\dot \cC^{\alpha}_{\pw,\Sigma_0}(\R^3)}\notag\\
    &+\frac{C}{\mu_*}\Big(\abs{\llbracket \mu \rrbracket }_{\dot \cC^{\alpha'}(\Sigma_0)}+ \big(\ell^{-\alpha'}_{\Sigma_0}+\mathfrak{P}_{\Sigma_0} \abs{\Sigma_0}_{\dot \cC^{1+\alpha'}}\big)\norm{\llbracket \mu\rrbracket}_{L^\infty(\Sigma_0)}\Big)\norm{\nabla v}_{L^\infty(\R^3)},\notag
\end{align}
and  \eqref{ap:eq45} leads to  
\begin{align}
    \norm{\nabla v,\, \nabla \mathbb P v}_{L^\infty(\R^3)}&\leqslant  C_{**}(\eps) \Big(\norm{\nabla v}_{L^r(\R^3)}+\norm{\nabla v}_{L^{3/(\alpha-\alpha')}(\R^3)}\Big)\notag\\
    &+C\norm{\nabla u_G, \mathbb P \nabla u_G}_{L^\infty(\R^3)}+\frac{1}{1-\varepsilon}\Big(\frac{C}{\mu_*}\norm{\llbracket \mu\rrbracket}_{L^\infty(\Sigma_0)}\Big)^{\frac{\alpha}{\alpha'}}\abs{\nabla u_G}_{\dot \cC^\alpha_{\pw,\Sigma_0}(\R^3)}\notag\\
    &+\Bigg[\Psi_\eps\Big(\frac{C}{\mu_*}\norm{\llbracket \mu\rrbracket}_{L^\infty(\Sigma_0)}\Big)^{\frac{\alpha}{\alpha'}}+\frac{C}{\mu_*}\Big(\abs{\llbracket \mu \rrbracket }_{\dot \cC^{\alpha'}(\Sigma_0)}+ \big(\ell^{-\alpha'}_{\Sigma_0}+\mathfrak{P}_{\Sigma_0} \abs{\Sigma_0}_{\dot \cC^{1+\alpha'}}\big)\norm{\llbracket \mu\rrbracket}_{L^\infty(\Sigma_0)}\Big)\Bigg]\norm{\nabla v,\nabla\mathbb P v}_{L^\infty(\R^3)} \label{ap:eq40}.
\end{align}
Assuming that 
\begin{gather}\label{ap:eq47}
\Psi_\eps\Big(\frac{C}{\mu_*}\norm{\llbracket \mu\rrbracket}_{L^\infty(\Sigma_0)}\Big)^{\frac{\alpha}{\alpha'}}+\frac{C}{\mu_*}\Big(\abs{\llbracket \mu \rrbracket }_{\dot \cC^{\alpha'}(\Sigma_0)}+ \big(\ell^{-\alpha'}_{\Sigma_0}+\mathfrak{P}_{\Sigma_0} \abs{\Sigma_0}_{\dot \cC^{1+\alpha'}}\big)\norm{\llbracket \mu\rrbracket}_{L^\infty(\Sigma_0)}\Big) <\eps'<1,
\end{gather}
the last term can be absorbed in the LHS and interpolation and Young's inequalities imply: 
\begin{gather}\label{ap:eq46}
    \norm{\nabla v,\, \nabla \mathbb P v}_{L^\infty(\R^3)}\leqslant  C_{**}(\eps,\eps') \norm{\nabla v}_{L^2(\R^3)}
    +C(\eps,\eps')\Big(\norm{\nabla u_G, \mathbb P \nabla u_G}_{L^\infty(\R^3)}+\abs{\nabla u_G}_{\dot \cC^\alpha_{\pw,\Sigma_0}(\R^3)}\Big).
\end{gather}
Upon inserting this estimate into \eqref{ap:eq45}, we obtain 
\begin{gather}
    \norm{\nabla v,\, \nabla \mathbb P v}_{L^\infty(\R^3)}\leqslant  C_{**}(\eps,\eps') \norm{\nabla v}_{L^2(\R^3)}
    +C(\eps,\eps')\Big(\norm{\nabla u_G, \mathbb P \nabla u_G}_{L^\infty(\R^3)}+\abs{\nabla u_G}_{\dot \cC^\alpha_{\pw,\Sigma_0}(\R^3)}\Big).
\end{gather}

The above two estimate  with \(\dvg G = \dvg F - \rho \dot{u}\), will replace estimate (3.11) in \cite{zodji2023well}. We shall use \eqref{ap:eq31} 
to derive an estimate for the $L^p(\R^3)$-norm of $\nabla u$ and $\nabla \dot u$,  and hence close the energy estimates for $\dot{u}$ and $\ddot{u}$. Here, we only rely on \eqref{ap:eq32}--\eqref{ap:eq47}, which is significantly more favorable than the smallness condition (3.5) in \cite{zodji2023well}, in the sense that only the smallness of the viscosities jump is required, not that of the full fluctuation of the viscosities. The remainder of the proof proceeds identically.

\begin{proof}[Proof of Lemma \ref{ap:commutator}] 
    To prove Lemma \ref{ap:commutator}, we first observe that since $w\in \cC^\alpha_{\pw,\Sigma}(\R^3)$, there exists a $\cC^\alpha$-regular extension  $w^e$ of $w$ that satisfies the following properties:
\begin{align}
\norm{w^e}_{L^\infty(\R^3)}&\leqslant C\norm{w}_{L^\infty(\R^3)},\quad w^e=w \text{ on } D^c,\label{ap:eq11}\\
\norm{w- w^e}_{L^q(\R^3)}&\leqslant C \ell^{\frac{1}{q}}_{\Sigma}\norm{\llbracket w\rrbracket}_{L^q(\Sigma)}\text{ for all } q\in [1,\infty],\label{ap:eq5}\\
\abs{w-w^e}_{\dot \cC^\beta_{\pw,\Sigma}(\R^3)}&\leqslant  C\abs{\llbracket w\rrbracket}_{\dot \cC^{\beta}(\Sigma)}+ C\ell^{-\beta}_{\Sigma}\norm{\llbracket w\rrbracket}_{L^\infty(\Sigma)}, \quad \beta\in (0,\alpha],\label{ap:eq15}\\
\abs{w^e}_{\dot \cC^\beta(\R^3)}&\leqslant  C\abs{w}_{\dot \cC^{\beta}_{\pw,\Sigma}(\R^3)}+ C\ell^{-\beta}_{\Sigma}\norm{\llbracket w\rrbracket}_{L^\infty(\Sigma)}, \quad \beta\in (0,\alpha].\label{ap:eq9}
\end{align}
The above extension result is an updated version of Lemma A.2 in \cite{zodji2023well}, originally established in \cite[Lemma 5.2]{hoff2002dynamics} for the two-dimensional case. 

\paragraph{\textbf{Proof of \eqref{ap:eq6}}}
Our approach consists in substituting $w$  with $w^e+ (w-w^e)$, which yields 
\begin{equation}\label{ap:eq3}
\big[\mathcal{T}, w\big]g=\big[\mathcal{T}, w^e\big]g+\big[\mathcal{T}, (w-w^e)\big]g.
\end{equation}
Given that $w^e$ is $\cC^\alpha(\R^3)$-regular (in the whole space), the first term is now more \emph{regular} than the LHS. In particular, we have
\begin{align*}
\abs{\big[\mathcal{T}, w^e\big]g(x)}&=\Bigg|\int_{\R^3} K(x-y)\big(w^e(y)-w^e(x)\big) g(y)dy\Bigg|\\
                                     &\leqslant C\abs{w^e}_{\dot \cC^{\alpha'}(\R^3)} \big(\abs{\cdot}^{-3+\alpha'} \ast \abs{g}\big)(x), \text{ for all } \alpha'\in (0,\alpha].
\end{align*}
For $0<\alpha'< \frac{3}{p'}$, we can apply Hardy-Littlewood-Sobolev's inequality (see \cite[Theorem 1.7]{bahouri2011fourier}) to obtain: 
\begin{align}
\norm{\big[\mathcal{T}, w^e\big] g}_{L^p(\R^3)}&\leqslant C\abs{w^e}_{\dot \cC^{\alpha'}(\R^3)}\norm{g}_{L^r(\R^3)}, \text{ recall }  \frac{1}{r}=\frac{\alpha'}{3}+\frac{1}{p}<1\notag\\
&\leqslant C\big(\abs{w}_{\dot \cC^{\alpha'}_{\pw,\Sigma}(\R^3)}+ \ell^{-\alpha'}_{\Sigma}\norm{\llbracket w\rrbracket}_{L^\infty(\Sigma)}\big)\norm{g}_{L^r(\R^3)},\text{ recall } \eqref{ap:eq9}.  \label{ap:eq7}
\end{align}
To estimate the $L^p(\R^3)$-norm of the last term of \eqref{ap:eq3}, we rely on the boundedness of commutators between 
Calder\'on–Zygmund operators and $BMO(\R^3)$ functions on $L^p(\R^3)$:
\begin{align}
    \norm{\big[\mathcal{T}, (w-w^e)\big]g}_{L^p(\R^3)}&\leqslant C\norm{w-w^e}_{L^\infty(\R^3)}\norm{g}_{L^p(\R^3)}\notag\\
    &\leqslant C \norm{\llbracket w\rrbracket}_{L^\infty(\Sigma)}\norm{g}_{L^p(\R^3)}, \quad \text{recall }\eqref{ap:eq5}.\label{ap:eq8}
\end{align}
Estimate \eqref{ap:eq6} follows directly from \eqref{ap:eq7}–\eqref{ap:eq8}. We now proceed to the proof of \eqref{ap:eq10}.

\paragraph{\textbf{Proof of \eqref{ap:eq10}}}
We recall that the first term of the expression  \eqref{ap:eq3} reads:
\begin{align*}
&\big(\big[\mathcal{T}, w^e\big]g\big)(x)=\int_{\R^3} K(x-z)\big(w^e(z)-w^e(x)\big) g(z)dz\\
                                        &=\int_{B(x,1)} K(x-z)\big(w^e(z)-w^e(x)\big) g(z)dz+\int_{\R^3\setminus B(x,1)} K(x-z)\big(w^e(z)-w^e(x)\big) g(z)dz.
\end{align*}
To estimate the first term, we use the $\cC^{\alpha'}(\R^3)$ regularity of $w^e$ to mitigate the singularity of the kernel $K$:
\begin{align*}
\Bigg|\int_{B(x,1)} K(x-z)&\big(w^e(z)-w^e(x)\big) g(z)dz\Bigg|\leqslant  C\abs{w^e}_{\dot \cC^{\alpha'}(\R^3)}\int_{B(x,1)}\abs{x-z}^{-3+\alpha'}\abs{g(z)}dz\\
&\leqslant C\abs{w^e}_{\dot \cC^{\alpha'}(\R^3)}\norm{g}_{L^r(\R^3)}\Bigg(\int_{0}^1\tau ^{(-3+\alpha')r'+2})d\tau\Bigg)^{\frac{1}{r'}}\\
&\leqslant C\abs{w^e}_{\dot \cC^{\alpha'}(\R^3)}\norm{g}_{L^r(\R^3)}, \quad \text{recall that  }\frac{1}{r}<\frac{\alpha'}{3},\\
&\leqslant C\big(\abs{w}_{\dot \cC^{\alpha'}_{\pw,\Sigma}(\R^3)}+ \ell^{-\alpha'}_{\Sigma}\norm{\llbracket w\rrbracket}_{L^\infty(\Sigma)}\big)\norm{g}_{L^r(\R^3)} \text{ recall }\eqref{ap:eq9}.
\end{align*}
 Using the decay of the kernel $K$, we derive:
\begin{align*}
    \bigg|\int_{\R^3\setminus B(x,1)} K(x-z)&\big(w^e(z)-w^e(x)\big) g(z)dz\bigg|\leqslant C\norm{w^e}_{L^\infty(\R^3)}\int_{\R^3\setminus B(x,1)} \abs{x-z}^{-3}\abs{g(z)}dz\\
    &\leqslant C\norm{w^e}_{L^\infty(\R^3)}\norm{g}_{L^p(\R^3)}\Bigg(\int_{1}^\infty\tau ^{-3p'+2}d\tau\Bigg)^{\frac{1}{p'}}\\
    &\leqslant C\norm{w^e}_{L^\infty(\R^3)}\norm{g}_{L^p(\R^3)}, \quad\text{recall that } p'>1,\\
    &\leqslant C\norm{w}_{L^\infty(\R^3)}\norm{g}_{L^p(\R^3)},\quad\text{recall } \eqref{ap:eq11}.
\end{align*}
It follows that 
\begin{equation}\label{ap:eq12}
    \norm{\big[\mathcal{T}, w^e\big]g}_{L^\infty(\R^3)}\leqslant C\Big(\abs{w}_{\dot \cC^{\alpha'}_{\pw,\Sigma}(\R^3)}+ \ell^{-\alpha'}_{\Sigma}\norm{\llbracket w\rrbracket}_{L^\infty(\Sigma)}\Big)\norm{g}_{L^r(\R^3)}+C\norm{w}_{L^\infty(\R^3)}\norm{g}_{L^p(\R^3)}.
\end{equation}
To estimate the $L^\infty(\R^3)$-norm of the last term of \eqref{ap:eq3}, we use the fact that $(w-w^e)\mathbb{1}_{D^c}=0$
to express: 
\begin{itemize}
    \item for $x\in D$,
\begin{align}
\big[\mathcal{T}, (w-w^e)\big]g(x)&= \big[\mathcal{T}, (w-w^e)\big]\big(g \mathbb{1}_{D}\big)(x)-(w-w^e)(x)\mathcal{T}\big(g \mathbb{1}_{D^c}\big)(x).\label{ap:eq14}
\end{align}
\item for $x\in D^c$,
\begin{equation}\label{ap:eq16}
\big[\mathcal{T}, (w-w^e)\big]g(x)=\mathcal{T}\big((w-w^e)g \mathbb{1}_{D}\big)(x).
\end{equation}
\end{itemize}
Since $(w-w^e)\in \cC^{\alpha'}(\overline{D})$, we follow the computations leading to \eqref{ap:eq12} to estimate the first term in the expression \eqref{ap:eq14}:
\begin{align*}
&\norm{\big[\mathcal{T}, (w-w^e)\big]\big(g \mathbb{1}_{D}\big)}_{L^\infty(D)}\\
&\leqslant C\Big(\abs{w-w^e}_{\dot \cC^{\alpha'}(\overline{D})}+ \ell^{-\alpha'}_{\Sigma}\norm{\llbracket w-w^e\rrbracket}_{L^\infty(\Sigma)}\Big)\norm{g}_{L^r(D)}+C\norm{w-w^e}_{L^\infty(D)}\norm{g}_{L^p(D)}\\
&\leqslant C\Big(\abs{\llbracket w\rrbracket}_{\dot \cC^{\alpha'}(\Sigma)}+\ell^{-\alpha'}_{\Sigma}\norm{\llbracket w\rrbracket }_{L^\infty(\Sigma)}\Big)\norm{g}_{L^r(\R^3)} +C\norm{\llbracket w\rrbracket }_{L^\infty(\Sigma)}\norm{g}_{L^p(\R^3)},\text{ recall } \eqref{ap:eq5}-\eqref{ap:eq15}.
\end{align*}
To estimate the last term of \eqref{ap:eq14} we express:
\begin{align*}
\mathcal{T}\big(g \mathbb{1}_{D^c}\big)(x)&=\int_{D^c\cap B(\check{x},\ell_{\Sigma})} K(x-y) \big(g(y)-g(\check{x})\big)dy+g(\check{x})\int_{D^c\cap B(\check{x},\ell_{\Sigma})} K(x-y)dy\\
&+\int_{D^c\cap B^c(\check{x},\ell_{\Sigma})} K(x-y) g(y)dy,
\end{align*}
where $\check{x}\in \partial D$ such that $d(x, D^c)=\abs{x-\check{x}}$. It follows that 
\[
\Bigg|\int_{D^c\cap B(\check{x},\ell_{\Sigma})} K(x-y) \big(g(y)-g(\check{x})\big)dy\Bigg|\leqslant C \abs{g}_{\dot \cC^{\alpha'}_{\pw,\Sigma}(\R^3)},
\]
and 
\[
\Bigg|\int_{D^c\cap B^c(\check{x},\ell_{\Sigma})} K(x-y) g(y)dy\Bigg|\leqslant C \ell_{\Sigma}^{-\frac{3}{p}}\norm{g}_{L^p(\R^3)}.
\]
Using the extra-cancellation of even-order Riesz kernels (\cite{Bertozzi1993,mateu2009extra}), we bound the remaining term as
\[
\Bigg|g(\check{x})\int_{D^c\cap B(\check{x},\ell_{\Sigma})} K(x-y)dy\Bigg|\leqslant  C \norm{g}_{L^\infty(\R^3)}.
\]
We thus obtain the following bound for the last term of \eqref{ap:eq14}:
\begin{equation}\label{ap:eq51}
\abs{(w-w^e)(x)\mathcal{T}\big(g \mathbb{1}_{D^c}\big)(x)}\leqslant C\norm{\llbracket w\rrbracket}_{L^\infty(\Sigma)}\Big(\norm{g}_{\cC^{\alpha'}_{\pw,\Sigma}(\R^3)}+\ell_{\Sigma}^{-\frac{3}{p}}\norm{g}_{L^p(\R^3)}\Big).
\end{equation}
For \eqref{ap:eq16}, we consider $\check{x}\in \partial D$ such that $d(x,D)=\abs{x-\check{x}}$ and we express:
\begin{align*}
    \big[\mathcal{T}, (w-w^e)\big]g(x)&=\int_{D\cap B(\check{x},\ell_{\Sigma})} K(x-y)\big( (w-w^e)(y)- (w-w^e)(\check{x})g(y)dy\\
    &+(w-w^e)(\check{x})\int_{D\cap B(\check{x},\ell_{\Sigma})} K(x-y)g(y)dy\\
    &+\int_{D \cap B^c(\check{x},\ell_{\Sigma})} K(x-y)(w-w^e)(y)g(y)dy.
\end{align*}
Estimate \eqref{ap:eq51} also applies to the last two terms above, whereas the first one can be bounded as follows:
\[
\Bigg|\int_{D \cap B(\check{x},\ell_{\Sigma})} K(x-y)\big( (w-w^e)(y)- (w-w^e)(\check{x})g(y)dy\Bigg|\leqslant C\Big(\abs{\llbracket w\rrbracket}_{\dot \cC^{\alpha'}(\Sigma)}+ \ell^{-\alpha'}_{\Sigma}\norm{\llbracket w\rrbracket}_{L^\infty(\Sigma)}\Big)\norm{g}_{L^r(\R^3)}.
\]

In sum, we obtain the following bound for the last term of \eqref{ap:eq3} (recall that $\frac{3}{r}<\alpha'$):
\begin{align}
    \norm{\big[\mathcal{T}, (w-w^e)\big]g}_{L^\infty(\R^3)}&\leqslant C\Big(\abs{\llbracket w\rrbracket}_{\dot \cC^{\alpha'}(\Sigma)}+\ell^{-\alpha'}_{\Sigma}\norm{\llbracket w\rrbracket }_{L^\infty(\Sigma)}\Big)\norm{g}_{L^r(\R^3)}\notag\\
    &+C\norm{\llbracket w\rrbracket}_{L^\infty(\Sigma)}\Big(\norm{g}_{\cC^{\alpha'}_{\pw,\Sigma}(\R^3)}+\ell_{\Sigma}^{-\alpha'}\norm{g}_{L^p(\R^3)}\Big).\label{ap:eq18}
\end{align}
Finally, \eqref{ap:eq10} follows directly from \eqref{ap:eq12}–\eqref{ap:eq18}. We now move on to the proof of \eqref{ap:eq19}.

\paragraph{\textbf{Proof of \eqref{ap:eq19}}}
    Let $x,y$ be two points located on the same side of $\Sigma$ and let us assume  $\delta=\abs{x-y}\in (0, \ell_\Sigma/2)$. Without lost of generality, assume that $x,y\in \overline{D}$. 
    By taking the difference of the first term in expression \eqref{ap:eq3}, we obtain:
\begin{align}
\big(\big[\mathcal{T}, w^e\big]g\big)(x)&-\big(\big[\mathcal{T}, w^e\big]g\big)(y)= \int_{B(y,2\delta)} K(x-z)\big(w^e(z)-w^e(x)\big) g(z)dz\notag\\
&-\int_{B(y,2\delta)} K(y-z)\big(w^e(z)-w^e(y)\big) g(z)dz\notag\\
&+\big(w^e(y)-w^e(x)\big)\int_{\R^3\setminus B(y,2\delta)} K(x-z) g(z)dz\notag\\
&+\int_{\R^3\setminus B(y, 2\delta)} \big(K(x-z) -K(y-z)\big)\big(w^e(z)-w^e(y)\big) g(z)dz.\label{ap:eq20}
\end{align}

One derives easily the following bound for the first two terms of the expression above: 
\begin{align*}
\Bigg|\int_{B(y,2\delta)} K(x-z)\big(w^e(z)-w^e(x)\big) g(z)dz\Bigg|&+\Bigg| \int_{B(y,2\delta)} K(y-z)\big(w^e(z)-w^e(y)\big) g(z)dz\Bigg|\\
&\leqslant C\abs{w^e}_{\dot \cC^\beta(\R^3)}\norm{g}_{L^\infty(\R^3)}\abs{x-y}^\beta.
\end{align*}

To estimate the third term on the RHS of \eqref{ap:eq20}, we express
\begin{align}
\int_{\R^3\setminus B(y,2\delta)} K(x-z) g(z)dz&=\mathcal{T} g(x)-\int_{D\cap B(y,2\delta)} K(x-z) g(z)dz-\int_{D^c\cap B(y,2\delta)} K(x-z) g(z)dz,\label{ap:eq49}
\end{align}
and we estimate: 
\begin{align}
\Bigg|\int_{D\cap B(y,2\delta)} K(x-z) g(z)dz\Bigg|&\leqslant \Bigg|\int_{D\cap B(y,2\delta)} K(x-z) \big(g(z)-g(x)\big)dz\Bigg|+\abs{g(x)}\Bigg|\int_{D\cap B(y,2\delta)} K(x-z)dz\Bigg|\notag\\
&\leqslant C \abs{g}_{\dot \cC^{\alpha'}_{\pw,\Sigma}(\R^3)}\abs{x-y}^{\alpha'}+ C\norm{g}_{L^\infty(\R^3)},\label{ap:eq50}
\end{align}
where we make use of the extra cancellation of even-order Riesz transform (see \cite{Bertozzi1993,mateu2009extra}) to bound the last term of the RHS of the first line. For the last term of \eqref{ap:eq49}, we express 
\[
\int_{D^c\cap B(y,2\delta)} K(x-z) g(z)dz=\int_{D^c\cap B(y,2\delta)} K(x-z) \big(g(z)-g(\check{x})\big)dz+g(\check{x})\int_{D^c\cap B(y,2\delta)} K(x-z)dz,
\]
where $\check{x}\in \partial D$ is such that $d(x,\partial D)=\abs{x-\check{x}}$. We observe that that the estimate \eqref{ap:eq50} also holds for the last term of \eqref{ap:eq49} and we obtain:
\begin{align*}
\Bigg|\big(w^e(y)-w^e(x)\big)&\int_{\R^3\setminus B(y,2\delta)} K(x-z) g(z)dz\Bigg|\\
&\leqslant C\Big(\abs{w^e}_{\dot \cC^\beta(\R^3)}\norm{g,\,\mathcal{T} g}_{L^\infty(\R^3)}+\abs{w^e}_{\dot \cC^{\beta-\alpha'}(\R^3)}\abs{g}_{\dot \cC^{\alpha'}_{\pw,\Sigma}(\R^3)}\Bigg) \abs{x-y}^\beta.
\end{align*}

To estimate the last term of \eqref{ap:eq20}, we observe that 
\begin{align*}
\abs{K(x-z)-K(y-z)}&\leqslant\abs{x-y}\int_0^1\abs{\nabla K\big(\tau x+ (1-\tau) y-z\big)}d\tau \\
&\leqslant C \int_0^1\frac{\abs{x-y}}{\abs{\tau x+ (1-\tau) y-z}^4}d\tau.
\end{align*}
Hence, for all $z\notin B(y, 2\delta)$, and $\tau\in (0,1)$, we have:
\begin{align*}
\abs{\tau x+ (1-\tau) y-z}&\geqslant \abs{y-z}-\tau \abs{x-y}\\
                          &\geqslant\frac{1}{2} \abs{y-z} +\Big(\frac{1}{2} \abs{y-z}-\tau \delta\Big)\\
                          &\geqslant \frac{1}{2} \abs{y-z} +\Big(\frac{1}{2} (2\delta)-\tau \delta\Big)\\
                          &\geqslant \frac{1}{2} \abs{y-z} .
\end{align*}
In consequence,  we have:
\begin{align*}
    \Bigg|\int_{\R^3\setminus B(y, 2\delta)} &\big(K(x-z) -K(y-z)\big)\big(w^e(z)-w^e(y)\big) g(z)dz\Bigg|\\
    &\leqslant C\abs{w^e}_{\dot \cC^\beta(\R^3)} \norm{g}_{L^\infty(\R^3)}\abs{x-y}\int_{\R^3\setminus B(y,2\delta)}\abs{y-z}^{\beta-4}dz\\
    &\leqslant C\abs{w^e}_{\dot \cC^\beta(\R^3)}\norm{g}_{L^\infty(\R^3)} \abs{x-y}^\beta.
\end{align*}
We finally obtain:
\begin{align}
    \abs{\big[\mathcal{T}, w^e\big]g}_{\dot \cC^\beta_{\pw,\Sigma}(\R^3)}&\leqslant C\big(\abs{w^e}_{\dot \cC^\beta(\R^3)}\norm{g,\mathcal{T} g}_{L^\infty(\R^3)}+\abs{w^e}_{\dot \cC^{\beta-\alpha'}(\R^3)}\abs{g}_{\dot \cC^{\alpha'}_{\pw,\Sigma}(\R^3)}\big)\notag\\
&\leqslant C\Big(\abs{w}_{\dot \cC^{\beta}_{\pw,\Sigma}(\R^3)}+ \ell^{-\beta}_{\Sigma}\norm{\llbracket w\rrbracket}_{L^\infty(\Sigma)}\Big)\norm{g,\mathcal{T} g}_{L^\infty(\R^3)}\notag\\
&+C\Big(\abs{w}_{\dot \cC^{\beta-\alpha'}_{\pw,\Sigma}(\R^3)}+ \ell^{-\beta+\alpha'}_{\Sigma}\norm{\llbracket w\rrbracket}_{L^\infty(\Sigma)}\Big)\abs{g}_{\dot \cC^{\alpha'}_{\pw,\Sigma}(\R^3)},\label{ap:eq22}
\end{align}
where we make use of \eqref{ap:eq9}. 

Using \eqref{ap:eq14} we express 
    \begin{align}
        \big[\mathcal{T}, (w-w^e)\big]g(x)&-\big[\mathcal{T}, (w-w^e)\big]g(y)\notag\\
        &=
    \int_{D\cap B(x, \delta)} K(x-z)\big( (w-w^e)(z)-(w-w^e)(x)\big) g(z)dz\notag\\
    &-\int_{D\cap B(x, \delta)}K(y-z)\big( (w-w^e)(z)-(w-w^e)(y)\big) g(z)dz\notag\\
    &+\int_{D\cap B^c(x, \delta)} \big( K(x-z)-K(y-z)\big)\big((w-w^e)(z)-(w-w^e)(x)\big) g(z)dz\notag\\
    &+\big((w-w^e)(y)-(w-w^e)(x)\big)\Bigg(\int_{D\cap B^c(x,\delta)} K(y-z)g(z)dz+ \int_{D^c} K(y-z)g(z)dz\Bigg)\notag\\ 
    &-(w-w^e)(x)\int_{D^c} \big(K(x-z)-K(y-z)\big)g(z)dz.\label{ap:eq21}
    \end{align}
As in the previous step, we estimate the first three terms on the RHS as follows:
\begin{align*}
&\Bigg|\int_{D\cap B(x, \delta)} K(x-z)\big( (w-w^e)(z)-(w-w^e)(x)\big) g(z)dz\Bigg|+\Bigg|\int_{D\cap B(x, \delta)}K(y-z)\big( (w-w^e)(z)-(w-w^e)(y)\big) g(z)dz\Bigg|\\
&\Bigg|\int_{D\cap B^c(x, \delta)} \big( K(x-z)-K(y-z)\big)\big((w-w^e)(z)-(w-w^e)(x)\big) g(z)dz\Bigg|\leqslant C\abs{w-w^e}_{\dot \cC^\beta_{\pw,\Sigma}(\R^3)}\norm{g}_{L^\infty(\R^3)}\abs{x-y}^{\beta}.
\end{align*}
To bound the last term, we express: 
\begin{align*}
    &\int_{D^c} \big(K(x-z)-K(y-z)\big)g(z)dz\\
    &=\int_{D^c\cap B(\check x,2\delta)}K(x-z)\big(g(z)-g(\check{x})\big)dz-\int_{D^c\cap B(\check x,2\delta)} K(y-z)\big(g(z)-g(\check{y})\big)dz\\
    &+\big(g(\check{y})-g(\check{x})\big)\int_{D\cap B(\check x,2\delta)}K(y-z)dz+\int_{D^c\cap B^c (\check{x}, 2\delta)} \big(K(x-z)-K(y-z)\big)\big(g(z)-g(\check{x})\big)dz\\
    &-g(\check{x})\int_{D} \big(K(x-z)-K(y-z)\big)dz. 
\end{align*}
Following the preceding step, all terms on the RHS, except for the last, are bounded by 
\[
C\abs{g}_{\dot \cC^\beta(D^c)}\abs{x-y}^{\beta}.
\]
Next, we apply Theorem 2.3 of \cite{gancedo2022quantitative} to estimate  the last term as 
\[
\bigg|g(\check{x})\int_{D} \big(K(x-z)-K(y-z)\big)dz\Bigg|\leqslant C\norm{g}_{L^\infty(\R^3)}\mathfrak{P}_{\Sigma} \abs{\Sigma}_{\dot \cC^{1+\beta}}\abs{x-y}^\beta,
\]
where $\mathfrak{P}_\Sigma$ is a polynomial in  $\abs{\Sigma}$, $\norm{\Sigma}_{\text{Lip}}$ and $\abs{\Sigma}_{\text{inf}}^{-1}$.
For the remaining term of \eqref{ap:eq21}, we express 
\[
\int_{D\cap B^c(x,\delta)} K(y-z)g(z)dz+ \int_{D^c} K(y-z)g(z)dz=\mathcal{T} g-\int_{D\cap B(x,\delta)} K(y-z)g(z)dz,
\]
and we estimate 
\[
\Bigg|\int_{D\cap B^c(x,\delta)} K(y-z)g(z)dz+ \int_{D^c} K(y-z)g(z)dz\Bigg|
\leqslant C\Big(\norm{g,\,\mathcal{T} g}_{L^\infty(\R^3)}+ \abs{g}_{\dot \cC_{\pw,\Sigma}^\beta(\R^3)}\abs{x-y}^{\beta}\Big).
\]
In sum, we obtain:
\begin{align}
    \abs{\big[\mathcal{T}, (w-w^e)\big]g}_{\dot \cC^\beta_{\pw,\Sigma}(\R^3)}&\leqslant C\norm{\llbracket w\rrbracket}_{L^\infty(\Sigma)}\abs{g}_{\dot \cC^\beta_{\pw,\Sigma}(\R^3)}\notag\\
    &+C\Big(\abs{\llbracket w\rrbracket}_{\dot \cC^{\beta}(\Sigma)}+ \big(\ell^{-\beta}_{\Sigma}+\mathfrak{P}_{\Sigma} \abs{\Sigma}_{\dot \cC^{1+\beta}}\big)\norm{\llbracket w\rrbracket}_{L^\infty(\Sigma)}\Big)\norm{\mathcal{T} g, g}_{L^\infty(\R^3)}.\label{ap:eq24}
\end{align}
Estimate \eqref{ap:eq19} simply follows from \eqref{ap:eq22}-\eqref{ap:eq24}. This completes the proof of Lemma \ref{ap:commutator}. 

Let us observe that following the same computations as above, one obtains ($r\in [1,\infty)$ and $\alpha'\in (0,\alpha]$): 
\begin{equation}\label{ap:eq17}
            \norm{\mathcal{T} g}_{L^\infty(\R^3)}\leqslant C\Big(\ell_\Sigma^{-\frac{3}{r}}\norm{g}_{L^r(\R^3)}+\norm{g}_{\cC^{\alpha'}_{\pw,\Sigma}(\R^3)}\Big),
\end{equation}
and
\begin{equation}\label{ap:eq23}
            \norm{\mathcal{T} g}_{\dot \cC^\alpha_{\pw,\Sigma}(\R^3)}\leqslant C \Big(\abs{g}_{\dot \cC^\alpha_{pw,\Sigma}(\R^3)}
            +\norm{g}_{L^\infty(\R^3)}\mathfrak{P}_{\Sigma} \abs{\Sigma}_{\dot \cC^{1+\alpha}}\Big).
\end{equation}
Also, we may notice that in the case $\beta\in (0, \alpha)$, the third term on the RHS of \eqref{ap:eq20} can be estimated as 
\begin{align*}
\Bigg|\big(w^e(y)-w^e(x)\big)\int_{\R^3\setminus B(y,2\delta)} K(x-z) g(z)dz\Bigg| &\leqslant C\abs{w^e}_{\dot\cC^\alpha(\R^3)}\abs{x-y}^{\alpha}\int_{\abs{x-z}\geqslant 2 \delta}\abs{x-z}^{-3}\abs{g(z)}dz\\
&\leqslant C\abs{w^e}_{\dot\cC^\alpha(\R^3)}\norm{g}_{L^{3/(\alpha-\beta)}(\R^3)}\abs{x-y}^{\beta} 
\end{align*}
leading to 
\begin{align}\label{ap:eq41}
    \abs{\big[\mathcal{T}, w\big]g}_{\dot \cC^\beta_{\pw,\Sigma}(\R^3)}&\leqslant C\norm{\llbracket w\rrbracket}_{L^\infty(\Sigma)}\abs{g}_{\dot \cC^\beta_{\pw,\Sigma}(\R^3)}+C\Big(\abs{w}_{\dot\cC^\alpha(\R^3)}+ \ell_{\Sigma}^{-\alpha}\norm{\llbracket w\rrbracket}_{L^\infty(\Sigma)}\Big)\norm{g}_{L^{3/(\alpha-\beta)}(\R^3)}\notag\\
    &+C\Big(\abs{\llbracket w\rrbracket }_{\dot \cC^{\beta}(\Sigma)}+ \big(\ell^{-\beta}_{\Sigma}+\mathfrak{P}_{\Sigma} \abs{\Sigma}_{\dot \cC^{1+\beta}}\big)\norm{\llbracket w\rrbracket}_{L^\infty(\Sigma)}\Big)\norm{\mathcal{T} g, g}_{L^\infty(\R^3)}.
\end{align}

\end{proof}
\bibliographystyle{acm}
\bibliography{BiblioThese}
\end{document}